\documentclass[10pt, leqno]{article}
\usepackage{amsmath,amsfonts,amsthm,amscd,amssymb}
\numberwithin{equation}{section}
\theoremstyle{plain}
\newtheorem{theo}{Theorem}[section]
\newtheorem{prop}[theo]{Proposition}
\newtheorem{coro}[theo]{Corollary} 
\newtheorem{lemm}[theo]{Lemma}
\theoremstyle{definition}
\newtheorem{defi}[theo]{Definition}
\newtheorem{rema}[theo]{Remark}
\newtheorem{theo-defi}[theo]{Theorem-Definition}
\newtheorem{prop-defi}[theo]{Proposition-Definition}
\newtheorem{rema-defi}[theo]{Remark-Definition}
\newtheorem{exem-defi} [theo]{Example-Definition}

\newtheorem{exem}[theo]{Example}
\newtheorem{conj}[theo]{Conjecture}
\newtheorem{prob}[theo]{Problem}

\def \bet{\beta}
\def \bul{\bullet}
\def \col{\colon}
\def \Del{\Delta}

\def \eps{\epsilon}
\def \Gam{\Gamma}

\def \kap{\kappa}
\def \Lam{\Lambda}
\def \lam{\lambda}

\def \Lo{\Longrightarrow}
\def \lo{\longrightarrow}
\def \lom{\longmapsto}
\def \mab{\mathbb}

\def \ol{\overline}
\def \os{\overset}
\def \parno{\par\noindent}

\def \sig{\sigma}

\def \sus{\subset}

\def \ul{\underline}
\def \us{\underset}

\def \vp{\varpi}

\def \vil{\varinjlim}
\def \vpl{\varprojlim}

\def \wh{\widehat}
\def \wt{\widetilde}
\bigskip

\newcommand{\getsfrom}{\ensuremath{
\longleftarrow\kern-.50em\lower.0ex\hbox%
{$\shortmid\,$}}}

\begin{document}

\title{The $l$-adic bifiltered El Zein-Steenbrink-Zucker complex
of a proper SNCL scheme with a relative SNCD}
\author{Yukiyoshi Nakkajima 
 \date{}\thanks{
 \endgraf}}
\maketitle

\bigskip
\parno
{\bf Abstract.---}
The aim of this paper is to give the log $l$-adic 
relative monodromy-weight conjecture and 
to prove that this conjecture is true in certain cases. 
This conjecture is a generalization of the 
famous $l$-adic monodromy-weight conjecture due to P.~Deligne and K.~Kato. 
To give our conjecture, we use our theory of the derived category of bifiltered complexes 
which is a generalization of theory of the derived category of bifiltered complexes in 
\cite{ilc}, \cite{dh2}, \cite{dh3} and the derived category of filtered complexes in 
\cite{blec} (and \cite{nh2}).

$${\bf Contents}$$
\parno
\S\ref{sec:intro}. Introduction
\bigskip
\parno 
Part I. Derived categories of bifiltered complexes
\medskip
\parno 
\S\ref{sec:nfil}. The definition of the derived category of bifiltered complexes 
\parno 
\S\ref{sec:sti}.  Strictly injective resolutions and $Rf_*$
\parno 
\S\ref{sfr}. Strictly flat resolutions and $Lf^*$
\parno 
\S\ref{sec:rhoml}. ${\rm RHom}^{\bul}$
\parno 
\S\ref{sec:otl}. $\otimes^L_{\cal A}$
\parno 
\S\ref{sec:comp}.  Complements
\parno
\S\ref{sec:rksfdc}. A remark on bifiltered derived categories 
\bigskip
\parno 
Part II.  $l$-adic relative monodromy-weight conjecture
\medskip
\parno 
\S\ref{sec:lbc}. 
$l$-adic bifiltered El Zein-Steenbrink-Zucker complex
\parno
\S\ref{sec:lwss}. 
$l$-adic weight spectral sequences 
\parno  
\S\ref{sec:confu}. Contravariant functoriality 
\parno
\S\ref{sec:fpl}. Fundamental properties of 
the $l$-adic weight filtrations and 
the $l$-adic weight spectral sequences 
\parno
\S\ref{sec:rmwc}. 
Log $l$-adic relative monodromy-weight conjecture 
\parno  
\S\ref{sec:ssf}. Strict semistable family 

\medskip 
\parno 
{\bf Appendix}
\medskip 
\parno 
\S\ref{sec:lbddmlif}. Edge morphisms between 
the $E_1$-terms of $l$-adic weight spectral sequences 
\bigskip
\parno 

\smallskip
\parno
References

\section{Introduction}\label{sec:intro} 
Let $V$ be an object of an abelian category 
with a finite increasing filtration $W$  and 
let $N \col (V,W) \lo (V,W)$ be a nilpotent filtered endomorphism.   
In \cite{dw2} Deligne has proved that 
there exists at most one monodromy filtration 
relative to $W$. That is, there exists at most
one finite increasing filtration $M$ on $V$ 
such that $N(M_kV) \subset M_{k-2}V$ 
$(\forall k\in {\mab Z})$ and 
such that $N^e\col V\lo V$ induces an isomorphism 
$$N^e \col {\rm gr}^M_{k+e}{\rm gr}^W_kV 
\os{\sim}{\lo} {\rm gr}^M_{k-e}{\rm gr}^W_kV$$ 
for any $e\in {\mab Z}_{\geq 1}$.  
Let $U$ be the complement of a smooth divisor $D$ 
defined by an equation $t=0$ 
in a smooth scheme $X$ over a finite field ${\mab F}_q$. 
Let ${\cal F}$ be a smooth $\ol{\mab Q}_l$-sheaf on $U$ tamely ramified along $D$.  
Let ${\cal F}[D]$ be the restriction to $D$ 
of a smooth extension $\ol{\cal F}$ on $X[t^{1/n}]$ $((\exists n,q)=1)$ 
of the pull-back of ${\cal F}$ to $X[t^{1/n}]$ (Abyankar's lemma). 
Let $\iota \col \ol{\mab Q}_l \os{\sim}{\lo} {\mab C}$ be 
an isomorphism of fields. 
Assume that ${\cal F}$ has a finite increasing filtration 
$W$ of smooth $\ol{\mab Q}_l$-subsheaves such that 
${\rm gr}_k^W{\cal F}$ is exactly $\iota$-pure of weight $k$. 
In [loc.~cit.] Deligne has proved that there exists 
the monodromy filtration $M$ on ${\cal F}[D]$ 
relative to $W$ and that ${\rm gr}_k^M{\cal F}[D]$ is 
exactly $\iota$-pure of weight $k$. 
Pursuing an $\infty$-adic analogue of this, he has defined 
the variation of mixed Hodge structures $(V,W,F)$ 
over the punctured unit disk $\Del^*$ over the complex number field
and he has posed a problem defining a class 
$(V,W,F)$ with the logarithm $N$ 
of a unipotent monodromy of $V$ such 
that there exists a monodromy filtration 
$M$ on $V$ relative to $W$ 
(See \cite{stz}, \cite{kas} and \cite{asm} 
for a more expanded statement: 
the admissibility of variations of mixed Hodge structures.).  
\par 
Let $\os{\circ}{\cal X}$ be a projective strict semistable family 
with horizontal SNCD(=simple normal crossing divisor) $\os{\circ}{\cal D}$ over 
the unit disk $\os{\circ}{\Del}$. 
In \cite{stz} (resp.~\cite{ezth}) 
Steenbrink and Zucker (resp.~El Zein)
have given an answer for Deligne's problem 
in this good geometric case.  
Let $\os{\circ}{X}$ and $\os{\circ}{D}$ be the special fibers of $\os{\circ}{\cal X}$ 
and $\os{\circ}{\cal D}$, respectively. 
To construct the monodromy filtration on the 
suitable log cohomology of  
$(\os{\circ}{X},\os{\circ}{D})$ with coefficient ${\mab Q}$ 
relative to the weight filtration arising from $\os{\circ}{D}$, 
they have used the weight filtration arising from 
both $\os{\circ}{X}$ and $\os{\circ}{D}$. 
One can  prove that the weight filtration 
has two characterizing properties of the relative monodromy filtration by using 
a certain double complex and using M.~Saito's result 
which tells us that the monodromy filtration coincides with  
the weight filtration in the case $D=\emptyset$ (\cite{sm}). 
Consequently the relative monodromy filtration exists in this geometric case. 
This main result in  \cite{stz} and \cite{ezth} is a generalization of 
that in \cite{st} (whose complete proof has been given by M.~Saito (\cite{sm})). 
Influenced by their work, 
M.~Saito has constructed the category of mixed Hodge modules 
(cf.~\cite[Introduction]{smmh}).   
\par 
In this paper we give an $l$-adic version of 
(a generalization of) El Zein-Steenbrink-Zucker's work 
in the equal characteristic case $p>0$ (without no conjecture) 
and in the mixed characteristics case  
under the assumption of the log $l$-adic monodromy-weight conjecture 
in the mixed characteristics case. 
In order to give the $l$-adic version, 
we use theory of the derived categories of bifiltered complexes in this paper, 
theory of the log geometry of Fontaine-Illusie-Kato in \cite{klog1} and 
theory of log \'{e}tale cohomologies in \cite{nale}. 
Our theory of the derived categories of bifiltered complexes  
can be considered as a generalization of the theory of filtered derived categories in 
\cite{blec} (and \cite{nh2}) and the theory of bifiltered derived categories in \cite{dh3}. 
\par 
In order to give relations between our work and preceding results in the $l$-adic geometric case, 
let us recall the case where $W$ is the trivial filtration, i.e., $W_{k-1}V=0$ and $W_kV=V$ for 
some $k\in {\mab Z}$ in the $l$-adic geometric case. 
\par  
Let $\os{\circ}{S}$ be the spectrum of 
a henselian discrete valuation ring ${\cal V}$. 
Let $\ol{K}$ be a separable closure of $K$ and 
let $\ol{\cal V}$ be the integral closure of ${\cal V}$ in $K$. 
Set $\os{\circ}{\ol{S}}:={\rm Spec}(\ol{\cal V})$. 
Let $\ol{\eta}$ be the generic point of $\os{\circ}{\ol{S}}$ 
and 
let $\os{\circ}{\ol{s}}$ be the closed point of $\os{\circ}{\ol{S}}$.  
Let $\os{\circ}{\cal X}$ be a proper strict semistable family 
over $\os{\circ}{S}$. 
Let $\os{\circ}{\cal X}_{\os{\circ}{\ol{s}}}$ 
(resp.~$\os{\circ}{\cal X}_{\ol{\eta}}$) 
be the geometric special fiber (resp.~the geometric generic fiber) of 
$\os{\circ}{\cal X}$ over $\os{\circ}{S}$. 
For a nonnegative integer $j$, let 
$(\os{\circ}{\cal X}_{\os{\circ}{\ol{s}}})^{(j)}$ be the disjoint union of 
the $(j+1)$-fold intersections of the irreducible components of 
$\os{\circ}{\cal X}_{\os{\circ}{\ol{s}}}$. 
Let $a^{(j)}\col (\os{\circ}{\cal X}_{\os{\circ}{\ol{s}}})^{(j)}\lo 
\os{\circ}{\cal X}_{\os{\circ}{\ol{s}}}$ be the natural morphism. 
Let $l$ be a fixed prime number which is different from the characteristic of 
the residue field. In \cite{rz} 
Rapoport and Zink have constructed the projective system 
$\{(A^{\bul}_{l^n},P)\}_{n=1}^{\infty}$ of 
filtered complexes of ${\mab Z}_l$-modules in 
the \'{e}tale topos 
$(\os{\circ}{\cal X}_{\os{\circ}{\ol{s}}})_{\rm et}$ of $\os{\circ}{\cal X}_{\os{\circ}{\ol{s}}}$ 
such that 
$${\rm gr}_k^PA^{\bul}_{l^n}
\simeq \us{j\geq {\rm max}\{-k,0\}}
{\bigoplus}a^{(2j+k)}_*(({\mab Z}/l^n)_{(\os{\circ}{\cal X}_{\os{\circ}{\ol{s}}})^{(2j+k)}})(-j-k)[-2j-k]$$ 
in the derived category 
${\rm D}^{\rm b}_{\rm ctf}(\os{\circ}{\cal X}_{\os{\circ}{\ol{s}}},{\mab Z}/l^n)$ 
and such that there exists a compatible family of isomorphisms 
$\theta_n \col R\Psi({\mab Z}/l^n) \os{\sim}{\lo}A^{\bul}_{l^n}$'s 
in 
${\rm D}^{\rm b}_{\rm ctf}(\os{\circ}{\cal X}_{\os{\circ}{\ol{s}}},{\mab Z}/l^n)$'s, 
where $R\Psi({\mab Z}/l^n)$ is the nearby cycle sheaf 
of ${\mab Z}/l^n$ on $\os{\circ}{\cal X}/\os{\circ}{S}$.  
Consequently we have the following $l$-adic weight spectral sequence of 
${\rm Gal}(\ol{\eta}/\eta)$-modules: 
\begin{equation*} 
E^{-k,q+k}_1:=
\us{j\geq {\rm max}\{-k,0\}}{\bigoplus}
H^{q-2j-k}_{\rm et}((\os{\circ}{\cal X}_{\os{\circ}{\ol{s}}})^{(2j+k)},
{\mab Z}_l)(-j-k) \Lo 
H^q_{\rm et}(\os{\circ}{\cal X}_{\ol{\eta}},{\mab Z}_l) 
\quad (q\in {\mab N}). 
\tag{1.0.1}\label{eqn:rzw}
\end{equation*}    
(See also some remarks in \cite{ndeg} 
about this spectral sequence.)  
However, unfortunately they have not defined their filtered complexes 
as well-defined objects of what category. 
One should consider them as objects of 
the derived categories 
${\rm D}^{\rm b}{\rm F}_{\rm ctf}(\os{\circ}{\cal X}_{\os{\circ}{\ol{s}}},{\mab Z}/l^n)$'s of 
bounded filtered complexes of ${\mab Z}/l^n$'s-modules 
whose graded complexes have 
finite tor-dimension in 
the \'{e}tale topos $(\os{\circ}{\cal X}_{\os{\circ}{\ol{s}}})_{\rm et}$ such that 
the cohomological sheaves of the graded complexes are constructible. 
(See  \cite{ilc}, \cite{dh3}, \cite{blec} and \cite{nh2} 
for the definition of the filtered derived category.) 
In fact, one should show that 
$\{(A^{\bul}_{l^n},P)\}_{n=1}^{\infty}$ defines 
a well-defined object of the projective 2-limit 
${\rm D}^{\rm b}{\rm F}_{\rm ctf}(\os{\circ}{\cal X}_{\os{\circ}{\ol{s}}},{\mab Z}_l)$ 
of ${\rm D}^{\rm b}{\rm F}_{\rm ctf}(\os{\circ}{\cal X}_{\os{\circ}{\ol{s}}},{\mab Z}/l^n)$'s. 
(If one ignores torsion, the induced filtration on 
$R\Psi({\mab Q}_l)[d]$ of ${\mab Q}_l$ on $\os{\circ}{\cal X}/\os{\circ}{S}$ 
(twisted by $d$) by $P$ is the well-defined monodromy filtration 
by the $l$-adic monodromy operator on the perverse sheaf 
$R\Psi({\mab Q}_l)[d]:=(R\vpl_nR\Psi({\mab Z}/l^n))\otimes^L_{{\mab Z}_l}{\mab Q}_l$, 
where $d=\dim \os{\circ}{\cal X}_{\os{\circ}{s}}$ (\cite{ilad}).)  
In the log crystalline case one has to determine the analogous filtered complex to 
$\{(A^{\bul}_{l^n},P)\}_{n=1}^{\infty}$ 
as an object of what category
because the analogous filtered complex depends on the affine covering of 
$\os{\circ}{\cal X}_{\os{\circ}{s}}$ and the local log smooth embedding a priori
(see \cite{nb} and \cite{nppf} for details). 
\par 
Let $P$ be the induced filtration on 
$H^q_{\rm et}(\os{\circ}{\cal X}_{\ol{\eta}},{\mab Q}_l)$ by (\ref{eqn:rzw}).  
By \cite{sga7} (cf.~\cite{rz}) the action of ${\rm Gal}(\ol{\eta}/\eta)$ 
on $H^q_{\rm et}(\os{\circ}{\cal X}_{\ol{\eta}},{\mab Q}_l)$ is tame. 
Let $N \col H^q_{\rm et}(\os{\circ}{\cal X}_{\ol{\eta}},{\mab Q}_l) 
\lo H^q_{\rm et}(\os{\circ}{\cal X}_{\ol{\eta}},{\mab Q}_l)(-1)$ 
be the $l$-adic monodromy operator.  
P.~Deligne has conjectured (what is called, 
the $l$-adic monodromy-weight conjecture in the case where 
$\os{\circ}{S}$ is of mixed characteristics) 
that $N$ induces an isomorphism 
\begin{equation*} 
N^e \col {\rm gr}_{q+e}^P
H^q_{\rm et}(\os{\circ}{\cal X}_{\ol{\eta}},{\mab Q}_l) 
\os{\sim}{\lo} 
{\rm gr}_{q-e}^P
H^q_{\rm et}(\os{\circ}{\cal X}_{\ol{\eta}},{\mab Q}_l)(-e) 
\quad (e\in {\mab Z}_{\geq 1}), 
\tag{1.0.2}\label{eqn:nkkph} 
\end{equation*} 
where $P$ on $H^q_{\rm et}(\os{\circ}{\cal X}_{\ol{\eta}},{\mab Q}_l)$ is 
the induced filtration by (\ref{eqn:rzw}) (cf.~\cite{dh1}). 
In the case where $\os{\circ}{S}$ is of equal characteristic $p>0$, this conjecture is true by 
\cite{dw2} and \cite{itp}. In the case where $\os{\circ}{S}$ is of mixed characteristics, 
this conjecture has not yet been solved, though this has been proved to be true in special cases 
(e.g., \cite{rz}, \cite{itwm}, \cite{scho}). 
\par 
In \cite{nd} C.~Nakayama has generalized Rapoport-Zink's result above as follows. 
\par 
Let $s$ be the log point of a field $\kap$, 
i.e., $s=({\rm Spec}(\kap),{\mab N}\oplus \kap^*\lo \kap)$, 
where the morphism ${\mab N}\oplus \kap^*\lo \kap$ is defined by 
$(n,a)\lom 0^na$ $(n\in {\mab N}, a\in \kap^*)$, where $0^0=1\in \kap$ and $0^n=0$ for $n\not=0$.  
Let $\kap_{\rm sep}$ be a separable closure of $\kap$ and 
set $\ol{s}:=s\otimes_{\kap}{\kap}_{\rm sep}$. 
Let $X$ be a proper SNCL(=simple normal crossing log) scheme over $s$ defined in \cite{ndw}. 
Set $X_{\ol{s}}:=X\times_s\ol{s}$. 
For a log scheme $Y$, let $\os{\circ}{Y}$ be the underlying scheme of $Y$. 
For a nonnegative integer $j$, let 
$\os{\circ}{X}{}^{(j)}_{\ol{s}}$ be the disjoint union of 
the $(j+1)$-fold intersections of the irreducible components of 
$\os{\circ}{X}_{\ol{s}}$. 
Let $l$ be a prime number which is different from the characteristic of $\kap$. 
For a commutative monoid $Q$ with unit element, denote 
by ${\rm Spec}^{\log}({\mab Z}[Q])$ a log scheme 
whose underlying scheme is 
${\rm Spec}({\mab Z}[Q])$ 
and whose log structure is the association of 
a natural inclusion $Q\os{\sus}{\lo} {\mab Z}[Q]$. 
Set $\ol{s}_{\frac{1}{l^m}}:=\ol{s}\times_{{\rm Spec}^{\log}({\mab Z}[{\mab N}])}
{\rm Spec}^{\log}({\mab Z}[l^{-m}{\mab N}])$ $(m\in {\mab N})$ 
and $\ol{s}_{\frac{1}{l^{\infty}}}:=\vpl_{m}\ol{s}_{\frac{1}{l^m}}$. 
For an fs(=fine and saturated) log scheme $Y$ over $s$, 
set $\ol{Y}_{\frac{1}{l^{\infty}}}:=Y\times_s\ol{s}_{\frac{1}{l^{\infty}}}$. 
In \cite{nd} Nakayama has constructed 
the projective system of filtered complexes producing 
the following $l$-adic weight spectral sequence whose 
convergent term is the Kummer log \'{e}tale cohomology 
$H^q_{\rm ket}(\ol{X}_{\frac{1}{l^{\infty}}},{\mab Z}_l)$ ($q\in {\mab N}$) 
defined in \cite{fk} and \cite{nale}: 
\begin{equation*} 
E^{-k,q+k}_1:=
\us{j\geq {\rm max}\{-k,0\}}{\bigoplus}
H^{q-2j-k}_{\rm et}(\os{\circ}{\ol{X}}{}^{(2j+k)},{\mab Z}_l)(-j-k) 
\Lo H^q_{\rm ket}(\ol{X}_{\frac{1}{l^{\infty}}},{\mab Z}_l). 
\tag{1.0.3}\label{eqn:nkw}
\end{equation*}   
However he has not defined his filtered complexes as objects of what category;  
one should define them as objects of 
the derived categories ${\rm D}^{\rm b}{\rm F}_{\rm ctf}(\os{\circ}{\ol{X}}_{\rm et},{\mab Z}/l^n)$'s of 
bounded filtered complexes of ${\mab Z}/l^n$'s-modules 
whose graded complexes have finite tor-dimension in 
the \'{e}tale topos $\os{\circ}{\ol{X}}_{\rm et}$ such that 
the cohomological sheaves of the graded complexes are constructible. 
If $X$ is the special fiber of $\os{\circ}{\cal X}$ 
with canonical log structure, then 
K.~Fujiwara and K.~Kato have proved that 
\begin{equation*} 
H^q_{\rm ket}(\ol{X}_{\frac{1}{l^{\infty}}},{\mab Z}_l)= 
H^q_{\rm et}(\os{\circ}{\cal X}_{\ol{\eta}},{\mab Z}_l) 
\tag{1.0.4}\label{eqn:fki} 
\end{equation*} 
(\cite{fk}) (See also \cite{nanc}, \cite{nd} and \cite{illl}.). 
Moreover we see that 
Nakayama has essentially proved that his filtered complexes 
are isomorphic to $\{(A^{\bul}_{l^n},P)\}_{n=1}^{\infty}$
in ${\rm D}^{\rm b}{\rm F}_{\rm ctf}(\os{\circ}{\ol{X}}_{\rm et},{\mab Z}_l)$
in \cite{nd}. 
In \cite{ndeg} we have proved that  
the weight spectral sequence (\ref{eqn:rzw}) 
is isomorphic to 
the weight spectral sequence (\ref{eqn:nkw}). 
As a result, (\ref{eqn:nkw}) turns out to be 
a generalization of (\ref{eqn:rzw}). 
In [loc.~cit.] he has also proved  
the degeneration at $E_2$ modulo torsion of (\ref{eqn:nkw}). 
Let 
$N \col H^q_{\rm ket}(\ol{X}_{\frac{1}{l^{\infty}}},{\mab Q}_l) 
\lo H^q_{\rm ket}(\ol{X}_{\frac{1}{l^{\infty}}},{\mab Q}_l)(-1)$ 
be the $l$-adic monodromy operator.  
In \cite{kln} (see also \cite{nd}, \cite{ndeg}) 
Kato has conjectured the following, 
which we call  
the {\it log $l$-adic monodromy-weight conjecture}: 

\begin{conj}[{\rm{\bf  Log $l$-adic monodromy-weight conjecture (\cite{kln})}}]\label{conj:nketph}
If $\os{\circ}{X}$ is projective over $\kap$, 
then $N$ induces an isomorphism 
\begin{equation*} 
N^e \col {\rm gr}_{q+e}^P
H^q_{\rm ket}(\ol{X}_{\frac{1}{l^{\infty}}},{\mab Q}_l) 
\os{\sim}{\lo} 
{\rm gr}_{q-e}^P
H^q_{\rm ket}(\ol{X}_{\frac{1}{l^{\infty}}},{\mab Q}_l)(-e) 
\quad (e\in {\mab Z}_{\geq 1}), 
\tag{1.1.1}\label{eqn:nketph} 
\end{equation*} 
where $P$ on $H^q_{\rm ket}(\ol{X}_{\frac{1}{l^{\infty}}},{\mab Q}_l)$ is 
the induced filtration by (\ref{eqn:nkw}). 
\end{conj}  
\par 
The aim of this paper is to give a generalized $l$-adic version of 
El Zein-Steenbrink-Zucker's work by using Nakayama's results in \cite{nd} and 
general results for the derived category of 
bifiltered complexes developed in this paper. 
Our work is a generalization of Nakayama's work.  
Our generalization is technically 
more difficult than the generalization \cite{stz} and \cite{ezth} of 
\cite{st} because torsion sheaves appear in our case 
(the existence of the derived tensor product of two bounded above 
(bi)filtered complexes 
is necessary in our case). 
The main result in this paper is to give 
the log $l$-adic relative monodromy-weight conjecture (\ref{conj:nkretph}) below;  
we construct an expected $l$-adic relative monodromy filtration 
on the $l$-adic Kummer log \'{e}tale cohomology in a geometric way.  
\par 
In the rest of this introduction, 
we give a quick explanation for 
our key bifiltered complex and we state 
the log $l$-adic relative monodromy-weight conjecture. 
\par 
Let $D$ be an SNCD on $X/s$ defined in \cite{ny}. 
Let $M(D)$ be the sheaf of invertible functions of $\os{\circ}{X}$ 
outside $\os{\circ}{D}$ 
in the \'{e}tale topos $\os{\circ}{X}_{\rm et}$.  
For a positive integer $j$, 
let $\os{\circ}{D}{}^{(j)}$ be the disjoint union of 
the $j$-fold intersections of the irreducible components of 
$\os{\circ}{D}$. Endow $\os{\circ}{D}{}^{(j)}$ with the pull-back of 
the log structure of $X$ and let $D^{(j)}$ be 
the resulting log scheme. Set $D^{(0)}:=X$. 
Set $(X,D):=X\times_{\os{\circ}{X}}(\os{\circ}{X},M(D))$ and 
$(X_{\frac{1}{l^m}},D_{\frac{1}{l^m}})
:=X_{\frac{1}{l^m}}\times_{\os{\circ}{X}}(\os{\circ}{X},M(D))$  
$(m\in {\mab N})$. 
Set $\ol{D}:=D\times_s\ol{s}$ and $(\ol{X}_{\frac{1}{l^{\infty}}},\ol{D}_{\frac{1}{l^{\infty}}}):=
\vpl_m(\ol{X}_{\frac{1}{l^m}},\ol{D}_{\frac{1}{l^m}})$. 
Let ${\rm D}^{\rm b}{\rm F}^2_{\rm ctf}(\os{\circ}{\ol{X}}_{\rm et},{\mab Z}_l)$ be the derived category 
of bounded bifiltered complexes of ${\mab Z}_l$-modules 
whose graded complexes have 
finite tor-dimension in 
the \'{e}tale topos $\os{\circ}{\ol{X}}_{\rm et}$ such that 
the cohomological sheaves of the graded complexes are constructible. 
(In the text we give the definition of 
${\rm D}^{\rm b}{\rm F}^2_{\rm ctf}(\os{\circ}{\ol{X}}_{\rm et},{\mab Z}_l)$.)
We construct a bifiltered complex  
$$(A_{l^{\infty}}((\ol{X}_{\frac{1}{l^{\infty}}},\ol{D}_{\frac{1}{l^{\infty}}})
/\ol{s}_{\frac{1}{l^{\infty}}}),P^{\ol{D}_{\frac{1}{l^{\infty}}}},P)
:=\{(A_{l^n}((\ol{X}_{\frac{1}{l^{\infty}}},\ol{D}_{\frac{1}{l^{\infty}}})/
\ol{s}_{\frac{1}{l^{\infty}}}),P^{\ol{D}_{\frac{1}{l^{\infty}}}},P)\}_{n=1}^{\infty}\in 
{\rm D}^{\rm b}{\rm F}^2_{\rm ctf}(\os{\circ}{\ol{X}}_{\rm et},{\mab Z}_l)$$ 
which plays a  key role in this paper. 
We call $(A_{l^{\infty}}((\ol{X}_{\frac{1}{l^{\infty}}},\ol{D}_{\frac{1}{l^{\infty}}})/\ol{s}_{\frac{1}{l^{\infty}}}),
P^{\ol{D}_{\frac{1}{l^{\infty}}}},P)$ 
the {\it $l$-adic bifiltered El Zein-Steenbrink-Zucker complex} of $(\ol{X},\ol{D})/\ol{s}$. 
Let $\pi_{(\ol{X}_{\frac{1}{l^{\infty}}},\ol{D}_{\frac{1}{l^{\infty}}})}\col 
(\ol{X}_{\frac{1}{l^{\infty}}},\ol{D}_{\frac{1}{l^{\infty}}})\lo 
(\ol{X},\ol{D})$ be a natural morphism and 
let $\eps_{(\ol{X},\ol{D})}\col (\ol{X},\ol{D})\lo 
(\os{\circ}{\ol{X}},\os{\circ}{\ol{D}})$ be 
a morphism forgetting the log structure of $(\ol{X},\ol{D})$. 
Then we prove that there exists a natural isomorphism 
\begin{align*} 
\theta \col R(\eps_{(\ol{X},\ol{D})}
\pi_{(\ol{X}_{\frac{1}{l^{\infty}}},\ol{D}_{\frac{1}{l^{\infty}}})})_*({\mab Z}_l)
\os{\sim}{\lo}  
A_{l^n}((\ol{X}_{\frac{1}{l^{\infty}}},\ol{D}_{\frac{1}{l^{\infty}}})/\ol{s}_{\frac{1}{l^{\infty}}}). 
\tag{1.1.2}\label{eqn:lexdad}   
\end{align*} 
The bifiltered complex 
$(A_{l^{\infty}}((\ol{X}_{\frac{1}{l^{\infty}}},
\ol{D}_{\frac{1}{l^{\infty}}})/\ol{s}_{\frac{1}{l^{\infty}}}),P^{\ol{D}_{\frac{1}{l^{\infty}}}},P)$ 
produces the following new spectral sequences of 
${\rm Gal}(\ol{s}/s)$-modules, 
respectively: 
\begin{align*} 
E_1^{-k,q+k} & =H^{q-k}_{\rm ket}(\ol{D}{}^{(k)}_{\frac{1}{l^{\infty}}},{\mab Z}_l)(-k)  
\Lo H^q_{\rm ket}((\ol{X}_{\frac{1}{l^{\infty}}},
\ol{D}_{\frac{1}{l^{\infty}}}),{\mab Z}_l)\quad (q\in {\mab N}),
\tag{1.1.3}\label{eqn:lidd}   
\end{align*} 
\begin{align*} 
E_1^{-k,q+k} & =
\bigoplus_{k'\leq k}
\bigoplus_{j\geq \max\{-k',0\}} 
H^{q-2j-k}_{\rm et}(\os{\circ}{\ol{X}}{}^{(2j+k')}\cap 
\os{\circ}{\ol{D}}{}^{(k-k')},{\mab Z}_l)(-j-k) 
\tag{1.1.4}\label{eqn:lintdd} \\ 
{} & \Lo 
H^q_{\rm ket}((\ol{X}_{\frac{1}{l^{\infty}}},\ol{D}_{\frac{1}{l^{\infty}}}),{\mab Z}_l)\quad (q\in {\mab N}). 
\end{align*} 
((\ref{eqn:lintdd}) is a generalization of (\ref{eqn:nkw}).)
Using the specialization argument as in \cite{nd}, 
we prove that (\ref{eqn:lintdd}) degenerates at $E_2$ modulo torsion. 
We also prove that (\ref{eqn:lidd}) degenerates at $E_2$ modulo torsion if 
$(X,D)$ is the log special fiber of 
a proper strict semistable family with a horizontal SNCD over a  
henselian discrete valuation ring of any characteristic.
Though we have not yet proved that 
(\ref{eqn:lidd}) degenerates at $E_2$ modulo torsion in the general case, 
we prove that the edge morphisms $\{d_r^{-k,q+k}\}_{k,q}$of the spectral sequence 
(\ref{eqn:lidd}) are strictly compatible with the induced filtration by the weight filtrations  
on $H^{q-k}_{\rm ket}(\ol{D}{}^{(k)}_{\frac{1}{l^{\infty}}},{\mab Z}_l)(-k)$'s.   
Furthermore, by pursuing the analogue of the log crystalline case in \cite{nppf}, 
we prove the contravariant functoriality of 
$(A_{l^{\infty}}((\ol{X}_{\frac{1}{l^{\infty}}},
\ol{D}_{\frac{1}{l^{\infty}}})/\ol{s}_{\frac{1}{l^{\infty}}}),P^{\ol{D}_{\frac{1}{l^{\infty}}}},P)$ 
with respect to 
the following commutative diagram of log schemes 
\begin{equation*} 
\begin{CD}
(X,D)@>{g}>> (Y,E)\\
@VVV @VVV\\
s@>{v}>>s', 
\end{CD}
\end{equation*} 
where $v$ is a certain morphism of log points such that ${\rm deg}(v)$ 
is not divisible by $l$ (see \cite{nb} for the definition of ${\rm deg}(v)$) and 
$g\col (X,D)/s\lo (Y,E)/s'$ is a morphism of SNCL log schemes with SNCD's. 
As a corollary of this contravariant functoriality, if 
for each irreducible component $\os{\circ}{X}_{\lam}$ of $\os{\circ}{X}$, 
there exists a smooth component $\os{\circ}{Y}_{\lam'}$ of $\os{\circ}{Y}$
such that 
$\os{\circ}{g}(\os{\circ}{X}_{\lam})\subset \os{\circ}{Y}_{\lam'}$
and if, 
for each irreducible component $\os{\circ}{D}_{\mu}$ of $\os{\circ}{D}$, 
there exists a smooth component $\os{\circ}{E}_{\mu'}$ of $\os{\circ}{E}$
such that 
$\os{\circ}{g}(\os{\circ}{D}_{\mu})\subset \os{\circ}{E}_{\mu'}$, 
then we obtain the contravariant functorialities of the spectral sequences (\ref{eqn:lidd}) and 
(\ref{eqn:lintdd}). 
That is, we obtain the following spectral sequences 
\begin{align*} 
E_1^{-k,q+k} & =H^{q-k}_{\rm ket}(\ol{D}{}^{(k)}_{\frac{1}{l^{\infty}}},{\mab Z}_l)(-k;v)  
\Lo H^q_{\rm ket}((\ol{X}_{\frac{1}{l^{\infty}}},\ol{D}_{\frac{1}{l^{\infty}}}),{\mab Z}_l)\quad (q\in {\mab N}),
\tag{1.1.5}\label{eqn:lided}   
\end{align*} 
\begin{align*} 
E_1^{-k,q+k} & =
\bigoplus_{k'\leq k}
\bigoplus_{j\geq \max\{-k',0\}} 
H^{q-2j-k}_{\rm et}(\os{\circ}{\ol{X}}{}^{(2j+k')}\cap 
\os{\circ}{\ol{D}}{}^{(k-k')},{\mab Z}_l)(-j-k';v)(-(k-k');g,\Del,\Del') 
\tag{1.1.6}\label{eqn:lindtdd} \\ 
{} & \Lo H^q_{\rm ket}((\ol{X}_{\frac{1}{l^{\infty}}},\ol{D}_{\frac{1}{l^{\infty}}}),{\mab Z}_l)\quad (q\in {\mab N}), 
\end{align*} 
where $(-k;v)$ is the $l$-adic analogue of the $D$-twist defined in \cite{nb} and \cite{nhir} and 
$(-(k-k');g,\Del,\Del')$ is the $l$-adic $D$-twist defined in the text, where 
$\Del=\{\os{\circ}{D}_{\mu}\}_{\mu}$ 
and $\Del'=\{\os{\circ}{E}_{\mu'}\}_{\mu'}$ are the irreducible components of $\os{\circ}{D}$ 
and $\os{\circ}{E}$, respectively. 
These contravariant functorialities on (\ref{eqn:lided}) and (\ref{eqn:lindtdd}) have 
not been (able) to be considered in \cite{rz}, \cite{nd}, \cite{itp} nor \cite{stwsl} even in 
the case $D=\emptyset$. 
We also prove the base change theorem of $(A_{l^{\infty}}((\ol{X}_{\frac{1}{l^{\infty}}},
\ol{D}_{\frac{1}{l^{\infty}}})/\ol{s}_{\frac{1}{l^{\infty}}}),P^{\ol{D}_{\frac{1}{l^{\infty}}}},P)$
under a mild condition.  
\par 
To construct $(A_{l^{\infty}}((\ol{X}_{\frac{1}{l^{\infty}}},
\ol{D}_{\frac{1}{l^{\infty}}})/\ol{s}_{\frac{1}{l^{\infty}}}),P^{\ol{D}_{\frac{1}{l^{\infty}}}},P)$, 
we use a bifiltered version of Berthelot's theory of 
the filtered derived category in \cite{blec} and \cite{nh2}. 
To give a bifiltered version of Berthelot's theory is a nontrivial work. 
We have to give appropriate definitions and appropriate formulations
about bifiltered complexes and 
we have to make quite complicated calculations about them patiently 
for the construction of our theory.   
The fundamental machines in \cite{blec}, \cite{nh2} and 
the bifiltered version of Berthelot's theory in the first part of this paper 
with fundamental facts for Kummer log \'{e}tale cohomologies in 
\cite{nale}, \cite{ktnk}, \cite{adv} 
give us several properties of 
$(A_{l^{\infty}}((\ol{X}_{\frac{1}{l^{\infty}}},\ol{D}_{\frac{1}{l^{\infty}}})/\ol{s}_{\frac{1}{l^{\infty}}}),
P^{\ol{D}_{\frac{1}{l^{\infty}}}},P)$. 
For example, by virtue of the adjunction formula for the derived homomorphism functor for 
bifiltered complexes proved in the first part,  
we obtain the contravariant functoriality and the base change theorem of 
$(A_{l^{\infty}}((\ol{X}_{\frac{1}{l^{\infty}}},\ol{D}_{\frac{1}{l^{\infty}}})
/\ol{s}_{\frac{1}{l^{\infty}}}),P^{\ol{D}_{\frac{1}{l^{\infty}}}},P)$ ((\ref{theo:func}), (\ref{theo:pwbcsbc})). 
\par 
Let 
$N \col H^q_{\rm ket}((\ol{X}_{\frac{1}{l^{\infty}}},\ol{D}_{\frac{1}{l^{\infty}}}),{\mab Q}_l) 
\lo H^q_{\rm ket}((\ol{X}_{\frac{1}{l^{\infty}}},\ol{D}_{\frac{1}{l^{\infty}}}),{\mab Q}_l)(-1)$ 
be the $l$-adic monodromy operator which will be defined in the text. 
We denote by $P^{\ol{D}_{\frac{1}{l^{\infty}}}}$ and $P$ 
the induced filtrations on 
$H^q_{\rm ket}((\ol{X}_{\frac{1}{l^{\infty}}},\ol{D}_{\frac{1}{l^{\infty}}}),{\mab Q}_l)$ by 
(\ref{eqn:lidd}) and (\ref{eqn:lintdd}), respectively. 
It is not difficult to prove that 
$N\col H^q_{\rm ket}((\ol{X}_{\frac{1}{l^{\infty}}},\ol{D}_{\frac{1}{l^{\infty}}}),{\mab Q}_l)\lo 
H^q_{\rm ket}((\ol{X}_{\frac{1}{l^{\infty}}},\ol{D}_{\frac{1}{l^{\infty}}}),{\mab Q}_l)(-1)$ 
induces a morphism 
$N\col P_kH^q_{\rm ket}((\ol{X}_{\frac{1}{l^{\infty}}},\ol{D}_{\frac{1}{l^{\infty}}}),{\mab Q}_l)\lo 
P_{k-2}H^q_{\rm ket}((\ol{X}_{\frac{1}{l^{\infty}}},\ol{D}_{\frac{1}{l^{\infty}}}),{\mab Q}_l)$ 
$(k\in {\mab Z})$ by using (\ref{eqn:lexdad}) and 
the filtered complex 
$(A_{l^{\infty}}((\ol{X}_{\frac{1}{l^{\infty}}},\ol{D}_{\frac{1}{l^{\infty}}})/\ol{s}_{\frac{1}{l^{\infty}}}),P)$.  
We conjecture the following, which we call
the {\it log $l$-adic relative monodromy-weight conjecture}: 

\begin{conj}[{\rm{\bf  Log $l$-adic relative monodromy-weight conjecture}}]\label{conj:nkretph}
Assume that $\os{\circ}{X}$ is projective over $\kap$. 
Then the $l$-adic relative monodromy filtration 
on $H^q_{\rm ket}((\ol{X}_{\frac{1}{l^{\infty}}},\ol{D}_{\frac{1}{l^{\infty}}}),{\mab Q}_l)$ 
relative to $P^{\ol{D}_{\frac{1}{l^{\infty}}}}$ exists and that it is equal to 
the induced filtration $P$ on 
$H^q_{\rm ket}((\ol{X}_{\frac{1}{l^{\infty}}},\ol{D}_{\frac{1}{l^{\infty}}}),{\mab Q}_l)$. 
This is equivalent to the following$:$ 
the morphism 
$N\col H^q_{\rm ket}((\ol{X}_{\frac{1}{l^{\infty}}},\ol{D}_{\frac{1}{l^{\infty}}}),{\mab Q}_l)\lo 
H^q_{\rm ket}((\ol{X}_{\frac{1}{l^{\infty}}},\ol{D}_{\frac{1}{l^{\infty}}}),{\mab Q}_l)(-1)$ 
induces the following isomorphism 
\begin{equation*}   
N^e \col {\rm gr}_{q+k+e}^P{\rm gr}_k^{P^{\ol{D}_{\frac{1}{l^{\infty}}}}}
H^q_{\rm ket}((\ol{X}_{\frac{1}{l^{\infty}}},\ol{D}_{\frac{1}{l^{\infty}}}),{\mab Q}_l) 
\os{\sim}{\lo} 
{\rm gr}_{q+k-e}^P{\rm gr}_k^{P^{\ol{D}_{\frac{1}{l^{\infty}}}}}
H^q_{\rm ket}((\ol{X}_{\frac{1}{l^{\infty}}},\ol{D}_{\frac{1}{l^{\infty}}}),{\mab Q}_l)(-e) 
\tag{1.2.1}\label{eqn:remtph} 
\end{equation*} 
for $k,q\in {\mab N}, e\in {\mab Z}_{\geq 1}$. 
\end{conj} 
Here note that the index $q+k+e$ of ${\rm gr}_{q+k+e}^P$ is different from 
the index $a+b$ of ${\rm gr}_{a+b}^W$ in \cite[(1.10.1)]{ezth} if $q\not=0$; 
the index $a+b$ in [loc.~cit.] is incorrect. 
If this conjecture is true, then the filtration $P$ on 
$H^q_{\rm ket}((\ol{X}_{\frac{1}{l^{\infty}}},\ol{D}_{\frac{1}{l^{\infty}}}),{\mab Q}_l)$
is equal to the relative monodromy filtration of $N$ with respect to the filtration 
$P^{\ol{D}_{\frac{1}{l^{\infty}}}}$ on 
$H^q_{\rm ket}((\ol{X}_{\frac{1}{l^{\infty}}},\ol{D}_{\frac{1}{l^{\infty}}}),{\mab Q}_l)$. 
Especially the relative monodromy filtration of $N$ with respect to the filtration 
$P^{\ol{D}_{\frac{1}{l^{\infty}}}}$ on 
$H^q_{\rm ket}((\ol{X}_{\frac{1}{l^{\infty}}},\ol{D}_{\frac{1}{l^{\infty}}}),{\mab Q}_l)$ exists.
By showing a key lemma and the strict compatibility 
of the edge morphisms of (\ref{eqn:lided}) with respect to 
the weight filtration which has been already stated, 
we prove that the conjecture (\ref{conj:nkretph}) is true if 
the log $l$-adic monodromy-weight conjecture 
(\ref{conj:nketph}) for $D^{(k)}$ for any $k\in {\mab Z}_{\geq 1}$ is true.  
Because (\ref{conj:nketph}) has been proved 
in the proper strict semistable case in the equal characteristic $p>0$ (\cite{dh2}, \cite{itp}), 
we see that the conjecture (\ref{conj:nkretph}) is true  
if 
$(X,D)/s$  
is the log special fiber 
of a proper strict semistable family over a henselian discrete valuation ring 
of equal characteristic $p>0$.   
We also prove that the conjecture (\ref{conj:nkretph}) is true  
when ${\rm dim}\os{\circ}{X}\leq 2$ by using 
Kajiwara-Achinger's result (\cite[(3.1)]{kaj}, \cite[Theorem 3.6]{ash}) 
and Rapoport-Zink-Mokrane's calculation (\cite{rz}, \cite{msemi}) 
for $D^{(k)}$ $(k\in {\mab Z}_{\geq 1})$. 
\par 
It is very natural to give  
the $p$-adic version of this paper.  
We have already given this in \cite{nppf}. 
Especially we have constructed the log crystalline analogue of 
$(A_{l^{\infty}}((\ol{X}_{\frac{1}{l^{\infty}}},\ol{D}_{\frac{1}{l^{\infty}}})/\ol{s}_{\frac{1}{l^{\infty}}}),
P^{\ol{D}_{\frac{1}{l^{\infty}}}},P)$. 
In the future we would like to discuss the $\infty$-adic version 
of this paper, which is a generalization of \cite{stz}. 
\medskip 
\par 
This paper consists of three parts. 
The aim of the first part 
(${\rm \S\ref{sec:nfil}} \sim {\rm \S\ref{sec:rksfdc}}$)
is to construct a general theory of the derived category of bifiltered complexes
which will be used in the second part. 
This part is the bifiltered version of Berthelot's theory (\cite{blec}, \cite{nh2}). 
For a morphism $f\col ({\cal T},{\cal A})\lo ({\cal T}',{\cal A}')$ 
of ringed topoi, we define four functors $Rf_*$, $Lf^*$, $RHom^{\bul}$, $\otimes^L$ 
in the derived categories of bifiltered complexes of ${\cal A}$-modules among  
so called Grothendieck's six functors. 
The aim of the second part (${\rm \S\ref{sec:lbc}} \sim {\rm \S\ref{sec:ssf}}$) 
is to give the log $l$-adic relative monodromy-weight conjecture 
and to prove that this conjecture is true in the cases already stated. 
This part is a generalization of Nakayama's work. 
The third part (\S\ref{sec:lbddmlif}) is an appendix, 
in which we give the explicit descriptions of 
the edge morphisms of the $E_1$-terms of the spectral sequences  
(\ref{eqn:lidd}) and (\ref{eqn:lintdd}). 
\par 
The content of each section of this paper is as follows. 
\par 
In \S\ref{sec:nfil} we give the definition of the derived category of bifiltered complexes. 
The notion of the strictly exactness of a bifiltered complex is a key notion for the definition. 
Our definition of 
the derived category of bifiltered complexes 
is a generalization of 
the derived category of bounded below biregular bifiltered complexes in \cite{dh3}. 
\par 
In \S\ref{sec:sti} we give the notion of a strictly injective module  
and we prove the existence of the strictly injective resolution of 
a bounded below bifiltered complex and the existence of $Rf_*$. 
\par 
In \S\ref{sfr} we give the notion of a strictly flat module  
and we prove the existence of 
the strictly flat resolution of a bounded above bifiltered complex 
and the existence of $Lf^*$.  
\par 
In \S\ref{sec:rhoml} we prove 
the existence of the derived homomorphism functor ${\rm RHom}^{\bul}$ 
from a bounded above bifiltered complex
to a bounded below bifiltered complex.  
We also prove 
the adjunction formula of the derived homomorphism functor.  
\par 
In \S\ref{sec:otl} we prove the existence of the derived tensor product $\otimes^L$ 
of two bounded above bifiltered complexes. 
To prove the  existence of $\otimes^L$ is a much more nontrivial work than 
to prove the existence of ${\rm RHom}^{\bul}$. 
The derived tensor product $\otimes^L$ in \cite{nh2} and 
the bifiltered derived category in this paper 
is necessary for the second part of this paper. 
\par 
In \S\ref{sec:comp} we give some complements. 
\par 
In \S\ref{sec:rksfdc} 
we give a simple remark related to the general theory of quasi-abelian categories of 
Schneiders (\cite{sc}, \cite{ss}). 
\par 
In \S\ref{sec:lbc} we construct the $l$-adic bifiltered El Zein-Steenbrink-Zucker complex 
$$(A_{l^{\infty}}((\ol{X}_{\frac{1}{l^{\infty}}},\ol{D}_{\frac{1}{l^{\infty}}})/\ol{s}_{\frac{1}{l^{\infty}}}),
P^{\ol{D}_{\frac{1}{l^{\infty}}}},P)
\in {\rm D}^{\rm b}{\rm F}^2_{\rm ctf}(\os{\circ}{\ol{X}}_{\rm et},{\mab Z}_l).$$ 
This section is a main part of this paper.  
\par 
In \S\ref{sec:lwss} we construct the spectral sequences (\ref{eqn:lidd}) and (\ref{eqn:lintdd}). 
\par 
In \S\ref{sec:confu} we prove the contravariant functoriality of 
$(A_{l^{\infty}}((\ol{X}_{\frac{1}{l^{\infty}}},\ol{D}_{\frac{1}{l^{\infty}}})/\ol{s}_{\frac{1}{l^{\infty}}}),
P^{\ol{D}_{\frac{1}{l^{\infty}}}},P)$ and 
we construct the spectral sequences (\ref{eqn:lided}) and (\ref{eqn:lindtdd}). 
\par 
In \S\ref{sec:fpl} we investigate fundamental properties of
the induced filtrations $P^{\ol{D}_{\frac{1}{l^{\infty}}}}$ and $P$ on  
$H^q_{\rm ket}((\ol{X}_{\frac{1}{l^{\infty}}},\ol{D}_{\frac{1}{l^{\infty}}}),{\mab Q}_l)$
by the spectral sequences (\ref{eqn:lidd}) and (\ref{eqn:lintdd}), respectively. 
\par 
In \S\ref{sec:rmwc} we give the conjecture (\ref{eqn:remtph}) 
and we prove that this is true in the case 
$\dim \os{\circ}{X}\leq 2$. 
\par 
In \S\ref{sec:ssf} we construct a similar bifiltered complex to 
$(A_{l^{\infty}}((\ol{X}_{\frac{1}{l^{\infty}}},\ol{D}_{\frac{1}{l^{\infty}}})/\ol{s}_{\frac{1}{l^{\infty}}},
P^{\ol{D}_{\frac{1}{l^{\infty}}}},P)$ in the proper strictly semistable case with a relative SNCD 
over ${\cal V}$ and we give a comparison theorem 
between the similar complex and 
$(A_{l^{\infty}}((\ol{X}_{\frac{1}{l^{\infty}}},\ol{D}_{\frac{1}{l^{\infty}}})/\ol{s}_{\frac{1}{l^{\infty}}}),
P^{\ol{D}_{\frac{1}{l^{\infty}}}},P)$. 
We also prove that the conjecture (\ref{conj:nkretph}) 
is true if $D^{(k)}/s$ for any $k\in {\mab N}$ is the log special fiber of 
a proper strict semistable family over 
a henselian discrete valuation ring 
of equal characteristic.

\par 
In \S\ref{sec:lbddmlif} 
we give the explicit expressions of the edge morphisms between 
the $E_1$-terms of $l$-adic weight spectral sequences (\ref{eqn:lidd}) and 
(\ref{eqn:lintdd}). 
\par 
\par
\bigskip
\parno
{\bf Acknowledgment.}
 I would like to express my sincere thanks to 
 K.~Kato for sending me a letter \cite{kln},  K.~Fujiwara  
for informing me of a technique in the proof of \cite[(6.1)]{itp} before 
the publication of the paper and C.~Nakayama 
for giving me a suggestion in (\ref{rema:gfs}) (2) in the text, respectively.

\bigskip
\parno
{\bf Notations.}
(1) For a log scheme $X$, $\os{\circ}{X}$ denotes the underlying scheme of $X$. 
For a morphism 
$\varphi \col X \lo Y$, $\os{\circ}{\varphi}$ denotes 
the underlying morphism 
$\os{\circ}{X} \lo \os{\circ}{Y}$ of $\varphi$. 
\par
(2) SNC(L)=simple normal crossing (log), 
SNCD=simple normal crossing divisor.  
\par
(3) For a complex $(E^{\bul}, d^{\bul})$ 
of objects in an exact additive category ${\cal A}$, 
we often denote  $(E^{\bul}, d^{\bul})$ only by 
$E^{\bul}$ as usual. 
\par 
(4) For a complex $(E^{\bul}, d^{\bul})$ in (3) and 
for an integer $n$, $(E^{\bul}\{n\}, d^{\bul}\{n\})$ 
denotes the following complex:
\begin{equation*}
\cdots \lo \us{q-1}{E^{q-1+n}}
\os{d^{q-1+n}}{\lo} 
\us{q}{E^{q+n}} \os{d^{q+n}}{\lo} 
\us{q+1}{E^{q+1+n}} \os{d^{q+1+n}}{\lo} 
\cdots.
\end{equation*}
Here the numbers under the objects 
above in ${\cal A}$ mean the degrees.
\par 
(5) For a complex $(E^{\bul}, d^{\bul})$ in (3), 
$\tau=\{\tau_k\}_{k\in {\mab Z}}$ denotes 
the canonical filtration on $(E^{\bul}, d^{\bul})$: 
\begin{equation*}
\tau_k(E^{\bul}):= 
(\cdots \lo E^{k-2} \lo E^{k-1} \lo {\rm Ker}(d^k) \lo 0 \lo 0 
\lo \cdots).    
\end{equation*} 
We fix the following isomorphism 
\begin{align*} 
{\rm gr}_k^{\tau}E^{\bul}\os{\sim}{\lo} {\cal H}^k(E^{\bul})[-k]
\end{align*}
induced by the projection ${\rm Ker}(E^k\lo E^{k+1}) \lo {\cal H}^k(E^{\bul})$. 
\par 
(6)  For a morphism 
$f \col  (E^{\bul},d^{\bul}_E) \lo (F^{\bul},d^{\bul}_F)$ 
of complexes, 
let ${\rm MF}(f)$ (resp.~{\rm MC}(f)) 
be the mapping fiber (resp.~the mapping cone) 
of $f$: ${\rm MF}(f):=E^{\bul} \oplus F^{\bul}[-1]$ 
with boundary morphism ``$(x,y)\lom (d_E(x),-d_F(y)+f(x))$"
(resp.~${\rm MC}(f):=E^{\bul}[1] \oplus F^{\bul}$ 
with boundary morphism ``$(x,y)\lom (-d_E(x),d_F(y)+f(x))$").  
\par
(7) Let $({\cal T},{\cal A})$ be a ringed topos. 
\medskip
\par
(a) $C({\cal T},{\cal A})$ (resp.~${\rm C}^{\pm}({\cal T},{\cal A})$, 
$C^{\rm b}({\cal T},{\cal A})$): the category of 
(resp.~bounded below, bounded above, 
bounded) complexes of ${\cal A}$-modules, 
\par
(b) $K({\cal T},{\cal A})$ (resp.~$K^{\pm}({\cal T},{\cal A})$, 
$K^{\rm b}({\cal T},{\cal A})$): the category of 
(resp.~bounded below, bounded above, 
bounded) complexes of ${\cal A}$-modules modulo homotopy, 
\par
(b) $D({\cal T},{\cal A})$ (resp.~$D^{\pm}({\cal T},{\cal A})$, 
$D^{\rm b}({\cal T},{\cal A})$): the derived category of
$K({\cal T},{\cal A})$ (resp.~$K^{\pm}({\cal T},{\cal A})$, 
$K^{\rm b}({\cal T},{\cal A})$). 
For an object $E^{\bul}$ of 
$C({\cal T},{\cal A})$ (resp.~$C^{\pm}({\cal T},{\cal A})$, 
$C^{\rm b}({\cal T},{\cal A})$), we denote 
simply by $E^{\bul}$
the corresponding object to $E^{\bul}$ in 
$D({\cal T},{\cal A})$ (resp.~$D^{\pm}({\cal T},{\cal A})$, 
$D^{\rm b}({\cal T},{\cal A})$). 
\par
(d) The additional notation ${\rm F}$ 
to the categories above means ``the filtered ''.
Here the filtration is 
an increasing filtration indexed by 
${\mab Z}$.   
For example,  
${\rm K}^+{\rm F}({\cal T},{\cal A})$ is 
the category 
of bounded below filtered complexes 
modulo filtered homotopy. 
\par
(e) ${\rm DF}^2({\cal T},{\cal A})$ 
(resp.~${\rm D}^{\pm}{\rm F}^2({\cal T},{\cal A})$, 
${\rm D}^{\rm b}{\rm F}^2({\cal T},{\cal A})$): 
the derived category of (resp.~bounded below, bounded above,  bounded) 
bifiltered complexes of ${\cal A}$-modules which will be defined. 
\par
(8) For an fs log scheme $X$, $X_{\rm ket}$ denotes 
the Kummer log \'{e}tale topos of $X$ by imitating 
``$(X/S)_{\rm cris}$'' as in the classical crystalline case (\cite{bb}). 
(We do not use the notation 
$\wt{X}_{\rm ket}$ which was used in references.)
For a scheme $Y$, 
$Y_{\rm et}$ denotes the \'{e}tale topos of $Y$ 
($Y_{\rm et}$ is not the \'{e}tale site of $Y$ in this paper).

\bigskip
\par\noindent
{\bf Conventions.}
We make the following conventions about signs  
(cf.~\cite{bbm}, \cite{con}). 
\par
Let ${\cal A}$ be an exact additive category.
\smallskip
\par
(1) (cf.~\cite[0.3.2]{bbm}, \cite[(1.3.2)]{con}) 
For a short exact sequence 
\begin{equation*}0\lo (E^{\bul},d^{\bul}_E) 
\os{f}{\lo} (F^{\bul},d^{\bul}_F) \os{g}{\lo} 
(G^{\bul},d^{\bul}_G) \lo 0
\end{equation*}
of complexes of objects in ${\cal A}$, 
the mapping fiber of $g$. 
We fix an isomorphism 
``$E^{\bul}\owns x \lom (f(x),0)\in 
{\rm MF}(g)=F^{\bul}\oplus G^{\bul}[-1]$'' 
in the derived category $D({\cal A})$.
\smallskip
\par
(2) (\cite[0.3.2]{bbm}, \cite[(1.3.3)]{con}) 
In the situation (1), the boundary morphism
$(G^{\bul},d^{\bul}_G) \lo (E^{\bul}[1],d^{\bul}_E[1])$ 
in $D^+({\cal A})$ 
is the following composite morphism 
\begin{equation*}
(G^{\bul},d^{\bul}_G) 
\os{\sim}{\longleftarrow} {\rm MC}(f) 
\os{{\rm proj.}}{\lo} (E^{\bul}[1],d^{\bul}_E[1])
\os{(-1)\times}{\lo} (E^{\bul}[1],d^{\bul}_E[1]).
\end{equation*}
More generally, we use only the similar boundary morphism 
for a triangle in a derived category. 
\par
(3) For a complex $(E^{\bul}, d^{\bul})$ of objects in 
${\cal A}$, 
the identity ${\rm id} \col E^q \lo E^q$ 
$(\forall q \in {\mab Z})$ induces an isomorphism 
${\cal H}^q((E^{\bul},-d^{\bul})) \os{\sim}{\lo} 
{\cal H}^q((E^{\bul},d^{\bul}))$
$(\forall q \in {\mab Z})$ of cohomologies. 
\par 
\par 
(4) For a ringed topos $({\cal T},{\cal A})$ and 
for two complexes $(A^{\bul},d^{\bul})$ and $(B^{\bul},d^{\bul})$ 
of ${\cal A}$-modules, the boundary morphism $d=\{d^n\}_{n\in {\mab Z}}$ of 
$A^{\bul}\otimes_{\cal A}B^{\bul}$ is defined as usual: 
$$d^n=\sum_{p+q=n}(d^p\otimes {\rm id}+(-1)^p{\rm id}\otimes d^q).$$

\bigskip
\bigskip
\bigskip
\parno 
{\bf {\Large Part I. Derived categories of bifiltered complexes}}
\bigskip
\parno  
In the Part I of this paper we construct 
theory of derived categories of bifiltered complexes. 
We need new various ideas to construct it. 
\bigskip

\section{The definition of the derived category of bifiltered complexes}\label{sec:nfil}
In this section we give the definition of the strictness of a morphism of 
bifiltered modules and the definition of the strictly exactness of bifiltered complexes. 
The latter notion is the most important one for 
the definition of the derived category of bifiltered complexes 
in this paper. 
To give the definitions is not an obvious work (see (\ref{rema:obstm}) below); 
the theory in \cite{blec} and \cite{nh2} does not imply the theory in this paper. 
Unfortunately the author has not yet given the appropriate definition of the strictly exactness of 
a complex with $n$-pieces of filtrations for a general positive integer $n$ 
because of the quite complicated description of the multi-graded complex of a complex with
$n$-pieces of filtrations for the case $n\geq 3$. 
\par
Let $({\cal T}, {\cal A})$ be  a ringed topos. 
Let $E$ be an 
${\cal A}$-module in ${\cal T}$. 
An increasing filtration on $E$ 
is, by definition, a
family $P:=\{P_k\}:=\{P_kE\}_{k\in {\mab Z}}$ of 
${\cal A}$-submodules of 
$E$ such that $P_kE \subset P_{k+1}E$ 
for any $k\in {\mab Z}$. 
As in \cite{blec} and \cite{nh2}, 
filtrations are not necessarily 
exhaustive nor separated unlike biregular filtrations in \cite{dh3}.  
(In \cite{nhir} we have already used results in \cite{blec} and \cite{nh2} essentially 
for non-biregular filtrations for the study of weight filtrations on 
the log crystalline cohomological sheaves of proper SNCL schemes in characteristic $p>0$.)
In this paper we consider only increasing filtrations 
and we call them filtrations shortly. 
Let $n$ be a positive integer and let 
$P^{(i)}$ $(i=1,\ldots, n)$ be filtrations on $E$. 
We denote by $(E,\{P^{(i)}\}_{i=1}^n)$ 
an ${\cal A}$-module $E$ 
with $n$-pieces of filtrations $P^{(1)}, \ldots, P^{(n)}$ on $E$. 
For simplicity of notation, we almost always denote 
$P^{(i)}$ by $E^{(i)}$ and 
$P^{(i)}_kE$ $(k\in {\mab Z})$ by $E^{(i)}_k$, respectively. 
For $1\leq i_1< i_2< \cdots < i_m\leq n$ $(1\leq m\leq n)$ and $k_1,\ldots, k_m\in {\mab Z}$, 
set 
$$E^{(i_1\cdots i_m)}_{k_1\cdots k_m}:=
E^{(i_1)}_{k_1}\cap \cdots \cap E^{(i_m)}_{k_m}.$$
For $1\leq i_1\leq i_2\leq \cdots \leq i_n\leq n$ 
and $k_1,\ldots, k_n\in {\mab Z}$, we also use 
the following convenient notation 
$$E^{(i_1\cdots i_n)}_{k_1\cdots k_n}:=
E^{(i_1)}_{k_1}\cap \cdots \cap E^{(i_n)}_{k_n}.$$ 
\par 
For an ${\cal A}$-module $E$,  
we mean by the trivial filtration on $E$ 
a filtration $\{P_k\}_{k\in {\mab Z}}$ 
such that $P_0E=E$ and $P_{-1}E=0$. 
In \cite{blec}, ${\cal A}$ has 
a nontrivial filtration and 
${\cal A}$ is not necessarily commutative; 
in this paper, we consider only 
the trivial filtration on 
${\cal A}$ and ${\cal A}$ is assumed to be commutative. 
\par  
A morphism $f\col (E,\{E^{(i)}\}_{i=1}^n)\lo (F,\{F^{(i)}\}_{i=1}^n)$ 
of modules of $n$-pieces of filtrations is defined to be 
a morphism $f\col E\lo F$ of ${\cal A}$-modules 
such that $f(E^{(i)}_k)\subset F^{(i)}_k$ for any $1\leq i\leq n$ and any $k\in {\mab Z}$. 
Let ${\rm MF}^n({\cal A})$ be the category of ${\cal A}$-modules with 
$n$-pieces of filtrations. 
Obviously ${\rm MF}^n({\cal A})$ is an additive category. 
We do not say that $f$ is strict even if 
$f\col (E,E^{(i)})\lo (F,F^{(i)})$ is strict 
(i.e., ${\rm Im}(f)\cap F^{(i)}_k=f(E^{(i)}_{k})$ $(k\in {\mab Z})$)
for $1\leq \forall i\leq n$: 

\begin{defi}\label{defi:fef}
Let $f \col (E,\{E^{(i)}\}_{i=1}^n) \lo (F,\{F^{(i)}\}_{i=1}^n)$ 
be a morphism in 
${\rm MF}^n({\cal A})$. 
Then we say that $f$ is {\it strict} if ${\rm Im}(f)\cap 
F^{(i_1\cdots i_n)}_{k_1\cdots k_n}=f(E^{(i_1\cdots i_n)}_{k_1 \cdots k_n})$ 
for $1 \leq i_1 \leq \cdots \leq i_n \leq n$ and $k_1,\ldots, k_n \in {\mab Z}$.
\end{defi}
It is obvious that, if  $f \col (E,\{E^{(i)}\}_{i=1}^n) \lo (F,\{F^{(i)}\}_{i=1}^n)$ is strict, 
then  
the underlying  morphism $f^{(i)}\col (E,E^{(i)})\lo (F,F^{(i)})$ of filtered ${\cal A}$-modules 
is strict for $1\leq \forall i\leq n$. 
\par 
We often use the following simple criterion in key points in the proofs of 
results in this paper:  

\begin{prop}\label{prop:usef}
Let $f \col (E,\{E^{(i)}\}_{i=1}^n) \lo (F,\{F^{(i)}\}_{i=1}^n)$ 
be a morphism in 
${\rm MF}^n({\cal A})$. 
For $1\leq i \leq n$, 
let $f^{(i)}\col (E,E^{(i)}) \lo (F,F^{(i)})$ 
be the filtered morphism. 
If $f$ is injective, then $f$ is strict if and only if 
$f^{(i)}\col (E,E^{(i)}) \lo (F,F^{(i)})$
is strict for $1\leq \forall i \leq n$. 
\end{prop}
\begin{proof} 
We have only to prove that ${\rm Im}(f)\cap F^{(i_1\cdots i_n)}_{k_1 \cdots k_n}\subset 
f(E^{(i_1\cdots i_n)}_{k_1 \cdots k_n})$.
Since $F^{(i_1\cdots i_n)}_{k_1 \cdots k_n}\subset F^{(i_m)}_{k_m}$ $(1\leq \forall m \leq n)$ 
and 
since 
$f^{(i)}\col (E,E^{(i)}) \lo (F,F^{(i)})$ $(1\leq i \leq n)$  is strict,  
${\rm Im}(f)\cap F^{(i_1\cdots i_n)}_{k_1 \cdots k_n}\subset \bigcap_{m=1}^n
f(E^{(i_m)}_{k_m})$.  
Since $f \col E\lo F$ is injective, it is very easy to check 
that 
$\bigcap_{m=1}^n
f(E^{(i_m)}_{k_m})=
f(E^{(i_1\cdots i_n)}_{k_1 \cdots k_n})$. 
\end{proof}

\begin{rema}\label{rema:obstm} 
Even if $f^{(i)}$ is strict for $1\leq \forall i \leq n$, 
$f$ is not necessarily strict in general because 
$E^{(i_1\cdots i_n)}_{k_1 \cdots k_n}$ may be too small compared with 
$F^{(i_1\cdots i_n)}_{k_1 \cdots k_n}$. 
Indeed, let $A$ be a nonzero commutative ring with unit element. 
Let $M$ be an $A$-module and let 
$0 \subsetneq N \subsetneq M$ be an $A$-submodule. 
Consider a filtration $P$ on $M$ defined by 
$P_kM:=0$ $(k<0)$, $P_0M:=N$ $(k=0)$ and $P_kM:=M$ $(k>0)$. 
Set $E:=M\oplus M$ and consider two filtrations 
$E^{(1)}$ and $E^{(2)}$ defined by 
\begin{equation*}
E^{(1)}_k= 
\begin{cases} 0 & (k<0), \\
N\oplus 0 & (k=0), \\ 
E & (k>0)  
\end{cases}
\end{equation*}
and 
\begin{equation*}
E^{(2)}_k= 
\begin{cases} 0 & (k<0), \\
0\oplus N & (k=0), \\ 
E & (k>0).  
\end{cases}
\end{equation*} 
Set $F:=M$ and consider two filtrations 
$F^{(1)}$ and $F^{(2)}$ on $F$ defined by 
$F^{(1)}:=P=:F^{(2)}$.  
Then the summation $+ \col M\oplus M \lo M$ 
induces a bifiltered morphism 
$f\col (E,(E^{(1)},E^{(2)})) \lo (F,(F^{(1)},F^{(2)}))$ of $A$-modules. 
The morphism $f^{(i)}\col (E,E^{(i)}) \lo (F,F^{(i)})$ for $i=1,2$ is strict, 
while $f$ is not: 
$f(E^{(12)}_0)=0\not=N={\rm Im}(f)\cap F^{(12)}_{00}$.  
Because 
$${\rm Im}(f)=(f(E), \{f(E)\cap F^{(1)}_k)\}_{k\in {\mab Z}}, 
\{f(E)\cap F^{(2)}_k\}_{k\in {\mab Z}})$$ 
and 
\begin{align*} 
{\rm Coim}(f)&=(E,P^{(1)},P^{(2)})/{\rm Ker}(f) \\
&=(E/f^{-1}(0), \{E^{(1)}_k+f^{-1}(0)/f^{-1}(0)\}_{k\in {\mab Z}},
\{E^{(2)}_k+f^{-1}(0)/f^{-1}(0)\}_{k\in {\mab Z}})\\
&= (f(E),\{f(E^{(1)}_k)\}_{k\in {\mab Z}},\{f(E^{(2)}_k)\}_{k\in {\mab Z}}),
\end{align*}  
${\rm Coim}(f)={\rm Im}(f)$. 
Hence $f$ is strict in the sense of \cite[\S1]{sc}.  
(Consequently our definition of the strictness for a morphism 
of bifiltered modules is not equal to the natural definition of the strictness for a morphism 
of bifiltered modules in [loc.~cit.].)   
However the induced morphism 
$f'\col (E,E^{(1)}, E^{(2)}, E^{(12)})\lo (F,F^{(1)}, F^{(2)}, F^{(12)})$ 
of trifiltered complexes by $f$ is not strict in the sense of \cite[\S1]{sc}. 
Here  
the category of trifiltered complexes obtained by bifiltered complexes is defined suitably. 
Note that the sets of indexes of the filtrations $E^{(12)}$ and $F^{(12)}$ 
are ${\mab Z}^2$. 
More generally, 
$f$ in  (\ref{defi:fef}) is strict if and only if 
the induced morphism 
$$f' \col 
(E,\{E^{(i_1\cdots i_n)}_{k_1 \cdots k_n}\}_{1\leq i_1\leq \cdots \leq i_n\leq n,k_1,\ldots, k_n\in {\mab Z}}) 
\lo (F,\{F^{(i_1\cdots i_n)}_{k_1 \cdots k_n}\}_{1\leq i_1\leq \cdots \leq i_n\leq n,k_1,\ldots, k_n\in {\mab Z}})$$ 
is strict in the sense of [loc.~cit.].
\par 
Note also that the morphism 
$M\owns x\lom (x,-x)\in M\oplus M$   
does not induce a filtered morphism $(F,F^{(i)})\lo (E,E^{(i)})$ $(i=1,2)$ if $N\not=0$. 
Hence we cannot consider an exact sequence 
``$0\lo (F,F^{(i)})\lo (E,E^{(i)})\os{+}{\lo} (F,F^{(i)})\lo 0$''.
\end{rema}

\par 
Let $f \col (E,\{E^{(i)}\}_{i=1}^n) \lo (F,\{F^{(i)}\}_{i=1}^n)$ 
be a morphism in ${\rm MF}^n({\cal A})$. 
We say that $f$ is a {\it strict  injective morphism} 
(resp.~{\it strict surjective morphism}) if $f$ is strict 
and if the induced morphism $E\lo F$ is injective (resp.~if $f$ is strict 
and if the induced morphism $E\lo F$ is surjective).

\par
A complex with $n$-pieces of filtrations is, by definition, a complex 
$(E^{\bul},d)$ with $n$-pieces $\{P^{(i)}\}_{i=1}^n$ of filtrations 
such that 
$d(P^{(i)}_kE^q)\subset P^{(i)}_kE^{q+1}$. 
A morphism of complexes with $n$-pieces of filtrations is defined in an obvious way. 
Let ${\rm CF}^n({\cal A})$ be 
the category of complexes of ${\cal A}$-modules with $n$-pieces of filtrations 
and let 
${\rm C}^+{\rm F}^n({\cal A})$, 
${\rm C}^-{\rm F}^n({\cal A})$, 
${\rm C}^{\rm b}{\rm F}^n({\cal A})$
be the categories of bounded below, 
bounded above and bounded complexes of 
${\cal A}$-modules with $n$-pieces of filtrations, respectively. 
We define the notion of the $n$-filtered homotopy 
in an obvious way. 
\par
Let ${\rm KF}^n({\cal A})$ be 
the category of complexes of ${\cal A}$-modules with $n$-pieces of filtrations 
modulo $n$-filtered homotopies and
let ${\rm K}^+{\rm F}^n({\cal A})$, 
${\rm K}^-{\rm F}^n({\cal A})$, 
${\rm K}^{\rm b}{\rm F}^n({\cal A})$ be 
the categories of  bounded below, bounded above and
bounded complexes of ${\cal A}$-modules with $n$-pieces of filtrations 
modulo $n$-filtered homotopies, respectively.
For an object of ${\rm CF}^n({\cal A})$ 
or ${\rm KF}^n({\cal A})$, we define 
the direct image and the inverse image 
of an object of ${\rm CF}^n({\cal A})$ or ${\rm KF}^n({\cal A})$
by a morphism of ringed 
topoi in an obvious way. 
Since ${\rm MF}^n({\cal A})$ 
is an additive category,
${\rm K}^{\star}{\rm F}^n({\cal A})$ 
($\star=+$, $-$, b, nothing) 
is a triangulated category. 
For a complex 
$(E^{\bul},\{P^{(i)}\}_{i=1}^n)\in {\rm CF}^n({\cal A})$ 
with $n$-pieces of filtrations 
and for a sequence $\ul{l}:=(l_1,\ldots,l_n)$, 
we define the shift 
$(E^{\bul},\{P^{(i)}\}_{i=1}^n)\langle \ul{l} \rangle$ 
by $P^{(i)}\langle l_i \rangle_kE^{\bul}:=P^{(i)}_{l_i+k}E^{\bul}$
(\cite[(1.1)]{dh2}).  
In the following we often denote 
$(E^{\bul},\{P^{(i)}\}_{i=1}^n)$ by 
$(E^{\bul},\{E^{\bul(i)}\}_{i=1}^n)$.  
\par

\begin{prop-defi}\label{prop-defi:we}
$(1)$ For  $1 \leq i_1 \leq \cdots \leq  i_n \leq n$ 
and $k_1,\ldots, k_n \in {\mab Z}$,  
the following {\it intersection functor} 
\begin{equation*}
\bigcap^{(i_1\cdots i_n)}_{k_1\cdots k_n}
\col {\rm KF}^n({\cal A}) \owns 
(E^{\bul},\{E^{\bul(i)}\}_{i=1}^n) 
\lom E^{\bul(i_1\cdots i_n)}_{k_1\cdots k_n}=
\bigcap_{j=1}^{n}E^{\bul(i_j)}_{k_j}
\in K({\cal A})
\tag{2.4.1}\label{prop-defi:intf}
\end{equation*}
is well-defined. Here $K({\cal A})$ is the category of complexes of ${\cal A}$-modules 
modulo homotopy. 
\par 
$(2)$ 
For  $1 \leq i_1 \leq \cdots \leq  i_n \leq n$ 
and $k_1,\ldots, k_n \in {\mab Z}$,  
the following {\it gr functor} 
\begin{equation*}
{\rm gr}^{P^{(i_1)}}_{k_1} \cdots,{\rm gr}^{P^{(i_n)}}_{k_n}
\col {\rm KF}^n({\cal A}) \owns 
(E^{\bul},\{E^{\bul(i)}\}_{i=1}^n) 
\lom {\rm gr}^{P^{(i_1)}}_{k_1} \cdots,{\rm gr}^{P^{(i_n)}}_{k_n} (E^{\bul})
\in  K({\cal A}) 
\tag{2.4.2}\label{prop-defi:igrntf}
\end{equation*}
is well-defined. 
\end{prop-defi}
\begin{proof} 
This is easy to prove. 
\end{proof}

The following is a generalization of the strictly exactness 
defined in \cite{blec} (and \cite{nh2}).

\begin{defi}\label{defi:nfdf}
Assume that $n\leq 2$. 
We say that a complex 
$(E^{\bul},\{E^{\bul(i)}\}_{i=1}^n)\in {\rm CF}^n({\cal A})$ 
with $n$-pieces of filtrations is  
{\it strictly exact} if  the complexes 
$E^{\bul}$ and $E^{\bul(i_1i_2)}_{k_1k_2}$ 
$(1 \leq \forall i_1 \leq  \forall i_2 \leq n, 
\forall k_1,\forall k_2 \in {\mab Z})$ are exact. 
\end{defi}

The following simple remark is very important: 

\begin{rema}\label{rema:whs} 
For the case $n=1$, we have said that   
$(E^{\bul},\{E^{\bul(i)}\}_{i=1}^n)$ is strictly exact if 
$E^{\bul}$ and $E^{\bul(1)}$ is exact in \cite{blec} and \cite{nh2}. 
For the case $n=2$,  
the following two conditions is {\it not} necessary equivalent 
((2) does not imply (1) in general):
\par 
(1) $(E^{\bul},\{E^{\bul(i)}\}_{i=1}^2)$ is strictly exact.  
\par 
(2) $(E^{\bul},E^{\bul(i)})$ is strictly exact for $1\leq \forall i\leq 2$. 
\par 
Indeed, let the notations be as in (\ref{rema:obstm}).  
Let us consider the following sequence 
$$0\lo {\rm Ker}(f) \lo (E,\{E^{(i)}\}_{i=1,2}) \os{f}{\lo} (F,\{F^{(i)}\}_{i=1,2})\lo 0.$$
Then it is easy to check that 
$$0\lo {\rm Ker}(f^{(i)}) \lo (E,E^{(i)}) \os{f}{\lo} (F,F^{(i)})\lo 0$$ 
is exact for $i=1,2$.  Indeed, 
$F_{-1}^{(i)}{\rm Ker}(f^{(i)})=M$, $E^{(i)}_{-1}=M\oplus M$, 
$F^{(i)}_{-1}=M$, 
$F_0^{(i)}{\rm Ker}(f^{(i)})=0$, $E^{(1)}_0=N\oplus 0$, 
$E^{(2)}_0=0\oplus N$ and $F^{(i)}_0=N$. 
However    
$$0\lo {\rm Ker}(f^{(12)}) \lo (E,E^{(12)}) \os{f}{\lo} (F,F^{(12)})\lo 0$$ 
is not exact since $E^{(12)}_0=0$ and $F^{(12)}_0=N\not=0$. 
\end{rema}

\begin{prop}\label{prop:ncc}
Assume that $n\leq 2$. 
If a complex $(E^{\bul},\{E^{\bul(i)}\}_{i=1}^n)$ 
with $n$-pieces of filtrations
is strictly exact and  if 
$(F^{\bul},\{F^{\bul(i)}\}_{i=1}^n)
\simeq 
(E^{\bul},\{E^{\bul(i)}\}_{i=1}^n)$ 
in ${\rm KF}^n({\cal A})$, 
then $(F^{\bul},\{F^{\bul(i)}\}_{i=1}^n)$ is 
also strictly exact. 
\end{prop} 
\begin{proof}
This is easy to prove.   
\end{proof} 

\begin{prop}\label{prop:nsq}
Assume that $n\leq 2$. Let $1 \leq  i_1 \leq  i_2 \leq n$ be integers. 
Let $k_1$ and $k_2$ be integers.  
If $E^{\bul(i_1i_2)}_{l_1l_2}$ is exact for $(l_1,l_2)=(k_1,k_2)$, $(k_1-1,k_2)$, $(k_1,k_2-1)$ 
and $(k_1-1,k_2-1)$,  
then 
${\rm gr}^{P^{(i_1)}}_{k_1}{\rm gr}^{P^{(i_2)}}_{k_2}E^{\bul}$ is exact. 
\end{prop} 
\begin{proof} 
In the case $n=1$, this is obvious.  
Because 
\begin{align*} 
{\rm gr}^{P^{(i_1)}}_{k_1}{\rm gr}^{P^{(i_2)}}_{k_2}E^{\bul}&=
E^{\bul(i_1i_2)}_{k_1k_2}/(E^{{\bul(i_1i_2)}}_{k_1-1,k_2}+E^{\bul(i_1i_2)}_{k_1,k_2-1}), 
\tag{2.8.1}\label{ali:ppe}
\end{align*} 
we have only to prove that 
$E^{{\bul(i_1i_2)}}_{k_1-1,k_2}+E^{\bul(i_1i_2)}_{k_1,k_2-1}$ is exact. 
This follows from the following exact sequence
\begin{align*} 
0& \lo E^{{\bul(i_1i_2)}}_{k_1-1,k_2-1}\os{{\rm inc}.\oplus -{\rm inc}.}{\lo}  
E^{{\bul(i_1i_2)}}_{k_1-1,k_2}\oplus E^{\bul(i_1 i_2)}_{k_1,k_2-1}
\os{+}{\lo}  
E^{{\bul(i_1i_2)}}_{k_1-1,k_2}+E^{\bul(i_1i_2)}_{k_1,k_2-1}\lo 0.
\tag{2.8.2}\label{ali:pipe}
\end{align*}
\end{proof}

\begin{rema}
Consider the case $n=3$. 
By the definition of the quotient filtration, we obtain the following formulas 
for $1\leq i_1\leq  i_2  \leq i_3\leq n$:  
\begin{align*} 
{\rm gr}^{P^{(i_1)}}_{k_1}{\rm gr}^{P^{(i_2)}}_{k_2} {\rm gr}^{P^{(i_3)}}_{k_3}E^{\bul}=&
E^{\bul(i_1i_2i_3)}_{k_1k_2k_3}/\{E^{\bul(i_1i_2i_3)}_{k_1-1,k_2k_3}
+
E^{\bul(i_1)}_{k_1}\cap (E^{\bul(i_2i_3)}_{k_2-1,k_3}+E^{\bul(i_2i_3)}_{k_2,k_3-1})\}. 
\end{align*} 
Hence we would like to consider the following ${\cal A}$-module: 
\begin{align*} 
E^{\bul(i_1i_2i_3)}_{k_1-1,k_2k_3}\cap 
\{E^{\bul(i_1)}_{k_1}\cap (E^{\bul(i_2i_3)}_{k_2-1,k_3}+E^{\bul(i_2i_3)}_{k_2,k_3-1})\}
&=E^{\bul(i_1)}_{k_1-1}\cap 
E^{\bul(i_2i_3)}_{k_2k_3}\cap 
(E^{\bul(i_2i_3)}_{k_2-1,k_3}+E^{\bul(i_2i_3)}_{k_2,k_3-1})\\
&=E^{\bul(i_1i_2i_3)}_{k_1-1,k_2k_3}\cap 
(E^{\bul(i_2i_3)}_{k_2-1,k_3}+E^{\bul(i_2i_3)}_{k_2,k_3-1}). 
\end{align*} 
In the case where $n=3$, it is reasonable to say that 
$(E^{\bul},\{E^{\bul(i)}\}_{i=1}^n)$ is strictly exact if 
${\cal H}^q(E^{\bul})=0$, 
${\cal H}^q(E^{\bul(i_1i_2i_3)}_{k_1k_2k_3})=0$ 
and 
${\cal H}^q(E^{\bul(i_1i_2i_3)}_{k_1-1,k_2k_3}\cap 
(E^{\bul(i_2i_3)}_{k_2-1,k_3}+E^{\bul(i_2i_3)}_{k_2,k_3-1}))=0$ 
$(\forall q\in {\mab Z}, 1 \leq \forall i_1 \leq  \forall i_2 \leq \forall i_3 \leq n=3, 
\forall k_1,\forall k_2,\forall k_3\in {\mab Z})$. 
However I have not yet proved the existence of 
the strictly injective resolution of a bounded below complex 
with $3$-pieces of filtrations because of the complicated term 
$E^{\bul(i_1i_2i_3)}_{k_1-1,k_2k_3}\cap 
(E^{\bul(i_2i_3)}_{k_2-1,k_3}+E^{\bul(i_2i_3)}_{k_2,k_3-1})$. 
In this reason we discuss the derived category of complexes with $n$-pieces of filtrations 
under the assumption $n\leq 2$ in this paper. 
\end{rema}

\begin{prop}\label{prop:exbcb}
Let the notations 
be as in {\rm (\ref{prop:nsq})}. 
Let $p$ be an integer. 
If the boundary morphism 
$d^{p-1}\col (E^{p-1},\{E^{p-1(i)}\}_{i=1}^n)\lo (E^{p},\{E^{p(i)}\}_{i=1}^n)$ 
is strict, 
then 
\begin{align*} 
{\rm Im}({\rm gr}^{P^{(i_1)}}_{k_1}{\rm gr}^{P^{(i_2)}}_{k_2}E^{p-1}
\lo {\rm gr}^{P^{(i_1)}}_{k_1}{\rm gr}^{P^{(i_2)}}_{k_2}E^{p})
={\rm gr}^{P^{(i_1)}}_{k_1}{\rm gr}^{P^{(i_2)}}_{k_2}
{\rm Im}(E^{p-1}\lo E^p).
\tag{2.10.1}\label{ali:pilpe}
\end{align*} 
\end{prop}
\begin{proof}
The left hand side of (\ref{ali:pilpe}) is equal to 
\begin{align*} 
{\rm Im}(E^{{p-1(i_1i_2)}}_{k_1k_2}/
(E^{{p-1(i_1i_2)}}_{k_1-1,k_2}+ E^{p-1(i_1 i_2)}_{k_1,k_2-1})
\lo 
E^{{p(i_1i_2)}}_{k_1k_2}/
(E^{{p(i_1i_2)}}_{k_1-1,k_2}+ E^{p(i_1 i_2)}_{k_1,k_2-1})).
\end{align*}
The right hand side of (\ref{ali:pilpe}) is equal to 
\begin{align*} 
{\rm Im}(E^{p-1}\lo E^p)^{(i_1i_2)}_{k_1k_2}/
({\rm Im}(E^{p-1}\lo E^p)^{(i_1i_2)}_{k_1-1,k_2}+
{\rm Im}(E^{p-1}\lo E^p)^{(i_1i_2)}_{k_1,k_2-1}). 
\end{align*}
Hence we have a natural morphism from the left hand side of (\ref{ali:pilpe}) 
to the right hand side of (\ref{ali:pilpe}). 
By the assumption of the strictness of $d^{p-1}$, 
we see that this natural morphism is an isomorphism. 
\end{proof}

\begin{defi}\label{defi:nfm}
Assume that $n\leq 2$. 
We say that a filtered morphism 
\begin{align*} 
f \col (E^{\bul},\{E^{\bul(i)}\}_{i=1}^n) 
\lo (F^{\bul},\{F^{\bul(i)}\}_{i=1}^n)
\tag{2.11.1}\label{eqn:gefz}
\end{align*} 
in ${\rm CF}^n({\cal A})$ is an $n$-{\it filtered quasi-isomorphism} 
(or simply a {\it filtered quasi-isomorphism}) if 
the induced morphisms
\begin{align*} 
f\col  E\lo F
\tag{2.11.2}\label{eqn:gbaqz}
\end{align*} 
and   
\begin{align*} 
f\col  E^{\bul(i_1i_n)}_{k_1 k_n}
\lo F^{\bul(i_1i_n)}_{k_1 k_n}
\tag{2.11.3}\label{eqn:gbqz}
\end{align*}   
are quasi-isomorphisms for 
$1 \leq  \forall i_1 \leq  \forall i_n \leq n$ and 
$\forall k_1, \forall k_n \in {\mab Z}$. 
\end{defi}


\begin{rema}\label{rema:ba}
For the morphism (\ref{eqn:gefz}) we can define the $n$-filtered complex 
$({\rm MC}(f), \{{\rm MC}(f)^{(i)}\}_{i=1}^n)\in {\rm CF}^n({\cal A})$ in a natural way.  
It is obvious that $f$ is an $n$-filtered quasi-isomorphism 
if and only if $({\rm MC}(f), \{{\rm MC}(f)^{(i)}\}_{i=1}^n)$ is strictly exact. 
\end{rema}

\par
By the following proposition, we see that our definition above 
is equivalent to the definition 
in \cite[(1.3.6) (i), (ii)]{dh2} in the case $n=1, 2$ 
if the filtrations are biregular. 

\begin{prop}\label{prop:pq}
Assume that $n\leq 2$. 
Let 
$f \col (E^{\bul},\{P^{(i)}\}_{i=1}^n)\lo (F^{\bul},\{Q^{(i)}\}_{i=1}^n)$ 
be an $n$-filtered morphism in ${\rm CF}^n({\cal A})$. 
Assume that the filtrations $P^{(i)}$ and $Q^{(i)}$
$(1 \leq i \leq n)$ are biregular.
Then the following are equivalent$:$
\par 
$(1)$ The morphism $f$ is an $n$-filtered quasi-isomorphism. 
\par 
$(2)$ 
The morphism 
\begin{equation*} 
{\rm gr}(f) \col 
{\rm gr}^{P^{(1)}}_{k_1}{\rm gr}^{P^{(m)}}_{k_m}
E^{\bul} \lo 
{\rm gr}^{Q^{(1)}}_{k_1}{\rm gr}^{Q^{(m)}}_{k_m}F^{\bul} 
\tag{2.13.1}\label{eqn:dgrfis}
\end{equation*}
for any $k_1, k_m \in {\mab Z}$ and for $m\leq n$ 
is a quasi-isomorphism.  
\par 
Furthermore, if $n=2$, then $(1)$ 
is equivalent to the following$:$
\par 
$(3)$ 
The morphism 
\begin{equation*} 
{\rm gr}(f) \col 
{\rm gr}^{P^{(1)}}_{k_1}{\rm gr}^{P^{(2)}}_{k_2}
E^{\bul} \lo 
{\rm gr}^{Q^{(1)}}_{k_1}{\rm gr}^{Q^{(2)}}_{k_2}F^{\bul} 
\tag{2.13.2}\label{eqn:dgrfqis}
\end{equation*}
for any $k_1, k_2 \in {\mab Z}$
is a quasi-isomorphism.  
\end{prop}
\begin{proof} 
Because (\ref{prop:pq})  is easy to prove in the case $n=1$, 
we prove (\ref{prop:pq}) only in the case $n=2$. 
First assume that $f$ is a bifiltered quasi-isomorphism. 
Then the morphism ${\rm gr}(f)$'s in (\ref{eqn:dgrfis}) and (\ref{eqn:dgrfqis})
are quasi-isomorphisms by 
(\ref{ali:ppe}) and (\ref{ali:pipe}). 
\par 
Because (2) implies (3), it suffices to prove that (3) implies (1). 
Assume that (\ref{eqn:dgrfqis}) is a quasi-isomorphism. 
Let $q$ be an integer. 
Let $\bul$ be $q$ or $q\pm 1$. 
Let $k_1, k_2, l_1, l_2$ be integers such that 
$E^{\bul(1)}_{k_1}=0$, $E^{\bul(1)}_{l_1}=E^{\bul}$, $E^{\bul(2)}_{k_2}=0$ 
and $E^{\bul(2)}_{l_2}=E^{\bul}$. 
Then 
$${\rm gr}^{P^{(1)}}_{k_1+1}{\rm gr}^{P^{(2)}}_{k_2+1}E^{\bul}=
E^{\bul(12)}_{k_1+1,k_2+1}$$
and  
$${\rm gr}^{P^{(1)}}_{k_1+m}{\rm gr}^{P^{(2)}}_{k_2+1}E^{\bul}=
E^{\bul(12)}_{k_1+m,k_2+1}/E^{\bul}_{k_1+m-1,k_2+1}\quad (m\in {\mab N}).$$ 
Using the induction on $k_1+m$, we see that the morphism 
$$E^{\bul(i_1i_2)}_{m_1,k_2+1}\lo F^{\bul(i_1i_2)}_{m_1,k_2+1}$$
is a quasi-isomorphism for any $m_1\in {\mab Z}$. 
Furthermore, 
$${\rm gr}^{P^{(1)}}_{k_1+1}{\rm gr}^{P^{(2)}}_{k_2+2}E^{\bul}=
E^{\bul(12)}_{k_1+1,k_2+2}/E^{\bul(12)}_{k_1+1,k_2+1}$$ 
and 
$${\rm gr}^{P^{(1)}}_{k_1+m}{\rm gr}^{P^{(2)}}_{k_2+2}E^{\bul}=
E^{\bul(12)}_{k_1+m,k_2+2}/
(E^{\bul(12)}_{k_1+m-1,k_2+2}+E^{\bul(12)}_{k_1+m,k_2+1}) \quad (m\in {\mab N}).$$ 
By using (\ref{ali:pipe}), the induction on $k_1+m$ tells us that the morphism 
$$E^{\bul(i_1i_2)}_{m_1,k_2+2}\lo F^{\bul(i_1i_2)}_{m_1,k_2+2}$$
is a quasi-isomorphism for any $m_1\in {\mab Z}$. 
More generally, 
\begin{align*} 
{\rm gr}^{P^{(i_1)}}_{k_1+m}{\rm gr}^{P^{(i_2)}}_{k_2+l}E^{\bul}&=
E^{\bul(i_1i_2)}_{k_1+m,k_2+l}/(E^{\bul(i_1i_2)}_{k_1+m-1,k_2+l}+
E^{\bul(i_1i_2)}_{k_1+m,k_2+l-1}). 
\end{align*} 
By using (\ref{ali:pipe}) again, the induction tells us that 
the morphism 
$$ E^{\bul(i_1i_2)}_{m_1m_2}\lo F^{\bul(i_1i_2)}_{m_1m_2}$$
is a quasi-isomorphism for any $m_1,m_2\in {\mab Z}$. 
Since the filtrations $P^{(1)}$ and $P^{(2)}$ (resp.~$Q^{(1)}$ and $Q^{(2)}$) on 
$E^{\bul}$ (resp.~$F^{\bul}$) for $\bul= q, q\pm 1$ are finite, 
the morphism (\ref{eqn:gbaqz})  is a quasi-isomorphism. 
\end{proof}

\par 
Assume that $n\leq 2$. 
Let us consider the set of morphisms 
$({\rm F}^n{\rm Qis})$ whose elements are 
the $n$-filtered quasi-isomorphisms  in 
${\rm KF}^n({\cal A})$.
Then it is easy to see that 
$({\rm F}^n{\rm Qis})$ forms a 
saturated multiplicative system 
which is compatible with the triangulation
in the sense of 
\cite[II SM1)$\sim$SM6) p.~112]{v}. 
Set  
${\rm D}^{\star}{\rm F}^n({\cal A})
:={\rm K}^{\star}{\rm F}^n({\cal A})_{({\rm F}^n{\rm Qis})}$ 
($\star=+$, $-$,  b, nothing).

\begin{defi} 
We call ${\rm D}^+{\rm F}^n({\cal A})$, ${\rm D}^-{\rm F}^n({\cal A})$ 
and ${\rm D}^{\rm b}{\rm F}^n({\cal A})$ the {\it derived category of bounded below 
complexes of ${\cal A}$-modules with $n$-pieces of filtrations}, 
the {\it derived category of bounded above 
complexes of ${\cal A}$-modules with $n$-pieces of filtrations} 
and the {\it derived category of bounded complexes of ${\cal A}$-modules 
with $n$-pieces of filtrations}, respectively. 
\end{defi}

\par
In the rest of this section, we give notations which will be necessary in later sections. 
\par 
Assume that $n\leq 2$. 
Let ${\rm K}^{\star}(\Gam({\cal T},{\cal A}))$ 
($\star=+$, $-$, b, nothing) be the category of 
complexes of $\Gam({\cal T},{\cal A})$-modules 
modulo homotopies  
with respect to $\star=+,-, {\rm b}, {\rm nothing}$.  
Let ${\rm D}^{\star}(\Gam({\cal T},{\cal A})):=
{\rm K}^{\star}(\Gam({\cal T},{\cal A}))_{({\rm Qis})}$ 
be its derived category. 
Let 
${\rm MF}^n(\Gam({\cal T},{\cal A}))$ and 
${\rm C}^{\star}{\rm F}^n(\Gam({\cal T},{\cal A}))$   
be the categories of 
$\Gam({\cal T},{\cal A})$-modules with $n$-pieces of filtrations 
and that of complexes of  
$\Gam({\cal T},{\cal A})$-modules with $n$-pieces of filtrations, respectively, 
with respect to $\star=+,-, {\rm b}, {\rm nothing}$.   
Let 
${\rm K}^{\star}{\rm F}^n(\Gam({\cal T},{\cal A}))$ 
be the category of complexes of  
$\Gam({\cal T},{\cal A})$-modules with $n$-pieces of filtrations
modulo $n$-filtered homotopies with respect to $\star=+,-, {\rm b}, {\rm nothing}$.  
By abuse of notation, let $({\rm F}^n{\rm Qis})$
be the set of morphisms whose elements are 
the $n$-filtered quasi-isomorphisms  in 
${\rm K}^{\star}{\rm F}^n(\Gam({\cal T},{\cal A}))$.
Let 
${\rm D}^{\star}{\rm F}^n(\Gam({\cal T},{\cal A}))$ 
be the derived 
category of ${\rm K}^{\star}{\rm F}^n
(\Gam({\cal T},{\cal A}))$ 
localized by the set $({\rm F}^n{\rm Qis})$ of $n$-filtered quasi-isomorphisms of 
complexes of $\Gam({\cal T},{\cal A})$-modules. 

\section{Strictly injective resolutions and $Rf_*$}\label{sec:sti} 
In this section we give the definitions of the specially injective resolution and 
the strictly injective resolution of a complex with $n$-pieces of filtrations, 
which are generalizations of the definitions in \cite{blec} and \cite{nh2}. 
We also define the right derived functor 
$Rf_*\col {\rm D}^+{\rm F}^n({\cal A})\lo {\rm D}^+{\rm F}^n({\cal A}')$ $(n=1,2)$
for a morphism $f\col ({\cal T},{\cal A})\lo ({\cal T}',{\cal A}')$ 
of ringed topoi.  
\par
The {\it special filtered module} in \cite{blec} (and \cite{nh2}) 
can be generalized in the following way for a general positive integer $n$. 
For ${\cal A}$-modules $F$, $F^{(1)}_{l_1},\ldots, F^{(n)}_{l_n}$ ($l_1,\ldots,l_n\in {\mab Z}$), 
we set 
$$\prod_{\cal A}(F,\{F^{(i)}\}_{i=1}^n):=
F\times  \us{l_1 \in {\mab Z}}{\prod}F^{(1)}_{l_1}  
\times \cdots \times \us{l_n \in {\mab Z}}{\prod}F^{(n)}_{l_n}$$ 
with $n$-pieces of filtrations  $\prod^{(i)}_{\cal A}$'s on $\prod_{\cal A}(F,\{F^{(i)}\}_{i=1}^n)$ 
defined by 
$$(\prod^{(i)}_{\cal A}(F,\{F^{(j)}\}_{j=1}^n))_k:= 
F\times  \prod_{m=1}^{i-1}\us{l_m \in {\mab Z}}{\prod}F^{(m)}_{l_m} 
\times 
\us{l_i \leq k}{\prod}F^{(i)}_{l_i} \times 
\prod_{m=i+1}^{n}\us{l_m \in {\mab Z}}{\prod}F^{(m)}_{l_m},$$
where $F^{(i)}:= \{F^{(i)}_{l_i}\}_{l_i\in {\mab Z}}$ $(i=1,\ldots,n)$. 
Then we have an ${\cal A}$-module with $n$-pieces of filtrations: 
$$
(\prod_{\cal A}(F,\{F^{(i)}\}_{i=1}^n)
,\{\prod^{(i)}_{\cal A}\}_{i=1}^n).$$
By abuse of notation, we denote this filtered module 
simply by $\prod_{\cal A}(F,\{F^{(i)}\}_{i=1}^n)$. 
We call $\prod_{\cal A}(F,\{F^{(i)}\}_{i=1}^n)$ the {\it special $n$-filtered module} of 
$F$, $F^{(1)}_{l_1},\ldots, F^{(n)}_{l_n}$.

The following formula is a generalization of the formula 
in \cite{blec} (and \cite[(1.1.0.2)]{nh2}):

\begin{prop}\label{prop:mor}
The following formula holds$:$ 
\begin{align*}
{\rm Hom}_{{\rm MF}^n({\cal A})}((E,\{E^{(i)}\}_{i=1}^n), 
\prod_{\cal A}(F,\{F^{(i)}\}_{i=1}^n))= 
& {\rm Hom}_{\cal A}(E,F)\times 
\tag{3.1.1}\label{eqn:imspm}
\\ 
{} & {\prod}_{i=1}^n
{\prod}_{k\in {\mab Z}}{\rm Hom}_{\cal A}
(E/E^{(i)}_{k-1}, F^{(i)}_k). 
\end{align*} 
\end{prop}
\begin{proof} 
Assume that  a morphism 
$f\col (E,\{E^{(i)}\}_{i=1}^n)\lo \prod_{\cal A}(F,\{F^{(i)}\}_{i=1}^n)$ 
in ${\rm MF}^n({\cal A})$ is given.   
Then we have morphisms $E\lo F$ and $E/E^{(i)}_{k-1}\lo F^{(i)}_k$ of 
${\cal A}$-modules by using the projections 
$\prod_{\cal A}(F,\{F^{(i)}\}_{i=1}^n)\lo F$ and 
$\prod_{\cal A}(F,\{F^{(i)}\}_{i=1}^n)\lo F^{(i)}_k$, respectively. 
\par 
Assume that morphisms 
$E\lo F$ and $E/E^{(i)}_{k-1}\lo F^{(i)}_k$ for any $1\leq i\leq n$ and $k\in {\mab Z}$ 
are given. 
Then we have a morphism 
$E\os{{\rm proj}.}{\lo}  E/E^{(i)}_{k-1}\lo F^{(i)}_k$. 
Hence we have a morphism 
$E\lo F\times  \us{l_1 \in {\mab Z}}{\prod}F^{(1)}_{l_1}   
\times \cdots \times \us{l_n \in {\mab Z}}{\prod}F^{(n)}_{l_n}$. 
This morphism induces a filtered morphism 
$(E,\{E^{(i)}\}_{i=1}^n)\lo \prod_{\cal A}(F,\{F^{(j)}\}_{j=1}^n)$ 
since the composite morphism $E/E^{(i)}_{k}\lo E/E^{(i)}_{l-1}\lo F^{(i)}_l$ 
for $l>k$ is a zero morphism.  
\end{proof}

\par 
For two objects $(E,\{E^{(i)}\}_{i=1}^n)$, 
$(F,\{F^{(i)}\}_{i=1}^n))\in {\rm MF}^n({\cal A})$, 
we define
\begin{align*} 
{\rm Hom}_{\cal A}((E,\{E^{(i)}\}_{i=1}^n),(F,\{F^{(i)}\}_{i=1}^n))
:=&
({\rm Hom}_{\cal A}(E,F), \{{\rm Hom}^{(i)}_{\cal A}(E,F)\}_{i=1}^n)
\tag{3.1.2}\label{ali:efa}\\
& \in {\rm MF}^n
(\Gam({\cal T},{\cal A}))
\end{align*} 
in a well-known way:  
\begin{align*}
{\rm Hom}^{(i)}_{\cal A}(E,F)_k:=
\{f \in {\rm Hom}_{\cal A}(E,F)~\vert~
f(E^{(i)}_l) \subset F^{(i)}_{l+k}~(\forall l \in {\mab Z})\}.
\tag{3.1.3}\label{ali:eifa}
\end{align*} 
For $1\leq i_1 \leq \cdots \leq i_n\leq n$ and $k_1, \ldots, k_n\in {\mab Z}$, 
it is obvious that the following formula holds: 
\begin{align*} 
{\rm Hom}^{(i_1)}_{\cal A}(E,F)_{k_1}  \cap \cdots \cap 
{\rm Hom}^{(i_n)}_{\cal A}(E,F)_{k_n}=
{\rm Hom}_{{\rm MF}^n({\cal A})}((E,\{E^{(i)}\}_{i=1}^n,
(F,\{F^{(i)}\}_{i=1}^n)\langle k_1,\ldots,k_n\rangle). 
\tag{3.1.4}\label{ali:ikpm}
\end{align*} 
\par
We obtain a similar object 
${\cal H}{\it om}_{\cal A}((E,\{E^{(i)}\}_{i=1}^n),
(F,\{F^{(i)}\}_{i=1}^n))\in 
{\rm MF}^n({\cal A})$
in a similar way.
\par
In the following we assume that $n\leq 2$. 
The following definition is a nontrivial generalization of 
the definition due to Berthelot (\cite{blec}):

\begin{defi}[{\bf \cite{blec} for the case $n=1$}]\label{defi:stinj}
Assume that $n\leq 2$. 
We say that  an object 
$(J,\{J^{(i)}\}_{i=1}^n)$ of ${\rm MF}^n({\cal A})$ is 
{\it strictly injective} if it satisfies the 
following two conditions:
\par 
(1) $J$ and $J^{(i_1i_n)}_{k_1k_n}$
are injective ${\cal A}$-modules 
for $1\leq \forall i_1 \leq  \forall i_n \leq n, 
\forall k_1, \forall k_n \in {\mab Z}$.  
\par 
(2) For a strictly injective morphism 
$$(E,\{E^{(i)}\}_{i=1}^n) 
\os{\sus}{\lo} (F,\{F^{(i)}\}_{i=1}^n)$$
in  ${\rm MF}^n({\cal A})$, 
the induced morphism 
$${\rm Hom}_{\cal A}((F,\{F^{(i)}\}_{i=1}^n),
(J,\{J^{(i)}\}_{i=1}^n)) \lo 
{\rm Hom}_{\cal A}((E,\{E^{(i)}\}_{i=1}^n),
(J,\{J^{(i)}\}_{i=1}^n))$$ 
is an epimorphism.
\end{defi}
\par
In the following we assume that $n\leq 2$. 
Set 
\begin{align*}
{\cal I}^n_{\rm flas}({\cal A})
:=\{(J,\{J^{(i)}\}_{i=1}^n)\in {\rm MF}^n({\cal A})~\vert & ~
J \text{ and } J^{(i_1i_n)}_{k_1k_n} \text{ are flasque } 
{\cal A}{\textrm -}\text{modules} \tag{3.2.1}\\  
{} & (1\leq 
\forall i_1 \leq \forall i_n \leq n, \forall k_1, \forall k_n \in {\mab Z})\}, 
\end{align*}

\begin{rema}
One need not assume that $J$ is an injective ${\cal A}$-module 
in the definition (\ref{defi:stinj}) 
because one can prove that the property (2) in (\ref{defi:stinj}) implies that 
$J$ is an injective ${\cal A}$-module. 
Indeed, assume that 
$(J,\{J^{(i)}\}_{i=1}^n)$ enjoys the property (2). 
Let $J\os{\subset}{\lo} K$ be an injective morphism into an injective ${\cal A}$-module. 
Consider the identity morphism ${\rm id} \col (J,\{J^{(i)}\}_{i=1}^n)\lo (J,\{J^{(i)}\}_{i=1}^n)$. 
Endow $K$ with the filtration $K^{(i)}:={\rm Im}(J^{(i)}\lo K)$ $(i=1,2)$. 
The induced morphism 
$(J,\{J^{(i)}\}_{i=1}^n) \lo (K,\{K^{(i)}\}_{i=1}^n)$ is strict by (\ref{prop:usef}). 
Hence there exists a projection 
$p\col (K,\{K^{(i)}\}_{i=1}^n\lo (J,\{J^{(i)}\}_{i=1}^n)$ of the strict injective morphism 
$(J,\{J^{(i)}\}_{i=1}^n) \lo (K,\{K^{(i)}\}_{i=1}^n)$. 
Set $L:={\rm Ker}(p\col K\lo J)$.  
Now it is obvious that $K=J\oplus L$  and it is easy to see that 
$J$ is an injective ${\cal A}$-module. 
\end{rema}

\begin{align*}
{\cal I}^n_{\rm inj}({\cal A})
:=\{(J,\{J^{(i)}\}_{i=1}^n)\in {\rm MF}^n({\cal A})~\vert & ~
J \text{ and } J^{(i_1i_n)}_{k_1k_n} \text{ are injective } 
{\cal A}{\textrm -}\text{modules} \tag{3.3.1}\\  
{} & (\forall k_1, \forall k_n \in {\mab Z},~1\leq \forall 
i_1 \leq \forall i_n \leq n)\}, 
\end{align*}

\begin{equation*}
{\cal I}^n_{\rm stinj}({\cal A}):=\{(J,\{J^{(i)}\}_{i=1}^n)\in {\rm MF}^n({\cal A})~\vert~
\text{$(J,\{J^{(i)}\}_{i=1}^n)$ 
is a strictly injective  
${\cal A}${\textrm -}module$\}$}, 
\tag{3.3.2}
\end{equation*}

\begin{equation*}
{\cal I}^n_{\rm spinj}({\cal A})
:=\{ \prod_{\cal A}(I,\{I^{(i)}\}_{i=1}^n)~\vert~
\text{$I$  and $I^{(i)}_k$ are injective 
${\cal A}${\textrm -}modules}~ 
(\forall k\in {\mab Z},1 \leq \forall i \leq n)\}. 
\tag{3.3.3}
\end{equation*}

\parno
Then ${\cal I}^n_{\rm stinj}({\cal A})
\subset {\cal I}^n_{\rm inj}({\cal A})
\subset {\cal I}^n_{\rm flas}({\cal A})$.
The following is a generalization of  a result of Berthelot 
\cite{blec} (and \cite[(1.1.2)]{nh2}):

\begin{prop}[{\bf \cite{blec} for the case $n=1$}]\label{prop:spijstij} 
Assume that $\prod_{\cal A}(I,\{I^{(i)}\}_{i=1}^n) 
\in {\cal I}^n_{\rm spinj}({\cal A})$.
Then $\prod_{\cal A}(I,\{I^{(i)}\}_{i=1}^n)\in 
{\cal I}^n_{\rm stinj}({\cal A})$.
\end{prop}
\begin{proof}
Note that, for $1\leq i_1< i_m\leq n$ $(1\leq m\leq n)$ 
and $k_1, \ldots, k_m\in {\mab Z}$,  
$$((\prod_{\cal A}(I,\{I^{(i)}\}_{i=1}^n)))^{(i_1i_m)}_{k_1k_m} 
=I\times  \us{l_1\leq k_1, l_m\leq k_m}{\prod}
(I^{(i_1)}_{l_1}  \times I^{(i_m)}_{l_m})
\times
\us{j \not\in \{i_1, i_m\}}{\prod}\us{k\in {\mab Z}}
{\prod}I^{(j)}_k.$$ 
Hence $(\prod_{\cal A}(I,\{I^{(i)}\}_{i=1}^n)))^{(i_1i_m)}_{k_1k_m}$ is 
an injective ${\cal A}$-module. 
\par 
Let $(E,\{E^{(i)}\}_{i=1}^n) 
\os{\sus}{\lo} (F,\{F^{(i)}\}_{i=1}^n)$ be a strictly injective morphism. 
Then the induced morphism  
$E/E^{(i)}_{k-1}\lo F/F^{(i)}_{k-1}$ is injective. 
Now (\ref{prop:spijstij}) is obvious by (\ref{eqn:imspm}) 
since 
$I$ and $I^{(i)}_k$ are injective. 
\end{proof}

\begin{defi}
We say that an ${\cal A}$-module 
$(J,\{J^{(i)}\}_{i=1}^n)\in {\rm MF}^n({\cal A})$ 
with $n$-pieces of filtrations is 
$n$-{\it filteredly flasque}, $n$-{\it filteredly injective} 
and $n$-{\it specially injective} 
if $(J,\{J^{(i)}\}_{i=1}^n)\in 
{\cal I}^n_{\rm flas}({\cal A})$, 
$\in {\cal I}^n_{\rm inj}({\cal A})$ 
and
$\in {\cal I}^n_{\rm spinj}({\cal A})$, 
respectively.
\end{defi}

\parno
The category ${\rm MF}^n({\cal A})$ has 
enough special injectives in the following sense
(the following is a generalization of a result in \cite{blec}
(and \cite[(1.1.4)]{nh2})):

\begin{prop}[{\bf \cite{blec} for the case $n=1$}]\label{prop:enspij}
For an object  
$(E,\{E^{(i)}\}_{i=1}^n)\in {\rm MF}^n({\cal A})$, 
there exists a strictly injective morphism 
$(E,\{E^{(i)}\}_{i=1}^n) 
\os{\sus}{\lo} \prod_{\cal A}(I,\{I^{(i)}\}_{i=1}^n)$ 
with 
$\prod_{\cal A}(I,\{I^{(i)}\}_{i=1}^n)
\in {\cal I}^n_{\rm spinj}({\cal A})$.
\end{prop}
\begin{proof}
Though the following proof is similar to that of \cite[(1.1.4)]{nh2}, 
we give the complete proof because we use (\ref{prop:usef}) 
and we use the following proof in the proof of (\ref{prop:exsir}) below. 
\par 
Let $E \os{\sus}{\lo} I$ and 
$E/E^{(i)}_{l-1} \os{\sus}{\lo} I^{(i)}_l$ 
be  injective morphisms into injective ${\cal A}$-modules.
Set $J :=I\times {\prod}_{i=1}^n\us{l \in{\mab Z}}{\prod}I^{(i)}_{l}$ and 
$J^{(i)}_k :=I\times 
\prod_{m=1}^{i-1}\us{l_m \in {\mab Z}}{\prod}I^{(m)}_{l_m} 
\times \us{l_i \leq k}{\prod}I^{(i)}_{l_i} \times 
\prod_{m=i+1}^{n}\us{l_m \in {\mab Z}}{\prod}I^{(m)}_{l_m}$. 
There are natural injective morphisms
$J^{(i)}_k \os{\subset}{\lo} J^{(i)}_{k+1}$ and $J^{(i)}_k 
\os{\subset}{\lo} J$. Since $I$ and $I^{(i)}_k$ 
are injective ${\cal A}$-modules, 
$(J, \{J^{(i)}\}_{i=1}^n)$ is an object of 
${\cal I}^{n}_{\rm spinj}({\cal A})$.
\par
Two morphisms $E \os{\sus}{\lo} I$ and 
$E\os{\rm proj.}{\lo}E/E^{(i)}_{l-1} \os{\sus}{\lo} I^{(i)}_l$ 
induce an injective morphism
$E \lo J$. Furthermore, two composite morphisms 
$E^{(i)}_k \os{\subset}{\lo}E\os{\sus}{\lo} I$ and 
$E^{(i)}_k\os{\subset}{\lo}E \os{{\rm proj}.}{\lo}
E/E^{(i)}_{l-1} \os{\sus}{\lo} I^{(i)}_l$  induce an injective morphism
$E^{(i)}_k \lo J^{(i)}_k$. 
\par
It remains to prove that
the morphism $(E,\{E^{(i)}\}_{i=1}^n) \lo 
(J,\{J^{(i)}\}_{i=1}^n)$ is strict. 
By (\ref{prop:usef}), it suffices to prove that 
the morphism $(E,E^{(i)}) \lo (J,J^{(i)})$ is strict.
Set $N^{(i)}_k:={\rm Im}(E \lo J)\cap J^{(i)}_k$. 
Then  $N^{(i)}_k$ is isomorphic to 
the kernel of the following composite morphism 
$$E \lo J \os{\rm proj.}{\lo} \us{l> k}{\prod}I^{(i)}_l.$$
This kernel is nothing but $E^{(i)}_k$ 
by the definition of $I^{(i)}_l$ $(l >k)$. 
Hence the morphism $(E,E^{(i)}) \lo (J,J^{(i)})$ is strict. 
\end{proof}

\begin{defi}
Assume that $n\leq 2$. Let $(E^{\bul},\{E^{\bul(i)}\}_{i=1}^n)$ be an object of 
${\rm K}^+{\rm F}^n({\cal A})$.  
\par 
$(1)$ {\bf (\cite{blec} for the case $n=1$)} 
We say that an object  
$(J^{\bul},\{J^{\bul(i)}\}_{i=1}^n)\in 
{\rm K}^+{\rm F}^n({\cal A})$ 
with an $n$-filtered morphism 
$(E^{\bul},\{E^{\bul(i)}\}_{i=1}^n) \lo 
(J^{\bul},\{J^{\bul(i)}\}_{i=1}^n)$
is a {\it strictly injective resolution} of  
$(E^{\bul},\{E^{\bul(i)}\}_{i=1}^n)$ if 
$(J^{q},\{J^{q(i)}\}_{i=1}^n)\in 
{\cal I}^n_{\rm stinj}({\cal A})$ 
for any $q\in {\mab Z}$  
and if the morphism 
$(E^{\bul},\{E^{\bul(i)}\}_{i=1}^n) 
{\lo} (J^{\bul},\{J^{\bul(i)}\}_{i=1}^n)$ 
is an $n$-filtered quasi-isomorphism which 
induces a strictly injective morphism 
$(E^q,\{E^{q(i)}\}_{i=1}^n) 
{\lo} (J^q,\{J^{q(i)}\}_{i=1}^n)$
for any $q \in {\mab Z}$. 
\par
$(2)$ We say that an object  
$(J^{\bul},\{J^{\bul(i)}\}_{i=1}^n)\in 
{\rm K}^{+}{\rm F}^n({\cal A})$ 
with an $n$-filtered morphism 
$(E^{\bul},\{E^{\bul(i)}\}_{i=1}^n) \lo 
(J^{\bul},\{J^{\bul(i)}\}_{i=1}^n)$ 
is an $n$-{\it filtered flasque resolution}, 
an $n$-{\it filtered injective resolution}
and an $n$-{\it specially injective resolution} 
of  
$(E^{\bul},\{E^{\bul(i)}\}_{i=1}^n)$ 
if $(J^{q},\{J^{q(i)}\}_{i=1}^n)
\in {\cal I}^n_{\rm flas}({\cal A})$, 
$\in {\cal I}^n_{\rm inj}({\cal A})$ and 
$\in {\cal I}^n_{\rm spinj}({\cal A})$, respectively, 
for any $q\in {\mab Z}$ 
and if the morphism 
$(E^{\bul},\{E^{\bul(i)}\}_{i=1}^n) 
{\lo} (J^{\bul},\{J^{\bul(i)}\}_{i=1}^n)$ 
is an $n$-filtered quasi-isomorphism which induces 
a strictly injective morphism 
$(E^q,\{E^{q(i)}\}_{i=1}^n) 
{\lo} (J^q,\{J^{q(i)}\}_{i=1}^n)$
for any $q \in {\mab Z}$. 
\end{defi}

\begin{rema}
Assume that $n\leq 2$. 
Let $({\cal T}, {\cal A})$ be 
a ringed topos with enough points.
Let $(F^{\bul},\{F^{\bul(i)}\}_{i=1}^n)$ be 
an object of 
${\rm C}^+{\rm F}^n({\cal A})$. 
For an integer $p$, let 
$(I^{pq}, I^{pq(i)}_k,d^{pq})_{q\in {\mab Z}_{\geq 0}}$
be the Godement resolution of $(F^p,F^{p(i)}_k)$. 
Then  the sequence 
$$0\lo (F^{p},F^{p(i)}_k) \lo 
(I^{p0(i)},I^{p0(i)}_k)\os{(-1)^pd^{p0}}{\lo} 
(I^{p1(i)},I^{p1(i)}_k) \os{(-1)^pd^{p1}}{\lo} 
\cdots \quad (p,k\in {\mab Z})$$ 
gives an $n$-filtered flasque resolution 
$s(I^{\bul \bul},\{I^{\bul \bul(i)}_k\}_{i=1}^n)$ of 
$(F^{\bul},\{F^{\bul(i)}\}_{i=1}^n)$ because 
taking a point of ${\cal T}$ and the direct image of abelian sheaves by 
a morphism of ringed topoi are compatible 
with the finite intersection of ${\cal A}$-modules. 
\end{rema}

The example in (\ref{rema:obstm}) and the remark (\ref{rema:ba}) tell us that, 
even if $f\col  (E^{\bul},E^{\bul(i)})\lo (F^{\bul},F^{\bul(i)})$ 
is a filtered quasi-isomorphism for $1\leq \forall i\leq n$,  
$f \col (E^{\bul},\{E^{\bul(i)}\}_{i=1}^n) \lo (F^{\bul},\{F^{\bul(i)}\}_{i=1}^n)$ 
is not necessarily an $n$-filtered quasi-isomorphism in ${\rm CF}^n({\cal A})$. 
The proposition \cite[(1.1.7)]{nh2}(=the following proposition for the case $n=1$)  
does not imply the following proposition for the case $n=2$.

\begin{prop}\label{prop:exsir}
Assume that $n\leq 2$. 
For an object  
$(E^{\bul},\{E^{\bul(i)}\}_{i=1}^n)
\in {\rm K}^+{\rm F}^n({\cal A})$, 
there exists a 
specially injective resolution 
$(I^{\bul},\{I^{\bul(i)}\}_{i=1}^n)$ of 
$(E^{\bul},\{E^{\bul(i)}\}_{i=1}^n)$. 
\end{prop}
\begin{proof} 
Because the following proof is not only an obvious imitation of \cite[(1.1.7)]{nh2}, 
we give the complete proof 
(we have to use the simple argument in the proof (\ref{prop:usef}) 
and we need an additional argument in order to 
obtain (\ref{ali:kriin}) below).  
\par 
We may assume that $E^q=0$ for $q<0$. 
Assume that we are given 
$$(J^0,\{J^{0(i)}\}_{i=1}^n), (J^1,\{J^{1(i)}\}_{i=1}^n), 
\ldots, (J^q,\{J^{q(i)}\}_{i=1}^n)\in {\cal I}_{\rm spinj}({\cal A}).$$
We consider ${\cal A}$-modules 
$J^q{\oplus}_{E^{q}}E^{q+1}$ and 
$J^{q(i)}_k{\oplus}_{E^{q(i)}_{k}}E^{q+1(i)}_k$ $(i=1,\ldots, n)$.
Using the strictness of the morphism 
$(E^q,E^{q(i)}) \lo (J^q,J^{q(i)})$, we can easily check that  
the natural morphism 
$J^{q(i)}_k{\oplus}_{E^{q(i)}_{k}}E^{q+1(i)}_k
\lo J^q{\oplus}_{E^{q}}E^{q+1}$ is injective. 
Hence 
$\{J^{q(i)}_k{\oplus}_{E^{q(i)}_{k}}E^{q+1(i)}_{k}\}_{k\in {\mab Z}}$ 
defines a filtration on $J^q{\oplus}_{E^{q}}E^{q+1}$. 
The natural morphism 
$E^{q+1}\owns s \lom 
(0,s) \in J^q{\oplus}_{E^{q}}E^{q+1}$ induces a 
filtered morphism 
$$(E^{q+1},\{E^{q+1(i)}\}_{i=1}^n) \lo 
(J^q{\oplus}_{E^{q}}E^{q+1}, 
\{\{J^{q(i)}_k{\oplus}_{E^{q(i)}_{k}}E^{q+1(i)}_k\}_{k\in {\mab Z}}\}_{i=1}^n).$$
It is immediate to check
that, for each $i$, the filtered morphism 
$$(E^{q+1},E^{q+1(i)}) \lo 
(J^q{\oplus}_{E^{q}}E^{q+1}, \{J^{q(i)}_k{\oplus}_{E^{q(i)}_{k}}E^{q+1(i)}_k\}_{k\in {\mab Z}})$$ 
is strict. 
Let $I^{q+1}$ and $I^{q+1(i)}_k$ be injective ${\cal A}$-modules such 
that there exist the following injective morphisms of 
${\cal A}$-modules:
\begin{equation}
J^q{\oplus}_{E^{q}}E^{q+1} \os{\subset}{\lo} I^{q+1}, \quad 
J^q{\oplus}_{E^{q}}E^{q+1}/J^{q(i)}_k{\oplus}_{E^{q(i)}_k}E^{q+1(i)}_k 
\os{\subset}{\lo} I^{q+1(i)}_{k+1}.
\tag{3.9.1}
\label{eqn:jeesi}
\end{equation} 
Set 
$J^{q+1} :=I^{q+1}\times  {\prod}_{i=1}^n\us{k \in{\mab Z}}{\prod}I^{q+1(i)}_k$ 
and 
$$J^{q+1(i)}_k :=I^{q+1}\times 
\prod_{1\leq m\not=i \leq n}\us{l_m \in {\mab Z}}{\prod}I^{q+1(m)}_{l_m} 
\times \us{l_i \leq k}{\prod}I^{q+1(i)}_{l_i}$$ 
for each $1\leq i\leq n$. 
Then $(J^{q+1},\{J^{q+1(i)}\}_{i=1}^n) \allowbreak \in {\cal I}_{\rm spinj}({\cal A})$. 
By (\ref{eqn:imspm}) and (\ref{eqn:jeesi}), 
we have a natural injective morphism
\begin{equation}
(J^q{\oplus}_{E^{q}}E^{q+1}, 
\{\{J^{q(i)}_k{\oplus}_{E^{q(i)}_k}E^{q+1(i)}_k\}_{k\in {\mab Z}}\}_{i=1}^n)
\os{\sus}{\lo} (J^{q+1},\{J^{q+1(i)}\}_{i=1}^n).
\tag{3.9.2}
\label{eqn:jq1cnst}
\end{equation}
Since the morphism $E^q \lo J^q$ is injective,  
the morphism $E^{q+1} \lo J^q{\oplus}_{E^q}E^{q+1}$ is injective, 
and so is the following composite morphism 
$$E^{q+1} \os{\sus}{\lo} J^q{\oplus}_{E^q}E^{q+1} 
\os{\subset}{\lo} J^{q+1}.$$ 
In fact, we have a morphism
$(E^{q+1},\{E^{q+1(i)}\}_{i=1}^n) \lo (J^{q+1},\{J^{q+1(i)}\}_{i=1}^n)$. 
Using a morphism 
$J^q\owns s \lom (s,0) \in
J^q{\oplus}_{E^{q}}E^{q+1}$ and the morphism 
(\ref{eqn:jq1cnst}), 
we have a morphism 
$(J^q,\{J^{q(i)}\}_{i=1}^n) \lo 
(J^{q+1},\{J^{q+1(i)}\}_{i=1}^n)$. 
For a fixed $i$, by the proof of (\ref{prop:enspij}) 
and by the strictness of the morphism 
$(E^{q+1},\{E^{q+1(i)}_k\}_{k\in {\mab Z}}) \lo 
(J^q{\oplus}_{E^q}E^{q+1},\{J^{q(i)}_k{\oplus}_{E^{q(i)}_k} \allowbreak 
E^{q+1(i)}_k\}_{k\in {\mab Z}})$,  
the morphism 
$(E^{q+1},\{E^{q+1(i)}_k\}_{k\in {\mab Z}}) 
\lo (J^{q+1},\{J^{q+1(i)}_k\}_{k\in {\mab Z}})$ is strict.
Hence we obtain $(J^{\bul}, \allowbreak \{J^{\bul(i)}_k\}_{k\in {\mab Z}})$ inductively. 
\par
We claim that $(J^{\bul},\{J^{\bul(i)}\}_{i=1}^n)$ 
is filteredly quasi-isomorphic to $(E^{\bul},\{E^{\bul(i)}\}_{i=1}^n)$. 
To prove this, we first note that 
\begin{align*} 
{\rm Ker}(J^{q(i_1 i_n)}_{k_1  k_n} \lo 
J^{q+1(i_1 i_n)}_{k_1  k_n}) 
={\rm Ker}(J^{q(i_1 i_n)}_{k_1  k_n}
{\oplus}_{E^{q(i_1 i_n)}_{k_1  k_n}}E^{q+1(i_1 i_n)}_{k_1  k_n}) \quad (q\in {\mab Z}). 
\tag{3.9.3}\label{ali:krn}
\end{align*} 
Indeed, the problem is local. 
Because the morphism 
\begin{align*} 
J^{q(i_l)}_{k_l}{\oplus}_{E^{q(i_l)}_{k_l}}E^{q+1(i_l)}_{k_l}
\lo J^{q+1(i_l)}_{k_l} \quad (1\leq l\leq n)
\end{align*} 
is injective, 
\begin{align*} 
{\rm Ker}(J^{q(i_1i_n)}_{k_1 k_n} \lo J^{q+1(i_1i_n)}_{k_1  k_n})=
\bigcap_{l=1}^n{\rm Ker}(J^{q(i_l)}_{k_l}\lo J^{q(i_l)}_{k_l}{\oplus}_{E^{q(i_l)}_{k_l}} E^{q+1(i_l)}_{k_l}).  
\tag{3.9.4}\label{ali:kjrn}
\end{align*} 
Let $s$ be a local section of the right hand side of (\ref{ali:kjrn}). 
Let $g^q\col E^{q}\lo J^{q}$ be the constructed morphism. 
Then we may assume that there exists a local section $t_l$ of $E^{q(i_l)}_{k_l}$ 
such that $s=g^q(t_l)$ and $d(t_l)=0$ for $1\leq \forall l\leq n$. 
Hence $g^q(t_l)=g^q(t_m)$ for $1\leq \forall l,\forall m\leq n$. 
Because $g^q$ is injective, $t_l=t_m\in 
E^{q(i_1i_n)}_{k_1  k_n}$. 
Hence 
\begin{align*} 
&{\rm Ker}(J^{q(i_1i_n)}_{k_1  k_n} 
\lo J^{q(i_1i_n)}_{k_1  k_n}
{\oplus}_{E^{q(i_1i_n)}_{k_1  k_n}}E^{q+1(i_1i_n)}_{k_1  k_n})
\tag{3.9.5}\label{ali:kriin}\\
&=\bigcap_{l=1}^n{\rm Ker}(J^{q(i_l)}_{k_l}\lo J^{q(i_l)}_{k_l}{\oplus}_{E^{q(i_l)}_{k_l}} E^{q+1(i_l)}_{k_l})\\
& ={\rm Ker}(J^{q(i_1i_n)}_{k_1k_n} \lo J^{q+1(i_1i_n)}_{k_1k_n}). 
\end{align*} 
Let $\triangle$ be nothing or $(i_1i_n)$ for $1\leq i_1 \leq i_n\leq n$ and 
let $\circ$ be nothing or $n$-pieces of integers $k_1  k_n$. 
By (\ref{ali:kriin}) we see that 
the morphism 
${\rm Ker}(E^{q\triangle }_{\circ} \lo E^{q+1\triangle}_{\circ}) 
\lo 
{\rm Ker}(J^{q\triangle}_{\circ} \lo J^{q+1\triangle}_{\circ}) $ 
is an epimorphism. In particular, 
the morphism 
${\cal H}^q(E^{\bul}_{\circ}) \lo {\cal H}^q(J^{\bul}_{\circ})$ 
is an epimorphism. 
Furthermore, the morphism 
$J^{q-1}_{\circ}  \lo J^{q}_{\circ}$ factors through 
$J^{q-1}_{\circ}  \lo 
J^{q-1}_{\circ}{\oplus}_{E^{q-1}_{\circ}}E^{q}_{\circ}$ by (\ref{ali:kriin}). 
Note again that 
$J^{q-1}_{\circ}{\oplus}_{E^{q-1}_{\circ}}E^{q}_{\circ} \lo 
J^q_{\circ}$ is an injective morphism. 
Because the inverse image of 
${\rm Im}(J^{q-1}_{\circ} \lo J^{q}_{\circ})$ by the morphism 
$E^{q}_{\circ} \lo J^{q}_{\circ}$ 
is equal to
the inverse image of 
${\rm Im}(J^{q-1}_{\circ} \lo 
J^{q-1}_{\circ}{\oplus}_{E^{q-1}_{\circ}}E^{q}_{\circ})$, 
the morphism 
${\rm Ker}(E^q_{\circ} \lo E^{q+1}_{\circ})
/{\rm Im}(E^{q-1}_{\circ} 
\lo E^{q}_{\circ}) \lo 
{\cal H}^q(J^{\bul}_{\circ})$ 
is an injective morphism. Consequently the morphism 
${\cal H}^q(E^{\bul}_{\circ}) 
{\lo} 
{\cal H}^q(J^{\bul}_{\circ})$ is an isomorphism. 
\end{proof}

\begin{prop}\label{prop:comf}
Assume that $n\leq 2$. 
Let $f^{\bul} \col 
(E^{\bul},\{E^{\bul(i)}\}_{i=1}^n) \lo 
(F^{\bul},\{F^{\bul(i)}\}_{i=1}^n)$ be a morphism in 
${\rm C}^+{\rm F}^n({\cal A})$. Then there exists a 
morphism 
$g^{\bul} \col (J^{\bul},\{J^{\bul(i)}\}_{i=1}^n) \lo 
(K^{\bul},\{K^{\bul(i)}\}_{i=1}^n)$ 
in ${\rm C}^+{\rm F}^n({\cal A})$ such that 
$(J^{\bul},\{J^{\bul(i)}\}_{i=1}^n)$ 
$($resp.~$(K^{\bul},\{K^{\bul(i)}\}_{i=1}^n))$ 
is a specially injective resolution of 
$(E^{\bul},\{E^{\bul(i)}\}_{i=1}^n)$ 
$($resp.~$(F^{\bul},\{F^{\bul(i)}\}_{i=1}^n))$
and such that the following diagram is commutative$:$
\begin{equation*}
\begin{CD}
(E^{\bul},\{E^{\bul(i)}\}_{i=1}^n) @>{\subset}>> 
(J^{\bul},\{J^{\bul(i)}\}_{i=1}^n)\\ 
 @V{f^{\bul}}VV  @VV{g^{\bul}}V \\
(F^{\bul},\{F^{\bul(i)}\}_{i=1}^n) @>{\subset}>> 
(K^{\bul},\{K^{\bul(i)}\}_{i=1}^n).
\end{CD}
\end{equation*}
\end{prop}
\begin{proof}The proof is the same as that of 
\cite[(1.1.8)]{nh2} by using the proof of (\ref{prop:exsir}). 
\end{proof}
\par 
For an additive full subcategory 
${\cal I}$ of ${\rm MF}^n({\cal A})$, 
let ${\rm K}^+{\rm F}^n({\cal I})$  
be the category of 
the bounded below complexes with $n$-pieces of filtrations 
whose components belong to 
${\cal I}$. 

\begin{lemm}\label{lemm:hom0}
Assume that $n\leq 2$. 
Let $(E^{\bul},\{E^{\bul(i)}\}_{i=1}^n)$ be 
a complex of ${\cal A}$-modules with $n$-pieces of filtrations 
and let 
$(I^{\bul},\{I^{\bul(i)}\}_{i=1}^n)$ be an object of 
${\rm K}^+{\rm F}^n({\cal I}^n_{\rm stinj}({\cal A}))$. 
Assume that $(E^{\bul},\{E^{\bul(i)}\}_{i=1}^n)$ is strictly exact. 
Let $f \col (E^{\bul},\{E^{\bul(i)}\}_{i=1}^n) \lo 
(I^{\bul},\{I^{\bul(i)}\}_{i=1}^n)$ 
be a morphism of complexes with $n$-pieces of filtrations. 
Then $f$ is $n$-filteredly homotopic to zero.
\end{lemm}
\begin{proof}
By the definition  of the strict injectivity, 
the same argument as 
that in the classical case works.
\end{proof}

\begin{lemm}\label{lemm:filis}
Assume that $n\leq 2$. Then the following hold$:$
\par 
$(1)$ Let $(I^{\bul},\{I^{\bul(i)}\}_{i=1}^n)$ be an object of 
${\rm K}^+{\rm F}^n({\cal I}^n_{\rm stinj}({\cal A}))$. 
Let 
$$s \col (E^{\bul},\{E^{\bul(i)}\}_{i=1}^n) 
\lo (F^{\bul},\{F^{\bul(i)}\}_{i=1}^n)$$ 
be an $n$-filtered quasi-isomorphism.
Then $s$ induces an isomorphism 
$$s^* \col 
{\rm Hom}_{{\rm KF}^n({\cal A})}((F^{\bul},\{F^{\bul(i)}\}_{i=1}^n), 
(I^{\bul},\{I^{\bul(i)}\}_{i=1}^n))$$ 
$$\os{\sim}{\lo}{\rm Hom}_{{\rm KF}^n({\cal A})}((E^{\bul},\{E^{\bul(i)}\}_{i=1}^n), 
(I^{\bul},\{I^{\bul(i)}\}_{i=1}^n)).$$ 
\par
$(2)$ If a morphism $s \col(I^{\bul},\{I^{\bul(i)}\}_{i=1}^n) \lo
(E^{\bul},\{E^{\bul(i)}\}_{i=1}^n)$ is 
an $n$-filtered quasi-isomorphism from an object of 
${\rm K}^+{\rm F}^n({\cal I}^n_{\rm stinj}({\cal A}))$ to 
a complex of ${\cal A}$-modules with $n$-pieces of filtrations, 
then $s$ has an $n$-filtered homotopy inverse. 
\end{lemm}
\begin{proof}
By using (\ref{lemm:hom0}),  
the proof is the same as that of \cite[(1.1.10)]{nh2}. 
\par
(2): The proof is the same as that of 
\cite[I (4.5)]{hard} by using 
(\ref{lemm:hom0}), though 
there is an error in signs in 
the proof of \cite[I (4.5)]{hard}
(see \cite[(1.1.11)]{nh2} for this).
\end{proof}

\begin{coro}\label{coro:dfkfmykf}
Assume that $n\leq 2$. 
\par 
$(1)$ The following equalities hold$:$
\begin{align*}
{\rm D}^+{\rm F}^n({\cal A})& =
{\rm K}^+{\rm F}^n({\cal I}^n_{\rm flas}({\cal A}))_{({\rm F}^n{\rm Qis})} =
{\rm K}^+{\rm F}^n({\cal I}^n_{\rm inj}({\cal A}))_{({\rm F}^n{\rm Qis})} \\
{} & =
{\rm K}^+{\rm F}^n({\cal I}^n_{\rm stinj}({\cal A}))=
{\rm K}^+{\rm F}^n({\cal I}^n_{\rm spinj}({\cal A})).
\end{align*}
\par
$(2)$ Set  ${\cal I}:={\cal I}^n_{\rm flas}({\cal A})$, 
${\cal I}^n_{\rm inj}({\cal A})$, 
${\cal I}^n_{\rm stinj}({\cal A})$ or ${\cal I}^n_{\rm spinj}({\cal A})$. 
Let $f \col ({\cal T}, {\cal A}) 
\lo  ({\cal T}', {\cal A}')$  be a morphism of ringed topoi.
Then there exists the right derived functor 
$$Rf_* \col{\rm D}^+{\rm F}^n({\cal A}) \lo {\rm D}^+{\rm F}^n({\cal A}')$$  
of $f_*$ such that
$Rf_*([(I^{\bul},\{I^{\bul(i)}\}_{i=1}^n)])=
[(f_*(I^{\bul}),\{f_*(I^{\bul(i)})\}_{i=1}^n)]$  
for an object $(I^{\bul},\{I^{\bul(i)}\}_{i=1}^n)\in  {\rm K}^+{\rm F}^n({\cal I})$. 
Here $[~]$ is  
the localization functor.  
\par
$(3)$ Let 
$f \col ({\cal T}, {\cal A}) \lo  ({\cal T}', {\cal A}')$ and 
$g \col ({\cal T}', {\cal A}') 
\lo  ({\cal T}'', {\cal A}'')$  be morphisms of ringed topoi.  
Then $R(gf)_*=Rg_*Rf_*$.
\end{coro}
\begin{proof}
(1): The first 
two equalities follow from 
(\ref{prop:exsir}) and the proof of 
\cite[I (5.1)]{hard}.
The last two equalities follow  
from the proof of \cite[I (4.7)]{hard} and 
(\ref{lemm:filis}) (2).
\par
(2): (2) follows from the argument in the proof of \cite[I (5.1)]{hard}.
\par
(3): (3) follows by setting 
${\cal I}:={\cal I}^n_{\rm flas}({\cal A})$ in (2).
\end{proof}

\begin{prop}\label{prop:stpmb}
Assume that $n\leq 2$. 
Let $\star$ be $+$, $-$, {\rm b} or nothing. 
For $1 \leq i_1  \leq i_n \leq n$  
and $k_1,k_n \in {\mab Z}$, 
the intersection functor 
\begin{equation*}
\bigcap^{(i_1i_n)}_{k_1k_n}
\col {\rm DF}^{\star n}({\cal A}) \owns 
[(E^{\bul},\{E^{\bul(i)}\}_{i=1}^n)] 
\lom 
[E^{\bul(i_1i_n)}_{k_1k_n}]\in 
D^{\star}({\cal A}) 
\tag{3.14.1}\label{prop-defi:indtf}
\end{equation*}
is well-defined. Here $D^{\star}({\cal A})
:=K^{\star}({\cal A})_{({\rm FQis})}$ as usual.  
\end{prop}
\begin{proof} 
This follows from the definition of ${\rm DF}^{\star n}({\cal A})$. 
\end{proof}

\section{Strictly flat resolutions and $Lf^*$}\label{sfr} 
In this section we give the definitions of 
the specially flat resolution and 
the strictly flat resolution of 
a bifiltered complex. These are generalizations of the definitions in 
\cite{blec} and \cite{nh2}. 
The definitions in this section are more complicated than those in the previous section 
because the definitions which are dual to those in the previous section are not appropriate. 
We also define the left derived functor 
$Lf^*\col {\rm D}^-{\rm F}^n({\cal A}')\lo {\rm D}^-{\rm F}^n({\cal A})$ $(n=1,2)$
for a morphism $f\col ({\cal T},{\cal A})\lo ({\cal T}',{\cal A}')$ 
of ringed topoi.  
\par 
First we define another {\it special filtered module}  
(see \cite{blec} and \cite{nh2} for the case $n=1$). 
\par 
For two objects $(E,\{E^{(i)}\}_{i=1}^n)$, 
$(F,\{F^{(i)}\}_{i=1}^n)\in {\rm MF}^n({\cal A})$, 
we define the $n$-filtered tensor product
$(E{\otimes}_{\cal A}F,\{(E\otimes_{\cal A} F)^{(i)}\}_{i=1}^n)$ 
of $(E,\{E^{(i)}\}_{i=1}^n)$ and
$(F,\{F^{(i)}\}_{i=1}^n)$ as follows: 
$$(E\otimes_{\cal A}F)^{(i)}_k:={\rm Im}(\bigoplus_{l+m=k}
E^{(i)}_l\otimes_{\cal A}F^{(i)}_m\lo E\otimes_{\cal A}F).$$ 
\par
Let ${\cal P}(n)$ be the set of nonempty subsets of the set $\{1,\ldots,n\}$ 
and set ${\cal P}(n)_i:=\{P\in {\cal P}(n)~\vert~i\in P\}$ for $1\leq i\leq n$. 
For $P\in {\cal P}(n)$, 
set $m(P):=\# P$. Let  
$i_1<\ldots <i_{m(P)}$ be the elements of $P$: $\{i_1,\ldots, i_{m(P)}\}=P$.
Assume that, for any $P\in {\cal P}(n)$, we are given an 
${\cal A}$-module $E^{(i_1\cdots  i_{m(P)})}_{l_1\cdots  l_{m(P)}}$.  
We set $E^{P}:=
\{E^{(i_1\cdots  i_{m(P)})}_{l_1\cdots  l_{m(P)}}\}_{l_1,\ldots,l_{m(P)}\in {\mab Z}}$
and 
$$\Sigma_{\cal A}(E,\{E^{P}\}_{P\in {\cal P}(n)})
:=E\oplus \bigoplus_{P\in {\cal P}(n)}
\bigoplus_{\{l_1,\ldots, l_{m(P)}\in {\mab Z}\}}
E^{P}_{l_1\cdots  l_{m_{(P)}}} $$ 
with $n$-pieces of filtrations $\Sigma^{(i)}_{\cal A}$'s defined by 
\begin{align*} 
&(\Sigma^{(i)}_{\cal A}(E,\{E^{P}\}_{P\in {\cal P}(n)}))_k:=\\
&
\bigoplus_{P=\{i_1,\ldots,i_{m(P)}\}\in {\cal P}(n)_i}
\bigoplus_{\{l_1\in {\mab Z},\ldots, l_{m(P)}\in {\mab Z}~\vert~l_{m}\leq k~
{\rm for}~i_m=i~{\rm for}~1\leq m\leq m(P)\}}E^P_{l_1\cdots  l_{m(P)}}
\end{align*} 
for $1\leq i\leq n$. 
Then $(\Sigma_{\cal A}(E,\{E^{P}\}_{P\in {\cal P}(n)}),
\Sigma^{(i)}_{\cal A}(E,\{E^{P}\}_{P\in {\cal P}(n)}))$ is 
an object of ${\rm MF}^n({\cal A})$. 
For simplicity of notation, we denote this filtered module by 
$\Sigma_{\cal A}(E,\{E^{P}\}_{P\in {\cal P}(n)})$. 
The filtered module $\Sigma_{\cal A}(E,\{E^{P}\}_{P\in {\cal P}(n)})$ 
is a highly nontrivial generalization of 
the special filtered module defined in \cite{blec} (and \cite{nh2}). 
By the definition of the filtration $\Sigma^{(i)}_{\cal A}(E,\{E^{P}\}_{P\in {\cal P}(n)})$, 
we obtain the following formula: 
\begin{align*} 
&(\Sigma^{(i_1\ldots i_n)}_{\cal A}(E,\{E^{P}\}_{P\in {\cal P}(n)}))_{k_1\cdots  k_n}
\tag{4.0.1}\label{eqn:espflhm}\\
&:=\bigcap_{j=1}^n(\Sigma^{(i_j)}_{\cal A}(E,\{E^{P}\}_{P\in {\cal P}(n)}))_{k_j}\\
&=\bigoplus_{P=\{e_1,\ldots,e_{m(P)}\}\in \cap_{j=i_1}^{i_n}{\cal P}(n)_j}
\bigoplus_{\{l_1\in {\mab Z},\ldots, l_{m(P)}\in {\mab Z}~\vert~l_j\leq k_m~
{\rm for}~e_j=i_m~{\rm for}~{\rm some}~1\leq m\leq n\}}
{\bigoplus}E^{P}_{l_1,\ldots,l_{{m(P)}}}. 
\end{align*} 
The following is a generalization of a formula 
in \cite{blec} and \cite[(1.1.12.1)]{nh2}: 

\begin{prop}\label{prop:sfa}
The following formula holds$:$ 
\begin{align*}  
&{\rm Hom}_{{\rm MF}^n({\cal A})}
({\Sigma}_{\cal A}(E,\{E^{P}\}_{P\in {\cal P}(n)}), 
(F,\{F^{(i)}\}_{i=1}^n))= \tag{4.1.1}\label{eqn:espflahm} \\ 
&{\rm Hom}_{\cal A}(E,F) \times 
{}  {\prod}_{P\in {\cal P}(n)}
\us{l_1\cdots l_{m(P)}\in {\mab Z}}{\prod}
{\rm Hom}_{\cal A}(E^{P}_{l_1\cdots l_{m(P)}},F^{P}_{l_1\cdots l_{m(P)}}) 
\end{align*}
\end{prop}
\begin{proof} 
Assume that we are given a filtered morphism 
${\Sigma}_{\cal A}(E,\{E^{P}\}_{P\in {\cal P}(n)})\lo (F,\{F^{(i)}\}_{i=1}^n)$. 
Then we have morphisms  
$E\lo F$ and $E^{P}_{l_1\cdots l_{m(P)}}\lo F$. 
By the definition of $\Sigma^{(i)}_{\cal A}(E,\{E^{P}\}_{P\in {\cal P}(n)})$, 
the latter morphism factors through $F^{P}_{l_1\cdots l_{m(P)}}$. 
\par 
Conversely assume that we are are given morphisms 
$E\lo F$ and $E^{P}_{l_1\cdots l_{m(P)}}\lo F^{P}_{l_1\cdots l_{m(P)}}$. 
Obviously we have the composite morphism 
$E^{P}_{l_1\cdots l_{m(P)}}\lo  F^{P}_{l_1\cdots l_{m(P)}}\os{\subset}{\lo} F$. 
This composite morphism and the morphism $E\lo F$
induces a filtered morphism 
${\Sigma}_{\cal A}(E,\{E^{P}\}_{P\in {\cal P}(n)})\lo (F,\{F^{(i)}\}_{i=1}^n)$. 
\end{proof} 

\par
The following is a highly nontrivial generalization of the definition of 
the strictly flatness in \cite{blec} (and \cite{nh2}). 
This definition plays a central role in the definition of 
the derived tensor product $\otimes^L_{\cal A}$ for two complexes of ${\cal A}$-modules 
with $n$-pieces of filtrations defined 
in \S\ref{sec:otl} below. 

\begin{defi}\label{defi:stfl} 
Assume that $n\leq 2$. 
We say that  an object 
$(Q,\{Q^{(i)}\}_{i=1}^n)$ of ${\rm MF}^n({\cal A})$ is 
{\it strictly flat} if it satisfies the 
following two conditions:
\par
(1)  $Q$  and 
$Q/\sum_{j=1}^NQ^{(i^j_1i^j_n)}_{k^j_1k^j_n}$ 
$(N\in {\mab Z}_{\geq 1},1\leq \forall i^j_1 \leq \forall i^j_n \leq n, 
\forall k^j_1, \forall k^j_n \in {\mab Z})$ 
are flat ${\cal A}${\textrm -}modules. 
\par
(2) For a strictly injective morphism  
$(E,\{E^{(i)\bul}\}_{i=1}^n) \os{\subset}{\lo} (F,\{F^{(i)\bul}\}_{i=1}^n)$, 
the induced morphism
$$(Q{\otimes}_{\cal A}E,\{(Q\otimes_{\cal A} E)^{(i)}\}_{i=1}^n)  
\lo (Q{\otimes}_{\cal A}F,\{(Q\otimes_{\cal A} F)^{(i)}\}_{i=1}^n)$$ 
is a strictly injective morphism.
\end{defi}

The following remark (1) is very important.

\begin{rema}\label{rema:dual}
(1) The dual definition of the first property of $(J,\{J^{(i)}\}_{i=1}^n)$ 
in (\ref{defi:stinj}) (1) is the following statement: 
\medskip 
\parno
``$Q$  and $Q/Q^{(i_1i_n)}_{k_1k_n}$ 
$(1\leq \forall i_1 \leq \forall i_n \leq n, 
\forall k_1, \forall k_n \in {\mab Z})$ 
are flat ${\cal A}${\textrm -}modules.'' 
\medskip 
\parno 
However this notion is not appropriate in the definition of the derived tensor product 
$\otimes^L_{\cal A}$ below. 
\par 
(2) Let $(J,\{J^{(i)}\}_{i=1}^n)$ be an object of 
${\cal I}^n_{\rm flas}({\cal A})$ for any $n\in{\mab Z}_{\geq 1}$ 
Then, for any positive integer $N$, 
$\sum_{j=1}^NJ^{(i^j_1\cdots i^j_n)}_{k^j_1\cdots k^j_n}$ is automatically flasque for 
$1\leq 
\forall i^j_1 \leq  \cdots \leq \forall i^j_n \leq n, \forall k^j_1,\ldots, \forall k^j_n \in {\mab Z}$. 
\end{rema}

In the following we assume that $n\leq 2$. 
Let us consider 
the following additive full subcategories 
of ${\rm MF}^n({\cal A})$: 

\begin{align*}
{\cal Q}^n_{\rm fl}({\cal A})
:=\{ (Q,\{Q^{(i)}\}_{i=1}^n)~\vert~ &
Q \text{ and } Q/\sum_{j=1}^NQ^{(i^j_1i^j_n)}_{k^j_1k^j_n} \text{ are flat } 
{\cal A}{\textrm -}\text{modules}
\tag{4.3.1}\label{eqn:qfla}\\
{} & ~(N\in {\mab Z}_{\geq 1},1\leq \forall i^j_1\leq \forall i^j_n \leq n, 
\forall k^j_1, \forall k^j_n \in {\mab Z})\}, 
\end{align*}

\begin{equation*}
{\cal Q}^n_{\rm stfl}({\cal A})
:=\{ (Q,\{Q^{(i)}\}_{i=1}^n)~\vert~
\text{$(Q,\{Q^{(i)}\}_{i=1}^n)$ is a strictly 
flat ${\cal A}${\textrm -}module$\}$}, 
\tag{4.3.2}\label{eqn:qstfla}
\end{equation*}

\begin{align*}
{\cal Q}^n_{\rm spfl}({\cal A}):=\{ 
\Sigma_{\cal A}(Q,\{Q^{P}\}_{P\in {\cal P}(n)})~\vert~&
\text{$Q$  and $Q^P_{k_1 k_{m(P)}}$ 
are flat ${\cal A}${\textrm -}modules}~{\rm for}~\tag{4.3.3}\label{eqn:qspfla}\\
& \forall P\in {\cal P}(n)~{\rm and}~\forall k_1,\forall k_{m(P)}\in {\mab Z}\}. 
\end{align*}
\parno
Then ${\cal Q}^n_{\rm stfl}({\cal A}) \subset 
{\cal Q}^n_{\rm fl}({\cal A})$.

\begin{defi}
Assume that $n\leq 2$. 
We say that an object 
$(Q,\{Q^{(i)}\}_{i=1}^n)\in {\rm MF}^n({\cal A})$ is 
$n$-{\it filteredly  flat} 
(resp.~{\it $n$-specially flat})
if $(Q,\{Q^{(i)}\}_{i=1}^n)\in {\cal Q}^n_{\rm fl}({\cal A})$ 
(resp.~$(Q,\{Q^{(i)}\}_{i=1}^n)\in 
{\cal Q}^n_{\rm spfl}({\cal A})$).
\end{defi}

\begin{lemm}[{\bf \cite{blec} for the case $n=1$}]\label{lemm:spfistf} 
Assume that $n\leq 2$. 
Then 
${\cal Q}^n_{\rm spfl}({\cal A}) \subset {\cal Q}^n_{\rm stfl}({\cal A})$.
\end{lemm}
\begin{proof} 
Let $\Sigma_{\cal A}(Q,\{Q^{P}\}_{P\in {\cal P}(n)})$ 
be an object of ${\cal Q}_{\rm spfl}({\cal A})$. 
It is easy to see that 
$\Sigma_{\cal A}(Q,\{Q^{P}\}_{P\in {\cal P}(n)})$ and 
$\Sigma_{\cal A}(Q,\{Q^{P}\}_{P\in {\cal P}(n)})
/\sum_{j=1}^N
(\Sigma_{\cal A}(Q,\{Q^{P}\}_{P\in {\cal P}(n)}))^{(i^j_1 i^j_n)}_{k^j_1k^j_n}$ are flat ${\cal A}$-modules.   
\par 
Let $\iota \col (E,\{E^{(i)}\}_{i=1}^n) \os{\subset}{\lo} 
(F,\{F^{(i)}\}_{i=1}^n)$ be a strictly injective morphism 
and let $k$ be an integer. 
Denote by the same symbol $\iota$ the induced injective 
morphism 
$\Sigma_{\cal A}(Q,\{Q^{P}\}_{P\in {\cal P}(n)})\otimes_{\cal A}E \os{\subset}{\lo}
\Sigma_{\cal A}(Q,\{Q^{P}\}_{P\in {\cal P}(n)})\otimes_{\cal A}F$.
Let $s$ be a local section of 
$\iota(\Sigma_{\cal A}(Q,\{Q^{P}\}_{P\in {\cal P}(n)})\otimes_{\cal A}E) \cap
(\Sigma^{(i)}_{\cal A}(Q,\{Q^{P}\}_{P\in {\cal P}(n)})\otimes_{\cal A}F)_k$.
By the definition of the filtration on the filtered tensor product, 
$s$ is a finite sum of  local sections of 
$$\bigoplus_{P=\{i_1i_{m(P)}\}\in {\cal P}(n)_i}
\bigoplus_{\{l_1, l_{m(P)}\in {\mab Z}~\vert~l_m\leq l~
{\rm for}~i_m=i~{\rm for}~1\leq m\leq m(P)\}}Q^P_{l_1 l_{m(P)}}\otimes_{\cal A}F_j$$
$(l+j \allowbreak \leq k)$. Because 
\begin{align}
(Q^P_{l_1 l_{m(P)}}\otimes_{\cal A}F_j)\cap (Q^P_{l_1 l_{m(P)}}\otimes_{\cal A}\iota(E))&  
=Q^P_{l_1 l_{m(P)}}\otimes_{\cal A}(F_j\cap \iota(E))\notag \\
{} &=Q^P_{l_1 l_{m(P)}}\otimes_{\cal A}\iota(E_j),\notag
\end{align}
$s$ is  a local section of 
$\iota((\Sigma^{(i)}_{\cal A}(Q,\{Q^{P}\}_{P\in {\cal P}(n)})\otimes_{\cal A}E)_k)$.
Now we can complete the proof of (\ref{lemm:spfistf})
by (\ref{prop:usef}). 
\end{proof}

\begin{prop}[{\bf \cite{blec} for the case $n=1$}]\label{prop:exspfl} 
Assume that $n\leq 2$. 
For an  ${\cal A}$-module 
$(E,\{E^{(i)}\}_{i=1}^n)$ with $n$-pieces of filtrations, 
there exists a strict epimorphism 
$$\Sigma_{\cal A}(Q,\{Q^{P}\}_{P\in {\cal P}(n)}) \lo 
(E,\{E^{(i)}\}_{i=1}^n)$$ 
with 
$\Sigma_{\cal A}(Q,\{Q^{P}\}_{P\in {\cal P}(n)})\in 
{\cal Q}^n_{\rm spfl}({\cal A})$.
\end{prop}
\begin{proof}
Recall the functor 
$L^0 \col \{{\cal A}\text{-modules}\}
\lo \{\text{flat }{\cal A}\text{-modules}\}$ 
(\cite[\S7]{bob}):
for an ${\cal A}$-module, 
$L^0(E)$ is, by definition, the sheafification of the presheaf 
$$(U \lom  \text{a } \text{free }\Gam(U,{\cal A})\text{-module 
with basis }\Gam(U,E)\setminus \{0\}).$$
The natural morphism $L^0(E) \lo E$ is an epimorphism.
\par
Let $Q \lo E$ and 
$Q^P_{l_1 l_{m (P)}}
\lo E^P_{l_1 l_{m(P)}}$ 
($P\in {\cal P}$)  
be epimorphisms from flat ${\cal A}$-modules.
Set 
$$R :=Q\oplus \bigoplus_{P\in {\cal P}}
\us{l_1, l_{m(P)}\in {\mab Z}}{\bigoplus}Q^P_{l_1 l_{m(P)}}$$  
and 
$$R^{(i)}_k :=\bigoplus_{P=\{i_1,i_{m(P)}\}\in {\cal P}_i}
\bigoplus_{\{l_1, l_{m(P)}\in {\mab Z}~\vert~l_j\leq k~
{\rm for}~i_m=i~{\rm for}~{\rm any}~1\leq m\leq m(P)\}}Q^P_{l_1 l_{m(P)}}.$$ 
Then  
$(R,\{R^{(i)}\}_{i=1}^n)$ is an object of ${\cal Q}_{\rm spfl}({\cal A})$. 
The morphisms $Q \lo E$ and 
$Q^P_{l_1 l_{m(P)}} {\lo} E^P_{l_1 l_{m(P)}} \os{\subset}{\lo} E$ 
induce an epimorphism
$R \lo E$. The morphism 
$Q^P_{l_1 l_{m(P)}} \lo E^P_{l_1 l_{m(P)}} \os{\subset}{\lo}E^P_{k_1 k_{m(P)}}$ 
$(l_j\leq k_j)$  induces an epimorphism
$R^P_{k_1 k_{m(P)}} \lo E^P_{k_1 k_{m(P)}}$. 
Obviously the morphism $R \lo E$ is strict. 
Thus (\ref{prop:exspfl}) follows. 
\end{proof}

\begin{defi}
Assume that $n\leq 2$. 
Let $(E^{\bul},\{E^{\bul(i)}\}_{i=1}^n)$ be an object of  
${\rm K}^-{\rm F}^n({\cal A})$.  
\par 
(1)  We say that an object 
$(Q^{\bul},\{Q^{\bul(i)}\}_{i=1}^n)\in {\rm K}^-{\rm F}^n({\cal A})$ 
with an $n$-filtered morphism 
$(Q^{\bul},\{Q^{\bul(i)}\}_{i=1}^n)\lo (E^{\bul},\{E^{\bul(i)}\}_{i=1}^n)$ 
is a {\it strictly flat resolution} of  
$(E^{\bul},\{E^{\bul(i)}\}_{i=1}^n)$ 
if $(Q^{q},Q^{q(i)}_k)\in {\cal Q}^n_{\rm stfl}({\cal A})$ for any 
$q\in {\mab Z}$ and if the morphism  
$(Q^{\bul},Q^{\bul(i)})\lo (E^{\bul},\{E^{\bul(i)}\}_{i=1}^n)$ is 
an $n$-filtered quasi-isomorphism  which induces  
a strict epimorphism 
$(Q^q,Q^{q(i)})\lo (E^q,\{E^{q(i)}\}_{i=1}^n)$
for any $q \in {\mab Z}$. 
\par
(2) We say that  an object 
$(Q^{\bul},\{Q^{\bul(i)}\}_{i=1}^n)\in {\rm K}^-{\rm F}^n({\cal A})$ 
with an $n$-filtered morphism 
$(Q^{\bul},\{Q^{\bul(i)}\}_{i=1}^n)\lo (E^{\bul},\{E^{\bul(i)}\}_{i=1}^n)$ 
is a {\it filtered flat resolution} 
(resp.~{\it specially flat resolution}) 
of  $(E^{\bul},\{E^{\bul(i)}\}_{i=1}^n)$ 
if $(Q^{q},\{Q^{q(i)}\}_{i=1}^n) \in {\cal Q}^n_{\rm fl}({\cal A})$ 
(resp.~$(Q^{q},\{Q^{q(i)}\}_{i=1}^n)\in {\cal Q}^n_{\rm spfl}({\cal A})$) 
for any $q\in {\mab Z}$ and if the morphism 
$(Q^{\bul},\{Q^{\bul(i)}\}_{i=1}^n)\lo (E^{\bul},\{E^{\bul(i)}\}_{i=1}^n)$ 
is an $n$-filtered quasi-isomorphism 
which induces a strict epimorphism 
$(Q^q,Q^{q(i)})\lo (E^q,\{E^{q(i)}\}_{i=1}^n)$
for any $q \in {\mab Z}$. 
\end{defi}

The following is a more nontrivial result than (\ref{prop:exsir}) at first glance 
because we cannot use (\ref{prop:usef}): 

\begin{prop}[{\bf \cite{blec} for the case $n=1$}]\label{prop:esfrl}
Assume that $n\leq 2$. 
For an object  
$(E^{\bul},\{E^{\bul(i)}\}_{i=1}^n)\in {\rm K}^-{\rm F}^n({\cal A})$, 
there exists a specially flat resolution 
$(Q^{\bul},\{Q^{\bul(i)}\}_{i=1}^n)$ of $(E^{\bul},\{E^{\bul(i)}\}_{i=1}^n)$.
\end{prop}
\begin{proof}
Let ${\triangle}$ and $\circ$ be as in the proof of (\ref{prop:exsir}). 
We may assume that $E^q=0$ for $q>0$.  
Assume that we are given 
$(Q^q,\{Q^{q(i)}\}_{i=1}^n),\ldots,(Q^0,\{Q^{0(i)}\}_{i=1}^n)$ for $q\in {\mab Z}_{<0}$. 
Let the notations be as in (\ref{prop:exsir}). 
Consider the fiber product 
$(Q^q\times_{E^q}E^{q-1},\{Q^{q(i)}\times_{E^{q(i)}}E^{q-1(i)}\}_{i=1}^n)$.
Obviously the morphism 
$${\rm Ker}(Q^{q\triangle}_{\circ}\times_{E^{q\triangle}_{\circ}}E^{q-1\triangle}_{\circ}\lo 
Q^{q\triangle}_{\circ})
\lo 
{\rm Ker}(E^{q-1\triangle}_{\circ}\lo E^{q\triangle}_{\circ})$$ 
is surjective. 
Consider the kernel $I^{q-1\triangle}_{\circ}$ of the following morphism 
\begin{align*} 
{\rm Ker}(Q^{q\triangle}_{\circ}
\times_{E^{q\triangle}_{\circ}}E^{q-1\triangle}_{\circ}\lo Q^{q\triangle}_{\circ})
\lo {\cal H}^{q-1}(E^{\bul\triangle}_{\circ}) 
\end{align*} 
for the case where $\triangle$ is nothing or $(i)$ for $1\leq i\leq n$. 
Set 
\begin{align*} 
Q^{q-1}:=L^0(I^{q-1})\bigoplus_{P\in {\cal P}}
\bigoplus_{\{l_1, l_{m(P)}\in {\mab Z}\}}
L^0(I^{P,q-1}_{l_1 l_{m(P)}})
\end{align*} 
and 
\begin{align*} 
Q^{q-1(i)}_k:=\bigoplus_{P=\{i_1,i_{m(P)}\}\in {\cal P}_i}
\bigoplus_{\{l_1\in {\mab Z},\ldots, l_{m-1}\in {\mab Z}, l_m\leq k, l_{m+1}\in {\mab Z},
\ldots,l_{{m(P)}}\in {\mab Z}~\vert~i_m=i~{\rm for}~{\rm some}~1\leq m\leq m(P)\}}
L^0(I^{P,q-1}_{l_1,l_{{m(P)}}})
\end{align*}
for $1\leq i\leq n$. 
Then $(Q^{q-1}, \{Q^{q-1(i)}\}_{i=1}^n)$ is an object of ${\cal Q}_{\rm spfl}({\cal A})$. 
By the definition of $Q^{q-1(i)}_k$ (cf.~(\ref{eqn:espflhm})), 
\begin{align*} 
Q^{q-1(i_1i_n)}_{k_1 k_n}
=
\bigoplus_{P=\{j_1,j_{m(P)}\}\in {\cal P}_i}
\bigoplus_{\{l_1\in {\mab Z},\ldots, l_{m-1}\in {\mab Z}, l_p\leq k_p , l_{m+1}\in {\mab Z},
\ldots,l_{{m(P)}}\in {\mab Z}~\vert~j_p=i_p~{\rm for}~p=1,n\}}
L^0(I^{P,q-1}_{l_1,l_{{m(P)}}}). 
\end{align*} 
Hence we see that the morphism 
\begin{align*} 
{\cal H}^{q-1}(Q^{\bul\triangle}_{\circ}) 
\lo {\cal H}^{q-1}(E^{\bul\triangle}_{\circ}) 
\tag{4.8.1}\label{ali:eqq}
\end{align*} 
is an isomorphism for the case $\triangle$ is nothing or $(i_1 i_n)$. 
\end{proof}

For an additive full subcategory 
${\cal Q}$ of ${\rm MF}^n({\cal A})$, 
let ${\rm K}^-{\rm F}^n({\cal Q})$  
be the category of 
the bounded above complexes with $n$-pieces of filtrations 
whose components belong to 
${\cal Q}$.

\begin{coro}\label{coro:lcal}
Assume that $n\leq 2$. Then the following hold$:$
\par 
$(1)$ The following equalities hold$:$
$${\rm D}^-{\rm F}^n({\cal A})={\rm K}^-{\rm F}^n
({\cal Q}^n_{\rm fl}({\cal A}))_{({\rm F}^n{\rm Qis})}=
{\rm K}^-{\rm F}^n
({\cal Q}^n_{\rm stfl}({\cal A}))_{({\rm F}^n{\rm Qis})}
={\rm K}^-{\rm F}^n
({\cal Q}^n_{\rm spfl}({\cal A}))_{({\rm F}^n{\rm Qis})}.$$
\par
$(2)$ Let ${\cal Q}'{}^n
:={\cal Q}^n_{\rm fl}({\cal A}')$, 
${\cal Q}^n_{\rm stfl}({\cal A}')$ 
or ${\cal Q}^n_{\rm spfl}({\cal A}')$.
Let $f \col ({\cal T}, {\cal A}) 
\lo  ({\cal T}', {\cal A}')$ be  
a morphism of ringed topoi. Then 
there exists the  left derived functor 
$Lf^* \col
{\rm D}^-{\rm F}^n({\cal A}') 
\lo {\rm D}^-{\rm F}^n({\cal A})$ such that 
$Lf^*[(Q^{\bul},\{Q^{\bul(i)}\}_{i=1}^n)]=
[(f^*(Q^{\bul}),\{f^*(Q^{\bul(i)})\}_{i=1}^n)]$ 
for an object 
$(Q^{\bul},\{Q^{\bul(i)}\}_{i=1}^n) \in 
{\rm K}^-{\rm F}^n({\cal Q}'{}^n)$.
\par
$(3)$  Let 
$f \col ({\cal T}, {\cal A}) 
\lo  ({\cal T}', {\cal A}')$ 
and $g \col ({\cal T}', {\cal A}') 
\lo  ({\cal T}'', {\cal A}'')$ 
be morphisms of ringed topoi.
Then $L(gf)^*=Lf^*Lg^*$.
\end{coro}
\begin{proof}
(1) and (2) are obvious. (3) follows 
by setting ${\cal Q}:=
{\cal Q}^n_{\rm fl}({\cal A})$ in (2).
\end{proof}

\section{${\rm RHom}^{\bul}$}\label{sec:rhoml}
In this section we define 
the derived homomorphism functor ${\rm RHom}^{\bul}$
from bounded above complexes with $n$-pieces of filtrations 
to bounded below complexes with $n$-pieces of filtrations for $n\leq 2$. 
The results in this section are generalizations of results in  \cite[(1.2)]{nh2}. 
\par 

\par
As in \cite[p.~63]{hard}, we set  
$${\rm Hom}^m_{\cal A}
((E^{\bul},\{E^{\bul(i)}\}_{i=1}^n), 
(F^{\bul},\{F^{\bul(i)}\}_{i=1}^n)):=$$ 
$$\us{q\in {\mab Z}}{\prod}{\rm Hom}_{\cal A}
((E^q,\{E^{q(i)}\}_{i=1}^n),(F^{q+m},\{F^{q+m(i)}\}_{i=1}^n))$$
for 
$(E^{\bul},\{E^{\bul(i)}\}_{i=1}^n), 
(F^{\bul},\{F^{\bul(i)}\}_{i=1}^n)\in {\rm CF}^n({\cal A})$.
Then we have an object  
$${\rm Hom}^{\bul}_{\cal A}
((E^{\bul},\{E^{\bul(i)}\}_{i=1}^n), 
(F^{\bul},\{F^{\bul(i)}\}_{i=1}^n)) \in 
{\rm CF}^n(\Gam({\cal T},{\cal A}))$$ 
of $\Gam({\cal T},{\cal A})$-modules;
the boundary morphism 
$${\rm Hom}^m_{\cal A}((E^{\bul},\{E^{\bul(i)}\}_{i=1}^n), 
(F^{\bul},\{F^{\bul(i)}\}_{i=1}^n)) \lo $$ 
$${\rm Hom}^{m+1}_{\cal A}
((E^{\bul},\{E^{\bul(i)}\}_{i=1}^n), (F^{\bul},\{F^{\bul(i)}\}_{i=1}^n))$$ 
is defined as in \cite[p.~4]{bbm} and \cite[p.~10]{con}: 
$$d^m:=\us{q\in {\mab Z}}{\prod}((-1)^{m+1}d_E^q+d_F^{q+m}).$$
(Recall the filtration (\ref{ali:eifa}).)
\par
For a sequence $\ul{k}=(k_1, k_n)$ of integers, 
an $m$-cocycle of 
$$\bigcap_{j=1}^n{\rm Hom}^{\bul(j)}_{\cal A}
((E^{\bul},\{E^{\bul(i)}\}_{i=1}^n), 
(F^{\bul},\{F^{\bul(i)}\}_{i=1}^n))_{k_j}$$ 
corresponds to an $n$-filtered morphism 
$E^{\bul} \lo F^{\bul}[m]
\langle \ul{k} \rangle$. 
An $m$-coboundary of
$$\bigcap_{j=1}^n{\rm Hom}^{\bul(j)}_{\cal A}
((E^{\bul},\{E^{\bul(i)}\}_{i=1}^n), 
(F^{\bul},\{F^{\bul(i)}\}_{i=1}^n))_{k_j}$$
corresponds to a morphism 
$E^{\bul} \lo F^{\bul}[m]\langle \ul{k}\rangle$ 
which is homotopic to zero.  
Hence  
\begin{equation*}
H^m(\bigcap_{j=1}^n{\rm Hom}^{\bul(j)}_{\cal A}
((E^{\bul},\{E^{\bul(i)}\}_{i=1}^n), 
(F^{\bul},\{F^{\bul(i)}\}_{i=1}^n))_{k_j})= 
\tag{5.0.1}\label{eqn:homcoh}
\end{equation*} 
$${\rm Hom}_{{\rm KF}^n({\cal A})}
((E^{\bul},\{E^{\bul(i)}\}_{i=1}^n), (F^{\bul},\{F^{\bul(i)}\}_{i=1}^n)
[m]\langle \ul{k} \rangle).$$
In particular,  
\begin{equation*}
H^0(\bigcap_{j=1}^n{\rm Hom}^{\bul(j)}_{\cal A}
((E^{\bul},\{E^{\bul(i)}\}_{i=1}^n), 
(F^{\bul},\{F^{\bul(i)}\}_{i=1}^n))_{0})
=
\tag{5.0.2}\label{eqn:h0homch}
\end{equation*}
$${\rm Hom}_{{\rm KF}^n({\cal A})}
((E^{\bul},\{E^{\bul(i)}\}_{i=1}^n), (F^{\bul},\{F^{\bul(i)}\}_{i=1}^n)).$$
More generally, 
for  $1\leq i_1<  i_p \leq n$ $(1\leq p\leq n)$ 
and a sequence $\ul{k}=(k_1,k_p)$ of integers, 
we have  
\begin{align*}
&H^m(\bigcap_{k_1k_p}^{(i_1i_p)}
{\rm Hom}^{\bul}_{\cal A}
((E^{\bul},\{E^{\bul(i)}\}_{i=1}^n), 
(F^{\bul},\{F^{\bul(i)}\}_{i=1}^n)))
\tag{5.0.3}\label{eqn:ghcoh}\\
&= 
{\rm Hom}_{{\rm KF}^n({\cal A})}
((E^{\bul},\{E^{\bul(i_q)}\}_{q=1}^p), 
(F^{\bul},\{F^{\bul(i_q)}\}_{q=1}^p)[m]\langle \ul{k}\rangle).
\end{align*}  
\par 
To define the derived functor of the functor
$${\rm Hom}^{\bul}_{\cal A}(\bul, \bul) \col
{\rm KF}^n({\cal A})^{\circ} \times {\rm K}^+{\rm F}^n({\rm A}) 
\lo {\rm KF}^n(\Gam({\cal T},{\cal A})),$$
we have to check the following:

\begin{lemm}\label{lemm:stex2}
Let $(E^{\bul},\{E^{\bul(i)}\}_{i=1}^n)$ be an object of ${\rm KF}^n({\cal A})$ 
and let $(I^{\bul},\{I^{\bul(i)}\}_{i=1}^n)$ be an object of 
${\rm K}^+{\rm F}^n({\cal I}^n_{\rm stinj}({\cal A}))$. 
Assume that one of the following two 
conditions holds. 
\par
$(1)$ $(I^{\bul},\{I^{\bul(i)}\}_{i=1}^n)$ is strictly exact.
\par
$(2)$ $(E^{\bul},\{E^{\bul(i)}\}_{i=1}^n)$ is strictly exact. 
\parno
Then  
${\rm Hom}^{\bul}_{\cal A}((E^{\bul},\{E^{\bul(i)}\}_{i=1}^n),
(I^{\bul},\{I^{\bul(i)}\}_{i=1}^n))$ 
is strictly exact.
\end{lemm}
\begin{proof} 
(1): By the definition of the strict injectivity, 
there exist ${\cal A}$-modules 
$J^q$ and $J^{q(i)}_k$ $(i=1, n, q,k \in {\mab Z})$ 
satisfying the following three conditions:
\medskip 
\par
(i) $J^{q(i)}_{k-1} \subset J^{q(i)}_k\subset J^{q(i)}$,  
\par 
(ii) $(I^{q(i)},\{I^{q(i)}\}_{i=1}^n)\simeq 
(J^{q-1},\{J^{q-1(i)}\}_{i=1}^n)\oplus (J^q,\{J^{q(i)}\}_{i=1}^n)$, 
\par
(iii)  the boundary morphism $d \col (I^q,\{I^{q(i)}\}_{i=1}^n) \lo (I^{q+1},\{I^{q+1(i)}\}_{i=1}^n)$
is identified with the induced morphism by the morphisms 
$J^{q-1} \lo 0$ and $J^q \os{\rm id}{\lo} J^q$.  
By (\ref{eqn:ghcoh}), we have only 
to construct a filtered homotopy for a morphism 
$f\in 
{\rm Hom}_{{\rm CF}^n({\cal A})}((E^{\bul},\{E^{\bul(i)}\}_{i=1}^n), 
(I^{\bul},\{I^{\bul(i)}\}_{i=1}^n))$, 
which is easy.
\par
(2): By (\ref{eqn:ghcoh}) and by 
the definition  of the strict injectivity, 
the same argument as 
that in the classical case works. 
\end{proof}

\parno
By (\ref{lemm:stex2}) we obtain 
the following derived functor
$${\rm RHom}^{\bul}_{\cal A} 
\col {\rm DF}^n({\cal A})^{\circ} \times 
{\rm D}^+{\rm F}^n({\cal A}) 
\lo {\rm DF}^n(\Gam({\cal T},{\cal A})).$$ 
By (\ref{prop:stpmb}) 
we obtain the following functor 
$$\bigcap^{(i_1i_n)}_{k_1 k_n}
{\rm RHom}^{\bul}_{\cal A} 
\col {\rm DF}^n({\cal A})^{\circ} \times 
{\rm D}^+{\rm F}^n({\cal A}) 
\lo D(\Gam({\cal T},{\cal A}))$$ 
for $1 \leq i_1 \leq i_n \leq n$ 
and $k_1, k_n \in {\mab Z}$.  

The following includes the adjunction formula in \cite{blec} and 
\cite[(1.2.2)]{nh2}: 

\begin{theo}[{\bf Adjunction formula}]\label{theo:adj} 
Let 
$f \col ({\cal T}, {\cal A}) 
\lo ({\cal T}', {\cal A}')$ be 
a morphism of ringed topoi.
Let $(E^{\bul},\{E^{\bul(i)}\}_{i=1}^n)$ 
$($resp.~$(F^{\bul},\{F^{\bul(i)}\}_{i=1}^n))$ 
be an object of 
${\rm K}^-{\rm F}^n({\cal A}')$ 
and ${\rm K}^+{\rm F}^n({\cal A})$.
Then there exists a canonical isomorphism
$${\rm RHom}^{\bul}_{\cal A}
(Lf^*((E^{\bul},\{E^{\bul(i)}\}_{i=1}^n)), 
(F^{\bul},\{F^{\bul(i)}\}_{i=1}^n)) 
\os{=}{\lo} $$ 
$${\rm RHom}^{\bul}_{{\cal A}'}
((E^{\bul},\{E^{\bul(i)}\}_{i=1}^n), 
Rf_*((F^{\bul},\{F^{\bul(i)}\}_{i=1}^n)))$$
in ${\rm DF}^n(\Gam({\cal T},{\cal A}))$. 
The isomorphism above satisfies 
the transitive condition {\rm (cf.~\cite[V Proposition 3.3.1]{bb})}.
\end{theo}
\begin{proof}
The following proof is the bifiltered version of that of \cite[(1.2.2)]{nh2} (cf.~\cite[V Proposition 3.3.1]{bb}).  
\par 
Let $(I^{\bul},\{I^{\bul(i)}\}_{i=1}^n)$
be  a strictly injective resolution of $(F^{\bul},\{F^{\bul(i)}\}_{i=1}^n)$.
Let $(Q^{\bul},\{Q^{\bul(i)}\}_{i=1}^n)$  be a filtered flat resolution 
of $(E^{\bul},\{E^{\bul(i)}\}_{i=1}^n)$. 
Let $(J^{\bul},\{J^{\bul(i)}\}_{i=1}^n)\in {\rm K}^+{\rm F}({\cal I}_{\rm stinj})$ 
be a strictly injective resolution of $f_*((I^{\bul},\{I^{\bul(i)}\}_{i=1}^n))$.
Then we have the following composite morphism 
\begin{equation*}
{\rm Hom}^{\bul}_{\cal A}(f^*((Q^{\bul},\{Q^{\bul(i)}\}_{i=1}^n)),
(I^{\bul},\{I^{\bul(i)}\}_{i=1}^n))
={\rm Hom}^{\bul}_{{\cal A}'}((Q^{\bul},\{Q^{\bul(i)}\}_{i=1}^n),
f_*((I^{\bul},\{I^{\bul(i)}\}_{i=1}^n)))
\tag{5.2.1}\label{eqn:cmadlr}
\end{equation*}
$$\lo {\rm Hom}^{\bul}_{{\cal A}'}
((Q^{\bul},\{Q^{\bul(i)}\}_{i=1}^n),(J^{\bul},\{J^{\bul(i)}\}_{i=1}^n))
\os{\sim}{\longleftarrow}
{\rm Hom}^{\bul}_{{\cal A}'}((E^{\bul},\{E^{\bul(i)}\}),
(J^{\bul},\{J^{\bul(i)}\}_{i=1}^n)).$$
Here the last quasi-isomorphism follows from (\ref{lemm:stex2}) (2).
\par
As in \cite[V Proposition 3.3.1]{bb}, by the transitive condition, 
we have only to prove 
that (\ref{theo:adj}) holds 
for a morphism 
$f \col ({\cal T}, {\cal A}) 
\lo  ({\cal T}, {\cal B})$ of ringed topoi such that 
$f={\rm id}_{\cal T}$ as 
a morphism of topoi. 
As in the trivial filtered case,  
consider the following 
functor $f^!$:
\begin{align*} 
f^{!} \col {\rm MF}({\cal B}) 
\owns (K,\{K^{(i)}\}_{i=1}^n) \lom & 
{\cal H}{\it om}_{\cal B}
(f_*({\cal A}), (K,\{K^{(i)}\}_{i=1}^n))\\
&={\cal H}{\it om}_{\cal B}({\cal A}, (K,\{K^{(i)}\}_{i=1}^n))
\in {\rm MF}({\cal A}).
\end{align*} 
Here we endow $f_*({\cal A})(={\cal A})$ 
with the trivial filtration.
The functor $f^!$ is the right adjoint functor of $f_*$:
\begin{align*}
{\rm Hom}_{\cal A}((M,\{M^{(i)}\}_{i=1}^n), 
f^{!}((K,\{K^{(i)}\}_{i=1}^n)))=
{\rm Hom}_{\cal B}(f_*((M,& \{M^{(i)}\}_{i=1}^n)), 
(K,\{K^{(i)}\}_{i=1}^n)) 
\tag{5.2.2}\label{eqn:bikkuri} \\ 
& (M,\{M^{(i)}\}_{i=1}^n)\in {\rm MF}({\cal A})).  
\end{align*}
By (\ref{eqn:bikkuri}), we see that, if $(K,\{K^{(i)}\}_{i=1}^n)\in {\rm MF}({\cal B})$ is 
a strictly injective ${\cal B}$-module, 
then $f^{!}((K,\{K^{(i)}\}_{i=1}^n))$ is a strictly injective ${\cal A}$-module.
Moreover, for a strictly injective morphism  
$f_*((M,\{M^{(i)}\}_{i=1}^n) \os{\subset}{\lo} (K,\{K^{(i)}\}_{i=1}^n)$ of ${\cal B}$-modules, 
the corresponding morphism
$(M,\{M^{(i)}\}_{i=1}^n) \os{\subset}{\lo} f^{!}((K,\{K^{(i)}\}_{i=1}^n)$ is 
a strictly injective morphism of 
${\cal A}$-modules, which is easily checked.
Hence, by the same proof as that of  
(\ref{prop:exsir})
(especially, by noting that the functor $f^{!}$ 
commutes with the direct product), we can take
$f^{!}((K^{\bul},\{K^{\bul(i)}\}_{i=1}^n))$ as $(I^{\bul},\{I^{\bul(i)}\}_{i=1}^n)$, 
where $(K^{\bul},\{K^{\bul(i)}\}_{i=1}^n)$ is a bounded below 
complex of strictly injective ${\cal B}$-modules.
\par
Let $R^{\bul}$ be a flat 
resolution of $f_*({\cal A})$ 
with the trivial 
filtration. Since the filtration 
on $R^{\bul}$ is trivial,  it is obvious 
that  the morphism
$(Q^q,\{Q^{q(i)}\}_{i=1}^n){\otimes}_{\cal B}R^{\bul}
{\lo} \allowbreak (Q^q,\{Q^{q(i)}\}_{i=1}^n ){\otimes}_{\cal B} f_*({\cal A})$
is a filtered quasi-isomorphism $(q\in {\mab Z})$.
By (\ref{lemm:stex2}) (2) we have the following isomorphism 
\begin{align*}
& {\rm Hom}^{\bul}_{\cal B}((Q^q,\{Q^{q(i)}\}_{i=1}^n){\otimes}_{\cal B}
f_*({\cal A}),(K^{\bul}, \{K^{\bul(i)}\}_{i=1}^n)) \tag{5.2.3}\label{eqn:qar} \\
\os{\sim}{\lo} & {\rm Hom}^{\bul}_{\cal B}((Q^q,\{Q^{q(i)}\}_{i=1}^n){\otimes}_{\cal B}
R^{\bul}, (K^{\bul},\{K^{\bul(i)}\}_{i=1}^n)). 
\end{align*}
(\ref{eqn:qar}) is equal to the following:
\begin{align*}
& {\rm Hom}^{\bul}_{\cal B}((Q^q,\{Q^{q(i)}\}_{i=1}^n), 
{\cal H}{\it om}_{\cal B}(f_*({\cal A}),(K^{\bul},\{K^{\bul(i)}\}_{i=1}^n))) 
\tag{5.2.4}\label{eqn:qfak} \\ 
\os{\sim}{\lo} & {\rm Hom}^{\bul}_{\cal B}((Q^q,\{Q^{q(i)}\}_{i=1}^n), 
{\cal H}{\it om}_{\cal B}^{\bul}(R^{\bul},(K^{\bul},\{K^{\bul(i)}\}_{i=1}^n))). 
\nonumber 
\end{align*}
Here 
${\cal H}{\it om}_{\cal B}
(f_*({\cal A}),(K^{\bul},\{K^{\bul(i)}\}_{i=1}^n))$ is 
considered as a filtered ${\cal B}$-module, which is nothing but 
$f_*f^{!}(K^{\bul},\{K^{\bul(i)}\}_{i=1}^n)$.
It is easy to check that 
${\cal H}{\it om}_{\cal B}(R^{q},(K^{q+n},\{K^{q+n(i)}\}_{i=1}^n))$ is a strictly 
injective ${\cal B}$-module; so 
is ${\cal H}{\it om}_{\cal B}^n(R^{\bul},(K^{\bul},\{K^{\bul(i)}\}_{i=1}^n))$ 
$(n\in {\mab Z})$. Therefore 
${\cal H}{\it om}_{\cal B}^{\bul}(R^{\bul},(K^{\bul},\{K^{\bul(i)}\}_{i=1}^n))$ 
is a strictly injective resolution of $f_*f^{!}(K^{\bul},\{K^{\bul(i)}\}_{i=1}^n)$ 
by the sheafification of (\ref{lemm:stex2}) (2).
Hence we can take 
${\cal H}{\it om}_{\cal B}^{\bul}
(R^{\bul},(K^{\bul},\{K^{\bul(i)}\}_{i=1}^n))$ as 
$(J^{\bul},\{J^{\bul(i)}\}_{i=1}^n)$, and 
we have a filtered quasi-isomorphism
\begin{equation*}
{\rm Hom}^{\bul}_{\cal B}
((Q^q,\{Q^{q(i)}\}_{i=1}^n), 
f_*f^{!}(K^{\bul},\{K^{\bul(i)}\}_{i=1}^n))
\os{\sim}{\lo} {\rm Hom}^{\bul}_{\cal B}((Q^q,\{Q^{q(i)}\}_{i=1}^n), (J^{\bul},\{J^{\bul(i)}\}_{i=1}^n)) 
\tag{5.2.5}\label{eqn:qffkj}
\end{equation*}
by (\ref{eqn:qfak}).
\par
Let $(C^{\bul},\{C^{\bul(i)}\}_{i=1}^n)$ be 
the mapping cone of the morphism 
$f_*f^{!}(K^{\bul},\{K^{\bul(i)}\}_{i=1}^n)\lo (J^{\bul},\{J^{\bul(i)}\}_{i=1}^n)$.
Then we have a triangle 
$${\rm Hom}_{\cal B}^{\bul}((Q^{\bul},\{Q^{\bul(i)}\}_{i=1}^n),f_*f^{!}(K^{\bul},\{K^{\bul(i)}\}_{i=1}^n)) 
\lo {\rm Hom}_{\cal B}^{\bul}((Q^{\bul},\{Q^{\bul(i)}\}_{i=1}^n),(J^{\bul},\{J^{\bul(i)}\}_{i=1}^n))$$
$${\lo}{\rm Hom}_{\cal B}^{\bul}((Q^{\bul},\{Q^{\bul(i)}\}_{i=1}^n),
(C^{\bul},\{C^{\bul(i)}\}_{i=1}^n))\os{+1}{\lo} \cdots.$$
By (\ref{eqn:qffkj}) the filtered complex 
${\rm Hom}_{\cal B}^{\bul}((Q^q,\{Q^{q(i)}\}_{i=1}^n),
(C^{\bul},\{C^{\bul(i)}\}_{i=1}^n))$ is strictly exact.
As in \cite[p.~327]{bb}, by noting that 
$(Q^{\bul},\{Q^{\bul(i)}\}_{i=1}^n)$ is 
bounded above, one can easily check that 
$${\rm Hom}_{\cal B}^{\bul}((Q^{\bul},\{Q^{\bul(i)}\}_{i=1}^n),
(C^{\bul},\{C^{\bul(i)}\}_{i=1}^n))$$ 
is also strictly exact.
Therefore we obtain
$${\rm Hom}^{\bul}_{\cal B}((Q^{\bul},\{Q^{\bul(i)}\}_{i=1}^n), 
f_*f^{!}(K^{\bul},\{K^{\bul(i)}\}_{i=1}^n))
\os{\sim}{\lo} 
{\rm Hom}^{\bul}_{\cal B}((Q^{\bul},\{Q^{\bul(i)}\}_{i=1}^n), 
(J^{\bul},\{J^{\bul(i)}\}_{i=1}^n)),$$ 
which enables us to 
finish the proof of (\ref{theo:adj}).
\end{proof}

Let $(E^{\bul},\{E^{\bul(i)}\}_{i=1}^n)$ 
(resp.~$(F^{\bul},\{F^{\bul(i)}\}_{i=1}^n)$) be 
an object of ${\rm KF}^n({\cal A})$ 
(resp.~${\rm K}^+{\rm F}^n({\cal A})$).
Set 
$${\rm Ext}^q_{{\cal A}}
((E^{\bul},\{E^{\bul(i)}\}_{i=1}^n),(F^{\bul},\{F^{\bul(i)}\}_{i=1}^n)):=
{\rm Hom}_{{\rm DF}^n({\cal A})}
((E^{\bul},\{E^{\bul(i)}\}_{i=1}^n),(F^{\bul},\{F^{\bul(i)}\}_{i=1}^n)[q]).$$

The following lemma is a 
$n$-filtered version of a classical lemma  
\cite[I (6.4)]{hard}.
\begin{lemm}\label{lemm:hqrhomdf}
The following formula holds$:$
\begin{equation*}
H^q(\bigcap^{(1,n)}_{0,0}{\rm RHom}^{\bul}_{\cal A}
((E^{\bul},\{E^{\bul(i)}\}_{i=1}^n),(F^{\bul},\{F^{\bul(i)}\}_{i=1}^n)))= 
\tag{5.3.1}\label{eqn:hqgnderhm}
\end{equation*}
$${\rm Ext}^q_{{\cal A}}
((E^{\bul},\{E^{\bul(i)}\}_{i=1}^n),(F^{\bul},\{F^{\bul(i)}\}_{i=1}^n)).$$
In particular,
\begin{equation*}
H^0(\bigcap^{(1,n)}_{0,0}
{\rm RHom}^{\bul}_{\cal A}((E^{\bul},\{E^{\bul(i)}\}_{i=1}^n), 
(F^{\bul},\{F^{\bul(i)}\}_{i=1}^n)))=
\tag{5.3.2}\label{eqn:h0derhm}
\end{equation*} 
$${\rm Hom}_{{\rm DF}^n({\cal A})}
((E^{\bul},\{E^{\bul(i)}\}_{i=1}^n), 
(F^{\bul},\{F^{\bul(i)}\}_{i=1}^n)).$$
\end{lemm}
\begin{proof} 
By using  (\ref{lemm:filis}) (1), (2)  and (\ref{eqn:homcoh}),  
the proof is the same as that of \cite[(1.2.3)]{nh2}. 
\end{proof}

\section{$\otimes^L_{\cal A}$}\label{sec:otl}
In this section we define the $n$-filtered derived tensor product $\otimes^L_{\cal A}$ of 
two bounded below complexes with $n$-pieces of filtrations.  
As in the previous section, we assume that $n\leq 2$. 
The results in this section are generalizations of results in  \cite[(1.2)]{nh2}. 
The construction of $\otimes^L_{\cal A}$ is more nontrivial than the construction of 
${\rm RHom}^{\bul}$ 
in the previous section because the dual notion (\ref{rema:dual}) (1) does not work 
in this section.

\par 
The following (2) is a key lemma for the definition $\otimes^L_{\cal A}$. 

\begin{lemm}\label{lemm:tensit}
Let $1\leq i_1< i_m\leq n$ $(1\leq m\leq n)$ and $k_1,k_m$ be 
integers. Then the following hold$:$
\par 
$(1)$ 
Let $(E,\{E^{(i)}\}_{i=1}^n)$ be an ${\cal A}$-module with $n$-pieces of filtrations.
Then 
\begin{align*} 
&{\rm gr}^{(i_{m})}_{q_m}E
/{\rm Fil}^{(i_{1})}_{q_1}({\rm gr}^{(i_{m})}_{q_m}E)=
E^{(i_m)}_{q_m}
/(E^{(i_1i_m)}_{q_1q_m}+E^{(i_m)}_{q_m-1}). 
\tag{6.1.1}\label{ali:grfr}
\end{align*} 
\par
$(2)$ 
Let $(E,\{E^{(i)}\}_{i=1}^n)$ 
and $(F,\{F^{(i)}\}_{i=1}^n)$ be ${\cal A}$-modules 
with $n$-pieces of filtrations.
Assume that $(F,\{F^{(i)}\}_{i=1}^n) 
\in {\cal Q}_{\rm fl}({\cal A})$ 
$((\ref{eqn:qfla}))$. 
Then the natural morphism 
\begin{align*} 
\bigoplus_{p_1+q_1=k_1, p_m+q_m=k_m}
{\rm gr}^{(i_1)}_{p_1}  {\rm gr}^{(i_m)}_{p_m}
E {\otimes}_{\cal A} {\rm gr}^{(i_1)}_{q_1} {\rm gr}^{(i_m)}_{q_m}F
{\lo} {\rm gr}^{(i_1)}_{k_1}  {\rm gr}^{(i_m)}_{k_m}(E\otimes_{\cal A}F)
\tag{6.1.2}\label{ali:grm}
\end{align*} 
is an isomorphism. 
\end{lemm}
\begin{proof}
(1):  We omit the proof. 
\par 
(2): 
We proceed on induction on $m$. 
When $m=1$, (2) is nothing but \cite[(1.2.4)]{nh2}. 
Hence we have only to prove (2) for the case $m=n=2$. 
Because $F/ {\rm Fil}^{(i_m)}_{q_m}F$ is a flat ${\cal A}$-module, 
we have the following isomorphism:  
\begin{align*} 
\bigoplus_{p_m+q_m=k_m}
{\rm gr}^{(i_m)}_{p_m}E {\otimes}_{\cal A}  {\rm gr}^{(i_m)}_{q_m}F
\os{\sim}{\lo} {\rm gr}^{(i_m)}_{k_m}(E\otimes_{\cal A}F)
\tag{6.1.3}\label{ali:garm}
\end{align*}  
by \cite[(1.2.4)]{nh2}.  
Hence we have the following isomorphism 
\begin{align*} 
\bigoplus_{p_m+q_m=k_m}{\rm gr}^{(i_1)}_{k_1}
({\rm gr}^{(i_m)}_{p_m}E {\otimes}_{\cal A} {\rm gr}^{(i_m)}_{q_m}F)
\os{\sim}{\lo} {\rm gr}^{(i_1)}_{k_1}{\rm gr}^{(i_m)}_{k_m}(E\otimes_{\cal A}F). 
\tag{6.1.4}\label{ali:gdmrm}
\end{align*} 
Because  
${\rm gr}^{(i_m)}_{q_m}F$ and 
$${\rm gr}^{(i_{m})}_{q_m}F
/{\rm Fil}^{(i_{1})}_{q_1}({\rm gr}^{(i_{m})}_{q_m}F)
=E^{(i_m)}_{q_m}/(E^{(i_1i_m)}_{q_1q_m}+E^{(i_m)}_{q_m-1})$$ 
are flat ${\cal A}$-modules 
by the definition of ${\cal Q}_{\rm fl}({\cal A})$, 
we have the following isomorphism by \cite[(1.2.4)]{nh2}: 
\begin{align*} 
{\rm gr}^{(i_1)}_{k_1}({\rm gr}^{(i_m)}_{p_m}
E {\otimes}_{\cal A}{\rm gr}^{(i_m)}_{q_m}F)
=\bigoplus_{p_1+q_1=k_1}{\rm gr}^{(i_1)}_{p_1} {\rm gr}^{(i_m)}_{p_m}
E {\otimes}_{\cal A} {\rm gr}^{(i_1)}_{q_1}{\rm gr}^{(i_m)}_{q_m}F. 
\tag{6.1.5}\label{ali:gdmprm}
\end{align*} 
(This is the point of this proof.)
By (\ref{ali:gdmrm}) and (\ref{ali:gdmprm}) we have the following isomorphism: 
\begin{align*} 
\bigoplus_{p_1+q_1=k_1, p_m+q_m=k_m}
{\rm gr}^{(i_1)}_{p_1} {\rm gr}^{(i_m)}_{p_m}
E {\otimes}_{\cal A} {\rm gr}^{(i_1)}_{q_1} {\rm gr}^{(i_m)}_{q_m}F
\os{\sim}{\lo} {\rm gr}^{(i_1)}_{k_1}  {\rm gr}^{(i_m)}_{k_m}(E\otimes_{\cal A}F). 
\end{align*} 
\end{proof}

\par 
Let $(E^{\bul},\{E^{\bul(i)}\}_{i=1}^{n})$ and 
$(F^{\bul},\{F^{\bul(i)}\}_{i=1}^{n})$ be objects of ${\rm CF}^{n}({\cal A})$. 
Set 
\begin{equation*}
(E^{\bul}\otimes_{\cal A}F^{\bul})_k^{r(i)}:=
{\rm Im}(\bigoplus_{l+m=k}\bigoplus_{p+q=r}
E^{p(i)}_l\otimes_{\cal A}F^{q(i)}_m\lo 
\bigoplus_{p+q=r}E^p\otimes_{\cal A}F^q).  
\tag{6.1.6} 
\end{equation*}
Then we have a complex 
$$(E^{\bul},\{E^{\bul(i)}\}_{i=1}^{n})
\otimes_{\cal A}(F^{\bul},\{F^{(i)\bul}\}_{i=1}^{n}):=
(E^{\bul}\otimes_{\cal A}F^{\bul},
\{(E^{\bul}\otimes_{\cal A}F^{\bul})_k^{(i)}\}_{i=1}^n)$$ 
of ${\cal A}$-modules with $n$-pieces of filtrations, 
where the boundary morphism
is defined by the following formula 
\begin{equation*}
d\vert_{E^p\otimes_{\cal A}F^q}
=(d^p_E\otimes 1) +(-1)^p(1\otimes d^q_F).
\tag{6.1.7} 
\end{equation*} 
The functor
$$ \otimes \col {\rm CF}^{n}({\cal A})\times {\rm CF}^{n}({\cal A}) \lo 
{\rm CF}^{n}({\cal A})$$
induces a functor 
$$ \otimes \col {\rm KF}^{n}({\cal A})\times {\rm KF}^{n}({\cal A}) \lo {\rm KF}^{n}({\cal A}).$$

\par 
As in \cite[II (4.1)]{hard}, \cite{blec} and \cite[(1.2.5)]{nh2}, 
we need the following key theorem   
to define the following $n$-filtered derived functor 
$$\otimes^L_{\cal A} \col {\rm D}^-{\rm F}^{n}({\cal A})\times 
{\rm D}^-{\rm F}^{n}({\cal A}) \lo {\rm D}^-{\rm F}^{n}({\cal A}).$$

\begin{theo}\label{theo:prtdrt}
Let $(E^{\bul},\{E^{\bul(i)}\}_{i=1}^n)$ 
and $(F^{\bul},\{F^{\bul(i)}\}_{i=1}^{n})$ be two complexes of  
${\cal A}$-modules with $n$-pieces of filtrations. 
Assume that $(F^{\bul},\{F^{\bul(i)}\}_{i=1}^{n})\in {\rm K}^-{\rm F}^{n}({\cal Q}^n_{\rm fl}({\cal A}))$. 
Assume that either 
\par
$({\rm a})$ $(E^{\bul},\{E^{\bul(i)}\}_{i=1}^n)$ is strictly exact 
\par
or 
\par
$({\rm b})$ $(F^{\bul},\{F^{\bul(i)}\}_{i=1}^n)$ is strictly exact 
\parno 
and assume also that 
either
\par
$({\rm c})$ $E^{\bul}$ is bounded above 
\par
or
\par
$({\rm d})$ $F^{\bul}$ is bounded below.
\parno
Then 
$(E^{\bul}\otimes_{\cal A}F^{\bul}, 
\{(E^{\bul}\otimes_{\cal A}F^{\bul})^{(i)}\}_{i=1}^n)$ 
is strictly exact.
\end{theo}
\begin{proof}
By \cite[II (4.1)]{hard}, 
$E^{\bul}\otimes_{\cal A}F^{\bul}$ 
is exact. 
By \cite[(1.2.5)]{nh2} 
$(E^{\bul}\otimes_{\cal A}F^{\bul})^{(i_1)}_{k_1}$ is exact for 
$k_1\in {\mab Z}$. 
Hence we have only to prove that 
$(E^{\bul}\otimes_{\cal A}F^{\bul})^{(i_1i_m)}_{k_1 k_m}$ 
is exact for $1\leq i_1 < i_m \leq n$ $(1\leq m\leq n)$ and 
$k_1, k_m\in {\mab Z}$.
\par
Let $G^{\bul \bul}$ be a double complex defined by
$G^{pq}:=E^p\otimes_{\cal A}F^q$ with $n$-pieces of filtrations
$G^{pq(i)}_k:={\rm Im}(\us{l+m=k}{\bigoplus}E^{p(i)}_l\otimes_{\cal A}F^{q(i)}_m\lo G^{pq})$ 
$(1\leq i\leq n)$. Set 
$G^{pq(i_1i_m)}_{k_1 k_m}:=
G^{pq(i_1)}_{k_1}\cap G^{pq(i_m)}_{k_m}$. 
Then we have the following two spectral sequences
$$E_2^{pq}={\cal H}^p_{\rm II}
{\cal H}^q_{\rm I}(G^{\bul \bul(i_1i_m)}_{k_1 k_m}) 
\Lo
{\cal H}^{p+q}
((E^{\bul}\otimes_{\cal A}F^{\bul})^{(i_1i_m)}_{k_1 k_m}),$$
$$E_2^{pq}={\cal H}^p_{\rm I}
{\cal H}^q_{\rm II}(G^{\bul \bul(i_1i_m)}_{k_1 k_m}) \Lo
{\cal H}^{p+q}
((E^{\bul}\otimes_{\cal A}F^{\bul})^{(i_1i_m)}_{k_1 k_m}).$$
The assumption (c) or (d) implies 
that the two spectral sequences above are 
bounded and regular.
\par
First, assume that (a) holds. 
Set 
$E_{\infty}^{p(i)}:=\bigcup_{k\in {\mab Z}}E_{k}^{p(i)}$ 
and $F_{\infty}^{q(i)}:=
\bigcup_{k\in {\mab Z}}F_{k}^{q(i)}$  $(p, q\in {\mab Z})$. 
Since $F^{q(i)}_k$ 
$(\forall k \in {\mab Z})$ is a flat ${\cal A}$-module, $F_{\infty}^{q(i)}$ is also so. 
Because the complex $E^{\bul(i)}_{\infty}$ is exact by (a), 
$E^{\bul(i)}_{\infty}\otimes_{\cal A}F^{q(i)}_{\infty}$ is also so.
We prove that $(E^{\bul}\otimes_{\cal A}F^q)^{(i_1i_m)}_{k_1 k_m}$ 
for any $1\leq i_1 < i_m\leq n$ $(1\leq m\leq n)$ 
and $k_1, k_m\in {\mab Z}$ is exact. 
We have only to prove that 
\begin{equation*}
{\rm Im}((E^{p-1}\otimes_{\cal A}F^{q})^{(i_1i_m)}_{k_1 k_m} 
\!\rightarrow\!
(E^{p}\otimes_{\cal A}F^{q})^{(i_1i_m)}_{k_1 k_m} ) \supset 
{\rm Ker}((E^p\otimes_{\cal A}F^q)^{(i_1i_m)}_{k_1 k_m} 
\!\rightarrow\! (E^{p+1}\otimes_{\cal A}F^q)^{(i_1i_m)}_{k_1 k_m}).
\tag{6.2.1}\label{eqn:imaske}
\end{equation*} 
Set 
$$E^{\bul(i_1i_m)}_{\infty \infty}:=\bigcap_{j=1}^m\bigcup_{k_j\in {\mab Z}}E^{\bul(i_j)}_{k_j}
=\bigcup_{k_1,k_m\in {\mab Z}}E^{\bul(i_1i_m)}_{k_1 k_m}$$ 
and 
$$F^{\bul(i_1i_m)}_{\infty \infty}:=\bigcap_{j=1}^m\bigcup_{k_j\in {\mab Z}}F^{\bul(i_j)}_{k_j}
=\bigcup_{k_1,k_m\in {\mab Z}}F^{\bul(i_1i_m)}_{k_1 k_m}.$$   
Because 
$(F^{\bul},\{F^{\bul(i)}\}_{i=1}^{n})\in {\rm K}^-{\rm F}^{n}({\cal Q}^n_{\rm fl}({\cal A}))$, 
$F^q/F^{q(i_1i_m)}_{k_1 k_m}$ and $F^{q(i_1i_m)}_{k_1 k_m}$ are flat ${\cal A}$-modules, 
$F^q/F^{q(i_1i_m)}_{\infty \infty}$ and 
$F^{q(i_1i_m)}_{\infty \infty}$ are also flat ${\cal A}$-modules. 
Consequently the following natural composite morphism 
$$E^{\bul(i_1i_m)}_{\infty \infty}\otimes_{\cal A}F^{q(i_1i_m)}_{\infty \infty}
\lo E^{\bul}\otimes_{\cal A}F^{q(i_1i_m)}_{\infty \infty}\lo 
E^{\bul}\otimes_{\cal A}F^q$$ 
is injective. 
Because the lower horizontal sequence of 
the following commutative diagram 
\begin{equation*}
\begin{CD}
(E^{p-1} {\otimes}_{\cal A} F^q)^{(i_1i_m)}_{k_1 k_m}
@>>> 
(E^p {\otimes}_{\cal A} F^q)^{(i_1i_m)}_{k_1 k_m} @>>> 
(E^{p+1} {\otimes}_{\cal A} F^q)^{(i_1i_m)}_{k_1 k_m}  \\ 
@VVV @VVV @VVV  \\ 
E^{p-1(i_1i_m)}_{\infty \infty} {\otimes}_{\cal A} F^{q(i_1i_m)}_{\infty \infty} 
@>>> E^{p(i_1i_m)}_{\infty \infty} {\otimes}_{\cal A} F^{q(i_1 i_m)}_{\infty \infty} 
@>>> E^{p+1(i_1 i_m)}_{\infty \infty} {\otimes}_{\cal A} F^{q(i_1 i_m)}_{\infty \infty} 
\end{CD}
\tag{6.2.2}\label{cd:ef6}
\end{equation*}
is exact by the assumption (a), we may assume that  
$E^{\bul}=E^{\bul(i_1i_m)}_{\infty \infty}$ and 
$F^{q}=F^{q(i_1i_m)}_{\infty \infty}$.
Set $B^{p(i_1i_m)}_{\circ}:={\rm Im}(E^{p-1(i_1i_m)}_{\circ} \lo 
E^{p(i_1 i_m)}_{\circ})$ $(\circ=k_1 k_m \in {\mab Z}^m$ or nothing).
Then the sequence 
$0\lo B^{p-1(i_1 i_m)}_{\circ} \lo E^{p(i_1 i_m)}_{\circ} \lo B^{p+1(i_1 i_m)}_{\circ} \lo 0$
is exact by the assumption (a). 
Moreover we have the following commutative diagram 
with lower exact row:

\begin{equation}
\begin{CD}
@.(B^{p} {\otimes}_{\cal A} F^q)^{(i_1 i_m)}_{k_1 k_m} @>>>
(E^p {\otimes}_{\cal A} F^q)^{(i_1 i_m)}_{k_1 k_m}  @>>>
(B^{p+1} {\otimes}_{\cal A} F^q)^{(i_1 i_m)}_{k_1 k_m}  @. \\ 
@. @V{\bigcap}VV @V{\bigcap}VV @V{\bigcap}VV  @.\\ 
0 @>>>\!\! B^{p} {\otimes}_{\cal A} F^q \!@>>>\! 
E^p {\otimes}_{\cal A} F^q \!@>>>\! 
B^{p+1} {\otimes}_{\cal A} F^q \!\!@>>> 0.
\end{CD}
\tag{6.2.3}\label{cd:kef2}
\end{equation}
We claim that, to prove (\ref{eqn:imaske}),  
it suffices to prove that the following sequence 
\begin{equation*}
0\lo {\rm gr}^{(i_1)}_{k_1} {\rm gr}^{(i_m)}_{k_m} (B^{p}\otimes_{\cal A}F^q)
\lo {\rm gr}^{(i_1)}_{k_1} {\rm gr}^{(i_m)}_{k_m}(E^p\otimes_{\cal A}F^q)
\lo {\rm gr}^{(i_1)}_{k_1} {\rm gr}^{(i_m)}_{k_m}(B^{p+1}\otimes_{\cal A}F^q) 
\tag{6.2.4}\label{eqn:grkef}
\end{equation*}
is exact. Indeed, let $s$ be a local 
section of the sheaf on the right hand 
side of (\ref{eqn:imaske}). 
Then, by the lower exact sequence of (\ref{cd:kef2}) 
and by the 
assumptions $E^{\bul(i_1 i_m)}=E^{\bul(i_1 i_m)}_{\infty}$ 
and $F^{q(i_1 i_m)}=F^{q(i_1 i_m)}_{\infty}$, 
there exist integers $k'_1 \geq k_1, k'_m\geq k_m$  such that 
$s\in (B^p\otimes_{\cal A}F^q)^{(i_1 i_m)}_{k'_1 k'_m}$. 
If $k'_j=k_j$ for any $1\leq j\leq m$, there is nothing to prove.
If there exists $k'_j> k_j$ for some $1\leq j\leq m$, 
then (\ref{eqn:grkef}) for $k'_j$ and either of $k'_{j+1}$ or $k'_{j-1}$ implies 
that 
$s\in (B^p\otimes_{\cal A}F^q)^{(i_1 i_m)}_{k'_1 k'_{j-1}k'_j-1k'_{j+1} k'_m}$ 
by the injectivity of the morphism 
${\rm gr}^{(i_1)}_{k'_1} {\rm gr}^{(i_m)}_{k'_m} (B^p\otimes_{\cal A}F^q)
\lo {\rm gr}^{(i_1)}_{k'_1} {\rm gr}^{(i_m)}_{k'_m}(E^p\otimes_{\cal A}F^q)$ in 
(\ref{eqn:grkef}) with the replacement of $k_1$ and $k_m$ by $k'_1$ and $k'_m$, respectively, 
since the image of $s$ in 
${\rm gr}^{(i_1)}_{k'_1} {\rm gr}^{(i_m)}_{k'_m}(E^p\otimes_{\cal A}F^q)$ is zero.
Repeating this process, we see that 
$s\in (B^p\otimes_{\cal A}F^q)^{(i_1 i_m)}_{k_1 k_m}$. 
This means that 
$s\in 
{\rm Im}((E^{p-1}\otimes_{\cal A}F^{q})^{(i_1 i_m)}_{k_1 k_m} 
\!\rightarrow\!
(E^{p}\otimes_{\cal A}F^{q})^{(i_1 i_m)}_{k_1 k_m} )$. 
\par
Now let us prove the exactness of (\ref{eqn:grkef}).
\par
By (\ref{lemm:tensit}) (2) we have only to 
prove that the following sequence 
\begin{align*} 
0&\lo 
\bigoplus_{l_1+l'_1=k_1,l_m+l'_m=k_m}
{\rm gr}^{(i_1)}_{l_1} {\rm gr}^{(i_m)}_{l_m}B^p\otimes_{\cal A}
{\rm gr}^{(i_1)}_{l'_1} {\rm gr}^{(i_m)}_{l'_m}(F^q)
\tag{6.2.5}\label{eqn:dsgrkef}\\
& \lo 
\bigoplus_{l_1+l'_1=k_1,l_m+l'_m=k_m}
{\rm gr}^{(i_1)}_{l_1} {\rm gr}^{(i_m)}_{l_m}E^p\otimes_{\cal A}
{\rm gr}^{(i_1)}_{l'_1} {\rm gr}^{(i_m)}_{l'_m}(F^q)\\
& \lo 
\bigoplus_{l_1+l'_1=k_1,l_m+l'_m=k_m}
{\rm gr}^{(i_1)}_{l_1} {\rm gr}^{(i_m)}_{l_m}B^{p+1}\otimes_{\cal A}
{\rm gr}^{(i_1)}_{l'_1} {\rm gr}^{(i_m)}_{l'_m}(F^q)
\end{align*} 
is exact. 
Because $(F^{\bul},\{F^{\bul(i)}\}_{i=1}^{n})\in {\rm K}^-{\rm F}^{n}({\cal Q}^n_{\rm fl}({\cal A}))$, 
${\rm gr}^{(i_1)}_{l'_1} {\rm gr}^{(i_m)}_{l'_m}(F^q)$ is a  
flat ${\cal A}$-module by (\ref{ali:ppe}). 
Hence (\ref{eqn:dsgrkef}) is 
exact by the assumption (a), (\ref{prop:nsq}) and (\ref{prop:exbcb}). 
\par
Next assume that (b) holds.
Since $F_{\infty \infty}^{\bul(i_1i_m)}$ is bounded above, 
$E^p_{\infty \infty}\otimes_{\cal A}F^{\bul(i_1i_m)}_{\infty \infty}$ is exact.
We prove that $(E^p_{\infty \infty}\otimes_{\cal A}
F^{\bul(i_1i_m)}_{\infty \infty})_{k_1 k_m}$ is exact. 
Set $K^q_{\circ}:={\rm Im}(F^{q-1}_{\circ} \lo 
F^{q}_{\circ})$ $( \circ=k_1 k_m\in {\mab Z}^m$ or $\infty \infty$).
Then $K^q_{\circ}$ is a flat ${\cal A}$-module 
since $F^{\bul}_{\circ}$ is bounded above.
As in the case (a), we have only 
to prove that the following sequence
\begin{align*} 
0&\lo \bigoplus_{l_1+l'_1=k_1,l_m+l'_m=k_m}
{\rm gr}^{(i_1)}_{l_1} {\rm gr}^{(i_m)}_{l_m}E^p\otimes_{\cal A}
{\rm gr}^{(i_1)}_{l'_1} {\rm gr}^{(i_m)}_{l'_m}K^{q}
\tag{6.2.6}\label{eqn:grefl}\\
&\lo \bigoplus_{l_1+l'_1=k_1,l_m+l'_m=k_m}
{\rm gr}^{(i_1)}_{l_1} {\rm gr}^{(i_m)}_{l_m}
E^p\otimes_{\cal A}{\rm gr}^{(i_1)}_{l'_1} {\rm gr}^{(i_m)}_{l'_m}F^q \\
& \lo \bigoplus_{l_1+l'_1=k_1,l_m+l'_m=k_m}
{\rm gr}^{(i_1)}_{l_1} {\rm gr}^{(i_m)}_{l_m}E^p\otimes_{\cal A}
{\rm gr}^{(i_1)}_{l'_1} {\rm gr}^{(i_m)}_{l'_m}K^{q+1} 
\lo 0
\end{align*}
is exact.
By the assumption (b), (\ref{prop:nsq}) and (\ref{prop:exbcb}), the following sequence 
\begin{align*} 
0\lo {\rm gr}^{(i_1)}_{l'_1} {\rm gr}^{(i_m)}_{l'_m}K^{q}
\lo  {\rm gr}^{(i_1)}_{l'_1} {\rm gr}^{(i_m)}_{l'_m}F^q \lo 
{\rm gr}^{(i_1)}_{l'_1} {\rm gr}^{(i_m)}_{l'_m}K^{q+1} 
\lo 0
\end{align*}
is exact. Since $F^{\bul}$ is bounded above, we see that 
${\rm gr}^{(i_1)}_{l'_1} {\rm gr}^{(i_m)}_{l'_m}K^{q}$ is a flat ${\cal A}$-module 
by descending induction on $q$.  
Hence (\ref{eqn:grefl}) is exact.
\par
We finish the proof. 
\end{proof}

\parno
By using (\ref{theo:prtdrt}) (b),   
we have the following derived functor:  
\begin{equation*}
\otimes^L_{\cal A} \col 
{\rm K}^-{\rm F}^n({\cal A})\times 
{\rm D}^-{\rm F}^n({\cal A}) \lo
{\rm D}^-{\rm F}^n({\cal A}).  
\tag{6.2.7}\label{eqn:dirtsr}
\end{equation*}
By (\ref{theo:prtdrt}) (a) 
the functor above induces the following derived functor 
(cf.~\cite[II \S4]{hard})
\begin{equation*}
\otimes^L_{\cal A} \col 
{\rm D}^-{\rm F}^n({\cal A})\times 
{\rm D}^-{\rm F}^n({\cal A}) \lo
{\rm D}^-{\rm F}^n({\cal A}).  
\tag{6.2.8}\label{eqn:drtsr}
\end{equation*} 


\begin{rema}
The derived tensor product $\otimes^L_{\cal A}$ for the case $n=1$ 
(resp.~$n=2$) has 
a key role for the construction of the $l$-adic weight spectral sequence 
(resp.~the construction of a fundamental bifiltered complex) of 
a proper SNCL scheme 
over a family of log points with a relative horizontal SNCD  
in the second part of this paper.   
\end{rema}

\section{Complements}\label{sec:comp}
This section is a complement of \S\ref{sec:sti} and \S\ref{sfr}.
Let the notations be as in \S\ref{sec:nfil}.  
\par
Assume that $n\leq 2$. 
Let $f \col ({\cal T},{\cal A}) \lo ({\cal T}',{\cal A}')$ 
be a morphism of ringed topoi. 
As in \cite[\S3]{nh2} we consider the following  larger 
full subcategory category ${\cal I}^n_{f_*{\textrm -}{\rm acyc}}({\cal A})$ than 
${\cal I}^n_{\rm flas}({\cal A})$ in 
${\rm MF}^n({\cal A})$ (cf.~\cite[(1.4.5)]{dh2}):
\begin{align*} 
{\cal I}^n_{f_*{\textrm -}{\rm acyc}}
({\cal A}):=\{(J,\{J^{(i)}\}_{i=1}^n)~\vert~&J~\text{and}~ 
J^{(i_1i_n)}_{k_1k_n}~\text{are}~f_*\text{-acyclic}~\text{for}
~1\leq i_1 \leq i_n\leq n\\
&~\text{and}~k_1,k_n \in {\mab Z}\}
\end{align*}

Then the following holds by 
(\ref{prop:exsir}):

\begin{prop}
The canonical morphism 
${\rm K}^+{\rm F}^n({\cal I}^n_{f_*{\textrm -}{\rm acyc}}({\cal A}))
\lo {\rm D}^+{\rm F}^n({\cal A})$ induces an equivalence  
$${\rm K}^+{\rm F}^n
({\cal I}^n_{f_*{\textrm -}{\rm acyc}}
({\cal A}))_{{(\rm F}^n{\rm Qis})}
\os{\sim}{\lo}
{\rm D}^+{\rm F}^n({\cal A})$$ 
of categories 
and the right derived functor
$Rf_*$ is calculated by the following formula
$Rf_*([(J^{\bul},\{J^{\bul(i)}\}_{i=1}^n)])=[f_*((J^{\bul},\{J^{\bul(i)}\}_{i=1}^n))]$ 
$((J^{\bul},\{J^{\bul(i)}\}_{i=1}^n) \in 
{\rm K}^+{\rm F}^n
({\cal I}^n_{f_*{\textrm -}{\rm acyc}}({\cal A})))$.
\end{prop}

\parno
We can consider the dual notion of the above as follows. 
Set 
\begin{align*}
{\cal Q}^n_{f^*{\textrm -}{\rm acyc}}({\cal A})
:=\{ (Q,\{Q^{(i)}\}_{i=1}^n)~\vert~ &
Q \text{ and } Q/\sum_{j=1}^NQ^{(i^j_1i^j_n)}_{k^j_1k^j_n}
\text{ are }f^*{\textrm-}\text{ acyclic}  \\
{} & ~(N\in {\mab Z}_{\geq 1},1\leq \forall i^j_1 \leq \forall i^j_n \leq n, 
\forall k^j_1, \forall k^j_n \in {\mab Z})\}. 
\end{align*}
Then the following holds by (\ref{prop:esfrl}):

\begin{prop}
The canonical morphism 
${\rm K}^-{\rm F}^n({\cal Q}^n_{f^*{\textrm -}{\rm acyc}}({\cal A}'))
\lo {\rm D}^-{\rm F}^n({\cal A}')$ induces an equivalence  
$${\rm K}^-{\rm F}^n
({\cal Q}^n_{f^*{\textrm -}{\rm acyc}}({\cal A}'))_{({\rm F}^n{\rm Qis})}
\os{\sim}{\lo}
{\rm D}^-{\rm F}^n({\cal A}')$$ 
of categories 
and the left derived functor
$Lf^*$ is calculated by the following formula
$Lf^*[(Q^{\bul},\{Q^{\bul(i)}\}_{i=1}^n)]=[f^*((Q^{\bul},\{Q^{\bul(i)}\}_{i=1}^n))]$
$((Q^{\bul},\{Q^{\bul(i)}\}_{i=1}^n) \in 
{\rm K}^-{\rm F}^n({\cal Q}^n_{f^*{\textrm -}{\rm acyc}}({\cal A}')))$.
\end{prop}
\par
Next we define the {\it gr}-{\it functor}. 
For a sequence $\ul{k}=(k_1, k_n)$ of integers 
and for $(i_1, i_n)$ $(1\leq i_1\leq i_n \leq n)$, 
there exists the following functor 
\begin{align*}
{\rm gr}^{(i_1)}_{k_1} {\rm gr}^{(i_n)}_{k_n}
 \col  
{\rm K}^{\star}{\rm F}^n({\cal A})\owns 
(E^{\bul},\{E^{\bul(i)}\}_{i=1}^n)\lom &  
{\rm gr}^{(i_1)}_{k_1}{\rm gr}^{(i_n)}_{k_n}E^{\bul}
\in{\rm K}^{\star}({\cal A}) \tag{7.2.1}\label{ali:grb}\\
{} & (\star=+,-,{\rm b},~{\rm nothing}), 
\end{align*}
which we call the {\it gr}-{\it functor}. 

\begin{lemm}\label{lemm:gris}
If $f \col (E^{\bul},\{E^{\bul(i)}\}_{i=1}^n) \lo 
(F^{\bul},\{F^{\bul(i)}\}_{i=1}^n)$ is 
an $n$-filtered quasi-isomorphism in ${\rm KF}^n({\cal A})$,  then 
${\rm gr}^{(i_1)}_{k_1} 
{\rm gr}^{(i_n)}_{k_n}(f) \col 
{\rm gr}^{(i_1)}_{k_1} 
{\rm gr}^{(i_n)}_{k_n}E^{\bul} \lo 
{\rm gr}^{(i_1)}_{k_1} {\rm gr}^{(i_n)}_{k_n}F^{\bul}$ 
is a quasi-isomorphism.
\end{lemm}
\begin{proof}
Using (\ref{ali:ppe}) and (\ref{ali:pipe}), one obtains
(\ref{lemm:gris}).
\end{proof}
By (\ref{lemm:gris}) the gr-functor (\ref{ali:grb}) induces 
a functor 
\begin{equation*}
{\rm gr}^{(i_1)}_{k_1} {\rm gr}^{(i_n)}_{k_n} \col  
{\rm D}^{\star}{\rm F}^n({\cal A}) 
\lo {\rm D}^{\star}({\cal A}).
\tag{7.3.1}\label{eqn:grdfa}
\end{equation*}
We also call this functor the {\it gr}-{\it functor}.

\begin{lemm}\label{lemm:grrdld}
For a morphism $f \col ({\cal T},{\cal A}) \lo 
({\cal T}',{\cal A}')$ of ringed topoi, 
the following diagrams are commutative$:$
\begin{equation*}
\begin{CD}
{\rm D}^+{\rm F}^n({\cal A}) 
@>{{\rm gr}^{(i_1)}_{k_1} 
{\rm gr}^{(i_n)}_{k_n}}>> {\rm D}^+({\cal A})\\ 
@V{Rf_*}VV  @VV{Rf_*}V \\
{\rm D}^+{\rm F}^n({\cal A}')  
@>{{\rm gr}^{(i_1)}_{k_1} {\rm gr}^{(i_n)}_{k_n}}>> {\rm D}^+({\cal A}'),
\end{CD}
\tag{7.4.1}\label{cd:grfcty}
\end{equation*} 
\begin{equation*}
\begin{CD}
{\rm D}^-{\rm F}^n({\cal A}) 
@>{{\rm gr}^{(i_1)}_{k_1}{\rm gr}^{(i_n)}_{k_n}}>> {\rm D}^-({\cal A})\\ 
@A{Lf^*}AA  @AA{Lf^*}A \\
{\rm D}^-{\rm F}^n({\cal A}')  
@>{{\rm gr}^{(i_1)}_{k_1} {\rm gr}^{(i_n)}_{k_n}}>> {\rm D}^-({\cal A}').
\end{CD}
\tag{7.4.2}\label{eqn:grfld}
\end{equation*}
\end{lemm}  
\begin{proof} 
First we prove that the diagram (\ref{cd:grfcty}) is commutative. 
Let $(E^{\bul},\{E^{\bul(i)}\}_{i=1}^n)$ be 
an object of ${\rm K}^+{\rm F}^n({\cal A})$ 
and 
let $(I^{\bul},\{I^{\bul(i)}\}_{i=1}^n)$ be 
an $n$-filtered flasque resolution of 
$(E^{\bul},\{E^{\bul(i)}\}_{i=1}^n)$. 
Then 
${\rm gr}^{(i_1)}_{k_1} 
{\rm gr}^{(i_n)}_{k_n}I^{\bul}$ 
is a flasque resolution of 
${\rm gr}^{(i_1)}_{k_1} 
{\rm gr}^{(i_n)}_{k_n}E^{\bul}$ 
by (\ref{lemm:gris}). Hence 
$Rf_*({\rm gr}^{(i_1)}_{k_1} 
{\rm gr}^{(i_n)}_{k_n}E^{\bul})= 
f_*({\rm gr}^{(i_1)}_{k_1} 
{\rm gr}^{(i_n)}_{k_n}I^{\bul})$. 
Because 
$\sum_{j=1}^n\bigcap^{(i_1,i_n)}_{k_1,k_j-1, k_n}
(I^{q},\{I^{q(i)}\}_{i=1}^n)$ is flasque,  
$$R^1f_*(\sum_{j=1}^n\bigcap^{(i_1,i_n)}_{k_1,k_j-1, k_n}
(I^{q},\{I^{q(i)}\}_{i=1}^n))=0 \quad (\forall q \in {\mab Z}).$$ 
Hence the following sequence
$$0\lo 
f_*(\sum_{j=1}^n\bigcap^{(i_1, i_n)}_{k_1,k_j-1, k_n}
(I^{q},\{I^{q(i)}\}_{i=1}^n))  \lo  
f_*(\bigcap^{(i_1,i_n)}_{k_1,k_n}
(I^{q},\{I^{q(i)}\}_{i=1}^n)) $$ 
$$\lo f_*({\rm gr}^{(i_1)}_{k_1} {\rm gr}^{(i_n)}_{k_n}I^q) \lo 0 
\quad (\forall q \in {\mab Z})$$ 
is exact and 
\begin{align*}
{\rm gr}^{(i_1)}_{k_1} 
{\rm gr}^{(i_n)}_{k_n}f_*((I^{\bul},\{I^{\bul(i)}\}_{i=1}^n))&=f_*(\bigcap^{(i_1, i_n)}_{k_1,k_n}
(I^{\bul},\{I^{\bul(i)}\}_{i=1}^n))
/f_*(\sum_{j=1}^n\bigcap^{(i_1, i_n)}_{k_1,k_j-1,k_n}
(I^{\bul},\{I^{\bul(i)}\}_{i=1}^n))\\
&=
f_*({\rm gr}^{(i_1)}_{k_1} {\rm gr}^{(i_n)}_{k_n}I^{\bul}). 
\end{align*} 
Hence 
$${\rm gr}^{(i_1)}_{k_1} 
{\rm gr}^{(i_n)}_{k_n}Rf_*((E^{\bul},\{E^{\bul(i)}\}_{i=1}^n))=
f_*({\rm gr}^{(i_1)}_{k_1} {\rm gr}^{(i_n)}_{k_n}I^{\bul})=
Rf_*({\rm gr}^{(i_1)}_{k_1} {\rm gr}^{(i_n)}_{k_n}E^{\bul}).$$
\par
Next  we prove that the diagram (\ref{eqn:grfld}) is commutative. 
Let $(E^{\bul},\{E^{\bul(i)}\}_{i=1}^n)$ be 
an object of ${\rm K}^+{\rm F}^n({\cal A}')$ 
and 
let $(Q^{\bul},\{Q^{\bul(i)}\}_{i=1}^n)$ be 
an $n$-filtered flat resolution of 
$(E^{\bul},\{E^{\bul(i)}\}_{i=1}^n)$. 
Then ${\rm gr}^{(i_1)}_{k_1} 
{\rm gr}^{(i_n)}_{k_n}Q^{\bul}$ 
is a flat resolution of ${\rm gr}^{(i_1)}_{k_1} 
{\rm gr}^{(i_n)}_{k_n}E^{\bul}$ 
by (\ref{lemm:gris}). Hence 
$Lf^*({\rm gr}^{(i_1)}_{k_1} 
{\rm gr}^{(i_n)}_{k_n}E^{\bul})= 
f^*({\rm gr}^{(i_1)}_{k_1} 
{\rm gr}^{(i_n)}_{k_n}Q^{\bul})$ 
and  
the following sequence
$$0\lo 
f^*(\sum_{j=1}^n\bigcap^{(i_1, i_n)}_{k_1,k_j-1,k_n}
(Q^{q},\{Q^{q(i)}\}_{i=1}^n))  \lo  
f^*(\bigcap^{(i_1,i_n)}_{k_1,k_n}
(Q^{q},\{Q^{q(i)}\}_{i=1}^n)) $$ 
$$\lo f^*({\rm gr}^{(i_1)}_{k_1} {\rm gr}^{(i_n)}_{k_n}Q^q) \lo 0 
\quad (\forall q \in {\mab Z})$$ 
is exact. 
This shows that 
\begin{align*}
{\rm gr}^{(i_1)}_{k_1} 
{\rm gr}^{(i_n)}_{k_n}
f^*((Q^{\bul},\{Q^{\bul(i)}\}_{i=1}^n))
&=f^*(\bigcap^{(i_1, i_n)}_{k_1,k_n}
(Q^{\bul},\{Q^{\bul(i)}\}_{i=1}^n))
/f^*(\sum_{j=1}^n\bigcap^{(i_1, i_n)}_{k_1,k_j-1,k_n}
(Q^{\bul},\{Q^{\bul(i)}\}_{i=1}^n))\\
&=
f^*({\rm gr}^{(i_1)}_{k_1} {\rm gr}^{(i_n)}_{k_n}Q^{\bul}).  
\end{align*} 
Hence 
$${\rm gr}^{(i_1)}_{k_1} 
{\rm gr}^{(i_n)}_{k_n}Lf^*((E^{\bul},\{E^{\bul(i)}\}_{i=1}^n))=
f^*({\rm gr}^{(i_1)}_{k_1} {\rm gr}^{(i_n)}_{k_n}Q^{\bul})=
Lf^*({\rm gr}^{(i_1)}_{k_1} {\rm gr}^{(i_n)}_{k_n}E^{\bul}).$$
\end{proof}

\par
Lastly we give the definition of  
the {\it taking the filtration functor} $\pi^{(i)}_k$ 
and 
the {\it forgetting filtration functor} $\pi$. 
\par 
Let $k$ be an integer.
The following morphisms 
\begin{equation*}
\pi^{(i)}_k \col 
{\rm K}^{\star}{\rm F}^n({\cal A})
\owns (E^{\bul},\{E^{\bul(i)}\}_{i=1}^n) 
\lom E^{\bul(i)}_{k} \in {\rm K}^{\star}({\cal A}) 
\quad (\star=+,-,\text{ b, nothing})
\tag{7.4.3}\label{eqn:ffltfl}
\end{equation*}
and 
\begin{equation*}
\pi \col 
{\rm K}^{\star}{\rm F}^n({\cal A})
\owns (E^{\bul},\{E^{\bul(i)}\}_{i=1}^n) 
\lom E^{\bul} \in {\rm K}^{\star}({\cal A}) 
\quad (\star=+,-,\text{ b, nothing})
\tag{7.4.4}\label{eqn:fflnfl}
\end{equation*} 
induce morphisms
\begin{equation*}
\pi^{(i)}_k \col 
{\rm D}^{\star}{\rm F}^n({\cal A}) 
\lo D^{\star}({\cal A}) 
\tag{7.4.5}\label{eqn:dcfn}
\end{equation*}
and 
\begin{equation*}
\pi \col {\rm D}^{\star}{\rm F}^n({\cal A}) 
\lo D^{\star}({\cal A}),
\tag{7.4.6}\label{eqn:dcnfn}
\end{equation*}
respectively. 
It is easy to check the following 
diagrams are commutative:
\begin{equation*}
\begin{CD}
{\rm D}^+{\rm F}^n({\cal A}) 
@>{\pi^{(i)}_{k}}>> D^+({\cal A})\\ 
@V{Rf_*}VV  @VV{Rf_*}V \\
{\rm D}^+{\rm F}^n({\cal A}') 
@>{\pi^{(i)}_{k}}>> D^+({\cal A}'),
\end{CD}
\tag{7.4.7}\label{cd:pidffdn}
\end{equation*}
\begin{equation*}
\begin{CD}
{\rm D}^+{\rm F}^n({\cal A}) 
@>{\pi}>> D^+({\cal A})\\ 
@V{Rf_*}VV  @VV{Rf_*}V \\
{\rm D}^+{\rm F}^n({\cal A}') 
@>{\pi}>> D^+({\cal A}'),
\end{CD}
\tag{7.4.8}\label{cd:pidffn}
\end{equation*}
\begin{equation*}
\begin{CD}
{\rm D}^-{\rm F}^n({\cal A}) 
@>{\pi^{(i)}_{k}}>> {\rm D}^-({\cal A})\\ 
 @A{Lf^*}AA  @AA{Lf^*}A \\
{\rm D}^-{\rm F}^n({\cal A}') 
@>{\pi^{(i)}_{k}}>> D^-({\cal A}')
\end{CD}
\tag{7.4.9}\label{cd:pildff}
\end{equation*}
\begin{equation*}
\begin{CD}
{\rm D}^-{\rm F}^n({\cal A}) 
@>{\pi}>> D^-({\cal A})\\ 
 @A{Lf^*}AA  @AA{Lf^*}A \\
{\rm D}^-{\rm F}^n({\cal A}') 
@>{\pi}>> D^-({\cal A}'). 
\end{CD}
\tag{7.4.10}\label{cd:pnildff}
\end{equation*}

\section{A remark on bifiltered derived categories}\label{sec:rksfdc}
In this section we prove some 
properties of complexes 
of ${\cal A}$-modules with $n$-pieces of filtrations 
by using the formulation of 
the derived category of complexes of objects of quasi-abelian categories in \cite{sc}.
See also \cite{ss} for the general theory of 
the derived category of complexes of objects of 
quasi-abelian categories. 
 
\par
Let the notations be as in \S\ref{sec:nfil}.
Let $n\leq 2$ be a positive integer. 
Let ${\cal I}$ be an additive full 
subcategory of ${\rm MF}^n({\cal A})$ satisfying 
the following three conditions which are the dual part of 
\cite[Definition 1.3.2 (a), (b), (c)]{sc}.
\bigskip
\parno
(8.0.1):  For any object 
$(E,\{E^{(i)}\}_{i=1}^n) \in{\rm MF}^n({\cal A})$, 
there exists an object  
$(I,\{I^{(i)}\}_{i=1}^n)\in {\cal I}$ with a 
strictly injective morphism 
$(E,\{E^{(i)}\}_{i=1}^n) \os{\subset}{\lo} (I,\{I^{(i)}\}_{i=1}^n)$.
\smallskip
\parno
(8.0.2):  In any strictly exact sequence 
$$0\lo (I,\{I^{(i)}\}_{i=1}^n) \lo (J,\{J^{(i)}\}_{i=1}^n) 
\lo (K,\{K^{(i)}\}_{i=1}^n) \lo 0$$
in ${\rm MF}^n({\cal A})$, if $(I,\{I^{(i)}\}_{i=1}^n) \in {\cal I}$ 
and $(J,\{J^{(i)}\}_{i=1}^n) \in {\cal I}$,
then  $(K,\{K^{(i)}\}_{i=1}^n) \in {\cal I}$.
\smallskip
\parno
(8.0.3): If a sequence 
$$0\lo (I,\{I^{(i)}\}_{i=1}^n) \lo (J,\{J^{(i)}\}_{i=1}^n) 
\lo (K,\{K^{(i)}\}_{i=1}^n) \lo 0$$
in ${\rm MF}^n({\cal A})$ 
is a strictly exact sequence  with 
$(I,\{I^{(i)}\}_{i=1}^n), (J,\{J^{(i)}\}_{i=1}^n), (K,\{K^{(i)}\}_{i=1}^n) 
\in {\cal I}$,
then, for any morphism 
$f \col ({\cal T},{\cal A})\lo 
({\cal T}',{\cal A}')$ of ringed topoi, 
the sequence 
$$0\lo f_*(I,\{I^{(i)}\}_{i=1}^n) \lo f_*(J,\{J^{(i)}\}_{i=1}^n) 
\lo f_*(K,\{K^{(i)}\}_{i=1}^n) \lo 0$$
is strictly exact.
\bigskip

\begin{prop}\label{prop:flijst}Let 
${\cal I}$ be ${\cal I}^n_{\rm flas}({\cal A})$, 
${\cal I}^n_{\rm inj}({\cal A})$ 
or ${\cal I}^n_{\rm stinj}({\cal A})$. 
Then ${\cal I}$ satisfies the 
conditions $(8.0.1)$,  $(8.0.2)$ and $(8.0.3)$.
\end{prop}
\begin{proof}
By (\ref{prop:enspij}), ${\cal I}$ satisfies the condition (8.0.1);
${\cal I}$ also satisfies the condition (8.0.3).
It is easy to check that 
the categories ${\cal I}^n_{\rm flas}({\cal A})$ and 
${\cal I}^n_{\rm inj}({\cal A})$
satisfy the condition (8.0.2).
Consider the exact sequence in (8.0.2) with
$(I,\{I^{(i)}\}_{i=1}^n)$, 
$(J,\{J^{(i)}\}_{i=1}^n)\in {\cal I}^n_{\rm stinj}({\cal A})$.
Then, by the definition of 
${\cal I}^n_{\rm stinj}({\cal A})$, 
there exists a splitting 
of the strictly injective morphism 
$(I,\{I^{(i)}\}_{i=1}^n) \os{\sus}{\lo} (J,\{J^{(i)}\}_{i=1}^n)$.
Hence $(J,\{J^{(i)}\}_{i=1}^n)
\simeq (I,\{I^{(i)}\}_{i=1}^n)\oplus (K,\{K^{(i)}\}_{i=1}^n)$.
Now it is easy to see that 
$(K,\{K^{(i)}\}_{i=1}^n) \in {\cal I}^n_{\rm stinj}({\cal A})$.
\end{proof}

\par
Let ${\cal I}$ be 
${\cal I}^n_{\rm flas}({\cal A})$, 
${\cal I}^n_{\rm inj}({\cal A})$ or 
${\cal I}^n_{\rm stinj}({\cal A})$.
Let ${\rm N}^+{\rm F}^n({\cal I}^n)$ be a full 
subcategory of ${\rm K}^+{\rm F}^n({\cal I})$ 
which consists of the strictly exact sequences of 
${\rm K}^+{\rm F}^n({\cal I})$.
We can prove the following as in the classical case 
(\cite[(1.3.4)]{sc}):
\begin{coro}\label{coro:knd}
The canonical functor 
$${\rm K}^+{\rm F}^n({\cal I})/{\rm N}^+{\rm F}^n({\cal I}) \lo 
{\rm D}^+{\rm F}^n({\cal A})$$ is an equivalence of categories.  
\end{coro}

\bigskip
\bigskip
\bigskip
\parno 
{\bf {\Large Part II.  $l$-adic relative monodromy-weight conjecture}}
\bigskip
\parno
In the Part II of this paper we construct a fundamental bifiltered complex which gives us 
the $l$-adic relative monodromy-weight conjecture.

\section{$l$-adic bifiltered El Zein-Steenbrink-Zucker complex}\label{sec:lbc}
Let $S$ be a family of log points defined in \cite[(1.1)]{nb}. 
That is, $S$ is locally isomorphic to 
$(\os{\circ}{S},{\mab N}\oplus {\cal O}_S^*\lo {\cal O}_S)$, 
where the morphism 
${\mab N}\oplus {\cal O}_S^*\lo {\cal O}_S$ is 
given by $(n,a)\lom 0^na$ 
$(n\in {\mab N}, a\in {\cal O}_S^*)$, where $0^n=0\in {\cal O}_S$ 
for $n\not =0$ and $0^0:=1\in {\cal O}_S$. 
In this paper we assume that $S$ is isomorphic to 
$(\os{\circ}{S},{\mab N}\oplus {\cal O}_S^*\lo {\cal O}_S)$ 
and we fix an isomorphism $S\os{\sim}{\lo} 
(\os{\circ}{S},{\mab N}\oplus {\cal O}_S^*\lo {\cal O}_S)$. 
We do not assume a condition on the characteristic of $S$. 
This $S$ is different from the $S$ in the introduction. 
For a monoid $P$ and a commutative ring of $A$, 
we denote by ${\rm Spec}^{\log}(A[P])$ the log scheme whose underlying scheme 
is ${\rm Spec}(A[P])$ and whose log structure is the association of 
the natural morphism 
$P\lo A[P]$. We have a natural morphism 
$S\lo {\rm Spec}^{\log}({\mab Z}[{\mab N}])$. 
Let $l$ be a prime number 
which is invertible on $\os{\circ}{S}$. 
Let $\mu_{l^m}$ $(m\in {\mab N})$ be the group of 
$l^m$-th roots of unity in $\ol{\mab Q}$. 
Set $\mu_{l^{\infty}}:=\vil_m\mu_{l^m}$. 
Assume that $\os{\circ}{S}$ is a scheme over 
${\rm Spec}({\mab Z}[l^{-1},\zeta_{l^\infty}~\vert~\zeta_{l^\infty}\in \mu_{l^{\infty}}])$. 
That is, ${\cal O}_S$ is assumed to be a ${\mab Z}[l^{-1}][T_m]/(T^{l^m}_{m}-1)$-algebra  
with $T_m\lom T_{m+1}^l$ for any $m\in {\mab N}$. 
Let $X$ be an SNCL scheme with a relative SNCD $D$ over $S$ 
defined in \cite[(6.1)]{ny}. 
(In \cite[\S2, \S3]{nppf} we have also recalled 
the definition of an SNCL scheme with a relative SNCD.)
Assume that $\os{\circ}{X}$ is quasi-compact. 
In this section we construct a bifiltered complex 
$(A_{l^{\infty}}((X_{\frac{1}{l^{\infty}}},D_{\frac{1}{l^{\infty}}})/S),P^{D_{\frac{1}{l^{\infty}}}},P)$
producing the generalizations of the $l$-adic weight spectral sequences
(\ref{eqn:lidd}) and (\ref{eqn:lintdd}). 
\par 
Let $\Del:=
\{\os{\circ}{X}_{\lam}\}_{\lam \in \Lam}$ and $\{\os{\circ}{D}_{\mu}\}_{\mu \in M}$ 
be decompositions of $\os{\circ}{X}$ and $\os{\circ}{D}$  
by smooth components, respectively (\cite[(1.1.9)]{nb}, \cite[\S6]{ny}). 
As in \cite[(9.13.1), (9.13.2)]{nh2}, 
for a nonnegative integer $k$ and a subset  
$\ul{\lam}=\{\lam_0,\cdots, \lam_k\}$ 
$(\lam_i \not= \lam_j~{\rm if}~i\not= j, \lam_i\in \Lam)$ of $\Lam$,  
set 
\begin{equation*}
\os{\circ}{X}_{\ul{\lam}} 
:=\os{\circ}{X}_{\lam_0}\cap \os{\circ}{X}_{\lam_2} \cap \cdots 
\cap \os{\circ}{X}_{\lam_k} 
\tag{9.0.1}\label{eqn:paxrlm}
\end{equation*} 
and 
\begin{equation*}
\os{\circ}{X}{}^{(k)} =  
\us{\# \ul{\lam}=k+1}{\coprod}
\os{\circ}{X}_{\ul{\lam}}. 
\tag{9.0.2}\label{eqn:kxntd} 
\end{equation*} 
For a positive integer $k$ and a subset  
$\ul{\mu}=\{\mu_1,\cdots, \mu_k\}$ 
$(\mu_i \not= \mu_j~{\rm if}~i\not= j, \mu_i\in M)$ of $M$, 
set 
\begin{equation*}
\os{\circ}{D}_{\ul{\mu}} 
:=\os{\circ}{D}_{\mu_1}\cap \os{\circ}{D}_{\mu_2} \cap \cdots 
\cap \os{\circ}{D}_{\mu_k} \quad (\mu_i \not= \mu_j~{\rm if}~i\not= j) 
\tag{9.0.3}\label{eqn:dparlm}
\end{equation*}
and 
\begin{equation*}
\os{\circ}{D}{}^{(k)} = 
\us{\# \ul{\mu}=k}{\coprod}
\os{\circ}{D}_{\ul{\mu}}.  
\tag{9.0.4}\label{eqn:kftd}
\end{equation*}
Set $\os{\circ}{D}{}^{(0)}:= \os{\circ}{X}$.  
Let $M(D)$ be the log structure in $\os{\circ}{X}_{\rm et}$ 
which is the pull-back of the log structure 
obtained by  $\os{\circ}{D}$ in \cite[\S6]{ny} by 
the natural morphism of topoi $\os{\circ}{X}_{\rm et}\lo \os{\circ}{X}_{\rm zar}$.

\begin{prop}
\label{prop:dkwdef}
For a positive integer $k$,  
$\os{\circ}{X}{}^{(k)}$ $($resp.~$\os{\circ}{D}{}^{(k)})$ 
is independent of the choice of the decomposition of 
$\os{\circ}{X}$ $($resp.~$\os{\circ}{D})$ by smooth components of 
$\os{\circ}{X}$ $($resp.~$\os{\circ}{D})$. 
In particular, the log scheme $D^{(k)}$ whose underlying scheme 
is $\os{\circ}{D}{}^{(k)}$ and whose log structure is the pull-back of 
$X$ is independent of the choice of  the decomposition of 
$\os{\circ}{D}$ by smooth components of $\os{\circ}{D}$.
\end{prop}
\begin{proof}
The proof is the same as that of \cite[(2.2.14), (2.2.15)]{nh2}. 
\end{proof}

\par 
As in \cite[(3.1.4)]{dh2}, let us define the orientation 
sheaves of the sets $\{\os{\circ}{X}_{\lam}\}_{\lam \in \Lam}$ 
and $\{\os{\circ}{D}_{\mu}\}_{\mu \in M}$. 
\par
Let $E$ be a finite set with cardinality $k\geq 0$.  
Set $\vp_E:=\bigwedge^k{\mab Z}^E$ if $k \geq 1$ 
and $\vp_E:={\mab Z}$ if $k =0$.
\par 
Let $k$ be a nonnegative integer. 
Let $P$ be a point of $\os{\circ}{X}{}^{(k)}$.
Let $\os{\circ}{X}_{\lam_0}, \ldots, \os{\circ}{X}_{\lam_{k}}$ 
be different smooth components of $\os{\circ}{X}$ such that
$\os{\circ}{X}_{\lam_0}\cap  \cdots \cap \os{\circ}{X}_{\lam_{k}}$ 
contains $P$.
Then the set 
$E:=\{\os{\circ}{X}_{\lam_0},\ldots,\os{\circ}{X}_{\lam_{k}}\}$ 
gives an abelian sheaf 
$$\vp_{\lam_0\cdots \lam_{k}{\rm zar}}
(\os{\circ}{X}/\os{\circ}{S}):=
\bigwedge^{k+1}
{\mab Z}^E_{\os{\circ}{X}_{\lam_0}\cap \cdots \cap 
\os{\circ}{X}_{\lam_{k}}}$$ 
on a local neighborhood of $P$ in $\os{\circ}{X}{}^{(k)}$. 
The sheaf 
$\vp_{\lam_0\cdots \lam_{k}{\rm zar}}
(\os{\circ}{X}/\os{\circ}{S})$ 
is globalized on $\os{\circ}{X}{}^{(k)}$; 
we denote this globalized abelian 
sheaf by the same symbol
$\vp_{\lam_0\cdots \lam_{k}{\rm zar}}
(\os{\circ}{X}/\os{\circ}{S})$. 
We denote a local 
section of 
$\vp_{\lam_0\cdots \lam_{k}{\rm zar}}
(\os{\circ}{X}/\os{\circ}{S})$
by the following way: 
$m(\lam_0\cdots \lam_{k})$ $(m\in {\mab Z})$.
Set $\vp^{(k)}_{\rm zar}(\os{\circ}{X}/\os{\circ}{S}):=
\bigoplus_{\{\lam_0,\ldots \lam_{k}\}}
\vp_{\lam_0\cdots \lam_{k}{\rm zar}}
(\os{\circ}{X}/\os{\circ}{S})$.
Set $\vp^{(0)}_{\rm zar}(\os{\circ}{X}/\os{\circ}{S})
:={\mab Z}_{\os{\circ}{X}}$.
\par
Let $k$ be a nonnegative integer. 
Let $Q$ be a point of $\os{\circ}{D}{}^{(k)}$.
Let $\os{\circ}{D}_{\mu_1}, \ldots, \os{\circ}{D}_{\mu_k}$ 
be different smooth components of $\os{\circ}{D}$ such that
$\os{\circ}{D}_{\mu_1}\cap  \cdots \cap \os{\circ}{D}_{\mu_k}$ 
contains $Q$.
Then the set 
$F:=\{\os{\circ}{D}_{\mu_1},\ldots,\os{\circ}{D}_{\mu_k}\}$ 
gives an abelian sheaf 
$$\vp_{\mu_1\cdots \mu_k{\rm zar}}(\os{\circ}{D}/\os{\circ}{S}):=
\bigwedge^{k}
{\mab Z}^F_{\os{\circ}{D}_{\mu_1}\cap \cdots 
\cap \os{\circ}{D}_{\mu_k}}$$ 
on a local neighborhood of $Q$ in $\os{\circ}{D}{}^{(k)}$. 
The sheaf 
$\vp_{\mu_1\cdots \mu_k{\rm zar}}(\os{\circ}{D}/\os{\circ}{S})$ 
is globalized on $\os{\circ}{D}{}^{(k)}$. 
By using $\vp_{\mu_1\cdots \mu_k{\rm zar}}
(\os{\circ}{D}/\os{\circ}{S})$, 
we have an analogous orientation sheaf 
$\vp^{(k)}_{\rm zar}(\os{\circ}{D}/\os{\circ}{S})
:=\bigoplus_{\{\mu_1,\ldots \mu_{k}\}}
\vp_{\mu_1\cdots \mu_{k}{\rm zar}}
(\os{\circ}{D}/\os{\circ}{S})$ 
in $\os{\circ}{D}{}^{(k)}_{\rm zar}$ 
for $k\in {\mab Z}_{\geq 0}$. 
Set 
$\vp^{(k),(k')}_{\rm zar}((\os{\circ}{X},\os{\circ}{D})/\os{\circ}{S}):=
\vp^{(k)}_{\rm zar}(\os{\circ}{X}/\os{\circ}{S})
\vert_{\os{\circ}{X}{}^{(k)}\cap \os{\circ}{D}{}^{(k')}} 
\otimes 
\vp^{(k')}_{\rm zar}(\os{\circ}{D}/\os{\circ}{S})
\vert_{\os{\circ}{X}{}^{(k)}\cap \os{\circ}{D}{}^{(k')}}$. 
The orientation sheaves  
$\vp^{(k)}_{\rm zar}(\os{\circ}{X}/\os{\circ}{S})$, 
$\vp^{(k)}_{\rm zar}(\os{\circ}{D}/\os{\circ}{S})$ and 
$\vp^{(k),(k')}_{\rm zar}((\os{\circ}{X},\os{\circ}{D})/\os{\circ}{S}))$
are non-canonically isomorphic to 
${\mab Z}_{\os{\circ}{X}{}^{(k)}}$, ${\mab Z}_{\os{\circ}{D}{}^{(k)}}$ 
and 
${\mab Z}_{\os{\circ}{X}{}^{(k)}\cap \os{\circ}{D}{}^{(k')}}$, respectively.  
Let $a^{(k)} \col \os{\circ}{X}{}^{(k)} \lo \os{\circ}{X}$, 
$c^{(k)} \col \os{\circ}{D}{}^{(k)} \lo \os{\circ}{X}$ and 
$a^{(k),(k')} \col \os{\circ}{X}{}^{(k)} \cap \os{\circ}{D}{}^{(k')} 
\lo \os{\circ}{X}$ be natural morphisms of schemes. 
(We do not use a symbol $b^{(k)}$ because we use the symbol 
$b^{(k)}$ for another morphism in the $p$-adic case in \cite{nppf}.)
Let $M_X$ be the log structure of $X$ in $\os{\circ}{X}_{\rm et}$. 
The orientation sheaves 
$\vp^{(k)}_{\rm zar}(\os{\circ}{X}/\os{\circ}{S})$,  
$\vp^{(k)}_{\rm zar}(\os{\circ}{D}/\os{\circ}{S})$ and 
$\vp^{(k),(k')}_{\rm zar}((\os{\circ}{X},\os{\circ}{D})/\os{\circ}{S})$
define \'{e}tale sheaves 
$\vp^{(k)}_{\rm et}(\os{\circ}{X}/\os{\circ}{S})$, 
$\vp^{(k)}_{\rm et}(\os{\circ}{D}/\os{\circ}{S})$ and 
$\vp^{(k),(k')}_{\rm et}((\os{\circ}{X},\os{\circ}{D})/\os{\circ}{S})$
in $(\os{\circ}{X}{}^{(k)})_{\rm et}$, 
$(\os{\circ}{D}{}^{(k)})_{\rm et}$, 
$(\os{\circ}{X}{}^{(k)}\cap \os{\circ}{D}{}^{(k')})_{\rm et}$, respectively.  
We have the following canonical isomorphisms  
\begin{equation*} 
a^{(k)}_{{\rm et}*}
(\vp^{(k)}_{\rm et}(\os{\circ}{X}/\os{\circ}{S})) 
\os{\sim}{\lo} 
\bigwedge^{k+1}(M_X^{\rm gp}/{\cal O}_X^*),   
\tag{9.1.1}\label{eqn:oremx} 
\end{equation*} 
\begin{equation*} 
c^{(k)}_{{\rm et}*}(\vp^{(k)}_{\rm et}(\os{\circ}{D}/\os{\circ}{S})) 
\os{\sim}{\lo} 
\bigwedge^k(M(D)^{\rm gp}/{\cal O}_X^*)   
\tag{9.1.2}\label{eqn:oremo} 
\end{equation*} 
and 
\begin{equation*} 
a^{(k),(k')}_{{\rm et}*}(\vp^{(k),(k')}_{\rm et}
((\os{\circ}{X},\os{\circ}{D})/\os{\circ}{S})) 
\os{\sim}{\lo} 
\bigwedge^{k+1}(M_X^{\rm gp}/{\cal O}_X^*)\otimes_{\mab Z}
\bigwedge^{k'}(M(D)^{\rm gp}/{\cal O}_X^*)    
\tag{9.1.3}\label{eqn:lco} 
\end{equation*} 
as abelian sheaves in $\os{\circ}{X}_{\rm et}$. 
We have a canonical isomorphism 
\begin{equation*} 
a^{(k),(k')}_{{\rm et}*}
(\vp^{(k),(k')}_{\rm et}((\os{\circ}{X},\os{\circ}{D})/\os{\circ}{S}))
\os{\sim}{\longleftarrow} 
a^{(k)}_{{\rm et}*}(\vp^{(k)}_{\rm et}(\os{\circ}{X}/\os{\circ}{S}))
\otimes_{\mab Z} 
c^{(k)}_{{\rm et}*}(\vp^{(k)}_{\rm et}(\os{\circ}{D}/\os{\circ}{S})).  
\tag{9.1.4}\label{eqn:etcab}
\end{equation*} 
Henceforth, we denote 
$a^{(k)}_{{\rm et}*}$, $c^{(k)}_{{\rm et}*}$ 
and $a^{(k),(k')}_{{\rm et}*}$ simply by 
$a^{(k)}_{*}$, $c^{(k)}_{*}$  
and $a^{(k),(k')}_{*}$, respectively. 
We apply the same rule for 
$a^{(k)*}_{\rm et}$, $c^{(k)*}_{\rm et}$ 
and $a^{(k),(k')*}_{\rm et}$. 
\par 
Set $(X,D):=X\times_{\os{\circ}{X}}(\os{\circ}{X},M(D))$ as in \cite[\S6]{ny} and 
$(\os{\circ}{X},\os{\circ}{D}):=(\os{\circ}{X},M(D))$. 
Let 
\begin{equation*} 
\eps_{D} \col (X,D) \lo X, \quad  
\eps_X \col X  \lo \os{\circ}{X} \quad 
\text{and} \quad  
\os{\circ}{\eps}_{D} \col 
(\os{\circ}{X},\os{\circ}{D}) \lo \os{\circ}{X} 
\end{equation*} 
be natural morphisms of log schemes forgetting log structures. 
Set 
${\eps}_{(X,D)}:=\eps_X \circ \eps_{D}$. 
The morphisms above induce 
the following morphisms of topoi:  
\begin{equation*} 
\eps_{D} \col (X,D)_{\rm ket} \lo X_{\rm ket}, \quad 
\eps_X \col X_{\rm ket}  \lo \os{\circ}{X}_{\rm et}, 
\end{equation*}  
\begin{equation*} 
{\eps}_{(X,D)}:={\eps}_{X} \circ {\eps}_{D} 
\col (X,D)_{\rm ket} \lo \os{\circ}{X}_{\rm et}
\quad   
\text{and}  \quad 
\os{\circ}{\eps}_{D} \col 
(\os{\circ}{X},\os{\circ}{D})_{\rm et} 
\lo \os{\circ}{X}_{\rm et}. 
\end{equation*}  
\par 
Following \cite{nd}, set $\os{\circ}{S}_{\frac{1}{l^m}}
:=\os{\circ}{S}\otimes_{{\mab Z}[{\mab N}]}
{\mab Z}[(l^m)^{-1}{\mab N}^{}]$.  
(In [loc.~cit.] C.~Nakayama has used the symbol 
${\mab N}^{1/l^m}$ instead of $(l^m)^{-1}{\mab N}$ in ${\mab Q}$.) 
The inclusion morphism  
$(l^m)^{-1}{\mab N}^{} \os{\sus}{\lo} 
{\mab Z}[(l^m)^{-1}{\mab N}^{}]$ gives an 
fs log structure on $\os{\circ}{S}_{\frac{1}{l^m}}$. 
Set $S_{\frac{1}{l^{\infty}}}:=\vpl_{m}S_{\frac{1}{l^m}}$. 
Set also $X_{\frac{1}{l^m}}:=X\times_SS_{\frac{1}{l^m}}$ and 
$X_{\frac{1}{l^{\infty}}}:=\vpl_{m}X_{\frac{1}{l^m}}$. 
Then we have the following natural morphisms of log schemes: 
\begin{equation*}  
\pi_{X_{\frac{1}{l^m}}} \col X_{\frac{1}{l^m}} \lo X \quad 
\text{and} \quad  
\pi_{X_{\frac{1}{l^{\infty}}}} \col X_{\frac{1}{l^{\infty}}} \lo X.
\end{equation*}   
Denote the group scheme 
$\mu_{l^m}:=\ul{\rm Spec}_{\os{\circ}{S}}({\cal O}_S[T_m]/(T^{l^m}_{m}-1))$ 
over $\os{\circ}{S}$ by ${\mab Z}/l^m(1)$. 
There exists a natural morphism 
$\mu_{l^{m+1}}\lo \mu_{l^m}$ defined by $T_m\lom T_{m+1}^l$. 
Set ${\mab Z}_l(1):=\vpl_n{\mab Z}/l^m(1)$. 
The group scheme ${\mab Z}_l(1)=\vpl_m\mu_{l^m}$ 
acts naturally on $S_{\frac{1}{l^{\infty}}}$: 
\begin{align*} 
(\zeta_{l^m})_{m\geq 1}\cdot (1\otimes \dfrac{1}{l^m}):=\zeta_{l^m}\otimes \dfrac{1}{l^m}. 
\tag{9.1.5}\label{ali:iin}
\end{align*} 
Here $(\zeta_{l^m})_{m\geq 1}(=(T_m)_{m\geq 1})$ in 
the left hand side is an element of ${\mab Z}_l(1)$ and 
$\zeta_{l^m}(=(T_m)_{m\geq 1})$ in the right hand side is the image of 
$\zeta_{l^m}\in {\mab Z}[l^{-1},\zeta_{l^\infty}~\vert~\zeta_{l^\infty}\in \mu_{l^{\infty}}]$ by 
the pull-back of the morphism 
$S\lo {\rm Spec}({\mab Z}[l^{-1},\zeta_{l^\infty}~\vert~\zeta_{l^\infty}\in \mu_{l^{\infty}}])$. 
This action of ${\mab Z}_l(1)$ on $S_{\frac{1}{l^{\infty}}}$ induces the action 
of ${\mab Z}_l(1)$ on $X_{\frac{1}{l^{\infty}}}$. 
For a sheaf $F$ in $X_{\rm ket}$, 
${\mab Z}_l(1)$ acts on 
$\pi_{X_{\frac{1}{l^{\infty}}}*}\pi_{X_{\frac{1}{l^{\infty}}}^{*}}(F)$ since 
${\mab Z}_l(1)$ acts on $X_{\frac{1}{l^{\infty}}}$.  
The group schemes ${\mab Z}/l^m(1)$ and 
${\mab Z}_l(1)=\vpl_m\mu_{l^m}$ define abelian sheaves in 
$\os{\circ}{X}_{\rm et}$. 
Set 
\begin{equation*} 
M_{(X_{\frac{1}{l^m}},D_{\frac{1}{l^m}})}
:=M_{X_{\frac{1}{l^m}}}\oplus_{{\cal O}^*_{X_{\frac{1}{l^m}}}}\os{\circ}{\pi}{}^{*}_{X_{\frac{1}{l^m}}}(M(D)) 
\end{equation*} 
and  
\begin{equation*} 
M_{(X_{\frac{1}{l^{\infty}}},D_{\frac{1}{l^{\infty}}})}:=
M_{X_{\frac{1}{l^{\infty}}}}
\oplus_{{\cal O}^*_{X_{\frac{1}{l^{\infty}}}}}\os{\circ}{\pi}{}^{*}_{X_{\frac{1}{l^{\infty}}}}(M(D)).
\end{equation*}  
Denote $(\os{\circ}{X}_{\frac{1}{l^m}},M_{(X_{\frac{1}{l^m}},D_{\frac{1}{l^m}})})$ by 
$(X_{\frac{1}{l^m}},D_{\frac{1}{l^m}})$  
and 
$(\os{\circ}{X}_{\frac{1}{l^{\infty}}},M_{(X_{\frac{1}{l^{\infty}}},D_{\frac{1}{l^{\infty}}})})$ 
by $(X_{\frac{1}{l^{\infty}}},D_{\frac{1}{l^{\infty}}})$.  
We also have the following natural morphisms of log schemes: 
\begin{equation*}  
\pi_{(X_{\frac{1}{l^m}},D_{\frac{1}{l^m}})} \col (X_{\frac{1}{l^m}},D_{\frac{1}{l^m}}) 
\lo (X,D) \quad 
\text{and} \quad 
\pi_{(X_{\frac{1}{l^{\infty}}},D_{\frac{1}{l^{\infty}}})} 
\col (X_{\frac{1}{l^{\infty}}},D_{\frac{1}{l^{\infty}}}) \lo (X,D). 
\end{equation*}  
By using the trivial action of ${\mab Z}_l(1)$ 
on $(\os{\circ}{X},\os{\circ}{D})$, 
we also have a natural action of ${\mab Z}_l(1)$ on 
$(X_{\frac{1}{l^{\infty}}},D_{\frac{1}{l^{\infty}}})$. 
Let $I^{\bul}_{\os{\circ}{D},l^n}$ $(n\in {\mab N})$ 
be an injective resolution of ${\mab Z}/l^n$ in $(\os{\circ}{X},\os{\circ}{D})_{\rm ket}$. 
Then we have a projective system 
$\{I^{\bul}_{\os{\circ}{D},l^n}\}_{n=1}^{\infty}$ fitting into the following commutative diagram 
\begin{equation*} 
\begin{CD}
{\mab Z}/l^{n+1}@>>>I^{\bul}_{\os{\circ}{D},l^{n+1}}\\
@V{\rm proj}.VV @VVV \\
{\mab Z}/l^{n}@>>>I^{\bul}_{\os{\circ}{D},l^n}. 
\end{CD}
\tag{9.1.6}\label{cd:iin}
\end{equation*} 

The following non-difficult proposition is important in this paper. 

\begin{prop}\label{prop:bav}  
There exists a projective system 
$\{(M^{\bul}_{l^n}(\os{\circ}{D}/\os{\circ}{S}),Q)\}_{n=1}^{\infty}$ 
of bounded filtered flat resolutions {\rm (\cite[(1.1.17) (2)]{nh2})} of the complexes 
$(\os{\circ}{\eps}_{D*}(I^{\bul}_{\os{\circ}{D},l^n}),\tau)$'s in $\os{\circ}{X}_{\rm et}$. 
\end{prop} 
\begin{proof} 
Let $k$ be an integer. 
Because the rank of $\ol{M}(D):=M(D)/{\cal O}_X^*$ is finite 
(since $\os{\circ}{X}$ is quasi-compact) and 
because $\os{\circ}{\eps}_D$ is proper, 
\begin{align*} 
{\rm gr}^\tau_k
(\os{\circ}{\eps}_{D*}(I^{\bul}_{\os{\circ}{D},l^n}))
& =R^k\os{\circ}{\eps}_{D*}({\mab Z}/l^n)[-k]=0
\end{align*} 
for $k>>0$ which depends only on the relative dimension of $X/S$ and 
the rank of $\ol{M}(D)$ by \cite[(7.2) (1)]{adv}. 
One may prove this vanishing in the following way in this case by \cite[(2.4)]{ktnk}. 
\par 
Consider the following exact sequence 
\begin{align*} 
0\lo {\mab Z}/l^m(1)\lo M^{\rm gp}(D)\os{l^m}{\lo}M^{\rm gp}(D)\lo 0. 
\end{align*}  
By \cite[(2.4)]{ktnk} we have 
\begin{align*} 
{\rm gr}^\tau_k
(\os{\circ}{\eps}_{D*}(I^{\bul}_{\os{\circ}{D},l^n}))
& =R^k\os{\circ}{\eps}_{D*}({\mab Z}/l^n)[-k]
=\bigwedge^k(M(D)^{\rm gp}/{\cal O}_X^*)
\otimes_{\mab Z}{\mab Z}/l^n(-k)[-k]\tag{9.2.1}\label{eqn:tktk1}\\ 
{} & =
c^{(k)}_{*}(\varpi^{(k)}_{\rm et}(\os{\circ}{D}/\os{\circ}{S}))\otimes_{\mab Z}{\mab Z}/l^n(-k)[-k].  
\end{align*} 
Because $\os{\circ}{X}$ is quasi-compact, there exists a positive integer $k_0$ such that 
$\varpi^{(k)}_{\rm et}(\os{\circ}{D}/\os{\circ}{S})=0$ for 
any $k > k_0$.  
Hence, if $k>k_0$, then 
$R^k\os{\circ}{\eps}_{D*}({\mab Z}/l^n)=0$.  
Consequently there exists a bounded above flat complex 
$M'^{\bul}_{l^n}$ of ${\mab Z}/l^n$-modules in $\os{\circ}{X}_{\rm et}$ 
with a filtered quasi-isomorphism 
$(M'^{\bul}_{l^n},\tau) 
\lo (\os{\circ}{\eps}_{D*}(I^{\bul}_{\os{\circ}{D},l^n}),\tau)$ 
(cf.~\cite[(1.1.18)]{nh2}). 
In fact, we can take the projective system 
$\{M'^{\bul}_{l^n}\}_{n=1}^{\infty}$ 
fitting into the following commutative diagram 
\begin{equation*} 
\begin{CD}
(M'^{\bul}_{l^{n+1}},\tau)@>>>(\os{\circ}{\eps}_{D*}(I^{\bul}_{\os{\circ}{D},l^{n+1}}),\tau)\\
@VVV @VVV \\
(M'^{\bul}_{l^n},\tau) @>>>(\os{\circ}{\eps}_{D*}(I^{\bul}_{\os{\circ}{D},l^{n}}),\tau)
\end{CD}
\end{equation*} 
by the dual argument of the proof of \cite[(1.1.8)]{nh2}. 
\par 
We claim that there exists an integer $q_0$ such that 
\begin{align*} 
{\cal T}{\it or}_q^{{\mab Z}/l^n}(\os{\circ}{\eps}_{D*}(I^{\bul}_{\os{\circ}{D},l^n}), G)
:={\cal H}^{-q}(\os{\circ}{\eps}_{D*}(I^{\bul}_{\os{\circ}{D},l^n})\otimes^L_{{\mab Z}/l^n}G)
= {\cal H}^{-q}(M'^{\bul}_{l^n}\otimes_{{\mab Z}/l^n}G)
=0
\end{align*}  
for any $q>q_0$ 
and for any ${\mab Z}/l^n$-module 
$G$ in $\os{\circ}{X}_{\rm et}$. 
Let $q_1$ be an integer such 
that $M'^q_{l^n}(\os{\circ}{D}/\os{\circ}{S})=0$ for any $q>q_1$. 
Obviously $\tau_{q_1}M'^{\bul}_{l^n}=M'^{\bul}_{l^n}$.  
Consider the following exact sequence 
$$0\lo \tau_{q_1-1}M'^{\bul}_{l^n}\lo \tau_{q_1}M'^{\bul}_{l^n} 
\lo{\rm gr}^{\tau}_{q_1}M'^{\bul}_{l^n}\lo 0.$$
By (\ref{eqn:tktk1}),  
${\cal T}{\it or}_q^{{\mab Z}/l^n}(\tau_{q_1}\os{\circ}{\eps}_{D*}
(I^{\bul}_{\os{\circ}{D},l^n}), G)
=
{\cal T}{\it or}_q^{{\mab Z}/l^n}(\tau_{q_1-1}\os{\circ}{\eps}_{D*}(I^{\bul}_{\os{\circ}{D},l^n}), G)$ 
except for finitely many $q$'s. Making this argument repeatedly until $\tau_0$, 
we see that 
${\cal T}{\it or}_q^{{\mab Z}/l^n}(\os{\circ}{\eps}_{D*}(I^{\bul}_{\os{\circ}{D},l^n}), G)
=0$ except for finitely many $q$'s. Hence our claim holds. 
By the same argument as that of \cite[II (4.2)]{hard}, 
there exists a bounded complex 
$M^{\bul}_{l^n}(\os{\circ}{D}/\os{\circ}{S})$
of flat ${\mab Z}/l^n$-modules which is isomorphic to 
$\os{\circ}{\eps}_{D*}(I^{\bul}_{\os{\circ}{D},l^n})$. 
By (\ref{eqn:tktk1}) again,  
we see that 
$(M^{\bul}_{l^n}(\os{\circ}{D}/\os{\circ}{S}),\tau)$ 
is a filtered flat resolution of 
$(\os{\circ}{\eps}_{D*}(I^{\bul}_{\os{\circ}{D},l^n}).\tau)$. 
\end{proof}

\begin{defi}\label{defi:dcd}
Let $({\cal T},{\cal A})$ be a ringed topos. Let $m\leq 2$ and $n$ be positive integers.  
Let $(K^{\bul},\{P^{(i)}\}_{i=1}^m)$ be an $m$-filtered complex 
of ${\cal A}$-modules. 
\par 
(1) We say that the cohomological sheaves of $(K^{\bul},\{P^{(i)}\}_{i=1}^m)$ are {\it constructible} if 
${\cal H}^q(K^{\bul})$ and ${\cal H}^q((P^{(1)}_{k_1}\cap P^{(m)}_{k_m})K^{\bul})$ 
$(\forall q, \forall k_1, \forall k_m\in {\mab Z})$ are constructible. 
\par 
(2) We say that $(K^{\bul},\{P^{(i)}\}_{i=1}^m)$ has {\it finite tor-dimension} if the 
${\cal A}$-modules 
$K^{\bul}$ and $K^{\bul}/(P^{(1)}_{k_1}\cap P^{(m)}_{k_m})K^{\bul}$ 
$(\forall k_1, \forall k_m\in {\mab Z})$ 
have finite tor-dimension. 
\end{defi}  

In the following  let $m$ be a positive integer less than or equal to $2$.  

\begin{prop}\label{prop:cfd}
Assume that the filtrations $P^{(1)}$ and $P^{(m)}$ are biregular. 
Then the following hold$:$
\par 
$(1)$ The cohomological sheaves of  $(K^{\bul},\{P^{(i)}\}_{i=1}^m)$ are constructible 
if and only if 
the cohomological sheaves of 
${\rm gr}^{P^{(1)}}_{k_1}{\rm gr}^{P^{(m)}}_{k_m}K^{\bul}$ is constructible 
for any $k_1,k_m\in {\mab Z}$. 
\par 
$(2)$ The filtered complex $(K^{\bul},\{P^{(i)}\}_{i=1}^m)$ has finite tor-dimension 
if and only if ${\rm gr}^{P^{(1)}}_{k_1}{\rm gr}^{P^{(m)}}_{k_m}K^{\bul}$ has finite tor-dimension 
for any $k_1,k_m\in {\mab Z}$.
\end{prop}
\begin{proof} 
(1): By \cite[IX (2.6) (ii)]{sga4-3}, the constructivility is stable 
under the extension of an exact sequence 
of ${\cal A}$-modules. 
By this fact and the easy argument in the proof of (\ref{prop:pq}), we obtain (1). 
\par 
(2): By the definition of the finite tor-dimension, the property of the finite tor-dimension is 
stable under the extension of an exact sequence of ${\cal A}$-modules. 
By this fact and the easy arguments in the proof of (\ref{prop:nsq}) and (\ref{prop:pq}), 
we obtain (2). 
\end{proof}


\begin{defi} 
Let $({\cal T},{\cal A})$ be a ringed topos. 
\par 
(1) Let ${\rm D}{\rm F}^m({\cal T},{\cal A})$ be the derived category of
$m$-filtered complexes of ${\cal A}$-modules. 
Let ${\rm D}^{\rm b}{\rm F}^m_{\rm ctf}({\cal T},{\cal A})$ be the full-subcategory of 
the derived category of bounded filtered complexes whose objects 
$(K^{\bul},\{P_k\}_{k=1}^m)$'s have finite tor-dimension and 
whose cohomological sheaves are constructible. 
\par 
(2) Assume that ${\cal A}$ is the projective limit ${\cal A}=\vpl_{n\in {\mab N}}{\cal A}_n$ for 
sheaves of commutative rings with unit elements in ${\cal T}$.  
Let ${\rm D}^{\rm b}{\rm F}_{\rm ctf}({\cal T},{\cal A})$ be the projective 2-limit of 
${\rm D}^{\rm b}{\rm F}_{\rm ctf}({\cal T},{\cal A}_n)$'s: 
an object of ${\rm D}^{\rm b}{\rm F}_{\rm ctf}({\cal T},{\cal A})$ 
is a projective system 
$\{(K^{\bul}_n,\{P_k\}_{k=1}^m)\}_{n=0}^{\infty}$, 
where $(K^{\bul}_n,\{P_k\}_{k=1}^m)$ is 
an object of ${\rm D}^{\rm b}{\rm F}_{\rm ctf}({\cal T},{\cal A}_n)$ 
such that 
$$(K^{\bul}_{n+1},\{P_k\}_{k=1}^m)\otimes^L_{{\cal A}_{n+1}}{\cal A}_n
=(K^{\bul}_n,\{P_k\}_{k=1}^m) \quad (\forall n\in {\mab N})$$
in ${\rm D}^{\rm b}{\rm F}_{\rm ctf}({\cal T},{\cal A}_n)$. 
Here $\otimes^L_{{\cal A}_{n+1}}$ is the derived tensor product of bounded below 
$m$-filtered complexes defined in (\ref{eqn:drtsr}) (cf.~\cite[(1.2.5.7)]{nh2}). 
A morphism in ${\rm D}^{\rm b}{\rm F}^m_{\rm ctf}({\cal T},{\cal A})$ 
is obviously defined. 
\end{defi} 

In the case $m=1$, we denote ${\rm D}^{\rm b}{\rm F}^m_{\rm ctf}$ 
by simply ${\rm D}^{\rm b}{\rm F}_{\rm ctf}$. 

\begin{coro}\label{coro:dpp}
The filtered complexes 
$(R\os{\circ}{\eps}_{D*}({\mab Z}/l^n),\tau)\in 
{\rm D}^+{\rm F}(\os{\circ}{X}_{\rm et},{\mab Z}/l^n)$'s in 
$\os{\circ}{X}_{\rm et}$ define an object 
$(R\os{\circ}{\eps}_{D*}({\mab Z}_l),\tau)$ of 
${\rm D}^{\rm b}{\rm F}_{\rm ctf}(\os{\circ}{X}_{\rm et},{\mab Z}_l)$.  
\end{coro}
\begin{proof} 
By the definition of $(M^{\bul}_{l^n}(\os{\circ}{D}/\os{\circ}{S}),Q)$ and 
by (\ref{eqn:tktk1}), we have the following formula: 
\begin{equation*} 
{\rm gr}_k^QM^{\bul}_{l^n}(\os{\circ}{D}/\os{\circ}{S})
=
{\mab Z}/l^n(-k)
\otimes_{\mab Z}
c^{(k)}_{*}(\varpi^{(k)}_{\rm et}(\os{\circ}{D}/\os{\circ}{S}))[-k] 
\tag{9.6.1}\label{eqn:grkqa}
\end{equation*} 
in $D^{\rm b}(\os{\circ}{X}_{\rm et},{\mab Z}/l^n)$. 
Hence the cohomological sheaves of the graded complexes of 
$(M^{\bul}_{l^n}(\os{\circ}{D}/\os{\circ}{S}),Q)$ are flat and smooth. 
We have to prove that the natural morphism 
$(M^{\bul}_{l^{n+1}}(\os{\circ}{D}/\os{\circ}{S}),Q)
\otimes_{{\mab Z}/l^{n+1}}{\mab Z}/l^n
\lo 
(M^{\bul}_{l^n}(\os{\circ}{D}/\os{\circ}{S}),Q)$ 
is a quasi-isomorphism. 
Because $(M^{\bul}_{l^{n+1}}(\os{\circ}{D}/\os{\circ}{S}),Q)$ 
is a bounded above filtered flat complex and the cohomological sheaves of 
${\rm gr}^Q_kM^{\bul}_{l^{n+1}}(\os{\circ}{D}/\os{\circ}{S})$ are flat, 
\begin{align*} 
{\cal H}^q({\rm gr}^Q_k(M^{\bul}_{l^{n+1}}(\os{\circ}{D}/\os{\circ}{S})
\otimes_{{\mab Z}/l^{n+1}}{\mab Z}/l^n))
&={\cal H}^q({\rm gr}^Q_kM^{\bul}_{l^{n+1}}(\os{\circ}{D}/\os{\circ}{S})
\otimes_{{\mab Z}/l^{n+1}}{\mab Z}/l^n) \tag{9.6.2}\label{eqn:gmmqa}\\
& ={\cal H}^q({\rm gr}^Q_kM^{\bul}_{l^{n+1}}(\os{\circ}{D}/\os{\circ}{S}))
\otimes_{{\mab Z}/l^{n+1}}{\mab Z}/l^n. 
\end{align*} 
The last sheaf is isomorphic to 
\begin{equation*} 
\begin{cases} 
0 \quad (q\not=k)\\
{\mab Z}/l^n(-k)
\otimes_{\mab Z}
c^{(k)}_{*}(\varpi^{(k)}_{\rm et}(\os{\circ}{D}/\os{\circ}{S})) \quad (q=k). 
\end{cases}
\end{equation*}  
This is nothing but 
${\cal H}^q({\rm gr}^Q_kM^{\bul}_{l^{n}}(\os{\circ}{D}/\os{\circ}{S}))$. 
Because the filtration $Q$ on $M^{\bul}_{l^{n}}(\os{\circ}{D}/\os{\circ}{S})$ is biregular, 
this means that 
$(M^{\bul}_{l^{n+1}}(\os{\circ}{D}/\os{\circ}{S}),Q)\otimes_{{\mab Z}/l^{n+1}}{\mab Z}/l^n
=(M^{\bul}_{l^{n}}(\os{\circ}{D}/\os{\circ}{S}),Q)$ 
in ${\rm D}^{\rm b}{\rm F}(\os{\circ}{X}_{\rm et},{\mab Z}/l^n)$.
By (\ref{eqn:tktk1}) the claim about the constructivility and the finite tor-dimension is obvious. 
\end{proof}

Next let us recall the Rapoport-Zink-Nakayama's 
double complex (\cite{rz}, \cite{nd}). 
\par 
Let $I^{\bul}_{X,l^n}$ $(n\in {\mab N})$ be an injective resolution of ${\mab Z}/l^n$ 
in $X_{\rm ket}$. 
Then we have a projective system 
$\{I^{\bul}_{X,l^n}\}_{n=1}^{\infty}$ fitting into the following commutative diagram 
\begin{equation*} 
\begin{CD}
{\mab Z}/l^{n+1}@>>>I^{\bul}_{X,l^{n+1}}\\
@V{\rm proj}.VV @VVV \\
{\mab Z}/l^{n}@>>>I^{\bul}_{X,l^n}. 
\end{CD}
\tag{9.6.3}\label{cd:ikin}
\end{equation*} 
Set 
\begin{equation*} 
K^{\bul}_{l^n}(X_{\frac{1}{l^{\infty}}}/S_{\frac{1}{l^{\infty}}}):=\eps_{X*}\pi_{X_{\frac{1}{l^{\infty}}}*}
\pi_{X_{\frac{1}{l^{\infty}}}}^*(I^{\bul}_{X,l^n})\in C^+(\os{\circ}{X}_{\rm et},{\mab Z}/l^n).  
\tag{9.6.4}\label{eqn:kdef}
\end{equation*} 
By (\ref{cd:ikin}) we have the projective system 
$\{K^{\bul}_{l^n}(X_{\frac{1}{l^{\infty}}}/S_{\frac{1}{l^{\infty}}})\}_{n=1}^{\infty}$.

\par 
The following has been stated in \cite[p.~723]{nd}: 

\begin{prop}\label{prop:lji} 
Let $I^{\bul}_{X_{\frac{1}{l^{\infty}}},l^n}$ be an injective resolution of 
${\mab Z}/l^n$ in $(X_{\frac{1}{l^{\infty}}})_{\rm ket}$. 
Then  the natural morphism 
\begin{equation*} 
(\eps_X\pi_{X_{\frac{1}{l^{\infty}}}})_*
\pi^{*}_{X_{\frac{1}{l^{\infty}}}}(I^{\bul}_{X,l^n}) 
\lo (\eps_X \pi_{X_{\frac{1}{l^{\infty}}}})_*(I^{\bul}_{X_{\frac{1}{l^{\infty}}},l^n}). 
\tag{9.7.1}\label{eqn:eij}
\end{equation*} 
is a quasi-isomorphism. Consequently 
the complex 
$K^{\bul}_{l^n}(X_{\frac{1}{l^{\infty}}}/S_{\frac{1}{l^{\infty}}})$ is quasi-isomorphic to 
$(\eps_X\pi_{X_{\frac{1}{l^{\infty}}}})_*(I^{\bul}_{X_{\frac{1}{l^{\infty}}},l^n})$. 
\end{prop}
\begin{proof}  
Let $I^{\bul}_{X_{\frac{1}{l^m}},l^n}$ be 
an injective resolution of ${\mab Z}/l^n$ 
in $(X_{\frac{1}{l^{m}}})_{\rm ket}$ for $m\in {\mab N}$. 
Then we have the following natural morphism 
\begin{equation*}  
\vil_{m\in {\mab N}}(\eps_X \pi_{X_{\frac{1}{l^m}}})_*(I^{\bul}_{X_{\frac{1}{l^{m}}},l^n})
\lo (\eps_X \pi_{X_{\frac{1}{l^{\infty}}}})_*(I^{\bul}_{X_{\frac{1}{l^{\infty}}},l^n}). 
\tag{9.7.2}\label{eqn:epxj}
\end{equation*} 
This morphism is a quasi-isomorphism by 
\cite[VI (8.7.3.1)]{sga4-2}. 
Since $X_{l^m}$ is an object of the Kummer \'{e}tale site of $X$, 
$\pi^{*}_{X_{\frac{1}{l^{m}}}}(I^{\bul}_{X,l^n})$ is injective 
(\cite[V (4.11) (1)]{sga4-2}). 
Hence we can set 
$I^{\bul}_{X_{\frac{1}{l^{m}}},l^n}:=\pi^{*}_{X_{\frac{1}{l^{m}}}}(I^{\bul}_{X,l^n})$ $(m\in {\mab N})$. 
In this case 
\begin{align*}  
\vil_{m\in {\mab N}}(\eps_X \pi_{X_{\frac{1}{l^{m}}}})_*(I^{\bul}_{X_{\frac{1}{l^{m}}},l^n})
& = \vil_{m\in {\mab N}}(\eps_X \pi_{X_{\frac{1}{l^{m}}}})_*
\pi^{*}_{X_{\frac{1}{l^{m}}}}(I^{\bul}_{X,l^n})
=\eps_{X*} \vil_{m\in {\mab N}}\pi_{X_{\frac{1}{l^{m}}}*}
\pi^{*}_{X_{\frac{1}{l^{m}}}}(I^{\bul}_{X,l^n}) \tag{9.7.3}\label{eqn:ep}\\ 
{} & =\eps_{X*} \pi_{X_{\frac{1}{l^{\infty}}}*}
\pi^{*}_{X_{\frac{1}{l^{\infty}}}}(I^{\bul}_{X,l^n}). 
\end{align*} 
Here we have used \cite[VI (8.5.5)]{sga4-2} as in \cite[p.~723]{nd}. 
Hence the morphism (\ref{eqn:eij}) is a quasi-isomorphism. 
\end{proof}

Denote by ${\mab Z}/l^m(1)$ the abelian sheaf in $\os{\circ}{X}_{\rm et}$ defined by the group scheme 
$\mu_{l^m}:=\ul{\rm Spec}_{\os{\circ}{X}}({\cal O}_X[T_m]/(T^{l^m}_{m}-1))$. 
Set ${\mab Z}_l(1):=\vpl_n{\mab Z}/l^m(1)$. 
Fix a ``generator'' $T:=(T_m)_{m\in {\mab N}}$ of ${\mab Z}_l(1)$. Because 
${\mab Z}_l(1)$ acts on $X_{\frac{1}{l^{\infty}}}$ naturally, 
we have an endomorphism 
\begin{equation*} 
T -1 \col \{K^{\bul}_{l^n}(X_{\frac{1}{l^{\infty}}}/S_{\frac{1}{l^{\infty}}})\}_{n=1}^{\infty} \lo 
\{K^{\bul}_{l^n}(X_{\frac{1}{l^{\infty}}}/S_{\frac{1}{l^{\infty}}})\}_{n=1}^{\infty} 
\tag{9.7.4}\label{eqn:knnd}
\end{equation*} 
of $\{K^{\bul}_{l^n}(X_{\frac{1}{l^{\infty}}}/S_{\frac{1}{l^{\infty}}})\}_{n=1}^{\infty}$. 
Let ${\rm MF}_{l^n}(T -1)={\rm MF}_{l^n}(T -1)^{\bul}$ 
be the mapping fiber of this endomorphism:  
\begin{align*}
{\rm MF}_{l^n}(T -1)& :=K^{\bul}_{l^n}(X_{\frac{1}{l^{\infty}}}/S_{\frac{1}{l^{\infty}}})\oplus 
K^{\bul}_{l^n}(X_{\frac{1}{l^{\infty}}}/S_{\frac{1}{l^{\infty}}})[-1]\tag{9.7.5}\label{ali:mftkdf}  \\
&=
s((K^{\bul}_{l^n}(X_{\frac{1}{l^{\infty}}}/S_{\frac{1}{l^{\infty}}}),d) 
\os{T -1}{\lo} (K^{\bul}_{l^n}(X_{\frac{1}{l^{\infty}}}/S_{\frac{1}{l^{\infty}}}),-d)).  
\end{align*} 
Here $d$ on the right hand side 
is the boundary morphism of 
$K^{\bul}_{l^n}(X_{\frac{1}{l^{\infty}}}/S_{\frac{1}{l^{\infty}}})$ 
and $s$ means the single complex of a double complex. 
Obviously we have the projective system 
$\{{\rm MF}_{l^n}(T-1)\}_{n=1}^{\infty}$.

\par 
Let $\theta \col \{{\rm MF}_{l^n}(T-1)\}_{n=1}^{\infty}  \lo 
\{{\rm MF}_{l^n}(T-1)(1)[1]\}_{n=1}^{\infty} $ be 
the following vertical morphism
\begin{equation*}
\footnotesize{
\begin{CD}
\{(K^{\bul}_{l^n}(X_{\frac{1}{l^{\infty}}}/S_{\frac{1}{l^{\infty}}})(1),-d(1))\}_{n=1}^{\infty}  @>{-(T-1)}>> 
\{(K^{\bul}_{l^n}(X_{\frac{1}{l^{\infty}}}/S_{\frac{1}{l^{\infty}}})(1),d(1))]\}_{n=1}^{\infty} 
@. \\
@. @A{{\rm id}\otimes T}AA  @.\\ 
@.  \{(K^{\bul}_{l^n}(X_{\frac{1}{l^{\infty}}}/S_{\frac{1}{l^{\infty}}}),d)]\}_{n=1}^{\infty}  
@>{T-1}>>\{(K^{\bul}_{l^n}(X_{\frac{1}{l^{\infty}}}/S_{\frac{1}{l^{\infty}}}),-d)]\}_{n=1}^{\infty}
\end{CD}
\tag{9.7.7}\label{cd:lttarl}}
\end{equation*} 
(\cite[(1.6)]{stwsl}).

\begin{lemm}[{\bf \cite[p.~723]{nd}}]\label{lemm:fq0}
The complex ${\rm MF}_{l^n}(T-1)$ is isomorphic to 
$R\eps_{X*}({\mab Z}/l^n)$ in $D^+(\os{\circ}{X}_{\rm et},{\mab Z}/l^n)$. 
In particular, if $q>>0$, then ${\cal H}^q({\rm MF}_{l^n}(T-1))=0$ for all $n\in {\mab N}$. 
\end{lemm}
\begin{proof}
By \cite[(1.3.1)]{nd} the following  sequence 
\begin{equation*}  
0 \lo I^{\bul}_{X,l^n} \lo 
\pi_{X_{l^{\infty}}*}\pi^{*}_{X_{l^{\infty}}}(I^{\bul}_{X,l^n}) 
\os{T-1}{\lo} 
\pi_{X_{l^{\infty}_*}}\pi^{*}_{X_{l^{\infty}}}(I^{\bul}_{X,l^n})  
\lo 0
\tag{9.8.1}\label{eqn:tiix}\\ 
\end{equation*}  
is exact. Because $I^{\bul}_{X,l^n}$ is a complex of injective ${\mab Z}/l^n$-modules in 
$\os{\circ}{X}_{\rm et}$, the sequence 
\begin{equation*}  
0 \lo \eps_{X*}(I^{\bul}_{X,l^n}) \lo 
(\eps_X\pi_{X_{l^{\infty}}})_*\pi^{*}_{X_{l^{\infty}}}(I^{\bul}_{X,l^n}) 
\os{T-1}{\lo} 
(\eps_X\pi_{X_{l^{\infty}}})_*\pi^{*}_{X_{l^{\infty}}}(I^{\bul}_{X,l^n})  
\lo 0
\tag{9.8.2}\label{eqn:tkix}\\ 
\end{equation*}  
is exact. Hence ${\rm MF}_{l^n}(T-1)$ is quasi-isomorphic to 
$\eps_{X*}(I^{\bul}_{X,l^n})$.  
Consequently, by \cite[(2.4)]{ktnk},  
\begin{align*} 
{\cal H}^q({\rm MF}_{l^n}(T-1))&=R^q\eps_{X*}({\mab Z}/l^n)=
\bigwedge^q(M_X^{\rm gp}/{\cal O}_X^*)
\otimes_{\mab Z}{\mab Z}/l^n(-q)
\tag{9.8.3}\label{eqn:tkx}\\ 
{} & =a^{(q-1)}_*(({\mab Z}/l^n)_{\os{\circ}{X}{}^{(q-1)}}(-q)
\otimes_{\mab Z}
\varpi^{(q-1)}_{\rm et}(X/S)).  
\end{align*} 

Because $\os{\circ}{X}$ is quasi-compact, 
${\cal H}^q({\rm MF}_{l^n}(T-1))=0$ 
($\exists q_0>0$, $\forall q>q_0$, $\forall n\in {\mab N}$). 
\end{proof}

\begin{coro}\label{coro:iind}
The projective system 
$\{({\rm MF}_{l^n}(T-1),\tau)\}_{n=1}^{\infty}$ defines a well-defined object of 
${\rm D}^{\rm b}{\rm F}_{\rm ctf}(\os{\circ}{X}_{\rm et},{\mab Z}_l)$. 
$($We denote the resulting filtered complex by 
${(\rm MF}_{l^{\infty}}(T-1),\tau)$ by abuse of notation.$)$
\end{coro} 
\begin{proof}By using (\ref{eqn:tkx}), 
the proof is the same as that of  (\ref{coro:dpp}). 
\end{proof}

\par 
Denote also by $\theta$ the following natural  morphism 
\begin{equation*} 
K^{\bul}_{l^n}(X_{\frac{1}{l^{\infty}}}/S_{\frac{1}{l^{\infty}}}) \owns x\lom (0,x\otimes T)\in 
{\rm MF}_{l^n}(T-1)(1)[1]. 
\tag{9.9.1}\label{ali:fln}
\end{equation*} 

\begin{rema}
The section $T$ in $(0,x\otimes T)$ in (\ref{ali:fln}) 
will turn out to be very important when we consider 
the contravariant functoriality in (\ref{theo:func}) below. 
\end{rema}

Let $A^{\bul \bul}_{l^n}(X_{\frac{1}{l^{\infty}}}/S_{\frac{1}{l^{\infty}}})$ 
be the double complex defined by 
the following formula and 
the following boundary morphisms; 
\begin{equation*} 
A^{ij}_{l^n}(X_{\frac{1}{l^{\infty}}}/S_{\frac{1}{l^{\infty}}}):=({\rm MF}_{l^n}(T-1)(j+1)/
\tau_j{\rm MF}_{l^n}(T-1)(j+1))^{i+j+1} \quad 
(i\in {\mab Z},j\in {\mab N}), 
\tag{9.10.1}\label{eqn:anxsmf}
\end{equation*} 
\begin{equation*} 
\begin{CD}
A^{i,j+1}_{l^n}(X_{\frac{1}{l^{\infty}}}/S_{\frac{1}{l^{\infty}}}) 
@>{-d}>> A^{i+1,+1}_{l^n}(X_{\frac{1}{l^{\infty}}}/S_{\frac{1}{l^{\infty}}})\\
@A{\theta}AA @AA{\theta}A \\
A^{ij}_{l^n}(X_{\frac{1}{l^{\infty}}}/S_{\frac{1}{l^{\infty}}}) @>{-d}>> 
A^{i+1,j}_{l^n}(X_{\frac{1}{l^{\infty}}}/S_{\frac{1}{l^{\infty}}}). 
\end{CD}
\tag{9.10.2}\label{cd:dlsgn}
\end{equation*} 
Note that the signs of our boundary morphisms are different from that of 
the boundary morphisms in the proof of \cite[p.~29]{rz} and \cite[the proof of (1.4)]{nd}. 
(We think that the signs in (\ref{cd:dlsgn}) are the best.) 
Let $A^{\bul}_{l^n}(X_{\frac{1}{l^{\infty}}}/S_{\frac{1}{l^{\infty}}})$ be the single complex of 
$A^{\bul \bul}_{l^n}(X_{\frac{1}{l^{\infty}}}/S_{\frac{1}{l^{\infty}}})$: 
\begin{align*} 
A^{\bul}_{l^n}(X_{\frac{1}{l^{\infty}}}/S_{\frac{1}{l^{\infty}}})= & 
s\{(({\rm MF}_{l^n}(T-1)(1)/\tau_0{\rm MF}_{l^n}(T-1)(1))^{\bul+1},-d)
\os{\theta}{\lo} \tag{9.10.3}\label{eqn:tdfef}\\
{} & (({\rm MF}_{l^n}(T-1)(2)/
\tau_1{\rm MF}_{l^n}(T-1)(2))^{\bul+2},-d)
\os{\theta}{\lo}  \cdots\}. 
\end{align*}  
Then we have the projective system 
$\{A^{\bul}_{l^n}(X_{\frac{1}{l^{\infty}}}/S_{\frac{1}{l^{\infty}}})\}_{n=1}^{\infty}$ and a morphism 
$\theta \col \{K^{\bul}_{l^n}(X_{\frac{1}{l^{\infty}}}/S_{\frac{1}{l^{\infty}}})\}_{n=1}^{\infty}\lo 
\{A^{\bul}_{l^n}(X_{\frac{1}{l^{\infty}}}/S_{\frac{1}{l^{\infty}}})\}_{n=1}^{\infty}$ 
of projective systems of complexes. 
\par 
In the proof of \cite[(1.4)]{nd} 
Nakayama has proved that the morphism (\ref{ali:fln}) 
induces the following isomorphism
\begin{align*}  
\theta \col R(\eps_X \pi_{X_{\frac{1}{l^{\infty}}}})_*({\mab Z}/l^n)
\os{\sim}{\lo}  
A^{\bul}_{l^n}(X_{\frac{1}{l^{\infty}}}/S_{\frac{1}{l^{\infty}}})
\tag{9.10.4}\label{ali:fiss}
\end{align*} 
in $D^+(\os{\circ}{X}_{\rm et},{\mab Z}/l^n)$. 
Here we prove this fact by the same proof as that of \cite[(5.13)]{stz} 
(and \cite[(3.17)]{fn}):

\begin{prop}\label{prop:qislws} 
The morphism 
$\theta \col K^{\bul}_{l^n}(X_{\frac{1}{l^{\infty}}}/S_{\frac{1}{l^{\infty}}})\lo  {\rm MF}_{l^n}(T-1)(1)[1]$ 
induces a quasi-isomorphism 
\begin{equation*} 
\theta \col 
K^{\bul}_{l^n}(X_{\frac{1}{l^{\infty}}}/S_{\frac{1}{l^{\infty}}})\lo 
A^{\bul}_{l^n}(X_{\frac{1}{l^{\infty}}}/S_{\frac{1}{l^{\infty}}})
\tag{9.11.1}\label{eqn:fqis} 
\end{equation*} 
in $C^+(\os{\circ}{X}_{\rm et},{\mab Z}/l^n)$.  
\end{prop}
\begin{proof} 
It is obvious that, for $i\geq 0$, the sequence  
\begin{equation*} 
0 \lo K^i_{l^n}(X_{\frac{1}{l^{\infty}}}/S_{\frac{1}{l^{\infty}}}) \os{\theta}{\lo} 
A^{i0}_{l^n}(X_{\frac{1}{l^{\infty}}}/S_{\frac{1}{l^{\infty}}}) \os{\theta}{\lo} 
A^{i1}_{l^n}(X_{\frac{1}{l^{\infty}}}/S_{\frac{1}{l^{\infty}}}) \os{\theta}{\lo} \cdots 
\end{equation*}
is exact. 
We have to prove that 
the complex $A^{-1\bul}_{l^n}(X_{\frac{1}{l^{\infty}}}/S_{\frac{1}{l^{\infty}}})$ is acyclic.
Consider the action of $T-1$  on 
$R^j\pi_{X_{\frac{1}{l^{\infty}}}*}({\mab Z}/l^n)$. 
Let $\ol{x}(\log)$ be the log geometric point of $X$. 
Then, 
by the log proper base change theorem 
\cite[(5.1)]{nale} and \cite[(4.1)]{nale}, 
$R^j\pi_{X_{\frac{1}{l^{\infty}}}*}({\mab Z}/l^n)_{\ol{x}(\log)}
=H^j({\mab Z}_l,{\mab Z}/l^n)$. 
This is equal to 
${\mab Z}/l^n$ $(j=0, 1)$ and $0$ $(j\geq 2)$ and the action of 
${\mab Z}_l$ on ${\mab Z}/l^n$ is trivial. 
Hence $T$ acts trivially on 
$R^j\pi_{X_{\frac{1}{l^{\infty}}}*}({\mab Z}/l^n)_{\ol{x}(\log)}$. 
Consequently $T$ acts trivially on  $R^j\pi_{X_{\frac{1}{l^{\infty}}}*}({\mab Z}/l^n)$. 
By the following Leray spectral sequence 
\begin{equation*}
E_2^{ij}=R^i\eps_{X*}R^j\pi_{X_{\frac{1}{l^{\infty}}}*}({\mab Z}/l^n)
\Lo 
R^{i+j}(\eps_X\pi_{X_{\frac{1}{l^{\infty}}}})_*({\mab Z}/l^n), 
\end{equation*} 
we see that $T$ acts trivially on 
${\cal H}^q(K^{\bul}_{l^n}(X_{\frac{1}{l^{\infty}}}/S_{\frac{1}{l^{\infty}}}))
=R^q(\eps_X\pi_X)_*({\mab Z}/l^n)$ $(q\in {\mab N})$. 
Hence $T={\rm id}$ on ${\cal H}^q(K^{\bul}_{l^n}(X_{\frac{1}{l^{\infty}}}/S_{\frac{1}{l^{\infty}}}))$. 
Set $K^{\bul}_{l^n}:=K^{\bul}_{l^n}(X_{\frac{1}{l^{\infty}}}/S_{\frac{1}{l^{\infty}}})$. 
Let $(x,y)$ be a local section of $K_{l^n}^j(j+1)\oplus K_{l^n}^{j-1}(j+1)$ 
$(j\in {\mab N})$ such that $\theta(x,y)=(0,x\otimes T)
\in {\rm Ker}(d: {\rm MF}_{l^n}(T-1)^{j+1}(j+2)
\lo {\rm MF}_{l^n}(T-1)^{j+2}(j+2))$. Then $dx=0$ and hence 
$x$ defines a local section of ${\cal H}^j(K^{\bul}_{l^n})(j+1)$. 
Because $T$ acts trivially on 
${\cal H}^j(K^{\bul}_{l^n})$, $(T-1)(x)=0$. 
Therefore there exists a local section $z\in K^{j-1}_{l^n}(j+1)$ 
such that $(T-1)(x)=dz$. We immediately see that 
$d(x,z)=0$ and $(x,y)=(x,z)+
\theta((y-z)\otimes \check{T},0) \in 
{\rm Ker}\,d+{\rm Im}(\theta)$, where 
$\check{T}$ is the dual basis of $T$.  
The last formula tells us that $A^{-1\bul}_{l^n}(X_{\frac{1}{l^{\infty}}}/S_{\frac{1}{l^{\infty}}})$ is acyclic. 
\par 
We complete the proof of (\ref{prop:qislws}). 
\end{proof}

\begin{coro}\label{coro:wda}
The complex $A^{\bul}_{l^n}(X_{\frac{1}{l^{\infty}}}
/S_{\frac{1}{l^{\infty}}})\in D^+(\os{\circ}{X}_{\rm et},{\mab Z}_l)$ 
is independent of the choice of $I^{\bul}_{X,l^n}$. 
\end{coro}

\par 
For $i\in {\mab Z}$ and $j\in {\mab N}$, 
set 
\begin{equation*} 
\{P_kA^{ij}_{l^n}(X_{\frac{1}{l^{\infty}}}/S_{\frac{1}{l^{\infty}}})\}_{n=1}^{\infty}
:=\{((\tau_{2j+k+1}+\tau_j){\rm MF}_{l^n}(T-1)(j+1)/
\tau_j{\rm MF}_{l^n}(T-1)(j+1))^{i+j+1}\}_{n=1}^{\infty} \quad 
\tag{9.12.1}\label{eqn:tdjef}
\end{equation*} 
and 
\begin{align*}
&\{P_kA^{\bul}_{l^n}(X_{\frac{1}{l^{\infty}}}/S_{\frac{1}{l^{\infty}}})\}_{n=1}^{\infty}
:=
\{s\{((\tau_{k+1}+\tau_0){\rm MF}_{l^n}(T-1)(1)/\tau_0{\rm MF}_{l^n}(T-1)(1))^{\bul+1},-d)
\os{\theta}{\lo} \tag{9.12.2}\label{eqn:tdef}\\
{} &
\cdots \os{\theta}{\lo}  
(\tau_{2j+k+1}+\tau_j){\rm MF}_{l^n}(T-1)(j)/\tau_j{\rm MF}_{l^n}(T-1)(j))^{\bul+j},-d)
\os{\theta}{\lo} \cdots  \}\}_{n=1}^{\infty}. 
\end{align*}    
(Note that $\tau_{2j+k+1}{\rm MF}_{l^n}(T-1)(j)/\tau_j{\rm MF}_{l^n}(T-1)(j)$ in a lot of literatures 
is incorrect since $2j+k+1< j$ may occur.) We have a filtered complex 
$(A^{\bul}_{l^n}(X_{\frac{1}{l^{\infty}}}/S_{\frac{1}{l^{\infty}}}),P)
\in {\rm C}^+{\rm F}(\os{\circ}{X}_{\rm et},{\mab Z}/l^n)$. 

\begin{prop}\label{prop:exf}
If $k>>0$, then the natural inclusion 
$P_kA^{\bul}_{l^n}(X_{\frac{1}{l^{\infty}}}
/S_{\frac{1}{l^{\infty}}})\os{\subset}{\lo}  A^{\bul}_{l^n}(X_{\frac{1}{l^{\infty}}}/S_{\frac{1}{l^{\infty}}})$
is a quasi-isomorphism for all $n\in {\mab N}$. 
If $k<<0$, then 
$P_kA^{\bul}_{l^n}(X_{\frac{1}{l^{\infty}}}/S_{\frac{1}{l^{\infty}}})$ 
is exact for all $n\in {\mab N}$. 
\end{prop} 
\begin{proof}
First we prove the former statement. 
Consider the following commutative diagram 
\begin{equation*} 
\begin{CD} 
E_1^{pq}={\cal H}^q((\tau_{2p+k+1}+\tau_p)
({\rm MF}_{l^n}(T-1)(p+1)/\tau_p{\rm MF}_{l^n}(T-1)(p+1))^{\bul+1})\\
@VVV \\
E_1^{pq}={\cal H}^q(({\rm MF}_{l^n}(T-1)(p+1)/\tau_{p}{\rm MF}_{l^n}(T-1)(p+1))^{\bul+1})\\
\Lo {\cal  H}^{p+q}(P_kA^{\bul}_{l^n}(X_{\frac{1}{l^{\infty}}}/S_{\frac{1}{l^{\infty}}}))\\
@VVV \\
\Lo {\cal  H}^{p+q}(A^{\bul}_{l^n}(X_{\frac{1}{l^{\infty}}}/S_{\frac{1}{l^{\infty}}})). 
\end{CD} 
\tag{9.13.1}\label{eqn:stpa}
\end{equation*} 
It suffices to prove that, for a fixed $p$, then 
${\cal H}^r(({\rm MF}_{l^n}(T-1)(p+1)/\tau_p{\rm MF}_{l^n}(T-1)(p+1))^{\bul+1})=0$ 
for $r>>0$. 
By the following exact sequence 
$$0\lo \tau_{p}{\rm MF}_{l^n}(T-1)\lo {\rm MF}_{l^n}(T-1)\lo 
{\rm MF}_{l^n}(T-1)/\tau_{p}{\rm MF}_{l^n}(T-1)\lo 0$$
and by (\ref{lemm:fq0}), 
it suffice to prove that 
${\cal H}^r(\tau_{p}{\rm MF}_{l^n}(T-1))=0$ if $r>>0$ $(\forall n\in {\mab N})$. 
However this is obvious by (\ref{eqn:tkx}). 
\par 
The latter statement follows from (\ref{lemm:fq0}) and  (\ref{eqn:tkx})
since,  
if $2j+k+1\geq j$ and if $k<<0$, 
then $j\geq -k>>0$ in (\ref{eqn:tdjef}). 
\end{proof}


\begin{prop}\label{prop:ni}
The filtered complex 
$(A^{\bul}_{l^n}(X_{\frac{1}{l^{\infty}}}/S_{\frac{1}{l^{\infty}}}),P)\in 
{\rm D}^+{\rm F}(\os{\circ}{X}_{\rm et},{\mab Z}/l^n)$ is independent of the choice of $I^{\bul}_{X,l^n}$. 
\end{prop} 
\begin{proof} 
Let $k'$ be a large integer. 
By (\ref{prop:exf}) it suffices to prove that 
$$(P_{k'}A^{\bul}_{l^n}(X_{\frac{1}{l^{\infty}}}/S_{\frac{1}{l^{\infty}}}),\{P_k\}_{k\leq k'})
\in {\rm D}^+{\rm F}(\os{\circ}{X}_{\rm et},{\mab Z}/l^n)$$
is independent of the choice of 
$I^{\bul}_{X,l^n}$. Let $I'{}^{\bul}_{\! \! \! X,l^n}$ 
be another injective resolution 
of ${\mab Z}/l^n$ in $X_{\rm ket}$. 
Then we have analogous complexes  
$K'{}^{\bul}_{\! \! \!  l^n}(X_{\frac{1}{l^{\infty}}}/S_{\frac{1}{l^{\infty}}})$ and ${\rm MF}'_{l^n}(T-1)$ as in 
(\ref{eqn:kdef}) and (\ref{ali:mftkdf}), respectively. 
Since we have a quasi-isomorphism 
$I^{\bul}_{X,l^n} \lo I'{}^{\bul}_{\! \! \! X,l^n}$ 
fitting into the following commutative diagram 
\begin{equation*} 
\begin{CD} 
{\mab Z}/l^n @>>> I^{\bul}_{X,l^n} \\
@| @VVV \\
{\mab Z}/l^n @>>> I'{}^{\bul}_{\! \! \! X,l^n}, 
\end{CD}
\end{equation*} 
we have natural morphisms  
$K^{\bul}_{l^n}(X_{\frac{1}{l^{\infty}}}/S)\lo K'^{\bul}_{l^n}(X_{\frac{1}{l^{\infty}}}/S)$ and  
${\rm MF}_{l^n}(T-1) \lo {\rm MF}'_{l^n}(T-1)$. 
In fact we have a filtered quasi-isomorphism
\begin{equation*} 
({\rm MF}_{l^n}(T-1),\tau)
\lo 
({\rm MF}'_{l^n}(T-1),\tau) 
\tag{9.14.1}\label{eqn:tmjmf} 
\end{equation*}  
By the definition of $P$, we have the following: 
\begin{align*} 
{\rm gr}_k^PA^{\bul}_{l^n}(X_{\frac{1}{l^{\infty}}}/S_{\frac{1}{l^{\infty}}})
=\bigoplus_{j\geq \max\{-k,0\}}
{\rm gr}_{2j+k+1}^{\tau}{\rm MF}_{l^n}(T-1)(j)\{j\}[1]
\tag{9.14.2}\label{ali:anxqads}
\end{align*} 
because the induced morphism 
$$\theta \col {\rm gr}_{2j+k+1}^{\tau}{\rm MF}_{l^n}(T-1)(j)\{j\}[1]
\lo {\rm gr}_{2(j+1)+k+1}^{\tau}{\rm MF}_{l^n}(T-1)(j+1)\{j+1\}[1]$$  
is a zero morphism. 
The isomorphism (\ref{ali:fiss}), (\ref{prop:exf}), 
(\ref{eqn:tkx}) and the descending induction on $k$ show (\ref{prop:ni}). 
\end{proof} 
By (\ref{eqn:tkx}) and (\ref{ali:anxqads}) we have 
\begin{align*} 
{\rm gr}^P_kA^{\bul}_{l^n}(X_{\frac{1}{l^{\infty}}}/S_{\frac{1}{l^{\infty}}})
& =\us{j\geq {\rm max}\{-k,0\}}{\bigoplus}
{\rm gr}^{\tau}_{2j+k+1}(A^{\bul j}_{l^n}(X_{\frac{1}{l^{\infty}}}/S_{\frac{1}{l^{\infty}}})\{-j\},-d) 
\tag{9.14.3}\label{eqn:grf}\\ 
{} & = \us{j\geq {\rm max}\{-k,0\}}{\bigoplus} 
{\cal H}^{2j+k+1}({\rm MF}_{l^n}(T-1),-d)
[-2j-k-1]\{1\}(j+1) \\ 
{} & \simeq 
\us{j\geq {\rm max}\{-k,0\}}{\bigoplus}
R^{2j+k+1}\eps_{X*}({\mab Z}/l^n)[-2j-k](j+1) \\ 
{} & =\us{j\geq {\rm max}\{-k,0\}}{\bigoplus}
{\mab Z}/l^n(-j-k)\otimes_{\mab Z}a^{(2j+k)}_*
(\varpi^{(2j+k)}_{\rm et}(X/S))[-2j-k]
\end{align*} 
in $D^+(\os{\circ}{X}_{\rm et},{\mab Z}/l^n)$. 
If $k>>0$, then $2j+k\geq k>>0$. 
Because $\os{\circ}{X}$ is quasi-compact, 
$\os{\circ}{X}{}^{(2j+k)}=\emptyset$ if $2j+k>>0$. 
Hence 
${\rm gr}^{P_k}A^{\bul}_{l^n}
(X_{\frac{1}{l^{\infty}}}/S_{\frac{1}{l^{\infty}}})=0$ if $k>>0$ $(\forall n\in {\mab N})$. 
If $k<<0$, then $2j+k\geq -2k+k=-k>>0$. 
Hence $\os{\circ}{X}{}^{(2j+k)}=\emptyset$ and 
${\rm gr}^P_kA^{\bul}_{l^n}
(X_{\frac{1}{l^{\infty}}}/S_{\frac{1}{l^{\infty}}})=0$ $(\forall n\in {\mab N})$. 
For a fixed $k\in {\mab Z}$, there are only finitely many $j$'s such that 
the last term in (\ref{eqn:grf}) is nonzero. 

\begin{coro}\label{coro:cnz}
$(1)$ The filtered complexes 
$(A^{\bul}_{l^n}(X_{\frac{1}{l^{\infty}}}/S_{\frac{1}{l^{\infty}}}),P)$'s 
for $n$'s define a well-defined object 
$(A^{\bul}_{l^{\infty}}(X_{\frac{1}{l^{\infty}}}/S_{\frac{1}{l^{\infty}}}),P)$ of 
${\rm D}^{\rm b}{\rm F}_{\rm ctf}(\os{\circ}{X}_{\rm et},{\mab Z}_l)$. 
\par
$(2)$ The filtered complexes $K^{\bul}_{l^n}(X_{\frac{1}{l^{\infty}}}/S_{\frac{1}{l^{\infty}}})$'s for $n$'s 
define a well-defined object of 
$K^{\bul}_{l^{\infty}}(X_{\frac{1}{l^{\infty}}}/S_{\frac{1}{l^{\infty}}})$ of 
$D^{\rm b}_{\rm ctf}(\os{\circ}{X}_{\rm et},{\mab Z}_l)$. 
\end{coro}
\begin{proof} 
(1): 
As in (\ref{coro:dpp}), this follows from (\ref{prop:exf}) and (\ref{coro:iind})
and the definition of $(A^{\bul}_{l^n}(X_{\frac{1}{l^{\infty}}}/S_{\frac{1}{l^{\infty}}}),P)$. 
Indeed, let $\{(F^{\bul}_n,\tau)\}_{n=1}^{\infty}$ be 
a representative of the resulting filtered complex 
obtained in (\ref{coro:iind}). 
The projective system of 
\begin{align*} 
\{(A_n^{\bul},P)\}_{n=1}^{\infty}:=\{(s(F^{\bul}_n/\tau_0\lo F^{\bul}_n/\tau_1\lo \cdots), 
\{s(\cdots (\tau_{2j+k+1}+\tau_j)/
\tau_jF_n^{i+j+1}\cdots )_{i\in {\mab Z},j\in {\mab N}}\}_{k\in {\mab Z}})\}_{n=1}^{\infty} 
\end{align*}
defines a well-defined object of 
${\rm D}^{\rm b}{\rm F}_{\rm ctf}(\os{\circ}{X}_{\rm et},{\mab Z}_l)$. 
\par 
(2): (2) follows from (1) and (\ref{prop:qislws}). 
\end{proof}

Henceforth we take representatives of 
$K^{\bul}_{l^{\infty}}(X_{\frac{1}{l^{\infty}}}/S_{\frac{1}{l^{\infty}}})
=\{K^{\bul}_{l^n}(X_{\frac{1}{l^{\infty}}}/S_{\frac{1}{l^{\infty}}})\}_{n=1}^{\infty}\in 
{\rm D}^{\rm b}_{\rm ctf}(\os{\circ}{X}_{\rm et},{\mab Z}_l)$ 
in $\vpl_n C^{\rm b}(\os{\circ}{X}_{\rm et},{\mab Z}/l^n)$
and 
$(A^{\bul}_{l^{\infty}}(X_{\frac{1}{l^{\infty}}}/S_{\frac{1}{l^{\infty}}}),P)
=\{(A^{\bul}_{l^n}(X_{\frac{1}{l^{\infty}}}/S_{\frac{1}{l^{\infty}}}),P)\}_{n=1}^{\infty}\in  
{\rm D}^{\rm b}{\rm F}_{\rm ctf}(\os{\circ}{X}_{\rm et},{\mab Z}_l)$ in 
$\vpl_n {\rm C}^{\rm b}{\rm F}(\os{\circ}{X}_{\rm et},{\mab Z}/l^n)$, 
respectively.
By abuse of notation, we use the same symbols for them: 
$$K^{\bul}_{l^n}(X_{\frac{1}{l^{\infty}}}/S_{\frac{1}{l^{\infty}}})
\in C^{\rm b}(\os{\circ}{X}_{\rm et},{\mab Z}/l^n) 
~{\textrm a}{\textrm n}{\textrm d}~  
(A^{\bul}_{l^n}(X_{\frac{1}{l^{\infty}}}/S_{\frac{1}{l^{\infty}}}),P)
\in {\rm C}^{\rm b}{\rm F}(\os{\circ}{X}_{\rm et},{\mab Z}/l^n).$$

Consider the following filtered complex
$(K^{\bul}_{l^n}((X_{\frac{1}{l^{\infty}}},D_{\frac{1}{l^{\infty}}})/S_{\frac{1}{l^{\infty}}}),
P^{D_{\frac{1}{l^{\infty}}}})$:  
\begin{equation*}
(K^{\bul}_{l^n}((X_{\frac{1}{l^{\infty}}},D_{\frac{1}{l^{\infty}}})/S_{\frac{1}{l^{\infty}}}),P^{D_{\frac{1}{l^{\infty}}}}):=
s(K^{\bul}_{l^n}(X_{\frac{1}{l^{\infty}}}/S_{\frac{1}{l^{\infty}}})\otimes_{{\mab Z}/l^n}
(M^{\bul}_{l^n}(\os{\circ}{D}/\os{\circ}{S}),Q))\in {\rm C}^{\rm b}{\rm F}(\os{\circ}{X}_{\rm et},{\mab Z}/l^n). 
\end{equation*} 
Here $K^{\bul}_{l^n}(X_{\frac{1}{l^{\infty}}}/S_{\frac{1}{l^{\infty}}})$ means the filtered complex 
$K^{\bul}_{l^n}(X_{\frac{1}{l^{\infty}}}/S_{\frac{1}{l^{\infty}}})$ with trivial filtration by abuse of notation. 
Obviously we have the projective system 
$\{(K^{\bul}_{l^n}((X_{\frac{1}{l^{\infty}}},D_{\frac{1}{l^{\infty}}})/S_{\frac{1}{l^{\infty}}}),
P^{D_{\frac{1}{l^{\infty}}}})\}_{n=1}^{\infty}$.

\begin{prop}\label{prop:lnvc} 
\begin{equation*} 
\{R(\eps_{(X,D)}\pi_{(X_{\frac{1}{l^{\infty}}},D_{\frac{1}{l^{\infty}}})})_*({\mab Z}/l^n)\}_{n=1}^{\infty}
=\{K^{\bul}_{l^n}((X_{\frac{1}{l^{\infty}}},D_{\frac{1}{l^{\infty}}})/S_{\frac{1}{l^{\infty}}})\}_{n=1}^{\infty}
\tag{9.16.1}\label{eqn:axdkf}
\end{equation*}
in $D^{\rm b}_{\rm ctf}(\os{\circ}{X}_{\rm et},{\mab Z}_l)$. 
\end{prop}
\begin{proof}  
Let $\eps\col (X,D)\lo (\os{\circ}{X},\os{\circ}{D})$ be the natural morphism of log schemes. 
Consider the following cartesian diagram: 
\begin{equation*} 
\begin{CD} 
(X_{\frac{1}{l^{m}}},D_{\frac{1}{l^{m}}}) 
@>{\eps \circ \pi_{(X_{\frac{1}{l^{m}}},D_{\frac{1}{l^{m}}})}}>> (\os{\circ}{X},\os{\circ}{D}) \\ 
@VVV @VV{\os{\circ}{\eps}_D}V \\
X_{\frac{1}{l^{m}}}@>{\eps_X \circ \pi_{X_{\frac{1}{l^{m}}}}}>> \os{\circ}{X}. 
\end{CD}
\tag{9.16.2}\label{eqn:xdcdx}
\end{equation*}  
Then, by the finiteness of the cohomological dimension 
(\cite[(7.2) (1)]{adv}) and the log K\"{u}nneth formula 
(\cite[(6.1)]{nale}),  we obtain the following 
\begin{align*}
R(\eps_{(X,D)}\pi_{(X_{\frac{1}{l^{m}}},D_{\frac{1}{l^{m}}})})_*
({\mab Z}/l^n) & = R(\eps_{X} \pi_{X_{\frac{1}{l^{m}}}})_*({\mab Z}/l^n)
\otimes^L_{{\mab Z}/l^n}
R\os{\circ}{\eps}_{D*}({\mab Z}/l^n)
\tag{9.16.3}\label{eqn:xdkf}\\ 
{} & =
(\eps_{X}\pi_{X_{\frac{1}{l^{m}}}})_*(I^{\bul}_{X_{\frac{1}{l^{m}}},l^n})
\otimes_{{\mab Z}/l^n}M^{\bul}_{l^n}(\os{\circ}{D}/\os{\circ}{S}) \\
{}&=
(\eps_{X}\pi_{X_{\frac{1}{l^{m}}}})_*(\pi^{*}_{X_{\frac{1}{l^{m}}}}(I^{\bul}_{X,l^n}))
\otimes_{{\mab Z}/l^n}M^{\bul}_{l^n}(\os{\circ}{D}/\os{\circ}{S}) 
\end{align*}
in $D^+_{\rm ctf}(\os{\circ}{X},{\mab Z}/l^n)$. 
Taking the inductive limit of (\ref{eqn:xdkf}) with respect to $m$ 
and using the 
quasi-isomorphism (\ref{eqn:epxj}), 
we have the formula (\ref{eqn:axdkf}). 
\end{proof}

Consider the following double complex and the single complex: 
\begin{equation*} 
(A^{\bul \bul}_{l^n}((X_{\frac{1}{l^{\infty}}},D_{\frac{1}{l^{\infty}}})/S_{\frac{1}{l^{\infty}}}),P^{D_{\frac{1}{l^{\infty}}}})
:=A^{\bul}_{l^n}(X_{\frac{1}{l^{\infty}}}/S_{\frac{1}{l^{\infty}}})
\otimes_{{\mab Z}/l^n}(M^{\bul}_{l^n}(\os{\circ}{D}/\os{\circ}{S}),Q) 
\tag{9.16.4}\label{eqn:atbnxds} 
\end{equation*} 
and 
\begin{equation*} 
(A^{\bul}_{l^n}((X_{\frac{1}{l^{\infty}}},D_{\frac{1}{l^{\infty}}})/S_{\frac{1}{l^{\infty}}}),P^{D_{\frac{1}{l^{\infty}}}})
:=s(A^{\bul \bul}_{l^n}((X_{\frac{1}{l^{\infty}}},D_{\frac{1}{l^{\infty}}})/S_{\frac{1}{l^{\infty}}}),P^{D_{\frac{1}{l^{\infty}}}})\in 
{\rm C}^{\rm b}{\rm F}(\os{\circ}{X}_{\rm et},{\mab Z}/l^n).    
\tag{9.16.5}\label{eqn:asbnxds} 
\end{equation*}  
Here $A^{\bul}_{l^n}(X_{\frac{1}{l^{\infty}}}/S_{\frac{1}{l^{\infty}}})$ means the filtered complex 
$A^{\bul}_{l^n}(X_{\frac{1}{l^{\infty}}}/S_{\frac{1}{l^{\infty}}})$ with trivial filtration by abuse of notation. 
\par 
Consider also the following filtered double complex and the single complex: 
\begin{equation*} 
(A^{\bul \bul}_{l^n}((X_{\frac{1}{l^{\infty}}},D_{\frac{1}{l^{\infty}}})/S_{\frac{1}{l^{\infty}}}),P)
:=(A^{\bul}_{l^n}(X_{\frac{1}{l^{\infty}}}/S_{\frac{1}{l^{\infty}}}),P)
\otimes_{{\mab Z}/l^n}(M^{\bul}_{l^n}(\os{\circ}{D}/\os{\circ}{S}),Q) 
\tag{9.16.6}\label{eqn:bapt}
\end{equation*}  
and 
\begin{equation*} 
(A^{\bul}_{l^n}((X_{\frac{1}{l^{\infty}}},D_{\frac{1}{l^{\infty}}})/S_{\frac{1}{l^{\infty}}}),P)
:=s(A^{\bul}_{l^n}((X_{\frac{1}{l^{\infty}}},D_{\frac{1}{l^{\infty}}})/S_{\frac{1}{l^{\infty}}}),P)
\in {\rm C}^{\rm b}{\rm F}(\os{\circ}{X}_{\rm et},{\mab Z}/l^n). 
\tag{9.16.7}\label{eqn:bast}
\end{equation*}  
Then we obtain the following bifiltered complex 
\begin{equation*} 
(A^{\bul}_{l^n}((X_{\frac{1}{l^{\infty}}},D_{\frac{1}{l^{\infty}}})/S_{\frac{1}{l^{\infty}}}),P^{D_{\frac{1}{l^{\infty}}}},P)
:=(A^{\bul}_{l^n}((X_{\frac{1}{l^{\infty}}},D_{\frac{1}{l^{\infty}}})/S_{\frac{1}{l^{\infty}}}),P^{D_{\frac{1}{l^{\infty}}}},P)\in 
{\rm C}^{\rm b}{\rm F}^2(\os{\circ}{X}_{\rm et},{\mab Z}/l^n). 
\tag{9.16.8}\label{eqn:bdast}
\end{equation*} 
Obviously we have the projective system 
$\{(A^{\bul}_{l^n}((X_{\frac{1}{l^{\infty}}},D_{\frac{1}{l^{\infty}}})/S_{\frac{1}{l^{\infty}}}),
P^{D_{\frac{1}{l^{\infty}}}},P)\}_{n=1}^{\infty}$.


\begin{prop}\label{prop:aqislws} 
The morphism 
\begin{equation*} 
\theta \otimes 1 \col K^{\bul}_{l^n}((X_{\frac{1}{l^{\infty}}},D_{\frac{1}{l^{\infty}}})/S_{\frac{1}{l^{\infty}}})=
K^{\bul}_{l^n}(X_{\frac{1}{l^{\infty}}}/S_{\frac{1}{l^{\infty}}})\otimes_{{\mab Z}/l^n}
M^{\bul}_{l^n}(\os{\circ}{D}/\os{\circ}{S}) \lo  
{\rm MF}_{l^n}(T-1)(1)[1]\otimes_{{\mab Z}/l^n}
M^{\bul}_{l^n}(\os{\circ}{D}/\os{\circ}{S}) 
\end{equation*} 
induces a filtered quasi-isomorphism 
\begin{equation*} 
\theta \otimes 1\col 
(K^{\bul}_{l^n}((X_{\frac{1}{l^{\infty}}},D_{\frac{1}{l^{\infty}}})/S_{\frac{1}{l^{\infty}}}),P^{D_{\frac{1}{l^{\infty}}}}) 
\lo (A^{\bul}_{l^n}((X_{\frac{1}{l^{\infty}}},D_{\frac{1}{l^{\infty}}})/S_{\frac{1}{l^{\infty}}}),P^{D_{\frac{1}{l^{\infty}}}}) 
\tag{9.17.1}\label{eqn:fqais} 
\end{equation*} 
in ${\rm C}^+{\rm F}(\os{\circ}{X}_{\rm et},{\mab Z}/l^n)$.  
\end{prop}
\begin{proof} 
Because $(M^{\bul}_{l^n}(\os{\circ}{D}/\os{\circ}{S}),Q)$ is 
a filtered flat complex of ${\mab Z}/l^n$-modules in $\os{\circ}{X}_{\rm et}$, 
(\ref{prop:aqislws}) follows from (\ref{prop:qislws}). 
\end{proof}

The  bifiltered complex (\ref{eqn:bdast}) defines an object 
of ${\rm D}^{\rm b}{\rm F}^2(\os{\circ}{X}_{\rm et},{\mab Z}/l^n)$,
By abuse of notation, we denote it 
by 
$(A^{\bul}_{l^n}((X_{\frac{1}{l^{\infty}}},D_{\frac{1}{l^{\infty}}})
/S_{\frac{1}{l^{\infty}}}),P^{D_{\frac{1}{l^{\infty}}}},P)$.

\begin{prop}\label{prop:era} 
There exists an isomorphism 
\begin{equation*}
R(\eps_{(X,D)}\pi_{(X_{{\frac{1}{l^{\infty}}}},D_{\frac{1}{l^{\infty}}})})_*({\mab Z}/l^n)
\os{\sim}{\lo} A^{\bul}_{l^n}((X_{\frac{1}{l^{\infty}}},D_{\frac{1}{l^{\infty}}})/S_{\frac{1}{l^{\infty}}})
\tag{9.18.1}\label{eqn:epaxds}
\end{equation*}
in $D^{\rm b}(\os{\circ}{X}_{\rm et},{\mab Z}/l^n)$.
\end{prop} 
\begin{proof} 
(\ref{prop:era}) immediately follows from 
(\ref{prop:lnvc}) and (\ref{prop:aqislws}). 
\end{proof} 

The following is a main result in this section. 

\begin{theo}\label{theo:wdbst}  
$(1)$ The bifiltered complex 
$(A^{\bul}_{l^n}((X_{\frac{1}{l^{\infty}}},D_{\frac{1}{l^{\infty}}})/S_{\frac{1}{l^{\infty}}}),P^{D_{\frac{1}{l^{\infty}}}},P)\in 
{\rm D}^{\rm b}{\rm F}^2(\os{\circ}{X}_{\rm et},{\mab Z}/l^n)$ 
is independent of the choices of $I^{\bul}_{X,l^n}$ and 
$(M^{\bul}_{l^n}(\os{\circ}{D}/\os{\circ}{S}),Q)$.  
This bifiltered complex is also independent of the choice of $T$.  
This is an object of ${\rm D}^{\rm b}{\rm F}^2_{\rm ctf}(\os{\circ}{X}_{\rm et},{\mab Z}/l^n)$.  
\par 
$(2)$ 
\begin{align*} 
(A^{\bul}_{l^{n+1}}((X_{\frac{1}{l^{\infty}}},D_{\frac{1}{l^{\infty}}})/S_{\frac{1}{l^{\infty}}}),P^{D_{\frac{1}{l^{\infty}}}},P)
\otimes^L_{{\mab Z}/l^{n+1}}{\mab Z}/l^{n}
=
(A^{\bul}_{l^{n}}((X_{\frac{1}{l^{\infty}}},D_{\frac{1}{l^{\infty}}})/S_{\frac{1}{l^{\infty}}}),P^{D_{\frac{1}{l^{\infty}}}},P)
\end{align*} 
in ${\rm D}^{\rm b}{\rm F}^2_{\rm ctf}(\os{\circ}{X}_{\rm et},{\mab Z}/l^n)$. 
Here $\otimes^L_{{\mab Z}/l^{n+1}}$ is the derived tensor product of bounded below 
bifiltered complexes defined in {\rm (\ref{eqn:drtsr})}. 
\end{theo}
\begin{proof} 
Let $({\cal T},{\cal A})$ be a ringed topos. Let 
$$L^0\col 
\{{\cal A}\text{-modules}\}\lo  \{\text{flat }{\cal A}\text{-modules}\}$$
be the the sheafification of the presheaf 
$$U \lom \text{free }\Gam(U,{\cal A})\text{-module 
with basis }\Gam(U,E)\setminus \{0\}.$$ 
First we fix a generator $T$ of ${\mab Z}_l(1)$. 
Let $I'^{\bul}_{X,l^n}$ be another injective resolution 
of ${\mab Z}/l^n$ in $X_{\rm ket}$ as in the proof of (\ref{prop:ni}). 
Let ${\rm MF}'_{l^n}(T-1)$ be the complex in the proof of (\ref{prop:ni}). 
Let $(M'^{\bul}_{l^n}(\os{\circ}{D}/\os{\circ}{S}),Q')$ be an analogous filtered flat complex to  
$(M^{\bul}_{l^n}(\os{\circ}{D}/\os{\circ}{S}),Q)$. 
Using $L^0$, we see that 
there exists an analogous filtered flat complex 
$(M''^{\bul}_{l^n}(\os{\circ}{D}/\os{\circ}{S})),Q'')$ to  
$(M^{\bul}_{l^n}(\os{\circ}{D}/\os{\circ}{S})),Q)$ 
such that there exist the following filtered quasi-isomorphisms: 
\begin{align*} 
(M^{\bul}_{l^n}(\os{\circ}{D}/\os{\circ}{S}),Q) \os{\sim}{\longleftarrow} 
(M''^{\bul}_{l^n}(\os{\circ}{D}/\os{\circ}{S}),Q'')\os{\sim}{\lo}  
(M'^{\bul}_{l^n}(\os{\circ}{D}/\os{\circ}{S}),Q').
\tag{9.19.1}\label{ali:anbxqds} 
\end{align*} 
Hence we may assume that there exists a filtered quasi-isomorphism
\begin{align*} 
(M^{\bul}_{l^n}(\os{\circ}{D}/\os{\circ}{S}),Q) 
\os{\sim}{\lo}  
(M'^{\bul}_{l^n}(\os{\circ}{D}/\os{\circ}{S}),Q').
\tag{9.19.2}\label{ali:anxqfids} 
\end{align*} 
Because the bifiltrations 
$P$ and $P^{D_{\frac{1}{l^{\infty}}}}$ are biregular, 
we have only to prove that  
${\rm gr}_{k'}^{P}{\rm gr}_{k}^{P^{D_{\frac{1}{l^{\infty}}}}}
A^{\bul}_{l^n}((X_{\frac{1}{l^{\infty}}},D_{\frac{1}{l^{\infty}}})/S_{\frac{1}{l^{\infty}}})$ is independent of the choice of 
$I^{\bul}_{X,l^n}$ and $(M^{\bul}_{l^n}(\os{\circ}{D}/\os{\circ}{S}),Q)$  
(cf.~the proof of (\ref{prop:pq})). 
Since 
${\rm gr}^Q_{k}M^{\bul}_{l^n}(\os{\circ}{D}/\os{\circ}{S})$ 
is a complex of flat ${\mab Z}/l^n$-modules, 
we have the following formula
\begin{align*} 
{\rm gr}_{k'}^{P}{\rm gr}_{k}^{P^{D_{\frac{1}{l^{\infty}}}}}
A^{\bul}_{l^n}((X_{\frac{1}{l^{\infty}}},D_{\frac{1}{l^{\infty}}})/S_{\frac{1}{l^{\infty}}})& ={\rm gr}_{k'}^{P}
(A^{\bul}_{l^n}(X_{\frac{1}{l^{\infty}}}/S_{\frac{1}{l^{\infty}}})\otimes_{{\mab Z}/l^n}
{\rm gr}^Q_{k}M^{\bul}_{l^n}(\os{\circ}{D}/\os{\circ}{S})) 
\tag{9.19.3}\label{ali:anxqds}\\ 
{} & =\bigoplus_{j\geq \max\{-k',0\}}
{\rm gr}_{2j+k'+1-k}^{\tau}{\rm MF}_{l^n}(T-1)(j)\{j\}[1]
\otimes_{{\mab Z}/l^n} 
{\rm gr}^Q_{k}M^{\bul}_{l^n}(\os{\circ}{D}/\os{\circ}{S})
\end{align*} 
by \cite[(1.2.4) (2)]{nh2}.  
Since 
${\rm gr}^Q_{k}M^{\bul}_{l^n}(\os{\circ}{D}/\os{\circ}{S})$ 
is a complex of flat ${\mab Z}/l^n$-modules, 
the following morphism 
\begin{align*} 
& {\rm gr}_{2j+k'+1-k}^{\tau}{\rm MF}_{l^n}(T-1)(j)\{j\}[1]
\otimes_{{\mab Z}/l^n} 
{\rm gr}^Q_{k}M^{\bul}_{l^n}(\os{\circ}{D}/\os{\circ}{S}) \\
& \lo 
({\rm gr}_{2j+k'+1-k}^{\tau}{\rm MF}'_{l^n}(T-1)(j)\{j\}[1]
\otimes_{{\mab Z}/l^n} 
{\rm gr}^Q_{k}M^{\bul}_{l^n}(\os{\circ}{D}/\os{\circ}{S}) 
\end{align*}
is a quasi-isomorphism by (\ref{eqn:tkx}). 
Furthermore, 
the following morphism 
\begin{align*} 
& {\rm gr}_{2j+k'+1-k}^{\tau}{\rm MF}'_{l^n}(T-1)(j)\{j\}[1]
\otimes_{{\mab Z}/l^n} 
{\rm gr}^{Q}_kM^{\bul}_{l^n}(\os{\circ}{D}/\os{\circ}{S}) \\
& \lo 
({\rm gr}_{2j+k'+1-k}^{\tau}{\rm MF}'_{l^n}(T-1)(j)\{j\}[1]
\otimes_{{\mab Z}/l^n} 
{\rm gr}^{Q'}_{k}M'^{\bul}_{l^n}(\os{\circ}{D}/\os{\circ}{S}) 
\end{align*}
is a quasi-isomorphism by (\ref{eqn:grkqa}). 
Now we have proved the independence of the choices of $I^{\bul}_{X,l^n}$ and 
$(M^{\bul}_{l^n}(\os{\circ}{D}/\os{\circ}{S}),Q)$. 
\par 
The independence of the choice of $T$ follows from (\ref{eqn:tkix}). 
\par 
By (\ref{prop:cfd}), (\ref{eqn:tkx}), (\ref{eqn:grf}) and  (\ref{ali:anxqds}),  
we see that the claim about 
the constructivility and the finite tor-dimension holds. 
\par 
(2): Obviously we have a natural morphism 
\begin{align*} 
(A^{\bul}_{l^{n+1}}((X_{\frac{1}{l^{\infty}}},D_{\frac{1}{l^{\infty}}})/S_{\frac{1}{l^{\infty}}}),P^{D_{\frac{1}{l^{\infty}}}},P)
\otimes^L_{{\mab Z}/l^{n+1}}{\mab Z}/l^{n}
\lo
(A^{\bul}_{l^{n}}((X_{\frac{1}{l^{\infty}}},D_{\frac{1}{l^{\infty}}})/S_{\frac{1}{l^{\infty}}}),P^{D_{\frac{1}{l^{\infty}}}},P).
\end{align*}
Let $R^{\bul}$ be a flat resolution of 
${\mab Z}/l^{n}$ of ${\mab Z}/l^{n+1}$-modules. 
Then 
$(A^{\bul}_{l^{n+1}}((X_{\frac{1}{l^{\infty}}},D_{\frac{1}{l^{\infty}}})/S_{\frac{1}{l^{\infty}}}),P^{D_{\frac{1}{l^{\infty}}}},P)
\otimes^L_{{\mab Z}/l^{n+1}}{\mab Z}/l^{n}=
(A^{\bul}_{l^{n+1}}((X_{\frac{1}{l^{\infty}}},D_{\frac{1}{l^{\infty}}})/S_{\frac{1}{l^{\infty}}}),
P^{D_{\frac{1}{l^{\infty}}}},P)\otimes_{{\mab Z}/l^{n+1}}R^{\bul}$. 
As in (1), it suffices to prove that 
\begin{align*} 
&{\rm gr}_{k'}^{P}{\rm gr}_{k}^{P^{D_{\frac{1}{l^{\infty}}}}}
\{(A^{\bul}_{l^{n+1}}((X_{\frac{1}{l^{\infty}}},D_{\frac{1}{l^{\infty}}})/S_{\frac{1}{l^{\infty}}}),
P^{D_{\frac{1}{l^{\infty}}}},P)\otimes_{{\mab Z}/l^{n+1}}R^{\bul}\}\\
& ={\rm gr}_{k'}^{P}{\rm gr}_{k}^{P^{D_{\frac{1}{l^{\infty}}}}}
(A^{\bul}_{l^{n}}((X_{\frac{1}{l^{\infty}}},D_{\frac{1}{l^{\infty}}})/S_{\frac{1}{l^{\infty}}}),P^{D_{\frac{1}{l^{\infty}}}},P).
\end{align*}  
This is obtained as follows: 
\begin{align*} 
&{\rm gr}_{k'}^{P}{\rm gr}_{k}^{P^{D_{\frac{1}{l^{\infty}}}}}
\{(A^{\bul}_{l^{n+1}}((X_{\frac{1}{l^{\infty}}},D_{\frac{1}{l^{\infty}}})/S_{\frac{1}{l^{\infty}}}),
P^{D_{\frac{1}{l^{\infty}}}},P)\otimes_{{\mab Z}/l^{n+1}}R^{\bul}\}\\
& = {\rm gr}_{k'}^{P}{\rm gr}_{k}^{P^{D_{\frac{1}{l^{\infty}}}}}
(A^{\bul}_{l^{n+1}}((X_{\frac{1}{l^{\infty}}},D_{\frac{1}{l^{\infty}}})/S_{\frac{1}{l^{\infty}}}),
P^{D_{\frac{1}{l^{\infty}}}},P)\otimes_{{\mab Z}/l^{n+1}}R^{\bul}\\
& =
\bigoplus_{j\geq \max\{-k',0\}}
{\rm gr}_{2j+k'+1-k}^{\tau}{\rm MF}_{l^{n+1}}(T-1)(j)\{j\}[1]
\otimes_{{\mab Z}/l^{n+1}} 
{\rm gr}^Q_{k}M^{\bul}_{l^{n+1}}(\os{\circ}{D}/\os{\circ}{S})
\otimes_{{\mab Z}/l^{n+1}}R^{\bul}\\
&\os{\sim}{\lo} 
\bigoplus_{j\geq \max\{-k',0\}}
{\rm gr}_{2j+k'+1-k}^{\tau}{\rm MF}_{l^{n}}(T-1)(j)\{j\}[1]
\otimes_{{\mab Z}/l^{n}} 
{\rm gr}^Q_{k}M^{\bul}_{l^{n}}(\os{\circ}{D}/\os{\circ}{S}). 
\end{align*} 
Here, to obtain the last quasi-isomorphism, we have used  (\ref{eqn:tkx}) and (\ref{eqn:grf}). 
\end{proof} 

\begin{defi}\label{defi:wtfd}   
We call the bifiltered complex 
$$(A^{\bul}_{l^{\infty}}((X_{\frac{1}{l^{\infty}}},D_{\frac{1}{l^{\infty}}})/S_{\frac{1}{l^{\infty}}}),
P^{D_{\frac{1}{l^{\infty}}}},P) :=\{(A^{\bul}_{l^n}((X_{\frac{1}{l^{\infty}}},D_{\frac{1}{l^{\infty}}})/S_{\frac{1}{l^{\infty}}}),
P^{D_{\frac{1}{l^{\infty}}}},P)\}_{n=1}^{\infty} 
\in {\rm D}^{\rm b}{\rm F}^2_{\rm ctf}(\os{\circ}{X}_{\rm et},{\mab Z}_l)$$ 
the $l$-{\it adic bifiltered El Zein-Steenbrink-Zucker complex} of $(X,D)$ 
over $S$. We call  $P^{D_{\frac{1}{l^{\infty}}}}$ and $P$ 
{\it the weight filtration with respect} to $D$ 
and 
{\it the weight filtration} on 
$A^{\bul}_{l^{\infty}}((X_{\frac{1}{l^{\infty}}},D_{\frac{1}{l^{\infty}}})/S_{\frac{1}{l^{\infty}}})$, respectively. 
(We do not use a traditional notation 
$W$ in this paper because 
we have to use the symbol $W$ 
for the Witt ring in other papers, 
e.~g.,~\cite{msemi}, \cite{ndw}.)
\end{defi}

We also have the following objects 
\begin{align*}
{\rm gr}_k^{P^{\os{\circ}{D}}}A^{\bul}_{l^{\infty}}((X_{\frac{1}{l^{\infty}}},D_{\frac{1}{l^{\infty}}})/S_{\frac{1}{l^{\infty}}})
:=\{{\rm gr}_k^{P^{\os{\circ}{D}}}A^{\bul}_{l^n}((X_{\frac{1}{l^{\infty}}},D_{\frac{1}{l^{\infty}}})/S_{\frac{1}{l^{\infty}}})\}_{n=1}^{\infty} 
\in D^{\rm b}_{\rm ctf}(\os{\circ}{X}_{\rm et},{\mab Z}_l)
\end{align*} 
and 
\begin{align*}
{\rm gr}_k^{P}A^{\bul}_{l^{\infty}}((X_{\frac{1}{l^{\infty}}},D_{\frac{1}{l^{\infty}}})/S_{\frac{1}{l^{\infty}}})
:=\{{\rm gr}_k^{P}A^{\bul}_{l^{\infty}}((X_{\frac{1}{l^{\infty}}},D_{\frac{1}{l^{\infty}}})/S_{\frac{1}{l^{\infty}}})\}_{n=1}^{\infty} 
\in D^{\rm b}_{\rm ctf}(\os{\circ}{X}_{\rm et},{\mab Z}_l). 
\end{align*} 


\begin{rema}\label{rema:cde}  
Assume that $\os{\circ}{S}$ is regular of dimension $\leq 1$. 
Let $f\col X\lo S$ be the structural morphism. 
Let ${\rm D}^{\rm bb}_{\rm ctf}(\os{\circ}{X}_{\rm et},{\mab Z}/l^n)$ 
and 
${\rm D}^{\rm bb}{\rm F}^2_{\rm ctf}(\os{\circ}{X}_{\rm et},{\mab Z}/l^n)$ 
be the derived category of bounded biregular complexes of 
${\mab Z}/l^n$-modules 
in $\os{\circ}{X}_{\rm et}$ and 
the derived category of bounded biregular bifiltered complexes of 
${\mab Z}/l^n$-modules. 
Let ${\rm D}^{\rm bb}{\rm F}^2_{\rm ctf}(\os{\circ}{X}_{\rm et},{\mab Z}_l)$ 
be the projective 2-limit of 
${\rm D}^{\rm bb}{\rm F}^2_{\rm ctf}(\os{\circ}{X}_{\rm et},{\mab Z}/l^n)$.  
As remarked in \cite[(1.1.2) c)]{dw2},  
${\rm D}^{\rm bb}_{\rm ctf}(\os{\circ}{X}_{\rm et},{\mab Z}/l^n)$ 
is stable under 
$R\os{\circ}{f}_*$, $\os{\circ}{f}{}^*$, $R\os{\circ}{f}_!$, $R\os{\circ}{f}{}^!$, 
$\otimes^L$ and ${\rm RHom}^{\bul}$ and 
commutes with reduction mod $l^n{\mab Z}$. 
Hence so is for 
${\rm D}^{\rm bb}{\rm F}^2_{\rm ctf}(\os{\circ}{X}_{\rm et},{\mab Z}_l)$.  
Because $(A^{\bul}_{l^{\infty}}((X_{\frac{1}{l^{\infty}}},D_{\frac{1}{l^{\infty}}})/S_{\frac{1}{l^{\infty}}}),P^{D_{\frac{1}{l^{\infty}}}},P)\in 
{\rm D}^{\rm bb}{\rm F}^2_{\rm ctf}(\os{\circ}{X}_{\rm et},{\mab Z}_l)$, 
we can define 
$$R\os{\circ}{f}_*((A^{\bul}_{l^{\infty}}((X_{\frac{1}{l^{\infty}}},D_{\frac{1}{l^{\infty}}})/S_{\frac{1}{l^{\infty}}}),
P^{D_{\frac{1}{l^{\infty}}}},P))\in 
{\rm D}^{\rm bb}{\rm F}^2_{\rm ctf}(\os{\circ}{S}_{\rm et},{\mab Z}_l).$$  
\end{rema}

\section{$l$-adic weight spectral sequences}\label{sec:lwss}
In this section we give generalizations of  
(\ref{eqn:lidd}) and  (\ref{eqn:lintdd}).

\begin{prop}\label{prop:zxbc}  
Let $\os{\circ}{Z}$ be a closed SNC scheme of 
$\os{\circ}{X}$ over $\os{\circ}{S}$. 
Set $Z:=X\times_{\os{\circ}{X}}\os{\circ}{Z}$. 
Assume that $Z$ is an SNCL scheme over $S$. 
Let $\iota \col Z \os{\sus}{\lo} X$ and 
$\iota_{\frac{1}{l^{\infty}}} \col Z_{\frac{1}{l^{\infty}}} \os{\sus}{\lo} X_{\frac{1}{l^{\infty}}}$
be the natural closed immersions. 
Let $I^{\bul}_{Z,l^n}$ be an injective resolution of 
${\mab Z}/l^n$ in $Z_{\rm ket}$. 
Set 
$K^{\bul}_{l^n}(Z_{\frac{1}{l^{\infty}}}/S_{\frac{1}{l^{\infty}}}):=\eps_{Z*}\pi_{Z_{\frac{1}{l^{\infty}}}*}
\pi_{Z_{\frac{1}{l^{\infty}}}}^*(I^{\bul}_{Z,l^n})\in C^+(\os{\circ}{Z}_{\rm et},{\mab Z}/l^n)$.  
Then the natural morphism 
$\os{\circ}{\iota}{}^{*}(K^{\bul}_{l^n}(X_{\frac{1}{l^{\infty}}}/S_{\frac{1}{l^{\infty}}})) 
\lo K^{\bul}_{l^n}(Z_{\frac{1}{l^{\infty}}}/S_{\frac{1}{l^{\infty}}})$ 
induced by a natural morphism 
$\iota^{*}(I^{\bul}_{X,l^n}) \lo I^{\bul}_{Z,l^n}$ 
is a quasi-isomorphism fitting into the following commutative diagram$:$
\begin{equation*} 
\begin{CD}
\os{\circ}{\iota}{}^{*}(K^{\bul}_{l^{n+1}}(X_{\frac{1}{l^{\infty}}}/S_{\frac{1}{l^{\infty}}})) @>>> 
K^{\bul}_{l^{n+1}}(Z_{\frac{1}{l^{\infty}}}/S_{\frac{1}{l^{\infty}}})\\
@VVV @VVV\\
\os{\circ}{\iota}{}^{*}(K^{\bul}_{l^n}(X_{\frac{1}{l^{\infty}}}/S_{\frac{1}{l^{\infty}}})) @>>> K^{\bul}_{l^n}(Z_{\frac{1}{l^{\infty}}}/S_{\frac{1}{l^{\infty}}}). 
\end{CD}
\end{equation*} 
\end{prop} 
\begin{proof} 
Consider the following two cartesian diagrams:  
\begin{equation*} 
\begin{CD} 
Z_{\frac{1}{l^{m}}} @>{\subset}>> X_{\frac{1}{l^{m}}} \\ 
@V{\pi_{Z_{\frac{1}{l^{m}}}}}VV @VV{\pi_{X_{\frac{1}{l^{m}}}}}V \\ 
Z @>{\subset}>> X \\ 
@V{\eps_Z}VV @VV{\eps_X}V \\ 
\os{\circ}{Z} @>{\subset}>> \os{\circ}{X}. 
\end{CD} 
\end{equation*} 
In $D^{\rm b}_{\rm ctf}(\os{\circ}{X}_{\rm et},{\mab Z}/l^n)$,  
$(\eps_X\pi_{X_{\frac{1}{l^{m}}}})_*\pi^*_{X_{\frac{1}{l^{m}}}}(I^{\bul}_{X,l^n})$ is equal to  
$R(\eps_X\pi_{X_{\frac{1}{l^{m}}}})_*({\mab Z}/l^n)$.  
Now (\ref{prop:zxbc}) follows from the quasi-isomorphisms (\ref{eqn:eij}), (\ref{eqn:epxj}) and 
the log proper base change theorem of 
Nakayama (\cite[(5.1)]{nale}).  
(It is easy to check the condition on the charts of the log structures of 
$X_{\frac{1}{l^{m}}}$, $X$ and $Z$ in [loc.~cit.] is satisfied.)
\end{proof}

\begin{prop}\label{prop:zxa}  
Let the notations be as in {\rm (\ref{prop:zxbc})}.   
Then 
$\os{\circ}{\iota}{}^{*}((A^{\bul}_{l^{\infty}}(X_{\frac{1}{l^{\infty}}}/S_{\frac{1}{l^{\infty}}}),P))
=(A^{\bul}_{l^{\infty}}(Z_{\frac{1}{l^{\infty}}}/S_{\frac{1}{l^{\infty}}}),P)$ 
in $D^{\rm b}_{\rm ctf}(\os{\circ}{Z}_{\rm et},{\mab Z}_l)$.  
\end{prop} 
\begin{proof} 
(\ref{prop:zxa}) 
follows from (\ref{prop:zxbc}) and 
the following commutative diagram:  
\begin{equation*} 
\begin{CD} 
\os{\circ}{\iota}{}^{*}(K^{\bul}_{l^n}(X_{\frac{1}{l^{\infty}}}/S_{\frac{1}{l^{\infty}}})) 
@>{T-1}>> \os{\circ}{\iota}{}^{*}(K^{\bul}_{l^n}(X_{\frac{1}{l^{\infty}}}/S_{\frac{1}{l^{\infty}}})) \\ 
@V{}VV @VV{}V \\ 
K^{\bul}_{l^n}(Z_{\frac{1}{l^{\infty}}}/S_{\frac{1}{l^{\infty}}})@>{T-1}>> 
K^{\bul}_{l^n}(Z_{\frac{1}{l^{\infty}}}/S_{\frac{1}{l^{\infty}}}). 
\end{CD} 
\end{equation*}   
\end{proof}

\begin{lemm}\label{lemm:grc} 
Let $k$ be a nonnegative integer. 
Then the following hold$:$
\par 
$(1)$ There exists an isomorphism 
\begin{equation*} 
{\rm gr}_k^{P^{D_{\frac{1}{l^{\infty}}}}}
A^{\bul}_{l^{\infty}}((X_{\frac{1}{l^{\infty}}},D_{\frac{1}{l^{\infty}}})
/S_{\frac{1}{l^{\infty}}})
\os{\sim}{\lo}c^{(k)}_{*}(A^{\bul}_{l^{\infty}}(D^{(k)}_{\frac{1}{l^{\infty}}}
/S_{\frac{1}{l^{\infty}}})(-k)\otimes_{\mab Z}
\varpi^{(k)}_{\rm et}(\os{\circ}{D}/\os{\circ}{S}))\{-k\} 
\tag{10.3.1}\label{eqn:grpdx}
\end{equation*} 
in $D^{\rm b}_{\rm ctf}(\os{\circ}{X}_{\rm et},{\mab Z}_l)$. 
\par 
$(2)$ There exists an isomorphism 
\begin{align*} 
{\rm gr}_k^{P}A^{\bul}_{l^{\infty}}((X_{\frac{1}{l^{\infty}}},D_{\frac{1}{l^{\infty}}})
/S_{\frac{1}{l^{\infty}}})&
\os{\sim}{\lo} 
\{\bigoplus^k_{k'=-\infty}
\bigoplus_{j\geq \max\{-k',0\}}
a^{(2j+k'),(k-k')}_{*}(({\mab Z}/l^n)_{\os{\circ}{X}{}^{(2j+k')}
\cap \os{\circ}{D}{}^{(k-k')}} \tag{10.3.2}\label{eqn:axd}\\ 
{} & \quad \otimes_{\mab Z}
\vp^{(2j+k'),(k-k')}_{\rm et}
((\os{\circ}{X},\os{\circ}{D})/\os{\circ}{S}))(-j-k)[-2j-k]\}_n
\end{align*}  
in $D^{\rm b}_{\rm ctf}(\os{\circ}{X}_{\rm et},{\mab Z}_l)$. 
\par 
$(3)$ There exists an isomorphism 
\begin{align*} 
& {\rm gr}_{k}^P{\rm gr}_{k'}^{P^{D_{\frac{1}{l^{\infty}}}}}
A^{\bul}_{l^{\infty}}((X_{\frac{1}{l^{\infty}}},D_{\frac{1}{l^{\infty}}})/S_{\frac{1}{l^{\infty}}})
= 
{\rm gr}_{k'}^{P^{D_{\frac{1}{l^{\infty}}}}}{\rm gr}_{k}^P
A^{\bul}_{l^{\infty}}((X_{\frac{1}{l^{\infty}}},D_{\frac{1}{l^{\infty}}})/S_{\frac{1}{l^{\infty}}})
\os{\sim}{\lo} \tag{10.3.3}\label{eqn:ana}\\ 
&\{\bigoplus_{j\geq \max\{-(k-k'),0\}}
a^{(2j+k-k'),(k')}_{*}(({\mab Z}/l^n)_{\os{\circ}{X}{}^{(2j+k-k')}
\cap \os{\circ}{D}{}^{(k')}} {}  \otimes_{\mab Z}
\vp^{(2j+k-k'),(k')}_{\rm et}
((\os{\circ}{X},\os{\circ}{D})/\os{\circ}{S}))\\
& (-j-k)[-2j-k]\}_n
\end{align*}  
in $D^{\rm b}_{\rm ctf}(\os{\circ}{X}_{\rm et},{\mab Z}_l)$. 
\end{lemm} 
\begin{proof} 
(1): Since $(M^q_{l^n}(\os{\circ}{D}/\os{\circ}{S}),Q)$ is 
a filtered flat ${\mab Z}/l^n$-module in $\os{\circ}{X}_{\rm et}$,  
${\rm gr}^Q_k(M^q_{l^n}(\os{\circ}{D}/\os{\circ}{S}))$ is 
a flat ${\mab Z}/l^n$-module in $\os{\circ}{X}_{\rm et}$. 
Hence we have the following equalities:  
\begin{align*} 
{\rm gr}_k^{P^{D_{\frac{1}{l^{\infty}}}}}
A^{\bul}_{l^n}((X_{\frac{1}{l^{\infty}}},D_{\frac{1}{l^{\infty}}})/S_{\frac{1}{l^{\infty}}}) 
{} & = s(A^{\bul}_{l^n}(X_{\frac{1}{l^{\infty}}}/S_{\frac{1}{l^{\infty}}})\otimes_{{\mab Z}/l^n}
{\rm gr}_k^QM^{\bul}_{l^n}(\os{\circ}{D}/\os{\circ}{S})) 
\tag{10.3.4}\label{ali:kpds}\\
{} & = 
A^{\bul}_{l^n}(X_{\frac{1}{l^{\infty}}}/S_{\frac{1}{l^{\infty}}})\otimes_{\mab Z}
c^{(k)}_{*}(\varpi^{(k)}_{\rm et}(\os{\circ}{D}/\os{\circ}{S}))(-k)\{-k\} \\  
{} & = 
c^{(k)}_{*}(c^{(k)*}
(A^{\bul}_{l^n}(X_{\frac{1}{l^{\infty}}}/S_{\frac{1}{l^{\infty}}})(-k))\otimes_{\mab Z}
\varpi^{(k)}_{\rm et}(\os{\circ}{D}/\os{\circ}{S}))\{-k\} \\ 
{} & = 
c^{(k)}_{*}(A^{\bul}_{l^n}(D^{(k)}_{\frac{1}{l^{\infty}}}/S_{\frac{1}{l^{\infty}}})(-k)\otimes_{\mab Z}
\varpi^{(k)}_{\rm et}(\os{\circ}{D}/\os{\circ}{S}))\{-k\}.  
\end{align*} 
Here, to obtain the second equality (resp.~the last equality), 
we have used (\ref{eqn:tktk1}) (resp.~(\ref{prop:zxa})).   
The compatibility with respect to $n$ is obvious. 
We complete the proof of (1). 
\par 
(2): We have the following equalities:  
\begin{align*} 
& {\rm gr}_k^{P}A^{\bul}_{l^n}((X_{\frac{1}{l^{\infty}}},D_{\frac{1}{l^{\infty}}})/S_{\frac{1}{l^{\infty}}}) 
\tag{10.3.5}\label{ali:anxds}\\
& =  \bigoplus^k_{k'=-\infty}
\bigoplus_{j\geq \max\{-k',0\}}
({\rm gr}_{2j+k'+1}^{\tau}A^{\bul j}_{l^n}(X_{\frac{1}{l^{\infty}}}/S_{\frac{1}{l^{\infty}}})[-j],-d)
\otimes_{{\mab Z}/l^n} 
{\rm gr}^Q_{k-k'}M^{\bul}_{l^n}(\os{\circ}{D}/\os{\circ}{S}) \\ 
{} & = \bigoplus^k_{k'=-\infty}
\bigoplus_{j\geq \max\{-k',0\}}
({\mab Z}/l^n)(-j-k')\otimes_{{\mab Z}}
a^{(2j+k')}_{*}
(\varpi^{(2j+k')}_{\rm et}(\os{\circ}{X}/\os{\circ}{S})) \\
{} & \quad [-2j-k'] \otimes_{{\mab Z}/l^n}
({\mab Z}/l^n)(-(k-k'))\otimes_{{\mab Z}}
c^{(k-k')}_{*}
(\varpi^{(k-k')}_{\rm et}(\os{\circ}{D}/\os{\circ}{S})))[-(k-k')] \\
{} & = \bigoplus^k_{k'=-\infty}
\bigoplus_{j\geq \max\{-k',0\}}
({\mab Z}/l^n)(-j-k)\otimes_{{\mab Z}}
a^{(2j+k')}_{*}(\varpi^{(2j+k')}_{\rm et}(\os{\circ}{X}/\os{\circ}{S}))\\
{} & \quad \otimes_{{\mab Z}}
c^{(k-k')}_{*}(\varpi^{(k-k')}_{\rm et}(\os{\circ}{D}/\os{\circ}{S}))[-2j-k]. 
\end{align*} 
Here, to obtain the first equality and the second one,  
we have used \cite[(1.2.4) (2)]{nh2} and the formula (\ref{eqn:grf}), respectively.  
Hence we obtain (\ref{eqn:axd}) by (\ref{eqn:etcab}). 
\par 
(3): The proof of (3) is similar to that of (2). 
\end{proof} 


Because 
$(A^{\bul}_{l^n}((X_{\frac{1}{l^{\infty}}},D_{\frac{1}{l^{\infty}}})
/S_{\frac{1}{l^{\infty}}}),P^{D_{\frac{1}{l^{\infty}}}},P) 
\in {\rm D}^{\rm b}{\rm F}^2_{\rm ctf}(\os{\circ}{X}_{\rm et},{\mab Z}/l^n)$, 
we can consider 
$$R\os{\circ}{f}_*(A^{\bul}_{l^n}((X_{\frac{1}{l^{\infty}}},
D_{\frac{1}{l^{\infty}}})/S_{\frac{1}{l^{\infty}}}),P^{D_{\frac{1}{l^{\infty}}}},P)\in 
{\rm D}^{\rm b}{\rm F}^2(\os{\circ}{S}_{\rm et},{\mab Z}/l^n).$$

\begin{coro}\label{coro:low} 
Let $f_{\os{\circ}{X}{}^{(k)}\cap 
\os{\circ}{D}{}^{(k')}/\os{\circ}{S}} \col 
\os{\circ}{X}{}^{(k)}\cap 
\os{\circ}{D}{}^{(k')} \lo \os{\circ}{S}$ $(k,k'\in {\mab N})$ 
be the structural morphism. 
Assume that $R^{q}f_{\os{\circ}{X}{}^{(k)}\cap 
\os{\circ}{D}{}^{(k')}/\os{\circ}{S}*}({\mab Z}/l^n
\otimes_{\mab Z}\varpi^{(k),(k')}_{\rm et}
((\os{\circ}{X},\os{\circ}{D})/\os{\circ}{S}))$ has finite stalks. 
Let $f_{(X_{\frac{1}{l^{\infty}}},D_{\frac{1}{l^{\infty}}})/S_{\frac{1}{l^{\infty}}}} \col 
(X_{\frac{1}{l^{\infty}}},D_{\frac{1}{l^{\infty}}}) \lo \os{\circ}{S}$ and 
$f_{D^{(k)}_{\frac{1}{l^{\infty}}}/S_{\frac{1}{l^{\infty}}}} \col D^{(k)}_{\frac{1}{l^{\infty}}} \lo \os{\circ}{S}$ 
be the structural morphisms.  
Then the following hold$:$ 
\par 
$(1)$ 
There exists the following spectral sequence$:$  
\begin{equation*} 
E_1^{-k,q+k}=R^{q-k}f_{D^{(k)}_{\frac{1}{l^{\infty}}}/S_{\frac{1}{l^{\infty}}}*}
({\mab Z}_l\otimes_{\mab Z}
\pi^{*}_{D^{(k)}_{\frac{1}{l^{\infty}}}}\eps^{*}_{D^{(k)}}
(\vp^{(k)}_{\rm et}(\os{\circ}{D}/\os{\circ}{S})))(-k) 
\Lo R^qf_{(X_{\frac{1}{l^{\infty}}},D_{\frac{1}{l^{\infty}}})/S_{\frac{1}{l^{\infty}}}*}({\mab Z}_l). 
\tag{10.4.1}\label{eqn:lehkd}  
\end{equation*} 
Here $R^{q-k}f_{D^{(k)}_{\frac{1}{l^{\infty}}}/S_{\frac{1}{l^{\infty}}}*}
({\mab Z}_l\otimes_{\mab Z}
\pi^{*}_{D^{(k)}_{\frac{1}{l^{\infty}}}}\eps^{*}_{D^{(k)}}
(\vp^{(k)}_{\rm et}(\os{\circ}{D}/\os{\circ}{S})))(-k) :=
\vpl_nR^{q-k}f_{D^{(k)}_{\frac{1}{l^{\infty}}}/S_{\frac{1}{l^{\infty}}}*}
({\mab Z}/l^n\otimes_{\mab Z}
\pi^{*}_{D^{(k)}_{\frac{1}{l^{\infty}}}}\eps^{*}_{D^{(k)}}
(\vp^{(k)}_{\rm et}(\os{\circ}{D}/\os{\circ}{S})))(-k)$ 
and 
$R^qf_{(X_{\frac{1}{l^{\infty}}},D_{\frac{1}{l^{\infty}}})/S_{\frac{1}{l^{\infty}}}*}({\mab Z}_l):=
\vpl_n R^qf_{(X_{\frac{1}{l^{\infty}}},D_{\frac{1}{l^{\infty}}})/S_{\frac{1}{l^{\infty}}}*}({\mab Z}/l^n)$. 
\par 
$(2)$ 
There exists the following spectral sequence$:$  
\begin{align*} 
E_1^{-k,q+k} & =
\bigoplus^k_{k'=-\infty}
\bigoplus_{j\geq \max\{-k',0\}} 
R^{q-2j-k}f_{\os{\circ}{X}{}^{(2j+k')}\cap 
\os{\circ}{D}{}^{(k-k')}/\os{\circ}{S}*}({\mab Z}_l
\otimes_{\mab Z} 
\tag{10.4.2}\label{eqn:lledd} \\ 
{} & \varpi^{(2j+k'),(k-k')}_{\rm et}
((\os{\circ}{X},\os{\circ}{D})/\os{\circ}{S}))(-j-k)\\
{} & \Lo R^qf_{(X_{\frac{1}{l^{\infty}}},D_{\frac{1}{l^{\infty}}})/S_{\frac{1}{l^{\infty}}}*}({\mab Z}_l). 
\end{align*} 
\par 
\par 
$(3)$ 
There exists the following spectral sequence for $k'\in {\mab N}:$  
\begin{align*} 
E_1^{-k,q+k} & =
\bigoplus_{j\geq \max\{-k,0\}} 
R^{q-2j-k}f_{\os{\circ}{X}{}^{(2j+k)}\cap 
\os{\circ}{D}{}^{(k')}/\os{\circ}{S}*}({\mab Z}_l
\otimes_{\mab Z} \varpi^{(2j+k),(k')}_{\rm et}
((\os{\circ}{X},\os{\circ}{D})/\os{\circ}{S}))(-j-k)
\tag{10.4.3}\label{eqn:lleddd} \\ 
{} & \Lo R^qf_{D^{(k')}_{\frac{1}{l^{\infty}}}/S_{\frac{1}{l^{\infty}}}*}({\mab Z}_l). 
\end{align*} 
\end{coro} 
\begin{proof}(First we prove (2).)
(2): 
We obtain the following identification of cohomologies: 
\begin{align*} 
& R^q\os{\circ}{f}_*({\rm gr}_k^{P}A^{\bul j}_{l^n}((X_{\frac{1}{l^{\infty}}},D_{\frac{1}{l^{\infty}}})
/S_{\frac{1}{l^{\infty}}})) 
\tag{10.4.4}\label{eqn:ncid}\\
& \os{R^q\os{\circ}{f}_*((\ref{eqn:axd}))}{{\os{\sim}{\lo}}} 
R^q\os{\circ}{f}_*(\bigoplus^k_{k'=-\infty}
\bigoplus_{j\geq \max\{-k',0\}}
a^{(2j+k'),(k-k')}_{*}(({\mab Z}/l^n)_{\os{\circ}{X}{}^{(2j+k')}
\cap \os{\circ}{D}{}^{(k-k')}})(-j-k) \\ 
{} & \quad 
\otimes_{\mab Z}
a^{(2j+k')}_{*}(\varpi^{(2j+k')}_{\rm et}(\os{\circ}{X}/\os{\circ}{S}))
\otimes_{\mab Z}
c^{(k-k')}_{*}(\varpi^{(k-k')}_{\rm et}(\os{\circ}{D}/\os{\circ}{S}))[-2j-k])\\ 
{} & =R^{q-2j-k}\os{\circ}{f}_*(\bigoplus^k_{k'=-\infty}
\bigoplus_{j\geq \max\{-k',0\}}
a^{(2j+k'),(k-k')}_*(({\mab Z}/l^n)_{\os{\circ}{X}{}^{(2j+k')}
\cap \os{\circ}{D}{}^{(k-k')}}\\ 
{} & \quad 
\otimes_{\mab Z}
\vp^{(2j+k'),(k-k')}_{\rm et}((\os{\circ}{X},\os{\circ}{D})/\os{\circ}{S}))(-j-k). 
\end{align*} 
Hence the projective system of $E_1$-terms of (\ref{eqn:lledd}) 
with respect to $n$ satisfies the Mittag-Leffler condition by the assumption of the finite stalks. 
Taking the projective limit of (\ref{eqn:ncid}) 
with respect to $n$,  
we obtain the spectral sequence (\ref{eqn:lledd}).    
\par 
(1): By (2),  
$R^{q}f_{D^{(k')}_{\frac{1}{l^{\infty}}}/S_{\frac{1}{l^{\infty}}}*}
({\mab Z}_l\otimes_{\mab Z}
\pi^{*}_{D^{(k')}_{\frac{1}{l^{\infty}}}}\eps^{*}_{D^{(k')}}
(\vp^{(k')}_{\rm et}(\os{\circ}{D}/\os{\circ}{S})))$ has finite stalks. 
We immediately obtain (\ref{eqn:lehkd}) 
by taking the projective limit of 
the spectral sequence 
obtained by (\ref{lemm:grc}) (1). 
\par 
(3): The proof of (3) is the same as that of (2). 
\end{proof}

\begin{defi}
We call the spectral sequences (\ref{eqn:lehkd}) and (\ref{eqn:lledd}) 
the {\it weight spectral sequence} of   
$R^qf_{(X_{\frac{1}{l^{\infty}}},D_{\frac{1}{l^{\infty}}})/S_{\frac{1}{l^{\infty}}}*}({\mab Z}_l)$ 
{\it relative to} $D_{\frac{1}{l^{\infty}}}$ 
and the {\it weight spectral sequence} of   
$R^qf_{(X_{\frac{1}{l^{\infty}}},D_{\frac{1}{l^{\infty}}})/S_{\frac{1}{l^{\infty}}}*}({\mab Z}_l)$, respectively. 
Set 
{\footnotesize{\begin{align*} 
P^{D_{\frac{1}{l^{\infty}}}}_{q+k}
R^qf_{(X_{\frac{1}{l^{\infty}}},D_{\frac{1}{l^{\infty}}})/S_{\frac{1}{l^{\infty}}}*}({\mab Z}_l)
&:={\rm Im}(R^qf_*(P^{D_{\frac{1}{l^{\infty}}}}_kA^{\bul}_{l^{\infty}}((X_{\frac{1}{l^{\infty}}},
D_{\frac{1}{l^{\infty}}})/S_{\frac{1}{l^{\infty}}}))
\lo R^qf_{*}(A^{\bul}_{l^{\infty}}((X_{\frac{1}{l^{\infty}}},D_{\frac{1}{l^{\infty}}})/S_{\frac{1}{l^{\infty}}})))\tag{10.5.1}\label{ali:rqpd}\\
&  \simeq {\rm Im}(R^qf_*(P^{D_{\frac{1}{l^{\infty}}}}_kA^{\bul}_{l^{\infty}}((X_{\frac{1}{l^{\infty}}},D_{\frac{1}{l^{\infty}}})
/S_{\frac{1}{l^{\infty}}}))
\lo R^qf_{(X_{\frac{1}{l^{\infty}}},D_{\frac{1}{l^{\infty}}})/S_{\frac{1}{l^{\infty}}}*}({\mab Z}_l))
\end{align*}}}
and 
{\footnotesize{\begin{align*} 
P_{q+k}R^qf_{(X_{\frac{1}{l^{\infty}}},D_{\frac{1}{l^{\infty}}})/S_{\frac{1}{l^{\infty}}}*}({\mab Z}_l)
&:={\rm Im}(R^q\os{\circ}{f}_*(P_kA^{\bul}_{l^{\infty}}((X_{\frac{1}{l^{\infty}}},D_{\frac{1}{l^{\infty}}})
/S_{\frac{1}{l^{\infty}}}))
\lo R^q\os{\circ}{f}_{*}
A^{\bul}_{l^{\infty}}((X_{\frac{1}{l^{\infty}}},D_{\frac{1}{l^{\infty}}})/S_{\frac{1}{l^{\infty}}}))\tag{10.5.2}\label{ali:rpaxd}\\
&\simeq 
{\rm Im}(R^q\os{\circ}{f}_*(P_kA^{\bul}_{l^{\infty}}((X_{\frac{1}{l^{\infty}}},D_{\frac{1}{l^{\infty}}})/S))
\lo R^qf_{(X_{\frac{1}{l^{\infty}}},D_{\frac{1}{l^{\infty}}})/S\frac{1}{l^{\infty}}*}({\mab Z}_l)).
\end{align*}}}
We call $P^{D_{\frac{1}{l^{\infty}}}}$ and $P$ the {\it weight filtration}  
on $R^qf_{(X_{\frac{1}{l^{\infty}}},D_{\frac{1}{l^{\infty}}})/S_{\frac{1}{l^{\infty}}}*}({\mab Z}_l)$
{\it relative to} $D_{\frac{1}{l^{\infty}}}$ and the {\it weight filtration} on 
$R^qf_{(X_{\frac{1}{l^{\infty}}},D_{\frac{1}{l^{\infty}}})/S_{\frac{1}{l^{\infty}}}*}({\mab Z}_l)$, 
respectively. 
\end{defi} 

\begin{exem}\label{exem:rd}
Assume that the relative dimension of $\os{\circ}{X}$ over $\os{\circ}{S}$ is of pure dimension $1$. 
Then the $E_1$-terms of (\ref{eqn:lledd}) with $d_1^{\bul \bul}$ are as follows: 
\begin{equation*} 
\footnotesize{\begin{CD}
\oplus_{i+j=1, i,j\geq 0}
R^0f_{\os{\circ}{X}{}^{(i)}\cap \os{\circ}{D}{}^{(j)}/\os{\circ}{S}*}({\mab Z}_l\otimes_{\mab Z}
 \varpi^{(i),(j)}_{\rm et}
((\os{\circ}{X},\os{\circ}{D})/\os{\circ}{S}))(-1)
@>>> R^2f_{\os{\circ}{X}{}^{(0)}/\os{\circ}{S}*}({\mab Z}_l) \\
@. R^1f_{\os{\circ}{X}{}^{(0)}/\os{\circ}{S}*}({\mab Z}_l)@. \\
@. R^0f_{\os{\circ}{X}{}^{(0)}/\os{\circ}{S}*}({\mab Z}_l) @>>> 
R^0f_{\os{\circ}{X}{}^{(1)}/\os{\circ}{S}*}({\mab Z}_l). 
\end{CD}}
\end{equation*} 
\end{exem}

\par 
Let $s$ be the log point of a field $\kap$.  
Let $\kap_{\rm sep}$ be a separable closure of $\kap$ and 
set $\ol{s}:=s\otimes_{\kap}{\kap}_{\rm sep}$. 
Let $(Y,E)$ be an SNCL over $s$ with a relative SNCD $E$ on $Y/s$. 
Set $(\ol{Y},\ol{E}):=(Y,E)\otimes_{\kap}{\kap}_{\rm sep}$. 
The Galois group ${\rm Gal}({\kap}_{\rm sep}/\kap)$ acts on 
$(\ol{Y},\ol{E})$ and hence on $H^q((\ol{Y}_{\frac{1}{l^{\infty}}},\ol{E}_{\frac{1}{l^{\infty}}}),{\mab Z}_l)$. 
It is easy to see that 
${\rm Gal}({\kap}_{\rm sep}/\kap)$ acts on 
$A^{\bul}_{l^{\infty}}((\ol{Y}_{\frac{1}{l^{\infty}}},\ol{E}_{\frac{1}{l^{\infty}}})/S))$. 
By (\ref{eqn:lehkd}), (\ref{eqn:lledd}) and (\ref{eqn:lleddd}) we have the following spectral sequences 
\begin{equation*} 
E_1^{-k,q+k}=H^{q-k}(\ol{E}^{(k)}_{\frac{1}{l^{\infty}}},{\mab Z}_l\otimes_{\mab Z}
\pi^{*}_{\ol{E}{}^{(k)}_{\frac{1}{l^{\infty}}}}\eps^{*}_{\ol{E}{}^{(k)}}
(\vp^{(k)}_{\rm et}(\os{\circ}{E}/\os{\circ}{s})))(-k) 
\Lo H^q((\ol{Y}_{\frac{1}{l^{\infty}}},\ol{E}_{\frac{1}{l^{\infty}}}),{\mab Z}_l), 
\tag{10.6.1}\label{eqn:lhksd}  
\end{equation*}  
\begin{align*} 
E_1^{-k,q+k} & =
\bigoplus^k_{k'=-\infty}
\bigoplus_{j\geq \max\{-k',0\}} H^{q-2j-k}((\os{\circ}{\ol{Y}}{}^{(2j+k')}\cap 
\os{\circ}{\ol{E}}{}^{(k-k')}),{\mab Z}_l
\otimes_{\mab Z} 
\tag{10.6.2}\label{eqn:llecdd} \\ 
{} & \varpi^{(2j+k'),(k-k')}_{\rm et}
((\os{\circ}{Y},\os{\circ}{E})/\os{\circ}{s}))(-j-k)\\
{} & \Lo H^q((\ol{Y}_{\frac{1}{l^{\infty}}},\ol{E}_{\frac{1}{l^{\infty}}}),{\mab Z}_l), 
\end{align*} 
\begin{align*} 
E_1^{-k,q+k} & =
\bigoplus_{j\geq \max\{-k,0\}} 
H^{q-2j-k}(\os{\circ}{\ol{Y}}{}^{(2j+k)}\cap \os{\circ}{\ol{E}}{}^{(k')},{\mab Z}_l
\otimes_{\mab Z} \varpi^{(2j+k),(k')}_{\rm et}
((\os{\circ}{Y},\os{\circ}{E})/\os{\circ}{S}))(-j-k)
\tag{10.6.3}\label{eqn:llecddd} \\ 
{} & \Lo H^q(\ol{E}^{(k')}_{\frac{1}{l^{\infty}}},{\mab Z}_l). 
\end{align*} 
These are spectral sequences of ${\rm Gal}({\kap}_{\rm sep}/\kap)$-modules.

\section{Contravariant functoriality}\label{sec:confu}
In this section we give the generalization of the contravariant functoriality of 
(weight) spectral sequences stated in the Introduction. 
In fact, we show the contravariant functoriality of 
$(A^{\bul}_{l^{\infty}}((X_{\frac{1}{l^{\infty}}},D_{\frac{1}{l^{\infty}}})/S_{\frac{1}{l^{\infty}}}),
P^{D_{\frac{1}{l^{\infty}}}},P)$. 
\par 
Let $v\col S'\lo S$ be a morphism of families of log points. 
Assume that $\os{\circ}{S}$ is a log scheme over 
${\rm Spec}({\mab Z}[l^{-1},\zeta_{l^\infty}~\vert~\zeta_{l^\infty}\in \mu_{l^{\infty}}])$. 
Consider $S'$ as a log scheme over 
${\rm Spec}({\mab Z}[l^{-1},\zeta_{l^\infty}~\vert~\zeta_{l^\infty}\in \mu_{l^{\infty}}])$ 
by the composite morphism 
$$S'\lo S\lo {\rm Spec}({\mab Z}[l^{-1},\zeta_{l^\infty}~\vert~\zeta_{l^\infty}\in \mu_{l^{\infty}}]).$$ 
Assume that the log structures of $S'$ and $S$ are constant and that 
the degree function $\deg v \col \os{\circ}{S}\lo {\mab Z}_{\geq 1}$ for $S'\lo S$ 
{\rm (\cite[(1.1.41)]{nb}, \cite[(2.1.15)]{nhir})}  
is not divisible by $l$.

\begin{defi}\label{defi:amfl} 
Let $w \col E \lo F$ 
be a morphism of  $({\mab Z}/l^n)_{\os{\circ}{X}}$-module. 
Let $k$ be an integer. 
The $D$-{\it twist}(:=degree twist) of $w$ by $k$  with respect to $v\col S'\lo S$ 
$$w(-k) \col {\cal E}(-k;v)\lo {\cal F}(-k;v)$$  
is, by definition, the morphism  
$\deg(v)^kw \col {\cal E} \lo {\cal F}$.  
This definition is well-defined for the derived category 
$D^+(({\mab Z}/l^n)_{\os{\circ}{X}})$. 
\end{defi}


\begin{exem}
Let $F_X \col X\lo X$ be the absolute Frobenius endomorphism 
over the absolute Frobenius endomorphism $F_S\col S\lo S$. 
Then the $D$-twist of $F_S$ is the usual Tate twist.  
We denote $(-k;F_S)$ $(k\in {\mab Z})$ simply by $(-k)$ as usual. 
\end{exem}

By virtue of the adjunction formula of ${\rm RHom}^{\bul}$ of bifiltered complexes 
((\ref{theo:adj})), we can obtain the following without difficulty
(though the proof is not difficult, the precise formulation is not easy): 

\begin{theo}[{\bf Contravariant Functoriality}]\label{theo:func} 
Let $v \col S'\lo S$ be as above. 
Assume that the morphism $v$ fits into 
the following commutative diagram
\begin{equation*} 
\begin{CD}
S'@>{v}>> S\\
@VVV @VVV \\
{\rm Spec}({\mab Z}[l^{-1},\zeta_{l^{\infty}}])@>>> 
{\rm Spec}({\mab Z}[l^{-1},\zeta_{l^{\infty}}]), 
\end{CD}
\tag{11.3.1}\label{cd:zml}
\end{equation*} 
where the lower horizontal morphism is induced by an endomorphism $\sig$ 
of the commutative ring ${\mab Z}[l^{-1},\zeta_{l^{\infty}}]$ such 
that $\sig(\zeta_{l^{\infty}})=\zeta_{l^{\infty}}^{\deg v}$ $(\zeta_{l^{\infty}}\in \mu_{l^{\infty}})$. 
For a commutative diagram 
\begin{equation*} 
\begin{CD}
(Y,E) @>{g}>> (X,D) \\ 
@VVV @VVV \\ 
S' @>{v}>> S
\end{CD} 
\tag{11.3.2}\label{cd:ssf}
\end{equation*} 
over 
\begin{equation*} 
\begin{CD}
(\os{\circ}{Y},\os{\circ}{E}) @>{\os{\circ}{g}}>> 
(\os{\circ}{X},\os{\circ}{D}) \\
@VVV @VVV \\ 
\os{\circ}{S}{}' @>{v}>> \os{\circ}{S}, 
\end{CD} 
\end{equation*} 
where $(Y,E)\lo (X,D)$ 
is a morphism of SNCL schemes with relative SNCD's over 
$v \col S' \lo S$,  
there exists a morphism 
\begin{equation*}
g^* \col 
(A^{\bul}_{l^{\infty}}((X_{\frac{1}{l^{\infty}}},D_{\frac{1}{l^{\infty}}})/S_{\frac{1}{l^{\infty}}}),
P^{D_{\frac{1}{l^{\infty}}}},P) 
\lo R\os{\circ}{g}_*((A^{\bul}_{l^{\infty}}((Y_{\frac{1}{l^{\infty}}},E_{\frac{1}{l^{\infty}}})/S'),
P^{E_{\frac{1}{l^{\infty}}}},P))
\tag{11.3.3}\label{eqn:bfe}
\end{equation*}
of bifiltered complexes in 
${\rm D}^+{\rm F}^2(\os{\circ}{X}_{\rm et},{\mab Z}_l)$ 
fitting into the following commutative diagram 
\begin{equation*}
\begin{CD} 
A^{\bul}_{l^{\infty}}((X_{\frac{1}{l^{\infty}}},D_{\frac{1}{l^{\infty}}})/S_{\frac{1}{l^{\infty}}})
@>{g^*}>> 
R\os{\circ}{g}_*(A^{\bul}_{l^{\infty}}((Y_{\frac{1}{l^{\infty}}},E_{\frac{1}{l^{\infty}}})/S'_{\frac{1}{l^{\infty}}}))\\
@A{\simeq}AA @AA{\simeq}A \\
R(\eps_{(X,D)}
\pi_{(X_{\frac{1}{l^{\infty}}},D_{\frac{1}{l^{\infty}}}))})_*({\mab Z}_l)
@>{g^*}>> R\os{\circ}{g}_* R(\eps_{(Y,E)}
\pi_{(Y_{\frac{1}{l^{\infty}}},E_{\frac{1}{l^{\infty}}})})_*({\mab Z}_l)
\end{CD}
\tag{11.3.4}\label{cd:appzln}
\end{equation*}
in $D^{\rm b}(\os{\circ}{X}_{\rm et},{\mab Z}_l)$.
The morphism $g^*$ in {\rm (\ref{eqn:bfe})} satisfies the obvious transitive law for 
the composition of $g$'s. 
\end{theo} 
\begin{proof}  
(The proof is not straightforward; we have to be careful for the definition of $g^*$, 
which has not been considered in references.)
Because we ignore the contravariant action $g^*$ in the definition of 
$(A^{\bul}_{l^{\infty}}((X_{\frac{1}{l^{\infty}}},D_{\frac{1}{l^{\infty}}})/S_{\frac{1}{l^{\infty}}}),
P^D,P)$ in previous sections and 
because we ignore the Galois action in this proof, we have to give another definition of 
$(A^{\bul}_{l^{\infty}}((X_{\frac{1}{l^{\infty}}},D_{\frac{1}{l^{\infty}}})/S_{\frac{1}{l^{\infty}}}), P^D,P)$. 
Fix a generator $T$ of ${\mab Z}_l(1)$; $T$ acts on $X_{\frac{1}{l^{\infty}}}$ 
and $Y_{\frac{1}{l^{\infty}}}$ naturally. 
For $i\in {\mab Z}$ and $j\in {\mab N}$,  denote 
${\rm MF}_{l^n}(T-1)(j+1)/
\tau_j{\rm MF}_{l^n}(T-1)(j+1)$ by 
${\rm MF}_{l^n}(T-1)(j+1)/\tau_j$ and 
set 
\begin{equation*} 
A^{ij}_{l^n}(X_{\frac{1}{l^{\infty}}}/S_{\frac{1}{l^{\infty}}}):=({\rm MF}_{l^n}(T-1)(j+1)
/\tau_j)^{i+j+1}(j+1;v) \quad 
(i\in {\mab Z},j\in {\mab N})
\tag{11.3.5}\label{eqn:andtmf}
\end{equation*} 
and 
\begin{align*} 
&\{P_kA^{ij}_{l^n}(X_{\frac{1}{l^{\infty}}}/S_{\frac{1}{l^{\infty}}})\}_{n=1}^{\infty}\tag{11.3.6}\label{eqn:tdajef}\\
&:=\{((\tau_{2j+k+1}+\tau_j){\rm MF}_{l^n}(T-1)(j+1)/\tau_j)^{i+j+1}(j+1;v)\}_{n=1}^{\infty}.  
\end{align*} 
(Note that both $(j+1)$ and $(j+1;v)$ are necessary.)
Consider the same boundary morphisms as those in (\ref{cd:dlsgn}). 
Set $(A^{\bul}_{l^{\infty}}(X_{\frac{1}{l^{\infty}}}/S_{\frac{1}{l^{\infty}}}),P):=
s((A^{\bul \bul}_{l^{\infty}}(X_{\frac{1}{l^{\infty}}}/S_{\frac{1}{l^{\infty}}}),P))$. 
Let $I^{\bul}_{X,l^n}$ (resp.~$I^{\bul}_{Y,l^n}$) 
be an injective resolution 
of ${\mab Z}/l^n$ in $X_{\rm ket}$ (resp.~$Y_{\rm ket}$). 
Then we have a natural morphism 
\begin{equation*} 
g^* \col I^{\bul}_{X,l^n} \lo g_*(I^{\bul}_{Y,l^n})
\tag{11.3.7}\label{eqn:gjj}
\end{equation*} 
of complexes. 
Let  ${\rm MF}'_{l^n}(T-1)$ 
be the mapping fiber of the morphism 
$T-1 \col I^{\bul}_{Y_{\frac{1}{l^{\infty}}},n} \lo I^{\bul}_{Y_{\frac{1}{l^{\infty}}},n}$. 
The morphism (\ref{eqn:gjj}) induces a filtered morphism 
\begin{equation*} 
g^{\bul*} \col \os{\circ}{g}{}^{*}(({\rm MF}_{l^n}(T-1),\tau)) \lo 
({\rm MF}'_{l^n}(T-1),\tau).  
\end{equation*} 
Because the pull-back $v^*(T)$ of the action $T$ on $S_{l^{\infty}}$ by  
$v^*$ is equal to $\deg (v)T$ on $S'_{l^{\infty}}$, 
we define the action $g^*$ on ${\mab Z}_l(1)$ as $\deg (v)T$. 
For $i,j\in {\mab N}$,  
set 
\begin{align*} 
g^{*ij}:=(\deg v)^{-(j+1)}g^{*i+j+1} \col &
\os{\circ}{g}{}^{*}({\rm MF}_{l^n}(T-1)(j+1)/\tau_j)^{i+j+1}(j+1;v)\\
&\lo 
({\rm MF}'_{l^n}(T-1)(j+1)/\tau_j)^{i+j+1}(j+1;v). 
\end{align*} 
Then we obtain the following commutative diagram: 
\begin{equation*}
\begin{CD}
(\os{\circ}{g}{}^{*}({\rm MF}_{l^n}(T-1))(j+2)/\tau_j)^{i+j+2}(j+2;v)@>{g^{*i,j+1}}>>   
({\rm MF}'_{l^n}(T-1)(j+2)/\tau_j)^{i+j+2}(j+2;v) \\ 
@A{g^*(\theta)}AA @AA{\theta}A\\
(\os{\circ}{g}{}^{*}({\rm MF}_{l^n}(T-1))(j+1)/\tau_j)^{i+j+1}(j+1;v)
@>{g^{*ij}}>> ({\rm MF}'_{l^n}(T-1)(j+1)/\tau_j)^{i+j+1}(j+1;v). 
\end{CD}
\tag{11.3.8}\label{cd:tdbjef}
\end{equation*} 
It is easy to check that the following diagram 
\begin{equation*}
\begin{CD}
S'_{\frac{1}{l^n}}@>{T}>> S'_{\frac{1}{l^n}}\\ 
@V{v}VV @VV{v}V\\
S_{\frac{1}{l^n}}@>{T}>> S_{\frac{1}{l^n}}
\end{CD}
\end{equation*} 
is commutative by the action (\ref{ali:iin}) and 
the commutative diagram (\ref{cd:zml}). 
Hence the following diagram is commutative: 
\begin{equation*}
\begin{CD}
(\os{\circ}{g}{}^{*}({\rm MF}_{l^n}(T-1))(j+1)/\tau_j)^{i+j+1}(j+1;v)@>{-d}>>   
(\os{\circ}{g}{}^{*}({\rm MF}_{l^n}(T-1))(j+1)/\tau_j)^{i+j+2}(j+1;v) \\ 
@V{g^{*ij}}VV @VVg^{*i+1,j}V\\
({\rm MF}'_{l^n}(T-1)(j+1)/\tau_j)^{i+j+1}(j+1;v)
@>{-d}>> ({\rm MF}'_{l^n}(T-1)(j+1)/\tau_j)^{i+j+2}(j+1;v). 
\end{CD}
\tag{11.3.9}\label{cd:tdcef}
\end{equation*} 
Taking a filtered injective resolution of 
$({\rm MF}'_{l^n}(T-1),\tau)$ (\cite[(1.1.5)]{nh2}) and using 
(\ref{cd:tdbjef}),  (\ref{cd:tdcef}) and 
(\ref{eqn:tdajef}), 
we have a filtered morphism 
\begin{equation*} 
\os{\circ}{g}{}^*((A^{\bul}_{l^n}(X_{\frac{1}{l^{\infty}}}/S_{\frac{1}{l^{\infty}}}),P))  
\lo 
(A^{\bul}_{l^n}(Y_{\frac{1}{l^{\infty}}}/S'_{\frac{1}{l^{\infty}}}),P).  
\tag{11.3.10}\label{eqn:aaxdgye}
\end{equation*} 
\par 
Let $I^{\bul}_{\os{\circ}{D},l^n}$ (resp.~$I'{}^{\bul}_{\os{\circ}{E},n}$)
be an injective resolution of ${\mab Z}/l^n$ 
in $(\os{\circ}{X},\os{\circ}{D})_{\rm ket}$ 
(resp.~$(\os{\circ}{Y},\os{\circ}{E})_{\rm ket}$). 
Then we have a morphism 
\begin{equation*} 
g^{*}(I^{\bul}_{\os{\circ}{D},l^n}) \lo 
(I'{}^{\bul}_{\os{\circ}{E},n}) 
\tag{11.3.11}\label{eqn:gpii}
\end{equation*} 
of complexes. 
By (\ref{prop:bav}) 
(see also \cite[(2.16), (2.18), (2.20)]{nh2}), 
there exist bounded filtered flat resolutions of 
$(M_{l^n}(\os{\circ}{D}/\os{\circ}{S}),Q)$ 
and 
$(M'_{l^n}(\os{\circ}{E}/\os{\circ}{S'}),Q')$ 
of 
$(I^{\bul}_{\os{\circ}{D},l^n},\tau)$ and 
$(I^{\bul}_{\os{\circ}{E},n},\tau)$, respectively, 
fitting into the following commutative diagram: 
\begin{equation*} 
\begin{CD}
g^*(\os{\circ}{\eps}_{\os{\circ}{D}*}(I^{\bul}_{\os{\circ}{D},l^n}),\tau) 
@>>>
(\os{\circ}{\eps}_{\os{\circ}{E}*}(I^{\bul}_{\os{\circ}{E},n}),\tau) \\
@AAA @AAA \\
g^{*}((M_{l^n}(\os{\circ}{D}/\os{\circ}{S}),Q)) @>>> 
(M_{l^n}(\os{\circ}{E}/\os{\circ}{S'}),Q'). 
\end{CD}
\tag{11.3.12}\label{cd:gqds}
\end{equation*} 
The morphism (\ref{eqn:aaxdgye}) and 
the upper horizontal morphism in (\ref{cd:gqds}) 
induce filtered morphisms  
\begin{align*} 
L\os{\circ}{g}{}^{*}((A^{\bul}_{l^n}(X_{\frac{1}{l^{\infty}}}/S_{\frac{1}{l^{\infty}}}),P)
\otimes^L_{{\mab Z}/l^n} 
(\os{\circ}{\eps}_{\os{\circ}{D}*}(I^{\bul}_{\os{\circ}{D},l^n}),\tau)&= 
L\os{\circ}{g}{}^{*}(A^{\bul}_{l^n}(X_{\frac{1}{l^{\infty}}}/S_{\frac{1}{l^{\infty}}}),P)
\otimes^L_{{\mab Z}/l^n} 
L\os{\circ}{g}{}^{*}
(\os{\circ}{\eps}_{\os{\circ}{D}*}(I^{\bul}_{\os{\circ}{D},l^n}),
\tau) \\
&\lo 
(A^{\bul}_{l^n}(Y_{\frac{1}{l^{\infty}}}/S_{\frac{1}{l^{\infty}}}),P)\otimes^L_{{\mab Z}/l^n} 
(\os{\circ}{\eps}_{\os{\circ}{E}*}(I^{\bul}_{\os{\circ}{E},n}),\tau) 
\end{align*} 
and 
\begin{align*} 
L\os{\circ}{g}{}^{*}(A^{\bul}_{l^n}(X_{\frac{1}{l^{\infty}}}/S_{\frac{1}{l^{\infty}}})
\otimes^L_{{\mab Z}/l^n} 
(\os{\circ}{\eps}_{\os{\circ}{D}*}(I^{\bul}_{\os{\circ}{D},l^n}),\tau)) 
&=
L\os{\circ}{g}{}^{*}A^{\bul}_{l^n}(X_{\frac{1}{l^{\infty}}}/S_{\frac{1}{l^{\infty}}})
\otimes^L_{{\mab Z}/l^n} 
L\os{\circ}{g}{}^{*}
(\os{\circ}{\eps}_{\os{\circ}{D}*}(I^{\bul}_{\os{\circ}{D},l^n}),\tau) \\
&\lo A^{\bul}_{l^n}(Y_{\frac{1}{l^{\infty}}}/S_{\frac{1}{l^{\infty}}})
\otimes^L_{{\mab Z}/l^n} 
(\os{\circ}{\eps}_{\os{\circ}{E}*}(I^{\bul}_{\os{\circ}{E},n}),\tau). 
\end{align*} 
In fact, these are underlying morphisms of the following morphism
\begin{equation*} 
L\os{\circ}{g}{}^{*}((A^{\bul}_{l^n}(((X_{\frac{1}{l^{\infty}}},
D_{\frac{1}{l^{\infty}}})/S_{\frac{1}{l^{\infty}}}),P^{D_{\frac{1}{l^{\infty}}}},P))
\lo 
(A^{\bul}_{l^n}(((Y_{\frac{1}{l^{\infty}}},E_{\frac{1}{l^{\infty}}})/S_{\frac{1}{l^{\infty}}}),
P^{E_{\frac{1}{l^{\infty}}}},P). 
\end{equation*} 
By the adjunction formula (\ref{theo:adj}) and (\ref{eqn:h0derhm}),  
we obtain the following morphism 
\begin{equation*} 
(A^{\bul}_{l^n}(((X_{\frac{1}{l^{\infty}}},D_{\frac{1}{l^{\infty}}})/S_{\frac{1}{l^{\infty}}}),
P^{D_{\frac{1}{l^{\infty}}}},P)
\lo 
R\os{\circ}{g}_*((A^{\bul}_{l^n}((Y_{\frac{1}{l^{\infty}}},E_{\frac{1}{l^{\infty}}})
/S_{\frac{1}{l^{\infty}}}),P^{E_{\frac{1}{l^{\infty}}}},P))
\end{equation*} 
of bifiltered complexes. 
\end{proof}

Let us recall the following: 

\begin{prop}[{\rm {\bf \cite[(4.3)]{nh3}}}]\label{prop:mmoo} 
Let $g\col Y \lo Y'$ be a morphism of 
fs log schemes.  Set $Z:=Y$ or $Y'$. 
Assume that $M_{Z,z}/{\cal O}^*_{Z,z}\simeq {\mab N}^r$
for any point $z$ of $\os{\circ}{Z}$ 
and for some $r\in {\mab N}$ depending on $z$. 
Let $b^{(k)}_Z \col \os{\circ}{D}{}^{(k)}(M_Z) \lo \os{\circ}{Z}$ 
$(k\in {\mab Z})$ be the morphism of schemes defined in 
{\rm \cite[p.~34]{nh3}}. 
Assume that, for each point $y\in \os{\circ}{Y}$ 
and for each member $m$ of the minimal generators 
of $M_{Y,y}/{\cal O}^*_{Y,y}$, there exists 
a unique member $m'$ of the minimal generators of 
$M_{Y',\os{\circ}{g}(y)}/{\cal O}^*_{Y',\os{\circ}{g}(y)}$ 
such that $g^*(m')\in m^{{\mab Z}_{>0}}$. 
Then there exists a canonical morphism 
$\os{\circ}{g}{}^{(k)}\col \os{\circ}{D}{}^{(k)}(M_Y) \lo 
\os{\circ}{D}{}^{(k)}(M_{Y'})$ fitting 
into the following commutative diagram of schemes$:$
\begin{equation*} 
\begin{CD} 
\os{\circ}{D}{}^{(k)}(M_Y) @>{\os{\circ}{g}{}^{(k)}}>> 
\os{\circ}{D}{}^{(k)}(M_{Y'}) \\ 
@V{b^{(k)}_Y}VV @VV{b^{(k)}_{Y'}}V \\ 
\os{\circ}{Y} @>{\os{\circ}{g}}>> \os{\circ}{Y}{}'. 
\end{CD} 
\end{equation*} 
\end{prop} 

Let 
\begin{equation*} 
\begin{CD}
Y @>{g}>> X \\ 
@VVV @VVV \\ 
S' @>{v}>> S
\end{CD} 
\tag{11.4.1}\label{cd:ssdgf}
\end{equation*} 
be a commutative diagram of SNCL schemes such that $g$ satisfies the assumption in 
(\ref{prop:mmoo}). In this case, $\os{\circ}{D}{}^{(k)}(M_X)=\os{\circ}{X}{}^{(k)}$ and
 $\os{\circ}{D}{}^{(k)}(M_Y)=\os{\circ}{Y}{}^{(k)}$
and the second assumption in (\ref{prop:mmoo}) is equivalent to 
the following: 
For each smooth component $\os{\circ}{Y}_{\lam'}$ of $\os{\circ}{Y}/\os{\circ}{S}{}'$, 
there exists a smooth component $\os{\circ}{X}_{\lam}$ of $\os{\circ}{X}/\os{\circ}{S}$
such that 
$g(\os{\circ}{Y}_{\lam'})\subset \os{\circ}{X}_{\lam}$ (cf.~\cite[(2.3)]{stwsl}). 
Let $\Lam'$ and $\Lam$ be the sets of indices 
of the $\lam'$'s and the $\lam$'s, respectively. 
Then we obtain a function 
$\phi \col \Lam' \owns \lam' \lom \lam \in \Lam$.
There exist positive integers $e({\lam}')$'s  
$(\lam' \in \Lam')$ such that 
there exist local equations $x_{\lam'}=0$ and 
$x_{\phi(\lam')}=0$ of 
$\os{\circ}{Y}_{\lam}$ and $\os{\circ}{X}_{\phi(\lam)}$, 
respectively,  such that $g^*(x_{\phi(\lam')})=x^{e({\lam'})}_{\lam'}$.  
In \cite[(1.5.7)]{nb} we have proved the following (the proof is easy): 

\begin{prop}\label{prop:xxle} 
Let the assumptions and the notations be as above. 
Set $\Lam'(y):=
\{\lam' \in \Lam'~\vert~y \in \os{\circ}{Y}_{\lam'}\}$.  
Then $\deg(v)_y=e({\lam}')$ for $\lam' \in \Lam'(y)$.  
In particular, 
$e({\lam}')$'s are independent of
the choice of an element of $\Lam'(y)$.  
\end{prop}

\par
The following new definition is a generalization of the usual Tate twist and 
the $l$-adic analogue of the $D$-twist(degree-twist) defined in \cite[(1.5.10)]{nb} 
and \cite[(5.1.5)]{nhir}:

\par
Let $E$ and $D$ be horizontal SNCD's on $Y/S'$ and $X/S$, respectively.  
Assume that, 
for each smooth component $\os{\circ}{E}_{\mu'}$ of $\os{\circ}{E}/\os{\circ}{S}{}'$, 
there exists a smooth component $\os{\circ}{D}_{\mu}$ of $\os{\circ}{D}/\os{\circ}{S}$
such that 
$\os{\circ}{g}(\os{\circ}{E}_{\mu'})\subset \os{\circ}{D}_{\mu}$. 
Let $M'$ and $M$ be the sets of indices 
of the $\mu'$'s and the $\mu$'s, respectively. 
Then we obtain a function 
$\psi \col M' \owns \mu' \lom \mu \in M$.
There exist positive integers $e({\mu}')$'s  
$(\mu' \in M')$ such that 
there exist local equations $y_{\mu'}=0$ and 
$y'_{\psi(\mu')}=0$ of 
$\os{\circ}{D}_{\mu}$ and $\os{\circ}{E}_{\psi(\mu)}$, 
respectively,  such that $g^*(y'_{\psi(\mu)})=y^{e({\mu})}_{\mu}$.

\begin{defi}
\label{defi:wcel}
(1) We call 
$\{e(\mu)\}_{\mu \in M}\in {\mab Z}^M_{>0}$ the {\it multi-degree} of 
$g$ with respect to a decomposition 
$\Del:=\{D_{\mu}\}_{\mu \in M}$ 
and $\Del':=\{E_{\mu'}\}_{\mu'\in \psi(M)}$ of $D$ and $E$, respectively.   
We  denote it by ${\rm deg}_{\Del,\Del'}(g) \in {\mab Z}^M_{>0}$. 
\par
(2) 
Let $u \col {\cal E}=\bigoplus_{\{\mu_1,\ldots, \mu_k\}}E_{\mu_1,\ldots, \mu_k}\lo 
{\cal F}=\bigoplus_{\{\mu_1,\ldots, \mu_k\}}F_{\mu_1,\ldots, \mu_k}$ 
be a morphism of  
$({\mab Z}/l^n)_{\os{\circ}{D}{}^{(k)}}$-modules, where 
$E_{\mu_1,\ldots, \mu_k}$ and $F_{\mu_1,\ldots, \mu_k}$ are 
$({\mab Z}/l^n)_{\os{\circ}{D}_{\mu_1,\ldots, \mu_k}}$-modules. 
Let $k$ be a nonnegative integer. 
The $k$-{\it twist}\index{twist} 
$$u(-k) \col {\cal E}(-k;g;\Del,\Del') 
\lo {\cal F}(-k;g;\Del,\Del')$$  
of $u$ with respect to $v$, $\Del$ and $\Del'$ 
is, by definition, the direct sum of the morphisms  
$e_{\mu_1}\cdots e_{\mu_k}u 
\col \allowbreak {\cal E}_{\mu_1,\ldots, \mu_k} \allowbreak \lo {\cal F}_{\mu_1,\ldots, \mu_k}$'s.
\end{defi}

As a corollary of (\ref{theo:func}) we obtain the following:  

\begin{coro}\label{coro:fptl}
The following hold$:$ 
\par 
$(1)$ 
There exists the following spectral sequence$:$  
\begin{equation*} 
E_1^{-k,q+k}=R^{q-k}f_{D^{(k)}_{\frac{1}{l^{\infty}}}/S_{\frac{1}{l^{\infty}}}*}
({\mab Z}_l\otimes_{\mab Z}
\pi^{*}_{D^{(k)}_{\frac{1}{l^{\infty}}}}\eps^{*}_{D^{(k)}}
(\vp^{(k)}_{\rm et}(\os{\circ}{D}/\os{\circ}{S})))(-k;g,\Del,\Del') 
\Lo R^qf_{(X_{\frac{1}{l^{\infty}}},D_{\frac{1}{l^{\infty}}})/S_{\frac{1}{l^{\infty}}}*}({\mab Z}_l). 
\tag{11.7.1}\label{eqn:lethkd}  
\end{equation*}  
\par 
$(2)$ 
There exists the following spectral sequence$:$  
\begin{align*} 
E_1^{-k,q+k} & =
\bigoplus^k_{k'=-\infty}
\bigoplus_{j\geq \max\{-k',0\}} 
R^{q-2j-k}f_{\os{\circ}{X}{}^{(2j+k')}\cap 
\os{\circ}{D}{}^{(k-k')}/\os{\circ}{S}*}({\mab Z}_l
\otimes_{\mab Z} 
\tag{11.7.2}\label{eqn:lletdd} \\ 
{} & \varpi^{(2j+k'),(k-k')}_{\rm et}
((\os{\circ}{X},\os{\circ}{D})/\os{\circ}{S}))(-j-k';v)(-(k-k');g,\Del,\Del')\\
{} & \Lo R^qf_{(X_{\frac{1}{l^{\infty}}},D_{\frac{1}{l^{\infty}}})/S_{\frac{1}{l^{\infty}}}*}({\mab Z}_l). 
\end{align*} 
\par 
\par 
$(3)$ 
Then there exists the following spectral sequence for $k\in {\mab N}:$  
\begin{align*} 
E_1^{-k,q+k} & =
\bigoplus_{j\geq \max\{-k,0\}} 
R^{q-2j-k}f_{\os{\circ}{X}{}^{(2j+k)}\cap 
\os{\circ}{D}{}^{(k')}/\os{\circ}{S}*}({\mab Z}_l
\otimes_{\mab Z} \varpi^{(2j+k),(k')}_{\rm et}
((\os{\circ}{X},\os{\circ}{D})/\os{\circ}{S}))(-j-k;v)
\tag{11.7.3}\label{eqn:lledtdd} \\ 
{} & \Lo R^qf_{D^{(k')}_{\frac{1}{l^{\infty}}}/S_{\frac{1}{l^{\infty}}}*}({\mab Z}_l). 
\end{align*} 
\end{coro} 
\begin{proof} 
This corollary follows from the consideration of the action $g^*$ on 
$\bigwedge^k(M(D)^{\rm gp}/{\cal O}_X^*)$ and 
$\bigwedge^q(M_X^{\rm gp}/{\cal O}_X^*)$ in 
the isomorphisms (\ref{eqn:tktk1}) and (\ref{eqn:tkx}), respectively. 
\end{proof}

\section{Fundamental properties of the $l$-adic weight filtrations and 
the $l$-adic weight spectral sequences}\label{sec:fpl} 
In this section we prove several fundamental properties of 
the $l$-adic weight filtrations and 
the $l$-adic weight spectral sequences defined in the previous sections. 
\par 
Let the notations be as in the previous section. 
The following is the $l$-adic analogue of the bifiltered base change theorem in 
\cite[(8.2)]{nppf}:  

\begin{theo}[{\bf Base change theorem of $(A,P^D,P)$}]\label{theo:pwbcsbc} 
Let $u \col S' \lo S$ be a solid morphism of log schemes. 
Set $(X'_{\frac{1}{l^{m}}},D'_{\frac{1}{l^{m}}})
:=(X_{\frac{1}{l^{m}}},D_{\frac{1}{l^{m}}})\times_{S}S'$ $(m\in {\mab N}\cup \{\infty\})$. 
Let $f'\col (X',D')\lo S'$ 
be the structural morphism. 
Assume that $\os{\circ}{X}{}'$ is quasi-compact. 
Assume also that $\os{\circ}{u}{}^*\col \os{\circ}{X}_{\rm et}\lo \os{\circ}{X}{}'_{\rm et}$ is exact.  
Then 
\begin{align*} 
L\os{\circ}{u}{}^*R\os{\circ}{f}_*((A^{\bul}_{l^n}((X_{\frac{1}{l^{\infty}}},
D_{\frac{1}{l^{\infty}}})/S_{\frac{1}{l^{\infty}}}),P^{D_{\frac{1}{l^{\infty}}}},P))= 
R\os{\circ}{f}{}'_*((A^{\bul}_{l^n}((X'_{\frac{1}{l^{\infty}}},
D'_{\frac{1}{l^{\infty}}})/S'_{\frac{1}{l^{\infty}}}),P^{D'_{\frac{1}{l^{\infty}}}},P)). 
\tag{12.1.1}\label{ali:bbbc}
\end{align*}   
\end{theo} 
\begin{proof} 
Because $R\os{\circ}{f}_*((A^{\bul}_{l^n}((X_{\frac{1}{l^{\infty}}},
D_{\frac{1}{l^{\infty}}})/S_{\frac{1}{l^{\infty}}}),P^{D_{\frac{1}{l^{\infty}}}},P))$ is bounded above, 
the left hand side of (\ref{ali:bbbc}) is well-defined. 
It is easy to see that there exists a canonical morphism from 
the left hand side of (\ref{ali:bbbc}) to the right hand side of (\ref{ali:bbbc})
by using the adjunction formula (\ref{theo:adj}) (and (\ref{eqn:h0derhm}))
and the contravariant functoriality 
(\ref{theo:func}).  
Because the filtrations $P^{D_{\frac{1}{l^{\infty}}}}$ and $P$ are biregular, 
it suffices to prove that 
the canonical morphism 
\begin{align*} 
L\os{\circ}{u}{}^*R\os{\circ}{f}_*({\rm gr}_{k'}^P{\rm gr}_{k}^{P^{D_{\frac{1}{l^{\infty}}}}}
A^{\bul}_{l^n}((X_{\frac{1}{l^{\infty}}},
D_{\frac{1}{l^{\infty}}})/S_{\frac{1}{l^{\infty}}})){\lo} 
R\os{\circ}{f}{}'_*(({\rm gr}_{k'}^P{\rm gr}_{k}^{P^{D'_{\frac{1}{l^{\infty}}}}}
A^{\bul}_{l^n}((X'_{\frac{1}{l^{\infty}}},
D'_{\frac{1}{l^{\infty}}})/S_{\frac{1}{l^{\infty}}})))
\tag{12.1.2}\label{ali:bgbgbc}
\end{align*}   
is an isomorphism. 
This follows from (\ref{eqn:ana}) and 
the smooth base change theorem 
(\cite[XVI (1.2)]{sga4-3}) because  
$\os{\circ}{X}{}^{(2j+k')}\cap 
\os{\circ}{D}{}^{(k-k')}$ is smooth over $\os{\circ}{S}$ 
because ${\cal H}^q(L\os{\circ}{u}{}^*E^{\bul})
=\os{\circ}{u}{}^*{\cal H}^q(E^{\bul})$ for a complex of 
$({\mab Z}/l^n)_{\os{\circ}{S}}$-modules since $\os{\circ}{u}{}^*$ is exact. 
\end{proof}


The following is a generalization of \cite[(1.5)]{nd}:

\begin{prop}\label{prop:pwsbc} 
Let the notations be as in {\rm (\ref{theo:pwbcsbc})}.
Assume that $\os{\circ}{X}{}'$ is quasi-compact. 
$($We do not assume also that $\os{\circ}{u}{}^*$ is exact.$)$   
Let $f_{(X'_{\frac{1}{l^{m}}},D'_{\frac{1}{l^{m}}})/S'_{\frac{1}{l^{\infty}}}} 
\col (X'_{\frac{1}{l^{m}}},D_{\frac{1}{l^{m}}}')\lo \os{\circ}{S}{}'$ 
be the structural morphism. 
Then the spectral sequences {\rm (\ref{eqn:lethkd})} 
and {\rm (\ref{eqn:lletdd})} for 
$R^qf_{(X'_{\frac{1}{l^{\infty}}},D'_{\frac{1}{l^{\infty}}})/S'_{\frac{1}{l^{\infty}}}}({\mab Z}_l)$ 
are canonically isomorphic to the inverse images of  
{\rm (\ref{eqn:lethkd})} and {\rm (\ref{eqn:lletdd})} by $\os{\circ}{u}$, 
respectively. 
\end{prop} 
\begin{proof}
As in (\ref{theo:pwbcsbc}) we have morphisms 
from the inverse images of  
(\ref{eqn:lethkd}) and (\ref{eqn:lletdd}) by $\os{\circ}{u}$ 
to the spectral sequences  (\ref{eqn:lethkd})
and (\ref{eqn:lletdd}) for 
$R^qf_{(X'_{\frac{1}{l^{\infty}}},D'_{\frac{1}{l^{\infty}}})/S'_{\frac{1}{l^{\infty}}}}({\mab Z}_l)$, respectively.  
As for (\ref{eqn:lletdd}), the smooth base change theorem 
(\cite[XVI (1.2)]{sga4-3}) does the job because  
$\os{\circ}{X}{}^{(2j+k')}\cap 
\os{\circ}{D}{}^{(k-k')}$ is smooth over $\os{\circ}{S}$. 
As for (\ref{eqn:lehkd}), we have just proved that 
$g^*R^{q-k}f_{D^{(k)}_{\frac{1}{l^{\infty}}}/S*_{\frac{1}{l^{\infty}}}}
({\mab Z}_l) 
= 
R^{q-k}f_{D'{}^{(k)}_{\frac{1}{l^{\infty}}}/S_{\frac{1}{l^{\infty}}}'*}
({\mab Z}_l)$. 
Now the claim for (\ref{eqn:lethkd}) follows. 
\end{proof}

The following is a generalization of the main result of \cite{nd}:

\begin{prop}\label{prop:stc} 
Assume that $\os{\circ}{X}$ is proper over $\os{\circ}{S}$. 
Then the spectral sequence {\rm (\ref{eqn:lledd})} 
degenerates at $E_2$ modulo torsion. 
\end{prop} 
\begin{proof} 
We use a standard specialization argument in \cite[(2.3)]{nd}. 
Though we can give the analogous proof to the proof in \cite{nd}, 
we give a quite or slightly (?) different proof from it because our strategy of 
the proof below is valid in the $p$-adic case in \cite{nppf}. 
We do not use the spectral sequence (\ref{eqn:llecddd}) of 
${\rm Gal}({\kap}_{\rm sep}/\kap)$-modules. 
\par  
It suffices to prove that 
(\ref{eqn:lletdd}) degenerates at $E_2$. 
By (\ref{prop:pwsbc}) we may assume that  
$S$ is a log point 
$s:=({\rm Spec}(\kap),{\mab N}\oplus \kap^*)$. 
As in \cite[(2.3)]{nd} we may assume that 
there exist a finitely generated subring $A$ over ${\mab Z}$ 
of $\kap$ and a proper SNCL scheme ${\cal X}$ with a 
relative SNCD ${\cal D}$ 
on ${\cal X}/S':=({\rm Spec}(A),{\mab N}\oplus A^*\lo A)$ 
such that there exists the following cartesian diagram 
\begin{equation*} 
\begin{CD} 
(X,D) @>>> ({\cal X},{\cal D}) \\
@VVV @VVV \\ 
s @>>> S'. 
\end{CD} 
\tag{12.3.1}\label{cd:xdst}
\end{equation*} 
By (\ref{prop:pwsbc})  
we have only to prove that 
the spectral sequence (\ref{eqn:lletdd}) 
for $({\cal X},{\cal D})/S'$ degenerates at $E_2$ modulo torsion.  
Take a closed point $\os{\circ}{t}={\rm Spec}({\mab F}_q)$ of $\os{\circ}{S'}$. 
Let $t$ be the log point whose underlying scheme is 
$\os{\circ}{t}$ and whose log structure is 
the pull-back of the log structure of $S'$. 
Because $E^{-k,q+k}_1$ is smooth, 
we may assume that $S'=t$. 
Set $\ol{t}:=t\otimes_{{\mab F}_q}\ol{\mab F}_q$. 
Let $F \col \ol{t}\lo \ol{t}$ be the $q$-th power Frobenius endomorphism. 
Then the morphism $F$ fits into 
the following commutative diagram
\begin{equation*} 
\begin{CD}
\ol{t}@>{F}>> \ol{t}\\
@VVV @VVV \\
{\rm Spec}({\mab Z}[l^{-1},\zeta_{l^{\infty}}])@>>> {\rm Spec}({\mab Z}[l^{-1},\zeta_{l^{\infty}}]), 
\end{CD}
\tag{12.3.2}\label{cd:zmql}
\end{equation*} 
where the lower horizontal morphism is induced by an endomorphism $\sig$ 
of the commutative ring ${\mab Z}[l^{-1},\mu_{l^{\infty}}]$ such 
that $\sig(\zeta_{l^{\infty}})=\zeta_{l^{\infty}}^q$ $(\zeta_{l^{\infty}}\in \mu_{l^{\infty}})$. 
Hence we obtain the spectral sequence (\ref{eqn:lletdd}) for 
$({\cal X}_{\ol{t}},{\cal D}_{\ol{t}})$, which is equal to the following 
in this proof: 
\begin{align*} 
E_1^{-k,q+k} & =
\bigoplus^k_{k'=-\infty}
\bigoplus_{j\geq \max\{-k',0\}} 
H^{q-2j-k}(\os{\circ}{\cal X}{}^{(2j+k')}_{\os{\circ}{\ol{t}}}\cap 
\os{\circ}{\cal D}{}^{(k-k')}_{\os{\circ}{\ol{t}}},{\mab Z}_l
\otimes_{\mab Z} 
\tag{12.3.3}\label{eqn:lletdqd} \\ 
{} & \varpi^{(2j+k'),(k-k')}_{\rm et}
((\os{\circ}{\cal X}_{\os{\circ}{\ol{t}}},\os{\circ}{\cal D}_{\os{\circ}{\ol{t}}})/\os{\circ}{\ol{t}}))(-j-k) 
\Lo H^q(({\cal X}_{\ol{t},\frac{1}{l^{\infty}}},{\cal D}_{\ol{t},\frac{1}{l^{\infty}}}),{\mab Z}_l). 
\end{align*} 
The weight filtration on 
$H^q(({\cal X}_{\ol{t},\frac{1}{l^{\infty}}},{\cal D}_{\ol{t},\frac{1}{l^{\infty}}}),{\mab Q}_l)$ 
induced by the spectral sequence (\ref{eqn:lletdqd})  
is the same as the weight filtration defined by the pull-back of 
the  Frobenius endomorphism of 
$({\cal X}_{\ol{t},\frac{1}{l^{\infty}}},{\cal D}_{\ol{t},\frac{1}{l^{\infty}}})$ by  
the purity of weight of the Frobenius endomorphism 
for the $l$-adic cohomology 
of a proper smooth scheme (\cite[(1.6)]{dw1}).  
Because the edge morphism $d^{-k,q+k}_r$ 
commutes with the Frobenius endomorphism, 
we see that 
the spectral sequence (\ref{eqn:lletdqd}) degenerates at $E_2$ modulo torsion because 
$E_1^{-k,q+k}\otimes_{{\mab Z}_l}{\mab Q}_l$ is of pure weight $q$. 
This implies that (\ref{eqn:lledd})
degenerates at $E_2$ modulo torsion. 
\end{proof} 

Set 
{\footnotesize{\begin{align*} 
P^{D_{\frac{1}{l^{\infty}}}}_{q+k}R^qf_{(X_{\frac{1}{l^{\infty}}},D_{\frac{1}{l^{\infty}}})/S_{\frac{1}{l^{\infty}}}*}({\mab Z}/l^n)
&:={\rm Im}(R^q\os{\circ}{f}{}_*(P^{\os{\circ}{D}}_k
A^{\bul}_{l^n}((X_{\frac{1}{l^{\infty}}},D_{\frac{1}{l^{\infty}}})/S_{\frac{1}{l^{\infty}}}))
\lo R^q\os{\circ}{f}_{*}(A^{\bul}_{l^n}((X_{\frac{1}{l^{\infty}}},D_{\frac{1}{l^{\infty}}})/S_{\frac{1}{l^{\infty}}})))\\
&  \simeq {\rm Im}(R^q\os{\circ}{f}{}_*(P^{\os{\circ}{D}}_kA^{\bul}_{l^n}((X_{\frac{1}{l^{\infty}}},
D_{\frac{1}{l^{\infty}}})/S_{\frac{1}{l^{\infty}}}))
\lo R^qf_{(X_{\frac{1}{l^{\infty}}},D_{\frac{1}{l^{\infty}}})/S_{\frac{1}{l^{\infty}}}*}({\mab Z}/l^n))
\end{align*}}} 
and 
{\footnotesize{\begin{align*} 
P_{q+k}R^qf_{(X_{\frac{1}{l^{\infty}}},D_{\frac{1}{l^{\infty}}})/S_{\frac{1}{l^{\infty}}}*}({\mab Z}/l^n)
&:={\rm Im}(R^q\os{\circ}{f}_*(P_kA^{\bul}_{l^n}((X_{\frac{1}{l^{\infty}}},D_{\frac{1}{l^{\infty}}})/S))
\lo R^qf_{*}
A^{\bul}_{l^n}((X_{\frac{1}{l^{\infty}}},D_{\frac{1}{l^{\infty}}})/S_{\frac{1}{l^{\infty}}}))\\
&\simeq 
{\rm Im}(R^q\os{\circ}{f}_*(P_kA^{\bul}_{l^n}((X_{\frac{1}{l^{\infty}}},D_{\frac{1}{l^{\infty}}})/S_{\frac{1}{l^{\infty}}}))
\lo R^qf_{(X_{\frac{1}{l^{\infty}}},D_{\frac{1}{l^{\infty}}})/S_{\frac{1}{l^{\infty}}}*}({\mab Z}/l^n)).
\end{align*}}}

\begin{prop}\label{prop:sms}
Assume that $\os{\circ}{X}$ is proper over $\os{\circ}{S}$.  
Then the following hold$:$ 
\par 
$(1)$ 
The sheaf 
$P^{D_{\frac{1}{l^{\infty}}}}_kR^qf_{(X_{\frac{1}{l^{\infty}}},D_{\frac{1}{l^{\infty}}})
/S_{\frac{1}{l^{\infty}}}*}({\mab Z}/l^n)$ 
$(k\in {\mab Z}, n\in {\mab Z}_{\geq 1})$ is a smooth sheaf in $\os{\circ}{S}_{\rm et}$. 
\par 
$(2)$  The sheaf 
$P_kR^qf_{(X_{\frac{1}{l^{\infty}}},D_{\frac{1}{l^{\infty}}})/S_{\frac{1}{l^{\infty}}}*}({\mab Z}/l^n)$ 
$(k\in {\mab Z}, n\in {\mab Z}_{\geq 1})$ is a smooth sheaf in $\os{\circ}{S}_{\rm et}$. 
\end{prop} 
\begin{proof}
First we prove (2).  
The $E_1$-terms of the spectral sequence (\ref{eqn:lledd}) 
are smooth by \cite[Arcata V (3.1)]{sga41/2}. 
Especially they are constructible. 
Because the constructivility is stable under the kernel  and the image 
of a morphism of constructible sheaves 
(\cite[IX (2.6) $({\rm i}_{\rm bis})$]{sga4-2}), 
the $E_r$-terms $(r \geq 1)$ are also constructible.  
Because the constructivility is stable under 
an extension of constructible sheaves ([loc.~cit., IX (2.6) (ii)]), 
the sheaf 
$P_kR^qf_{(X_{\frac{1}{l^{\infty}}},D_{\frac{1}{l^{\infty}}})/S_{\frac{1}{l^{\infty}}}*}({\mab Z}/l^n)$ is 
also constructible. 
Let $\ol{S}(\ol{x}{}')\lo \ol{S}(\ol{x})$ be the specialization map 
over $\os{\circ}{S}$ ([loc.~cit., VIII (7.2)]). 
Then we have the following specialization map 
\begin{align*} 
(\ref{eqn:lledd})_{\ol{x}}\lo (\ref{eqn:lledd})_{\ol{x}{}'}
\end{align*} 
of spectral sequences by (\ref{theo:func}). 
Since the $E_1$-terms of the spectral sequence (\ref{eqn:lledd}) 
are smooth, the specialization morphism 
$(E_1^{-k,q+k})_{\ol{x}} \lo (E_1^{-k,q+k})_{\ol{x}{}'}$ 
is an isomorphism. 
Hence the specialization morphism 
$$P_kR^qf_{(X_{\frac{1}{l^{\infty}}},D_{\frac{1}{l^{\infty}}})/S_{\frac{1}{l^{\infty}}}*}({\mab Z}/l^n)_{\ol{x}}\lo 
P_kR^qf_{(X_{\frac{1}{l^{\infty}}},D_{\frac{1}{l^{\infty}}})/S_{\frac{1}{l^{\infty}}}*}({\mab Z}/l^n)_{\ol{x}{}'}$$
is an isomorphism and consequently 
$P_kR^qf_{(X_{\frac{1}{l^{\infty}}},D_{\frac{1}{l^{\infty}}})/S_{\frac{1}{l^{\infty}}}*}({\mab Z}/l^n)$ 
is smooth by the criterion [loc.~cit., IX (2.11)]. 
\par
(1): By (2) we see that the $E_1$-terms of the spectral sequence 
(\ref{eqn:lehkd}) are smooth. The rest of the proof is the same as that of (2). 
\end{proof}

\begin{prop}\label{prop:sfsc}
Assume that $\os{\circ}{X}$ is proper over $\os{\circ}{S}$. 
Endow the $E_1$-terms of {\rm (\ref{eqn:lehkd})} with the weight filtration 
$P$ by using the equality {\rm (\ref{ali:rpaxd})} for the case where the horizontal SNCD is empty. 
Endow the $E_r$-terms of {\rm (\ref{eqn:lehkd})} $(r\geq 2)$ with the induced filtration by $P$. 
Then the edge morphism 
\begin{equation*} 
d^{-k,q+k}_r \col E^{-k,q+k}_r \lo E^{-k+1,q+k}_r 
\quad (r \in {\mab Z}_{\geq 1}) 
\tag{12.5.1}\label{eqn:estkd}
\end{equation*} 
of the spectral sequence 
$(\ref{eqn:lehkd})\otimes_{{\mab Z}_l}{\mab Q}_l$ 
is strictly compatible with the 
weight  filtration $P$ on the $E_r$-terms. 
\end{prop}
\begin{proof} 
We have already seen that  $E^{-k,q+k}_1$ is smooth on $\os{\circ}{S}$. 
Hence we may assume that $S'=t$ as in the proof of (\ref{prop:stc}). 
The rest of the proof is the same as that of 
(\ref{prop:stc}). 
\end{proof}

\begin{rema}
It is obvious that $d^{-k,q+k}_r$ in (\ref{eqn:estkd}) vanishes for $r\geq 2$ if $\dim \os{\circ}{X}=1$. 
We do not know whether $d^{-k,q+k}_r$ in (\ref{eqn:estkd}) vanishes for $r\geq 2$ in general. 
However we prove that $d^{-k,q+k}_r=0$ for $r\geq 2$ in \S\ref{sec:ssf} below
if $(X,D)$ is a the log special fiber of a proper strict semistable family 
with a horizontal simple normal crossing divisor
over a henselian discrete valuation ring of any characteristic. 
\end{rema}

\begin{theo}[{\bf Strict compatibility}]\label{theo:stpbgb}
Let the notations be as in {\rm (\ref{theo:func})}. 
Let $q$ be an integer. 
Set $R^qf_{(X_{\frac{1}{l^{\infty}}},D_{\frac{1}{l^{\infty}}})/S_{\frac{1}{l^{\infty}}}*}({\mab Q}_l):=
R^qf_{(X_{\frac{1}{l^{\infty}}},D_{\frac{1}{l^{\infty}}})/S_{\frac{1}{l^{\infty}}}*}({\mab Z}_l)\otimes_{{\mab Z}_l}{\mab Q}_l$ 
and 
$R^qf_{(Y_{\frac{1}{l^{\infty}}},E_{\frac{1}{l^{\infty}}})/S_{\frac{1}{l^{\infty}}}'*}({\mab Q}_l):=
R^qf_{(Y_{\frac{1}{l^{\infty}}},E_{\frac{1}{l^{\infty}}})/S'_{\frac{1}{l^{\infty}}}*}({\mab Z}_l)\otimes_{{\mab Z}_l}{\mab Q}_l$. 
Then the induced morphism 
\begin{equation*}
g^* \col v^{*}R^qf_{(X_{\frac{1}{l^{\infty}}},D_{\frac{1}{l^{\infty}}})/S_{\frac{1}{l^{\infty}}}*}({\mab Q}_l)
\lo R^qf'_{(Y_{\frac{1}{l^{\infty}}},E_{\frac{1}{l^{\infty}}})/S_{\frac{1}{l^{\infty}}}'*}({\mab Q}_l) \quad 
(h\in {\mab Z}_{\geq 0})
\tag{12.7.1}\label{eqn:gbstn}
\end{equation*} 
is strictly compatible with the weight filtrations $P$'s.
\end{theo}
\begin{proof}
This follows from the specialization argument as in 
the proof of (\ref{prop:stc}).  
\end{proof}

\section{Log $l$-adic relative monodromy-weight conjecture}\label{sec:rmwc}
Let the notations be as in the previous section. 
In this section we give the log $l$-adic relative monodromy-weight conjecture
and we prove that this is true for the case where the relative dimension of $\os{\circ}{X}/\os{\circ}{S}$ 
is less than or equal to $2$. 
\par 
Consider the following morphism 
\begin{equation*} 
(T-1)\otimes {\rm id}\otimes \check{T} 
\col K^{\bul}_{l^n}((X_{\frac{1}{l^{\infty}}},D_{\frac{1}{l^{\infty}}})/S_{\frac{1}{l^{\infty}}}) \lo 
K^{\bul}_{l^n}((X_{\frac{1}{l^{\infty}}},D_{\frac{1}{l^{\infty}}})/S_{\frac{1}{l^{\infty}}})(-1).   
\tag{13.0.1}\label{eqn:ikntd}
\end{equation*} 
We also have the following morphism 
\begin{equation*} 
\nu:= (T-1)\otimes {\rm id}\otimes \check{T} 
\col A^{\bul}_{l^n}((X_{\frac{1}{l^{\infty}}},D_{\frac{1}{l^{\infty}}})/S_{\frac{1}{l^{\infty}}}) \lo 
A^{\bul}_{l^n}((X_{\frac{1}{l^{\infty}}},D_{\frac{1}{l^{\infty}}})/S_{\frac{1}{l^{\infty}}})(-1).  
\tag{13.0.2}\label{eqn:ad}
\end{equation*} 
The morphism $\nu$ is compatible 
with the filtration $P^{D_{\frac{1}{l^{\infty}}}}$.  
\par 
The following holds as in \cite[(1.7)]{rz}: 

\begin{prop}\label{prop:lhm} 
The morphism $\nu$ is homotopic to a morphism 
$A^{\bul}_{l^n}((X_{\frac{1}{l^{\infty}}},D_{\frac{1}{l^{\infty}}})/S_{\frac{1}{l^{\infty}}}) 
\lo A^{\bul}_{l^n}((X_{\frac{1}{l^{\infty}}},D_{\frac{1}{l^{\infty}}})/S_{\frac{1}{l^{\infty}}})(-1)$ defined by 
\begin{align*} 
\wt{\nu}:= {\rm proj}. \otimes {\rm id} \col 
&A^{ij}_{l^n}(X_{\frac{1}{l^{\infty}}}/S_{\frac{1}{l^{\infty}}})\otimes M_{l^n}(\os{\circ}{D}/\os{\circ}{S})^k \lo 
A^{i-1,j+1}_{l^n}(X_{\frac{1}{l^{\infty}}}/S_{\frac{1}{l^{\infty}}})(-1)\otimes M_{l^n}(\os{\circ}{D}/\os{\circ}{S})^k\\
&(i\in {\mab N}, j\in {\mab Z}, k\in {\mab N}).
\end{align*} 
\end{prop}
\begin{proof} 
(Though the proof is the same as that of \cite[(1.7)]{rz}, we give the proof because the signs of 
our horizontal and vertical boundary morphisms are different from theirs.)  
We omit the Tate twist in this proof. 
It suffices to prove (\ref{prop:lhm}) for the case $D=\emptyset$. 
Denote $A^{ij}_{l^n}(X_{\frac{1}{l^{\infty}}}/S)$ simply by $A^{ij}$. 
Let $\sig \col A^{i+1,j-1}\lo A^{i,j-1}$ be a morphism defined by $(x,y)\lom (y,0)$, where 
$x\in K_{l^n}(X/S)^{i+j+1}, y\in K_{l^n}(X/S)^{i+j}$. 
Then $\sig$ gives us a homotopy form $\nu$ to $(T-1)$. 
Indeed, 
\begin{align*} 
&\{(-d+\theta)\sig+\sig(-d+\theta)\}(x,y)\\
&=(-dy,-(T-1)y)+(0,y)+\sig(-dx,-(T-1)x+dy)+\sig(0,x)\\
& =(-dy,-(T-1)y+y)+(-(T-1)x+dy+x,0)\\
&=(-(T-1)x+x,y-(T-1)y)=\{\nu-(T-1)\}(x,y).
\end{align*}

\end{proof}

\begin{prop}\label{coro:llha} 
$(1)$ The  morphism 
\begin{equation*} 
\nu \col  R^qf_{(X_{\frac{1}{l^{\infty}}},D_{\frac{1}{l^{\infty}}})/S_{\frac{1}{l^{\infty}}}*}({\mab Z}_l)
\lo  R^qf_{(X_{\frac{1}{l^{\infty}}},D_{\frac{1}{l^{\infty}}})/S_{\frac{1}{l^{\infty}}}*}({\mab Z}_l)(-1)
\tag{13.2.1}\label{eqn:llnu}
\end{equation*}  
is nilpotent. 
\par 
$(2)$ For $k\in {\mab Z}$, 
$\nu(P_kR^qf_{(X_{\frac{1}{l^{\infty}}},D)/S_{\frac{1}{l^{\infty}}}*}({\mab Z}_l)) 
\subset (P_{k-2}R^qf_{(X_{\frac{1}{l^{\infty}}},D)/S_{\frac{1}{l^{\infty}}}*}
({\mab Z}_l))(-1)$. 
\end{prop} 
\begin{proof} 
Obvious. 
\end{proof}

\par  
Because $(T-1)\otimes \check{T} \col 
R^qf_{(X_{\frac{1}{l^{\infty}}},D_{\frac{1}{l^{\infty}}})/S_{\frac{1}{l^{\infty}}}*}({\mab Z}_l)  \lo 
R^qf_{(X_{\frac{1}{l^{\infty}}},D_{\frac{1}{l^{\infty}}})/S_{\frac{1}{l^{\infty}}}*}({\mab Z}_l)(-1)$ 
is nilpotent, the morphism  
\begin{equation*} 
N := \log T\otimes \check{T}
\col R^qf_{(X_{\frac{1}{l^{\infty}}},D_{\frac{1}{l^{\infty}}})/S_{\frac{1}{l^{\infty}}}*}({\mab Q}_l) 
\lo R^qf_{(X_{\frac{1}{l^{\infty}}},D_{\frac{1}{l^{\infty}}})/S_{\frac{1}{l^{\infty}}}*}({\mab Q}_l)(-1) 
\tag{13.2.2}\label{eqn:lhlq}
\end{equation*} 
is well-defined. 

\begin{defi} 
We call $\nu$ and $N$ the $l$-{\it adic quasi-monodromy operator} on 
$R^qf_{(X_{\frac{1}{l^{\infty}}},D_{\frac{1}{l^{\infty}}})/S_{\frac{1}{l^{\infty}}}*}({\mab Q}_l)$
and the $l$-{\it adic monodromy operator}
on $R^qf_{(X_{\frac{1}{l^{\infty}}},D_{\frac{1}{l^{\infty}}})/S_{\frac{1}{l^{\infty}}}*}({\mab Z}_l)$, respectively. 
\end{defi}

\par 
Because the morphism $\nu$ is compatible 
with the filtration $P^{D_{\frac{1}{l^{\infty}}}}$, the morphisms 
(\ref{eqn:llnu}) and (\ref{eqn:lhlq}) are compatible with  
$P^{D_{\frac{1}{l^{\infty}}}}$. 

Let us recall the following definition: 

\begin{defi}[{\bf \cite[(1.6.13)]{dw2}, \cite[(2.5)]{stz}}]
\label{defi:vqnp} 
Let $(V,Q)$ be a filtered vector space. 
Let $N$ be a nilpotent endomorphism of $V$. 
A {\it monodromy filtration of} $N$ {\it relative to} $Q$  
is a filtration $M$ such that $N(M_kV) \subset M_{k-2}V$ 
and the induced filtration $M$ on ${\rm gr}_k^QV$ 
is the monodromy filtration of the nilpotent endomorphism 
$N$ on ${\rm gr}_k^QV$. 
(In \cite[(2.5)]{stz} Steenbrink and Zucker have called the 
filtration $M$ the weight filtration of $N$ relative to $Q$.)  
\end{defi}

\par 
Imitating the relative monodromy filtration 
in characteristic $p>0$ in \cite[(1.8.5)]{dw2} and 
the relative monodromy filtration 
over the complex number field in 
\cite{stz} and \cite{ezth}, 
we conjecture the following which we call 
the {\it log $l$-adic relative monodromy-weight conjecture}:

\begin{conj}[{\bf log $l$-adic relative monodromy-weight conjecture}]
\label{conj:remc} 
Let $q$ be a nonnegative integer.  
Assume that $\os{\circ}{X}$ is projective over $\os{\circ}{S}$. 
Then the filtration $P$ on 
$R^qf_{(X_{\frac{1}{l^{\infty}}},D_{\frac{1}{l^{\infty}}})/S_{\frac{1}{l^{\infty}}}*}({\mab Q}_l)$ 
is the monodromy filtration of $N$ relative to 
$P^{D_{\frac{1}{l^{\infty}}}}$, especially,  
the induced morphism 
\begin{equation*} 
N^e \col {\rm gr}^P_{q+k+e}{\rm gr}^{P^{D_{\frac{1}{l^{\infty}}}}}_{k}
R^qf_{(X_{\frac{1}{l^{\infty}}},D_{\frac{1}{l^{\infty}}})/S_{\frac{1}{l^{\infty}}}}*({\mab Q}_l)
\lo {\rm gr}^P_{q+k-e}{\rm gr}^{P^{D_{\frac{1}{l^{\infty}}}}}_{k}
R^qf_{(X_{\frac{1}{l^{\infty}}},D_{\frac{1}{l^{\infty}}})/S_{\frac{1}{l^{\infty}}}*}({\mab Q}_l)(-e)  
\tag{13.5.1}\label{eqn:clrm} 
\end{equation*} 
for $e,k\in {\mab N}$ 
is an isomorphism.  
\end{conj}

\begin{rema}
It is easy to check that (\ref{conj:remc}) is equivalent to 
the following: 
the following induced morphism 
\begin{align*} 
\nu^e \col & {\rm gr}^P_{q+k+e}{\rm gr}^{P^{D_{\frac{1}{l^{\infty}}}}}_{k}
R^qf_{(X_{\frac{1}{l^{\infty}}},D_{\frac{1}{l^{\infty}}})/S_{\frac{1}{l^{\infty}}}*}({\mab Q}_l)
\lo \tag{13.6.1}\label{eqn:clrnm} \\
& {\rm gr}^P_{q+k-e}{\rm gr}^{P^{D_{\frac{1}{l^{\infty}}}}}_{k} 
R^qf_{(X_{\frac{1}{l^{\infty}}},D_{\frac{1}{l^{\infty}}})/S_{\frac{1}{l^{\infty}}}*}({\mab Q}_l)(-e)  
\quad (e,k\in {\mab N}) 
\end{align*} 
by the morphism $\nu \col 
R^qf_{(X_{\frac{1}{l^{\infty}}},D_{\frac{1}{l^{\infty}}})/S_{\frac{1}{l^{\infty}}}*}({\mab Q}_l)\lo 
R^qf_{(X_{\frac{1}{l^{\infty}}},D_{\frac{1}{l^{\infty}}})/S_{\frac{1}{l^{\infty}}}*}({\mab Q}_l)(-1)$ is an isomorphism. 
\end{rema} 

We also recall the following conjecture by K.~Kato 
(\cite[(2.4) (2)]{nd}, \cite[(2.0.9;$l$)]{ndeg}): 

\begin{conj}[{\bf log $l$-adic monodromy-weight conjecture}] 
\label{conj:log} 
Let $q$ be a nonnegative integer.  
Let $f_{X_{\frac{1}{l^{\infty}}}/S_{\frac{1}{l^{\infty}}}}\col 
X_{\frac{1}{l^{\infty}}}\lo \os{\circ}{S}$ be the structural morphism. 
Assume that $\os{\circ}{X}$ is projective over $\os{\circ}{S}$. 
Then the filtration $P$ on 
$R^qf_{X_{\frac{1}{l^{\infty}}}/S_{\frac{1}{l^{\infty}}}*}({\mab Q}_l)$ 
is the monodromy filtration of $N$, especially,  
the induced morphism 
\begin{equation*} 
N^e \col {\rm gr}^P_{q+e}
R^qf_{X_{\frac{1}{l^{\infty}}}/S_{\frac{1}{l^{\infty}}}*}({\mab Q}_l)
\lo {\rm gr}^P_{q-e}
R^qf_{X_{\frac{1}{l^{\infty}}}/S_{\frac{1}{l^{\infty}}}*}({\mab Q}_l)(-e)  
\quad (e\in {\mab N}) 
\tag{13.7.1}\label{eqn:ulmw} 
\end{equation*} 
is an isomorphism.  
\end{conj}

\begin{rema}\label{rema:lmirm} 
(1) In \cite{kln} Kato has conjectured (\ref{conj:log}) in the case where 
$\os{\circ}{S}$ is a point. 
The proposition (\ref{prop:sms}) (2) tells us that, if  
Kato's original conjecture is true, then (\ref{conj:log}) is also true.
\par 
(2) In \cite{kln} Kato kindly suggested to me that 
the weight filtration and the monodromy filtration on 
the first log $l$-adic cohomology of the
degeneration of a $p$-adic analogue of the Hopf surface 
(this is a proper SNCL surface) are different. 
However the proof in [loc.~cit.] is not complete. 
A generalization of his suggestion including the $p$-adic case 
has been given in \cite[(6.5)]{ndeg}.
The proof in \cite[(6.5)]{ndeg} is totally different from his proof. 
\end{rema}

It is evident that (\ref{conj:remc}) is a generalization of 
(\ref{conj:log}).  Conversely (\ref{conj:log}) implies (\ref{conj:remc}): 

\begin{theo}\label{theo:mimr} 
Assume that $\os{\circ}{X}$ is projective over $\os{\circ}{S}$.
If {\rm (\ref{conj:log})} is true,  
then {\rm (\ref{conj:remc})} is true. 
Consequently 
there exists a monodromy filtration 
$M$ on $R^qf_{(X_{\frac{1}{l^{\infty}}},D_{\frac{1}{l^{\infty}}})/S_{\frac{1}{l^{\infty}}}*}({\mab Q}_l)$ 
relative to $P^{D_{\frac{1}{l^{\infty}}}}$ and the relative monodromy filtration 
$M$ is equal to $P$. 
\end{theo}
\begin{proof} 
By (\ref{coro:llha}) (2), 
$N(P_kR^qf_{(X_{\frac{1}{l^{\infty}}},D_{\frac{1}{l^{\infty}}})/S_{\frac{1}{l^{\infty}}}*}({\mab Q}_l))\subset 
P_{k-2}R^qf_{(X_{\frac{1}{l^{\infty}}},D_{\frac{1}{l^{\infty}}})/S_{\frac{1}{l^{\infty}}}*}({\mab Q}_l)$ $(k\in {\mab Z})$.  
It suffices to prove that the morphism 
(\ref{eqn:clrm}) is an isomorphism. 
In this proof we ignore orientation sheaves. 
By (\ref{prop:pwsbc}) we may assume that $S$ is the log point 
$({\rm Spec}({\kap}),{\mab N}\oplus {\kap}^*\lo \kap)$. 
Consider the $E_1$-term $\{E_1^{-k,q+k}\}_{k,q\in {\mab Z}}$ 
of (\ref{eqn:lehkd})$\otimes_{{\mab Z}_l}{\mab Q}_l$. 
Then, by the assumption, 
$N$ induces an isomorphism 
\begin{equation*} 
N^e \col {\rm gr}^P_{q+k+e}E_1^{-k,q+k}\os{\sim}{\lo} 
{\rm gr}^P_{q+k-e}E_1^{-k,q+k}(-e) 
\tag{13.9.1}\label{eqn:mwd}
\end{equation*} 
($E_1^{-k,q+k}=
H^q_{\rm ket}(D^{(k)}_{\frac{1}{l^{\infty}}},{\mab Q}_l))$ for any $k,q\in {\mab Z}$. 
By (\ref{prop:sfsc}) 
the edge morphism 
$$d^{-k,q+k}_r\col E_r^{-k,q+k}
\lo E_r^{-k+r,q+k+r-1}\quad (r\geq 1)$$ 
is strictly compatible with $P$. 
Apply the key lemma (\ref{lemm:easy}) below by setting $f\col U:=d^{-k-1,q+k}_r\col 
E_r^{-k-1,q+k}\lo V:=E_r^{-k,q+k}$ and $g:=d^{-k,q+k}_r\col V:=E_r^{-k,q+k}\lo 
W:=E_r^{-k+1,q+k}$ 
and consider the following commutative diagram: 
\begin{equation*} 
\begin{CD}
E_r^{-k-r,q+k+r-1}@>{d^{-k-r,q++r-1k}_r}>> E_r^{-k,q+k} @>{d^{-k,q+k}_r}>>E_r^{-k+r,q+k-r+1}\\
@V{N^e}VV @V{N^e}VV @V{N^e}VV \\
E_r^{-k-r,q+k+r-1}(-e)@>{d^{-k-r,q+k+r-1}_r}>> E_r^{-k,q+k}(-e) @>{d^{-k,q+k}_r}>>E_r^{-k+r,q+k-r+1}(-e). 
\end{CD}
\end{equation*} 
(Note that, 
because 
$$N\col (A^{\bul}_{l^{\infty}}((X_{\frac{1}{l^{\infty}}},D_{\frac{1}{l^{\infty}}})/S_{\frac{1}{l^{\infty}}}),
P^{D_{\frac{1}{l^{\infty}}}},P) \lo 
(A^{\bul}_{l^{\infty}}((X_{\frac{1}{l^{\infty}}},D_{\frac{1}{l^{\infty}}})/S_{\frac{1}{l^{\infty}}}),
P^{D_{\frac{1}{l^{\infty}}}},P\langle -2\rangle)(-1)$$ 
is a morphism of bifiltered complexes, 
$N$ and $d^{-k,q+k}_r$'s indeed commutes in the diagram above.)
By the key lemma below, 
we see that the morphism 
\begin{equation*} 
N^e \col {\rm gr}^P_{q+k+e}E_r^{-k,q+k} \lo 
{\rm gr}^P_{q+k-e}E_r^{-k,q+k}(-e) 
\end{equation*} 
is an isomorphism. 
\end{proof} 

\begin{lemm}\label{lemm:easy} 
Let 
\begin{equation*} 
\begin{CD}
(U,P^U) @>{f}>> (V,P^V) @>{g}>> (W,P^W)\\ 
@V{N}VV @V{N}VV @V{N}VV \\
(U,P^U\langle -2\rangle) @>{f}>> (V,P^V\langle -2\rangle) @>{g}>> 
(W,P^W\langle -2\rangle)
\end{CD} 
\end{equation*} 
be a commutative diagram of filtered objects of 
an abelian category such that $g\circ f=0$. 
Set $\Lam=U, V,W$.  
Assume that $P^U$ is finite and 
that $f$ and $g$ are strict with respect to $P^{\Lam}$'s. 
Assume that the induced morphisms 
$N^e \col {\rm gr}^{P^{\Lam}}_e\Lam\os{\sim}{\lo} {\rm gr}^{P^{\Lam}}_{-e}\Lam$  
are isomorphisms. 
Then the  
induced morphism 
\begin{equation*} 
N^e \col {\rm gr}^{P^V}_e({\rm Ker}(g)/{\rm Im}(f)) 
{\lo} 
{\rm gr}^{P^V}_{-e}({\rm Ker}(g)/{\rm Im}(f)) 
\tag{13.10.1}\label{eqn:meg} 
\end{equation*} 
is an isomorphism.  
\end{lemm} 
\begin{proof} 
For a finitely filtered morphism $F\col (V,P)\lo (V',P')$, $F$ is strict if and only if 
the sequence $$0\lo {\rm gr}^P{\rm Ker}(F)\lo {\rm gr}^PV\lo {\rm gr}^{P'}V'\lo 
{\rm gr}^{P'}{\rm Coker}(F)\lo 0$$
of graded objects is exact. 
Hence the morphisms  
\begin{align*} 
N^e\col  {\rm gr}^P_e{\rm Ker}(g)\lo {\rm gr}^P_{-e}{\rm Ker}(g)
\tag{13.10.2}\label{eqn:kgeg} 
\end{align*} 
and 
\begin{align*} 
N^e\col  {\rm gr}^P_e{\rm Ker}(f)\lo {\rm gr}^P_{-e}{\rm Ker}(f)
\tag{13.10.3}\label{eqn:kfeg} 
\end{align*} 
are isomorphisms. 
Here we denote $P^{\Lam}$ simply by $P$. 
Consequently  
\begin{align*} 
N^e\col  {\rm gr}^P_e{\rm Im}(f)\lo {\rm gr}^P_{-e}{\rm Im}(f)
\tag{13.10.4}\label{eqn:kimg} 
\end{align*} 
is an isomorphism. 
Because 
\begin{align*} 
{\rm gr}^{P}_{\pm e}({\rm Ker}(g)/{\rm Im}(f)) &=
(P_{\pm e}V\cap {\rm Ker}(g))/(P_{\pm e-1}V\cap {\rm Ker}(g)+
P_{\pm e}V\cap  {\rm Im}(f))\\
&(= (P_{\pm e}V\cap {\rm Ker}(g))/(P_{\pm e-1}V\cap {\rm Ker}(g)+
f(P_{\pm e}U))), 
\end{align*} 
the surjectivity of the morphism (\ref{eqn:meg}) follows from that of 
(\ref{eqn:kgeg}). 
We can show the injectivity of the morphism (\ref{eqn:meg}) as follows.
\par 
By the surjectivity of the morphism (\ref{eqn:kimg}) for various $e$'s, 
we obtain the following inclusions 
\begin{align*} 
f(P_{-e}U)  & \subset N^e(f(P_eU))+f(P_{-e-1}U)\subset 
N^e(f(P_eU))+N^{e+1}f(P_{-e-1}U)+f(P_{-e-2}U) \tag{13.10.5}\label{ali:sne}\\
& \subset \cdots  \subset \sum_{k=0}^mN^{e+k}f(P_{e+k}U)(=
\sum_{k=0}^mg(N^{e+k}P_{e+k}U))
\subset N^e(f(P_eU)) 
\end{align*} 
for some $m>>0$ because the filtration $P$ on $U$ is finite 
(and hence $N\col U\lo U$ is nilpotent). 
By Mitchell's embedding theorem 
we may assume that $f$ and $g$ are morphisms of filtered $R$-modules, 
where $R$ is a ring with unit element. 
Assume that, for $x\in P_eV\cap {\rm Ker}(g)$,  
$$N^e(x) \subset 
P_{-e-1}V\cap {\rm Ker}(g)+f(P_{-e}U).$$
By the inclusion (\ref{ali:sne}) 
there exists an element $y\in f(P_eU)$ such that $N^e(x-y) \in 
P_{-e-1}V\cap {\rm Ker}(g)$. 
By the injectivity of (\ref{eqn:kgeg}),  
$x-y\in P_{e-1}{\rm Ker}(g)$ and hence 
$x\in f(P_eU)+P_{e-1}{\rm Ker}(g)$.  
This shows that the injectivity of (\ref{eqn:meg}).
\end{proof} 

\begin{rema}
To prove the $\infty$-adic analogue of the conjecture (\ref{conj:remc}) for the case of a projective 
strict semistable family over the unit disk, Steenbrink and Zucker have used the $E_2$-degeneration 
of the $\infty$-adic analogue of (\ref{eqn:lehkd}) modulo torsion in \cite[p.~527]{stz}. 
We have not used the $E_2$-degeneration of (\ref{eqn:lehkd}) modulo torsion; 
we do not know whether it holds in general. 
Instead we use the strict compatibility of $d_r^{-k,q+k}$ for any $r\geq 1$ in the proof of 
(\ref{theo:mimr}). 
\end{rema}

\begin{prop}\label{coro:nm}
If $\dim \os{\circ}{X}\leq 2$, then {\rm (\ref{conj:log})} is true.  
\end{prop} 
\begin{proof} 
By the main result in \cite{nd} or (\ref{prop:stc}),  
the $l$-adic weight spectral sequence 
(\ref{eqn:lledd}) for the case 
$D=\emptyset$ degenerates at $E_2$. 
The conjecture (\ref{conj:log}) for the case $q=1$ has been proved by 
Kajiwara-Achinger  (\cite[(3.1)]{kaj} or \cite[Theorem 3.6]{ash}). 
Hence the conjecture for the case $q=3$  is also true 
by the classical Poincar\'{e} duality and the description of the edge morphisms 
of $E_1$-terms of (\ref{prop:stc}) (cf.~(\ref{eqn:glbd}) below for the case $D=\emptyset$) 
because 
the induced morphism 
$\nu^k \col E_1^{-k,q+k}\lo E_1^{k,q-k}(-k)$ by $\nu \col 
R^qf_{(X_{\frac{1}{l^{\infty}}},D)/S_{\frac{1}{l^{\infty}}}*}({\mab Q}_l)\lo 
R^qf_{(X_{\frac{1}{l^{\infty}}},D)/S_{\frac{1}{l^{\infty}}}*}({\mab Q}_l)(-1)$ 
is identified with the tensorization of the identity morphism with $\check{T}^{\otimes k}$. 
The conjecture for the case $q=2$  holds by the proof of \cite[(6.2.1)]{msemi}. 
\end{proof}

\begin{coro}\label{theo:c2oj2}
If $\dim \os{\circ}{X}\leq 2$, then {\rm (\ref{conj:remc})} is true.  
\end{coro}
\begin{proof}
(\ref{theo:c2oj2}) follows from (\ref{coro:nm}) and (\ref{theo:mimr}). 
\end{proof}

\begin{prob}\label{prob:indlp}
Let ${\cal V}$ be a complete discrete valuation ring of mixed characteristics 
$(0,p)$ with perfect residue field. Let $K$ be the fraction field of ${\cal V}$. 
Set $B=({\rm Spf}({\cal V}),{\cal V}^*)$.   
Let $S$ be a $p$-adic formal family of log points over $B$ 
such that $\os{\circ}{S}$ is a ${\cal V}/p$-scheme.  
Let $(X,D)/S$ be a proper SNCL scheme with a relative SNCD.  
In \cite{nppf} we have constructed the following spectral sequence
\begin{equation*} 
E_1^{-k,q+k}:=\bigoplus_{k'\leq k}\bigoplus_{j\geq \max \{-k',0\}} 
R^{q-2j-k}f_{\os{\circ}{X}{}^{(2j+k')}\cap \os{\circ}{D}{}^{(k-k')}
/K}(
{\cal O}_{\os{\circ}{X}{}^{(2j+k')}\cap \os{\circ}{D}{}^{(k-k')}
/\os{\circ}{S}}\otimes_{\mab Z}
\tag{13.14.1}\label{eqn:llpedd} 
\end{equation*} 
$$\vp^{(2j+k',k-k')}
((\os{\circ}{X},\os{\circ}{D})/\os{\circ}{S}))(-j-k) \Lo 
R^qf_{(X,D)/S}({\cal O}_{(X,D)/S})\otimes_{\cal V}K 
\quad (q\in {\mab Z}).$$
Assume that $\os{\circ}{S}$ is connected. 
Let $q$ be a nonnegative integer. 
Is the rank of $P_kR^qf_{(X,D)/S*}({\cal O}_{(X,D)/S})\otimes_{\cal V}K$ $(k\in {\mab N})$ 
is equal to the rank of 
$P_kR^qf_{(X_{\frac{1}{l^{\infty}}},D_{\frac{1}{l^{\infty}}})/S_{\frac{1}{l^{\infty}}}*}({\mab Q}_l)$? 
(If one knows that the rank of  
$R^qf_{(X,D)/K*}({\cal O}_{(X,D)/S})\otimes_{\cal V}K$ is equal to 
that of 
$R^qf_{(X_{\frac{1}{l^{\infty}}},D_{\frac{1}{l^{\infty}}})/S_{\frac{1}{l^{\infty}}}*}({\mab Q}_l)$, then 
one can prove that this problem is affirmatively solved by the specialization argument.) 
\end{prob}

\section{Strict semistable family}\label{sec:ssf} 
Let ${\cal V}$ be a henselian discrete valuation ring with residue field $\kap$. 
Let $K$ be the fraction field of ${\cal V}$. 
Let $\ol{\cal V}$ be the integral closure of ${\cal V}$ 
in a separable closure $\ol{K}$ of $K$. 
Let $\ol{\kap}$ be the residue field of $\ol{\cal V}$. 
Let $\pi$ be a uniformizer of ${\cal V}$.   
Let $\os{\circ}{\cal X}$ be 
a strict semistable family over ${\cal V}$. 
Endow $\os{\circ}{\cal X}$ with the canonical log structure and 
let ${\cal X}$ be the resulting log scheme. 

\begin{defi}\label{defi:dxsnc} 
Let $\os{\circ}{\cal D}$ be a closed subscheme of 
$\os{\circ}{\cal X}$. 
We call $\os{\circ}{\cal D}$ a {\it horizontal SNCD} on 
$\os{\circ}{\cal X}$ if 
the following three conditions are satisfied:  
\medskip 
\parno
(14.1.1): $\os{\circ}{\cal D}$ is an effective Cartier 
divisor on  $\os{\circ}{\cal X}$ over ${\cal V}$. 
\medskip 
\parno
(14.1.2): $\os{\circ}{\cal D}$ is a union of smooth  divisors  
$\{\os{\circ}{\cal D}_{\mu}\}_{\mu \in M}$ on $\os{\circ}{\cal X}/{\cal V}$.  
\medskip 
\parno 
(14.1.3): There exists an \'{e}tale morphism 
$\os{\circ}{\cal X} {\lo} 
{\rm Spec}({\cal V}[x_0,\ldots, x_d]/(x_0\cdots x_{a-1}-\pi))$ 
zariski locally on $\os{\circ}{\cal X}$
such that 
$\os{\circ}{\cal D}$  is locally the fiber product of 
the following  diagram $(a +b \leq d+1)$:
\begin{equation*}
\begin{CD}
@.  \os{\circ}{\cal X}\\ 
@.  @VVV \\
{\rm Spec}({\cal V}[x_0, \ldots, x_d]
/(\prod_{i=0}^{a-1}x_i-\pi, \prod_{i=a}^{a+b-1}x_i)) 
@>{\sus}>> {\rm Spec}({\cal V}[x_0, \ldots, x_d]
/(\prod_{i=0}^{a-1}x_i-\pi))  
\end{CD}
\tag{14.1.4}\label{eqn:locsd}
\end{equation*} 
and such that
$\{\os{\circ}{\cal D}_{\mu}\}_{\mu \in M}
=\{\{x_i=0\}_{i=a}^{a+b-1}\}$ in the 
diagram (\ref{eqn:locsd}). 
Here the left hand $\{\os{\circ}{\cal D}_{\mu}\}_{\mu \in M}$ of this equality 
means the set of non-empty ${\cal D}_{\mu}$'s in the local situation above. 
\end{defi} 

Set ${\cal U}:=\os{\circ}{\cal X}\setminus \os{\circ}{\cal D}$ and 
$(\os{\circ}{\ol{X}},\os{\circ}{\ol{D}})
:=(\os{\circ}{\cal X},\os{\circ}{\cal D})\otimes_{\ol{\cal V}}\ol{\kap}$.  
Denote $\vp^{(k)}_{\rm et}(\os{\circ}{\ol{D}}/{\rm Spec}~\ol{\kap})$, 
$\vp^{(k)}_{\rm et}(\os{\circ}{\ol{X}}/{\rm Spec}~\ol{\kap})$ and 
$\vp^{(k),(k')}_{\rm et}(\os{\circ}{\ol{X}}/{\rm Spec}~\ol{\kap})$
$(k\in {\mab N})$ simply by 
$\vp^{(k)}_{\rm et}(\os{\circ}{\ol{D}}/\ol{\kap})$, 
$\vp^{(k)}_{\rm et}(\os{\circ}{\ol{X}}/\ol{\kap})$ and 
$\vp^{(k),(k')}_{\rm et}((\os{\circ}{\ol{X}},\os{\circ}{\ol{D}})/\ol{\kap})$, 
respectively. 
\par 
Set $\eta:={\rm Spec}(K)$ and 
$\ol{\eta}:={\rm Spec}(\ol{K})$. 
Set ${\cal X}_{{\eta}}:={\cal U}\otimes_{\cal V}{{K}}$, 
${\cal D}_{{\eta}}:={\cal D}\otimes_{\cal V}{{K}}$ 
and ${\cal U}_{{\eta}}:={\cal U}\otimes_{\cal V}{{K}}$. 
Set ${\cal X}_{\ol{\eta}}:={\cal X}\otimes_{\cal V}{\ol{K}}$, 
${\cal D}_{\ol{\eta}}:={\cal D}\otimes_{\cal V}{\ol{K}}$  
and ${\cal U}_{\ol{\eta}}:={\cal U}\otimes_{\cal V}{\ol{K}}$. 
In this section we construct two spectral sequences  
converging to $H^q({\cal U}_{\ol{\eta}},{\mab Z}_l)$ 
$(q,n\in {\mab N})$. 
These spectral sequences turn out to be 
canonically isomorphic to 
(\ref{eqn:lehkd}) and (\ref{eqn:lledd}), respectively.

\par 
Set $\os{\circ}{\cal X}_{\ol{\cal V}}:=\os{\circ}{\cal X}\otimes_{\cal V}\ol{\cal V}$ and 
$\os{\circ}{\cal X}_{\ol{\eta}}:=\os{\circ}{\cal X}\otimes_{\cal V}\ol{K}$. 
Set 
$\os{\circ}{\cal D}_{\ol{\cal V}}:=\os{\circ}{\cal D}\otimes_{\cal V}\ol{\cal V}$ and 
$\os{\circ}{\cal D}_{\ol{\eta}}:=\os{\circ}{\cal D}\otimes_{\cal V}\ol{K}$. 
As in \S\ref{sec:lbc} we can define the orientation sheaves 
$\vp^{(k)}_{\rm et}(\os{\circ}{\cal D}_{\ol{\cal V}}/\ol{\cal V})$ 
and 
$\vp^{(k)}_{\rm et}(\os{\circ}{\cal D}_{\ol{\eta}}/\ol{K})$ 
$(k\in {\mab N})$ 
on $(\os{\circ}{\cal D}_{\ol{\cal V}})_{\rm et}$ and 
$(\os{\circ}{\cal D}_{\ol{\eta}})_{\rm et}$, respectively.  
\par 
Let $\iota \col X  \os{\sus}{\lo} {\cal X}$ 
and  $j \col {\cal X}_K  \os{\sus}{\lo} 
{\cal X}$ be the natural closed immersion 
and the natural open immersion, respectively. 
Let $\ol{\iota} \col \ol{X}  \os{\sus}{\lo} {\cal X}_{\ol{\cal V}}$ 
and  $\ol{j} \col {\cal X}_{\ol{\cal V}} \os{\sus}{\lo} 
{\cal X}_{\ol{\eta}}$ be the base changes of $\iota$ and $j$ 
over $\ol{\cal V}$ and $\ol{\eta}$, 
respectively. 
\par 
Let $R\os{\circ}{\Psi}({\mab Z}/l^n):=\os{\circ}{\ol{\iota}}{}^*R\os{\circ}{\ol{j}}_*({\mab Z}/l^n)$  
be the classical nearby cycle sheaf of ${\mab Z}/l^n$ on 
$\os{\circ}{X}$. 
Let $K^{\bul}_{l^n}({\cal X}_{\ol{\cal V}}/\ol{\cal V})$ be a representative of 
$R\os{\circ}{\Psi}({\mab Z}/l^n)$ obtained by an injective resolution of 
${\mab Z}/l^n$ in $({\cal X}_{\ol{\eta}})_{\rm et}$. 
Let $T$ be a basis of ${\mab Z}_l(1)$; 
$T$ acts on $K^{\bul}_{l^n}({\cal X}_{\ol{\cal V}}/\ol{\cal V})$ by the proof of \cite[(2.24)]{rz}.  
Consider the mapping fiber of 
$T-1 \col  K^{\bul}_{l^n}({\cal X}_{\ol{\cal V}}/\ol{\cal V}) \lo K^{\bul}_{l^n}({\cal X}_{\ol{\cal V}}/\ol{\cal V})$:
${\rm MF}_{l^n}(T-1)_{\rm ss}:=s((K^{\bul}_{l^n}({\cal X}_{\ol{\cal V}}/\ol{\cal V}),d) 
\os{T-1}{\lo}(K^{\bul}_{l^n}({\cal X}_{\ol{\cal V}}/\ol{\cal V}),-d))$, 
where $s$ means the single complex of a double complex. 
Here ${\rm ss}$ in ${\rm MF}_{l^n}(T-1)_{\rm ss}$ is the abbreviation of the semistability. 
As in (\ref{cd:lttarl}), we define a morphism  
\begin{equation*} 
\theta \col {\rm MF}_{l^n}(T-1)_{\rm ss}^{\bul} \lo 
{\rm MF}_{l^n}(T-1)_{\rm ss}^{\bul}(1)[1] . 
\end{equation*} 
Set 
\begin{equation*} 
A^{ij}_{l^n}({\cal X}_{\ol{\cal V}}/\ol{\cal V})
:=({\rm MF}_{l^n}(T-1)_{\rm ss}(j+1)/
\tau_j{\rm MF}_{l^n}(T-1)_{\rm ss}(j+1))^{i+j+1} 
\tag{14.1.5}\label{eqn:ssxsmf}
\end{equation*} 
and we define the following boundary morphisms as in (\ref{cd:dlsgn}):
\begin{equation*} 
\theta \col A^{ij}_{l^n}({\cal X}_{\ol{\cal V}}/\ol{\cal V}) \lo 
A^{i,j+1}_{l^n}({\cal X}_{\ol{\cal V}}/\ol{\cal V}), \quad 
{-d} \col A^{ij}_{l^n}({\cal X}_{\ol{\cal V}}/\ol{\cal V}) \lo 
A^{i+1,j}_{l^n}({\cal X}_{\ol{\cal V}}/\ol{\cal V}). 
\tag{14.1.6}\label{cd:stssd}
\end{equation*} 
Let $A^{\bul}_{l^n}({\cal X}_{\ol{\cal V}}/\ol{\cal V})$ be the single complex of 
the double complex 
$A^{\bul \bul}_{l^n}({\cal X}_{\ol{\cal V}}/\ol{\cal V})$: 
\begin{align*} 
A^{\bul}_{l^n}({\cal X}_{\ol{\cal V}}/\ol{\cal V})  = &
s({\rm MF}_{l^n}(T-1)_{\rm ss}(1)/\tau_0{\rm MF}_{l^n}(T-1)_{\rm ss}(1)[1]
\os{\theta}{\lo} \tag{14.1.7}\label{eqn:sssgle} \\ 
& {\rm MF}_{l^n}(T-1)_{\rm ss}(2)/\tau_1{\rm MF}_{l^n}(T-1)_{\rm ss}(2)[2]
\os{\theta}{\lo}  \cdots). 
\end{align*}  
As in \cite{rz},  set 
\begin{equation*}
P^{\cal X}_kA^{ij}_n({\cal X}_{\ol{\cal V}}/\ol{\cal V})
:=((\tau_{2j+k+1}+\tau_j){\rm MF}_{l^n}(T-1)_{\rm ss}(j+1)
/\tau_j{\rm MF}_{l^n}(T-1)_{\rm ss}(j+1))^{i+j+1}. 
\end{equation*} 
Then we have a filtered complex 
$(A^{\bul}_{l^n}({\cal X}_{\ol{\cal V}}/\ol{\cal V}),P^{{\cal X}_{\ol{\cal V}}})$. 
By [loc.~cit., p.~35]
the natural morphism 
\begin{equation*} 
\os{\circ}{\ol{i}}{}^*R\os{\circ}{\ol{j}}_*({\mab Z}/l^n)  \lo 
{\rm MF}_{l^n}(T-1)_{\rm ss}
\end{equation*} 
is an isomorphism. 
Let $a^{(k)} \col \os{\circ}{\ol{X}}{}^{(k)}\lo \os{\circ}{\ol{X}}$ be the natural morphism. 
The Kummer sequence 
\begin{equation*}
0 \lo {\mab Z}/l^n(1) \lo {\mab G}_m \lo {\mab G}_m 
\lo 0
\tag{14.1.8}\label{eqn:mugm}
\end{equation*}
in $\os{\circ}{\cal X}_{\ol{\eta},{\rm et}}$ gives 
an isomorphism 
\begin{equation*}
a^{(1)}_*(({\mab Z}/l^n)_{\os{\circ}{\ol{X}}{}^{(1)}}(-1)
\otimes_{\mab Z}
\varpi^{(1)}_{\rm et}(\os{\circ}{\ol{X}}/\ol{\kap})) =
a^{(1)}_*(({\mab Z}/l^n)_{\os{\circ}{\ol{X}}{}^{(1)}})(-1) 
\os{\sim}{\lo} 
\os{\circ}{\ol{i}}{}^*R^1\os{\circ}{\ol{j}}_*({\mab Z}/l^n).  
\tag{14.1.9}\label{eqn:1rz}
\end{equation*} 
By [loc.~cit., (3.7)] the cup product induce 
the following isomorphism: 
\begin{equation*}
\bigwedge^r\os{\circ}{\ol{i}}{}^*R^1\os{\circ}{\ol{j}}_*({\mab Z}/l^n) 
\os{\sim}{\lo}
\os{\circ}{\ol{i}}{}^*R^r\os{\circ}{\ol{j}}_*({\mab Z}/l^n) 
\quad (r\in {\mab N}). 
\tag{14.1.10}\label{eqn:cupr1}
\end{equation*} 
Hence we have an isomorphism 
\begin{equation*}
\os{\circ}{\ol{i}}{}^*R^r\os{\circ}{\ol{j}}_*({\mab Z}/l^n) 
\os{\sim}{\longleftarrow} 
a^{(r)}_*(({\mab Z}/l^n)_{\os{\circ}{\ol{X}}{}^{(r)}}(-r)
\otimes_{\mab Z}
\varpi^{(r)}_{\rm et}(\os{\circ}{\ol{X}}/\ol{\kap})) \quad (r\in {\mab N}). 
\tag{14.1.11}\label{eqn:gap}
\end{equation*} 
and we have the following isomorphism 
in $D^{\rm b}_{\rm ctf}(\os{\circ}{\ol{X}}_{\rm et},{\mab Z}/l^n)$ 
as in (\ref{eqn:grf}):
\begin{align*} 
{\rm gr}^{P^{\cal X}}_kA^{\bul}_{l^n}({\cal X}_{\ol{\cal V}}/\ol{\cal V})
& =  \us{j\geq {\rm max}\{-k,0\}}{\bigoplus}
{\mab Z}/l^n(-j-k)
\otimes_{\mab Z}a^{(2j+k)}_*
(\varpi^{(2j+k)}_{\rm et}(\os{\circ}{\ol{X}}/\ol{\kap}))[-2j-k].  
\tag{14.1.12}\label{ali:ssgr} 
\end{align*} 
By the same proof as that of (\ref{prop:qislws}) 
we see that 
the morphism 
$\theta \col K^{\bul}_{l^n}({\cal X}_{\ol{\cal V}}/\ol{\cal V})\lo  {\rm MF}_{l^n}(T-1)_{\rm ss}(1)[1]$ 
induces a quasi-isomorphism 
\begin{equation*} 
\theta \otimes 1\col 
K^{\bul}_{l^n}({\cal X}_{\ol{\cal V}}/\ol{\cal V})\lo A^{\bul}_{l^n}({\cal X}_{\ol{\cal V}}/\ol{\cal V})
\tag{14.1.13}\label{eqn:fqsis} 
\end{equation*} 
in $C^+(\os{\circ}{\ol{X}}_{\rm et},{\mab Z}/l^n)$.  
\par 
Set $(\os{\circ}{\cal X}_{\ol{\cal V}},\os{\circ}{\cal D}_{\ol{\cal V}})
:=(\os{\circ}{\cal X}_{\ol{\cal V}},M({\cal D}_{\ol{\cal V}}))$ and   
consider the following morphism 
\begin{equation*} 
\eps_{\os{\circ}{\cal D}_{\ol{\cal V}}} \col 
(\os{\circ}{\cal X}_{\ol{\cal V}},\os{\circ}{\cal D}_{\ol{\cal V}}) \lo 
\os{\circ}{\cal X}_{\ol{\cal V}}. 
\end{equation*} 
forgetting the log structure $M({\cal D}_{\ol{\cal V}})$.  
Henceforth, assume that $\os{\circ}{X}$ is quasi-compact.  
As in (\ref{prop:bav}), there exists a bounded filtered flat resolution  
$(M^{\bul}_{l^n}(\os{\circ}{\cal D}_{\ol{\cal V}}/\ol{\cal V}),R)$ of a 
representative in 
${\rm C}^{\rm b}{\rm F}((\os{\circ}{\cal X}_{\ol{\cal V}})_{\rm et},{\mab Z}/l^n)$ of 
the filtered complex 
$(R\os{\circ}{\eps}_{\os{\circ}{\cal D}_{\ol{\cal V}}*}({\mab Z}/l^n),\tau)$.  
By using the adjunction morphism, 
we have a natural morphism 
\begin{equation*}
\os{\circ}{{\ol{\iota}}}{}^{*}
(R\os{\circ}{\eps}_{\os{\circ}{\cal D}_{\ol{{\cal V}}}*}({\mab Z}/l^n))  
\lo R\os{\circ}{\eps}_{\os{\circ}{\ol{D}*}}({\mab Z}/l^n) 
\tag{14.1.14}\label{eqn:osire}
\end{equation*} 
and, in fact, this is an isomorphism by \cite[(2.4)]{ktnk}. 
Hence we can take 
$(M^{\bul}_{l^n}(\os{\circ}{D}_{\ol{\cal V}}/\ol{\cal V}),Q)$ in 
\S\ref{sec:lbc} as  
$\os{\circ}{\ol{\iota}}{}^*
(M^{\bul}_{l^n}(\os{\circ}{\cal D}/\os{\circ}{S}),R)$ in the semistable case in this section.  
\par 
Consider the following filtered double complex 
and the single complex: 
\begin{equation*} 
(A^{\bul \bul}_{l^n}(({\cal X}_{\ol{\cal V}},{\cal D}_{\ol{\cal V}})/\ol{\cal V}),P^{{\cal D}_{\ol{\cal V}}})
:=A^{\bul}_{l^n}({\cal X}_{\ol{\cal V}}/\ol{\cal V})
\otimes_{{\mab Z}/l^n}\os{\circ}{\ol{\iota}}{}^*
((M^{\bul}_{l^n}(\os{\circ}{\cal D}_{\ol{\cal V}}/\ol{\cal V}),R))
\end{equation*} 
and 
\begin{equation*} 
(A^{\bul}_{l^n}(({\cal X}_{\ol{\cal V}},{\cal D}_{\ol{\cal V}})/\ol{\cal V}),P^{{\cal D}_{\ol{\cal V}}})
:=s(A^{\bul \bul}_{l^n}(({\cal X}_{\ol{\cal V}},{\cal D}_{\ol{\cal V}})/\ol{\cal V}),P^{{\cal D}_{\ol{\cal V}}}) 
\end{equation*} 
in ${\rm C}^{\rm b}{\rm F}((\os{\circ}{\ol{X}})_{\rm et},{\mab Z}/l^n)$. 
Consider also the following filtered double complex and the single complex: 
\begin{equation*} 
(A^{\bul \bul}_{l^n}(({\cal X}_{\ol{\cal V}},{\cal D}_{\ol{\cal V}})/\ol{\cal V}),P)
:=(A^{\bul}_{l^n}({\cal X}_{\ol{\cal V}}/\ol{\cal V}),P^{{\cal X}_{\ol{\cal V}}})
\otimes_{{\mab Z}/l^n}\os{\circ}{\ol{\iota}}{}^*((M^{\bul}_{l^n}
(\os{\circ}{\cal D}_{\ol{\cal V}}/\ol{\cal V}),R))
\end{equation*} 
and 
\begin{equation*} 
(A^{\bul}_{l^n}(({\cal X}_{\ol{\cal V}},{\cal D}_{\ol{\cal V}})/\ol{\cal V}),P)
:=s(A^{\bul \bul}_{l^n}(({\cal X}_{\ol{\cal V}},{\cal D}_{\ol{\cal V}})/\ol{\cal V}),P) 
\end{equation*} 
in ${\rm C}^{\rm b}{\rm F}((\os{\circ}{\ol{X}})_{\rm et},{\mab Z}/l^n)$. 

\begin{prop}\label{prop:wdbss}  
The image of the complex 
$(A^{\bul}_{l^n}(({\cal X}_{\ol{\cal V}},{\cal D}_{\ol{\cal V}})/\ol{\cal V}),P^{{\cal D}_{\ol{\cal V}}},P)
\in {\rm C}^{\rm b}{\rm F}^2((\os{\circ}{\ol{X}})_{\rm et},{\mab Z}/l^n)$ 
in ${\rm D}^{\rm b}{\rm F}^2((\os{\circ}{\ol{X}})_{\rm et},{\mab Z}/l^n)$ 
is independent of the choice of 
$K^{\bul}_{l^n}({\cal X}_{\ol{\cal V}}/\ol{\cal V})$, $T$ and 
$\os{\circ}{\ol{\iota}}{}^*((M^{\bul}_{l^n}(\os{\circ}{\cal D}_{\ol{\cal V}}/\ol{\cal V}),R))$ 
up to canonical isomorphisms.  
It is an object of 
${\rm D}^{\rm b}{\rm F}^2_{\rm ctf}((\os{\circ}{\ol{X}})_{\rm et},{\mab Z}/l^n)$.  
The family 
$\{(A^{\bul}_{l^n}({\cal X}_{\ol{\cal V}},{\cal D}_{\ol{\cal V}})/\ol{\cal V}),
P^{{\cal D}_{\ol{\cal V}}},P)\}_{n\in {\mab N}}$ defines an object of 
${\rm D}^{\rm b}{\rm F}^2_{\rm ctf}((\os{\circ}{\ol{X}})_{\rm et},{\mab Z}_l)$.  
\end{prop}
\begin{proof}  
The proof is the same as that of (\ref{theo:wdbst}). 
\end{proof}

\begin{lemm}\label{lemm:ssgrc} 
The following hold$:$
\par 
$(1)$ There exists an isomorphism 
\begin{equation*} 
{\rm gr}_k^{P_{{\cal D}_{\ol{\cal V}}}}
A^{\bul}_{l^n}(({\cal X}_{\ol{\cal V}},{\cal D}_{\ol{\cal V}})/\ol{\cal V})
\os{\sim}{\lo}
c^{(k)}_{*}(A^{\bul}_{l^n}({\cal D}^{(k)}_{\ol{\cal V}}/\ol{\cal V})(-k)\otimes_{\mab Z}
\varpi^{(k)}_{\rm et}(\os{\circ}{\ol{D}}/\ol{\kap}))\{-k\} \quad 
(k\in {\mab Z}). 
\tag{14.3.1}\label{eqn:ssgrpdx}
\end{equation*} 
This isomorphism is compatible with $n$'s. 
\par 
$(2)$ There exists an isomorphism 
\begin{align*} 
{\rm gr}_k^{P}A^{\bul}_{l^n}(({\cal X}_{\ol{\cal V}},{\cal D}_{\ol{\cal V}})/\ol{\cal V})&
\os{\sim}{\lo} 
\bigoplus^k_{k'=-\infty}
\bigoplus_{j\geq \max\{-k',0\}}
a^{(2j+k'+1),(k-k')}_{*}(({\mab Z}/l^n)_{\os{\circ}{\ol{X}}{}^{(2j+k'+1)}
\cap \os{\circ}{\ol{D}}{}^{(k-k')}} \tag{14.3.2}\label{eqn:ssaxd}\\ 
{} & \quad 
\otimes_{\mab Z}
\varpi^{(2j+k'+1),(k-k')}_{\rm et}((\os{\circ}{\ol{X}},\os{\circ}{\ol{D}})/\ol{\kap}))
(-j-k)[-2j-k] \quad (k\in {\mab Z}). 
\end{align*}  
This isomorphism is compatible with $n$'s. 
\end{lemm} 
\begin{proof} 
The proof of (1) and  
(2) is the same as that of (\ref{lemm:grc}). 
\end{proof}

Next we generalize Rapoport-Zink's 
$l$-adic weight spectral sequence
(see (\ref{theo:xpssf}) below). 
For the generalization, we need some lemmas.

\begin{lemm}\label{lemm:jebc} 
Let $\os{\circ}{\cal X}{}'$ be a 
closed subscheme of $\os{\circ}{\cal X}$. 
Assume that $\os{\circ}{\cal X}{}'$ is strictly semistable over ${\cal V}$. 
Set ${\cal X}':= \os{\circ}{\cal X}{}'\times_{\os{\circ}{\cal X}}{\cal X}$. 
Set $\os{\circ}{\ol{X}}{}':=\os{\circ}{\cal X}{}'_{\ol{\cal V}}\otimes_{\ol{\cal V}}\ol{\kap}$.  
Let $R\os{\circ}{\Psi}({\mab Z}/l^n)$ 
and $R\os{\circ}{\Psi}{}'({\mab Z}/l^n)$
be the classical nearby cycle sheaf of ${\mab Z}/l^n$ on 
$\os{\circ}{\cal X}_{\ol{\cal V}}$ 
and $\os{\circ}{\cal X}{}'_{\ol{\cal V}}$, respectively. 
Let 
$\os{\circ}{\ol{\iota}}_{\os{\circ}{\ol{X}},\os{\circ}{\ol{X}}{}'}\col 
\os{\circ}{\ol{X}}{}' \os{\sus}{\lo} \os{\circ}{\ol{X}}$ 
be the closed immersion. 
Then 
\begin{equation*} 
\os{\circ}{\ol{\iota}}{}^{*}_{\os{\circ}{\ol{X}},\os{\circ}{\ol{X}}{}'}
R\os{\circ}{\Psi}({\mab Z}/l^n)= 
R\os{\circ}{\Psi}{}'({\mab Z}/l^n). 
\tag{14.4.1}\label{eqn:nbcyb}
\end{equation*}
\end{lemm}
\begin{proof} 
Let 
$\os{\circ}{\ol{\iota}}_{\os{\circ}{\cal X}_{\ol{\cal V}},
\os{\circ}{\cal X}{}'_{\ol{\cal V}}} 
\col \os{\circ}{\cal X}{}'_{\ol{\cal V}} 
\os{\sus}{\lo} \os{\circ}{\cal X}_{\ol{\cal V}}$ 
be the closed immersion. 
Let ${\cal X}_{\ol{\cal V}}$ and ${\cal X}'_{\ol{\cal V}}$ 
be log schemes whose underlying schemes are 
$\os{\circ}{\cal X}_{\ol{\cal V}}$ and 
$\os{\circ}{\cal X}{}'_{\ol{\cal V}}$  and whose log structures 
are the inductive limits of canonical log structures.  
Let $\eps_{{\cal X}_{\ol{\cal V}}} \col 
{\cal X}_{\ol{\cal V}} \lo \os{\circ}{\cal X}_{\ol{\cal V}}$ 
and 
$\eps_{{\cal X}'_{\ol{\cal V}}} \col 
{\cal X}'_{\ol{\cal V}} \lo \os{\circ}{\cal X}{}'_{\ol{\cal V}}$ 
be the morphisms forgetting the log structures. 
Let 
$\ol{j} \col \os{\circ}{\cal X}_{\ol{\eta}} 
\os{\sus}{\lo} {\cal X}_{\ol{\cal V}}$ 
and 
$\ol{j}{}' \col \os{\circ}{\cal X}{}'_{\ol{\eta}} 
\os{\sus}{\lo} {\cal X}'_{\ol{\cal V}}$ 
be the open immersions of log schemes. 
Then, by \cite[(3.1)]{fk} (see also \cite[(7.4), (7.5)]{illl}), 
$R\os{\circ}{\ol{j}}_*({\mab Z}/l^n)=
R\eps_{{\cal X}_{\ol{\cal V}}*}({\mab Z}/l^n)$. 
and $R\os{\circ}{\ol{j}}{}'_*({\mab Z}/l^n)=
R\eps_{{\cal X}'_{\ol{\cal V}}*}({\mab Z}/l^n)$. 
By the log proper base change theorem 
(\cite[(5.1)]{nale}) for the following commutative diagram 
\begin{equation*} 
\begin{CD} 
{\cal X}'_{\ol{\cal V}} @>{\subset}>> {\cal X}_{\ol{\cal V}} \\
@V{\eps_{{\cal X}'_{\ol{\cal V}}}}VV 
@VV{\eps_{{\cal X}_{\ol{\cal V}}}}V \\ 
\os{\circ}{\cal X}{}'_{\ol{\cal V}} 
@>{\subset}>> \os{\circ}{\cal X}_{\ol{\cal V}}
\end{CD}
\end{equation*} 
we obtain 
\begin{equation*} 
\os{\circ}{\ol{\iota}}{}^{*}_{\os{\circ}{\cal X}_{\ol{\cal V}},
\os{\circ}{\cal X}{}'_{\ol{\cal V}}} 
R\os{\circ}{\ol{j}}_*({\mab Z}/l^n)
= R\os{\circ}{\ol{j}}{}'_*({\mab Z}/l^n).  
\tag{14.4.2}\label{eqn:bcvc}
\end{equation*}
Hence 
$\os{\circ}{\ol{\iota}}{}^{*}_{\os{\circ}{\ol{X}},\os{\circ}{\ol{X}}{}'}
R\os{\circ}{\Psi}({\mab Z}/l^n)= 
R\os{\circ}{\Psi}{}'({\mab Z}/l^n)$.  
\end{proof}

The following is the dual of \cite[(2.7.2)]{nh2}: 

\begin{lemm}\label{lemm:canhigh}
Let 
$f \col ({\cal T}, {\cal A}) 
\lo ({\cal T}', {\cal A}')$ 
be a morphism of ringed topoi. 
Then, for an object $E'{}^{\bul}$ in 
$D^-({\cal A}')$, 
there exists a canonical morphism 
\begin{equation}
Lf^*((E^{\bul},\tau))\lo 
(Lf^*(E^{\bul}),\tau)  
\tag{14.5.1}\label{eqn:taure}
\end{equation} 
in ${\rm D}^-{\rm F}({\cal A}').$
\end{lemm}
\begin{proof}
By the adjunction 
there exits a natural morphism 
\begin{equation}
(E^{\bul},\tau)\lo 
(Rf_*Lf^*(E^{\bul}),\tau).   
\tag{14.5.2}
\label{eqn:tanure}
\end{equation} 
By \cite[(2.7.2)]{nh2} we have the following natural morphism 
\begin{equation}
(Rf_*Lf^*(E^{\bul}),\tau)\lo 
Rf_*((Lf^*(E^{\bul}),\tau)).   
\tag{14.5.3}
\label{eqn:taunsre}
\end{equation} 
Hence we have the following composite morphism 
\begin{equation}
(E^{\bul},\tau)\lo 
Rf_*((Lf^*(E^{\bul}),\tau)).  
\tag{14.5.4}
\label{eqn:tanusre}
\end{equation} 
By applying the adjunction formula \cite[(1.2.2)]{nh2} or (\ref{theo:adj}) 
for (\ref{eqn:tanusre}), 
we obtain the morphism (\ref{eqn:taure}). 
\end{proof}

Consider the following cartesian diagram of 
fs log schemes: 
\begin{equation*} 
\begin{CD} 
Y @>{q}>> Z \\ 
@V{p}VV @VV{\alpha}V \\ 
W @>{\bet}>> V.
\end{CD}
\tag{14.5.5}\label{cd:yzwv}
\end{equation*}
Assume that $R\alpha_*({\mab Z}/l^n)$ 
and $R\bet_*({\mab Z}/l^n)$ are bounded above. 
Then we can construct 
the following ``filtered K\"{u}nneth morphism''  
\begin{equation*} 
\cup \col R\bet_*({\mab Z}/l^n)\otimes^L_{{\mab Z}/l^n}
(R\alpha_*({\mab Z}/l^n),\tau) 
\lo 
R\bet_*((Rp_*({\mab Z}/l^n),\tau)).  
\tag{14.5.6}\label{eqn:fkcst} 
\end{equation*}
Indeed, by the adjunction formula \cite[(1.2.2)]{nh2}, 
we have only to construct a morphism 
\begin{equation*} 
L\bet^*R\bet_*({\mab Z}/l^n)\otimes^L_{{\mab Z}/l^n}
L\bet^*(R\alpha_*({\mab Z}/l^n),\tau) 
\lo (Rp_*({\mab Z}/l^n),\tau). 
\end{equation*} 
The following composite morphism 
\begin{align*} 
&L\bet^*R\bet_*({\mab Z}/l^n)\otimes^L_{{\mab Z}/l^n}
L\bet^*(R\alpha_*({\mab Z}/l^n),\tau) 
\lo {\mab Z}/l^n\otimes^L_{{\mab Z}/l^n}
L\bet^*((R\alpha_*({\mab Z}/l^n),\tau)) \\
& =L\bet^*((R\alpha_*({\mab Z}/l^n),\tau))
\lo (L\bet^*R\alpha_*({\mab Z}/l^n),\tau) 
\lo (Rp_*Lq^*({\mab Z}/l^n),\tau) = (Rp_*({\mab Z}/l^n),\tau)
\end{align*}
is a desired morphism. 
Here we have used the  morphism (\ref{eqn:taure}).  


\par 
Let $({\cal X},\os{\circ}{\cal D})$ be as in 
{\rm (\ref{defi:dxsnc})}. 
Consider the following cartesian diagrams  
\begin{equation*} 
\begin{CD} 
(\os{\circ}{\ol{X}},\os{\circ}{\ol{D}})
@>{\subset}>>
(\os{\circ}{\cal X}_{\ol{\cal V}},\os{\circ}{\cal D}_{\ol{\cal V}}) 
@<{\supset}<< 
({\cal X}_{\ol{\eta}},\os{\circ}{\cal D}_{\ol{\eta}}) \\ 
@V{\os{\circ}{\eps}_{\os{\circ}{\ol{D}}}}VV 
@V{\os{\circ}{\eps}_{\os{\circ}{\cal D}_{\ol{\cal V}}}}VV 
@V{\os{\circ}{\eps}_{\os{\circ}{\cal D}_{\ol{\eta}}}}VV \\ 
\os{\circ}{\ol{X}} @>{\os{\circ}{\ol{\iota}}}>> 
\os{\circ}{\cal X}_{\ol{\cal V}}
@<{\os{\circ}{\ol{j}}}<< {\cal X}_{\ol{\eta}},  
\end{CD}
\tag{14.5.7}\label{eqn:co} 
\end{equation*} 
where the vertical morphisms are morphisms forgetting log structures 
obtained by $\os{\circ}{\ol{D}}$, $\os{\circ}{\cal D}_{\ol{\cal V}}$ 
and $\os{\circ}{\cal D}_{\ol{\eta}}$. 
Then we have the following  K\"unneth morphism 
\begin{equation*} 
\cup \col R\os{\circ}{\ol{j}}_*({\mab Z}/l^n)
\otimes^L_{{\mab Z}/l^n}
R\os{\circ}{\eps}_{\os{\circ}{\cal D}_{\ol{\cal V}}*}({\mab Z}/l^n) 
\lo R(\os{\circ}{\ol{j}}
\os{\circ}{\eps}_{\os{\circ}{\cal D}_{\ol{\eta}}})_*({\mab Z}/l^n). 
\tag{14.5.8}\label{eqn:njdez} 
\end{equation*}
To prove that the morphism  
(\ref{eqn:njdez}) is an isomorphism,   
we cannot use the log K\"{u}nneth formula 
(\cite[(6.1)]{nale})  for the right cartesian diagram in (\ref{eqn:co}) 
since $\os{\circ}{\ol{j}}$ is not proper. 
However we can prove the following:

\begin{lemm}\label{lemm:cbdxve} 
The filtered K\"{u}nneth morphism 
{\rm (\ref{eqn:fkcst})} 
\begin{equation*} 
\cup \col R\os{\circ}{\ol{j}}_*({\mab Z}/l^n)
\otimes^L_{{\mab Z}/l^n}
(R\os{\circ}{\eps}_{\os{\circ}{\cal D}_{\ol{\cal V}}*}({\mab Z}/l^n),\tau) 
\lo R\os{\circ}{\ol{j}}_*
((R\os{\circ}{\eps}_{\os{\circ}{\cal D}_{\ol{\eta}}*}({\mab Z}/l^n),\tau))  
\tag{14.6.1}\label{eqn:jdjzez} 
\end{equation*}
for the right cartesian diagram of {\rm (\ref{eqn:co})} 
is an isomorphism. 
In particular the  K\"unneth morphism {\rm (\ref{eqn:njdez})} 
is an isomorphism. 
\end{lemm} 
\begin{proof} 
By abuse of notation, we denote 
$R\os{\circ}{\ol{j}}_*({\mab Z}/l^n)$ by 
a representative of $R\os{\circ}{\ol{j}}_*({\mab Z}/l^n)$. 
Because $R^k\os{\circ}{\eps}_{\os{\circ}{\cal D}_{\ol{\cal V}}*}({\mab Z}/l^n)$ 
$(k\in {\mab Z})$ is a flat ${\mab Z}/l^n$-module, 
we have only to prove that the following morphism 
\begin{align*}
&{\rm gr}_k^\tau(\cup)[k] \col 
R\os{\circ}{\ol{j}}_*({\mab Z}/l^n)
\otimes_{{\mab Z}/l^n}
R^k\os{\circ}{\eps}_{\os{\circ}{\cal D}_{\ol{\cal V}}*}({\mab Z}/l^n) 
\lo 
{\rm gr}_k^\tau R\os{\circ}{\ol{j}}_*
((R\os{\circ}{\eps}_{\os{\circ}{\cal D}_{\ol{\eta}}*}({\mab Z}/l^n),\tau))[k] 
\tag{14.6.2}\label{ali:grfi}\\
 &=
R\os{\circ}{\ol{j}}_*({\rm gr}_k^\tau 
(R\os{\circ}{\eps}_{\os{\circ}{\cal D}_{\ol{\eta}}*}({\mab Z}/l^n),\tau)[k])
=R\os{\circ}{\ol{j}}_*(R^k\os{\circ}{\eps}_{\os{\circ}{\cal D}_{\ol{\eta}}*}
({\mab Z}/l^n))
\end{align*}
is an isomorphism. 
Let $d^{(k)} \col 
\os{\circ}{\cal D}{}^{(k)}_{\ol{\cal V}} 
\lo \os{\circ}{\cal X}_{\ol{\cal V}}$ 
and 
$e^{(k)}\col \os{\circ}{\cal D}{}^{(k)}_{\ol{\eta}} 
\lo \os{\circ}{\cal X}_{\ol{\eta}}$ $(k\in {\mab N})$
be the natural morphisms. 
Then the source of the morphism (\ref{ali:grfi}) 
is 
$R\os{\circ}{\ol{j}}_*({\mab Z}/l^n)
\otimes_{{\mab Z}}d^{(k)}_*
(\vp^{(k)}_{\rm et}(\os{\circ}{\cal D}_{\ol{\cal V}}/\ol{\cal V}))(-k)$, 
while the target of the morphism (\ref{ali:grfi}) 
is $R\os{\circ}{\ol{j}}_*
({\mab Z}/l^n\otimes_{\mab Z}
e^{(k)}_*(\vp^{(k)}_{\rm et}(\os{\circ}{\cal D}_{\ol{\eta}}/{\ol{K}})))(-k)$.  
Fix a total order on the irreducible components of 
${\cal D}_{\ol{\cal V}}$. 
Then the source and the target are isomorphic to 
$R\os{\circ}{\ol{j}}_*({\mab Z}/l^n)
\otimes_{{\mab Z}/l^n}d^{(k)}_*
(({\mab Z}/l^n)_{\os{\circ}{\cal D}_{\ol{\cal V}}})(-k)$
and 
$R\os{\circ}{\ol{j}}_*
(e^{(k)}_*(({\mab Z}/l^n)_{\os{\circ}{\cal D}_{\ol{\eta}}}))(-k)$, 
respectively. 
Let $\os{\circ}{\ol{j}}{}^{(k)}
\col  \os{\circ}{\cal D}{}^{(k)}_{\ol{\eta}}
\os{\sus}{\lo} \os{\circ}{\cal D}{}^{(k)}_{\ol{\cal V}}$ 
be the natural open immersion. 
Then the source is equal to 
$d^{(k)}_*R\os{\circ}{\ol{j}}{}^{(k)}_*({\mab Z}/l^n)(-k)$ 
by (\ref{eqn:bcvc}). 
Since $d^{(k)}\circ \os{\circ}{\ol{j}}{}^{(k)}= \os{\circ}{\ol{j}}\circ e^{(k)}$, 
the target is $d^{(k)}_*R\os{\circ}{\ol{j}}{}^{(k)}_*({\mab Z}/l^n)(-k)$. 
Hence the source and the target of (\ref{ali:grfi}) 
are the same.  
\end{proof}

\begin{lemm}\label{lemm:qxa} 
\begin{equation*} 
H^q(\os{\circ}{X}_{\rm et},
A^{\bul}_{l^n}({\cal X}_{\ol{\cal V}}/\ol{\cal V})\otimes^L_{{\mab Z}/l^n}
R\os{\circ}{\eps}_{\os{\circ}{\ol{D}}*}({\mab Z}/l^n))
=H^q({\cal U}_{\ol{\eta}},{\mab Z}/l^n) \quad (q,n\in {\mab N}).   
\tag{14.7.1}\label{eqn:isouz}
\end{equation*} 
\end{lemm}
\begin{proof} 
Because (\ref{eqn:osire}) is an isomorphism, 
we have the following equality:  
\begin{align*} 
&H^q(\os{\circ}{X},A^{\bul}_{l^n}({\cal X}_{\ol{\cal V}}/\ol{\cal V})\otimes^L_{{\mab Z}/l^n}
R\os{\circ}{\eps}_{\os{\circ}{\ol{D}}*}({\mab Z}/l^n)) 
 = H^q(\os{\circ}{X},R\os{\circ}{\Psi}({\mab Z}/l^n)\otimes^L_{{\mab Z}/l^n}
\os{\circ}{\ol{\iota}}{}^{*}R{\eps}_{\os{\circ}{\cal D}_{\ol{{\cal V}}}*}
({\mab Z}/l^n)) \tag{14.7.2}\label{ali:xaxv}\\ 
{} & = H^q(\os{\circ}{\cal X}_{\ol{{\cal V}}},
R\os{\circ}{\ol{j}}_*({\mab Z}/l^n)\otimes^L_{{\mab Z}/l^n}
R\eps_{\os{\circ}{\cal D}_{\ol{{\cal V}}}*}({\mab Z}/l^n))=H^q_{\rm ket}
(({\cal X}_{\ol{\eta}},\os{\circ}{\cal D}_{\ol{\eta}}),{\mab Z}/l^n)\\
&=H^q_{\rm et}({\cal U}_{\ol{\eta}},{\mab Z}/l^n). 
\end{align*}
Here the third equality is obtained by   
the isomorphism (\ref{eqn:njdez}) 
and the last equality follows 
from Gabber's purity (\cite[\S8, third Consequence]{fup}) 
(cf.~\cite{fk}, \cite[(7.5)]{illl}). 
\end{proof}

The following is a generalization of 
Rapoport-Zink's $l$-adic weight spectral sequence 
(see also (\ref{coro:sspc}) below): 

\begin{theo}\label{theo:xpssf} 
There exists the following spectral sequences$:$ 
\begin{equation*} 
E_1^{-k,q+k}=H^{q-k}_{\rm et}
({\cal D}^{(k)}_{\ol{\eta}},{\mab Z}_l\otimes_{\mab Z}
\vp^{(k)}_{\rm et}(\os{\circ}{\cal D}_{\ol{\eta}}/\ol{K}))(-k) 
\Lo H^q_{\rm et}({\cal U}_{\ol{\eta}},{\mab Z}_l). 
\tag{14.8.1}\label{eqn:lghkd}  
\end{equation*}  
\begin{align*} 
E_1^{-k,q+k} & =
\bigoplus^k_{k'=-\infty}
\bigoplus_{j\geq \max\{-k',0\}} 
H^{q-2j-k}_{\rm et}(\os{\circ}{\ol{X}}{}^{(2j+k'+1)}\cap 
\os{\circ}{\ol{D}}{}^{(k-k')},{\mab Z}_l
\otimes_{\mab Z} 
\tag{14.8.2}\label{eqn:llgdd} \\ 
{} & \varpi^{(2j+k'+1),(k-k')}_{\rm et}((\os{\circ}{\ol{X}},\os{\circ}{\ol{D}})/\ol{\kap}))(-j-k)\\
{} & \Lo H^q_{\rm et}({\cal U}_{\ol{\eta}},{\mab Z}_l). 
\end{align*} 
\end{theo} 
\begin{proof} 
(\ref{theo:xpssf}) immediately 
follows from (\ref{lemm:ssgrc}), (\ref{eqn:fqsis}) and (\ref{eqn:isouz}).  
\end{proof}

\begin{prop}\label{prop:acomp}
There exists a natural isomorphism 
\begin{equation*} 
(A^{\bul}_{l^n}(({\cal X}_{\ol{\cal V}},\os{\circ}{{\cal D}}_{\ol{\cal V}})/\ol{\cal V}), 
P^{{\cal D}_{\ol{\cal V}}},P))\os{\sim}{\lo} 
(A^{\bul}_{l^n}((\ol{X},\ol{D})/\ol{s}),P^{\ol{D}},P).
\end{equation*}
\end{prop} 
\begin{proof} 
(\ref{prop:acomp}) immediately follows from 
\cite[(1.9)]{nd}, which tells us that 
there exists a natural isomorphism 
\begin{equation*} 
(A^{\bul}_{l^n}({\cal X}_{\ol{\cal V}}/\ol{\cal V}),P^{{\cal X}_{\ol{\cal V}}})\os{\sim}{\lo} 
(A^{\bul}_{l^n}(X_{\ol{\kap}}/\ol{s}),P^{X_{\ol{\kap}}}).
\end{equation*}
\end{proof}

\begin{coro}\label{coro:sspc} 
The spectral sequences 
{\rm (\ref{eqn:lghkd})} and {\rm (\ref{eqn:llgdd})} 
are canonically isomorphic to 
{\rm (\ref{eqn:lehkd})} and {\rm (\ref{eqn:lehkd})}, 
respectively.   
\end{coro} 
\begin{proof} 
By \cite[(5.6)]{ndeg} the following diagram is commutative:  
\begin{equation*}
\begin{CD}
R\eps_{\ol{X}*}({\mab Z}/l^n) @>{\sim}>> 
s((K^{\bul}_{l^n}(\ol{X}/\ol{s}),d) \os{T-1}{\lo}  (K^{\bul}_{l^n}(\ol{X}/\ol{s}),-d)) \\
@A{\simeq}AA @AA{\simeq}A \\ 
\os{\circ}{\ol{i}}{}^*R\os{\circ}{\ol{j}}_*({\mab Z}/l^n) @>{\sim}>> 
s((K^{\bul}_{l^n}({\cal X}_{\ol{\cal V}}/\ol{\cal V}),d) \os{T-1}{\lo}
(K^{\bul}_{l^n}({\cal X}_{\ol{\cal V}}/\ol{\cal V}),-d)). 
\end{CD}
\tag{14.10.1}\label{eqn:ttskrp}
\end{equation*} 
By \cite[(5.7)]{ndeg} the induced morphism 
$\os{\circ}{i}{}^*R^r\os{\circ}{\ol{j}}_*({\mab Z}/l^n)
\lo R^r\eps_{\ol{X}*}({\mab Z}/l^n)$ by the 
left vertical isomorphism in 
{\rm (\ref{eqn:ttskrp})} fits 
into the following commutative diagram: 
\begin{equation*}
\begin{CD}
\bigwedge^r(M_{\ol{X}}^{\rm gp}/{\cal O}_{\ol{X}}^*)
\otimes_{\mab Z}{\mab Z}/l^n(-r) 
@>{\us{\sim}{(\ref{eqn:tkx})}}>> 
R^r\eps_{\ol{X}*}({\mab Z}/l^n) \\
@| @AA{\simeq}A \\ 
({\mab Z}/l^n)_{\os{\circ}{\ol{X}}{}^{(r)}}(-r) 
@>{\us{\sim}{(\ref{eqn:gap})}}>> 
\os{\circ}{\ol{i}}{}^*R^r\os{\circ}{\ol{j}}_*({\mab Z}/l^n). 
\end{CD}
\tag{14.10.2}\label{eqn:gdff}
\end{equation*} 
(\ref{coro:sspc}) immediately follows from 
(\ref{prop:acomp}), 
\end{proof}

\begin{theo}\label{theo:e2dgh} 
Let $S$ be a family of log points. 
Assume that the log structure of $S$ is constant and that $\os{\circ}{S}$ is connected. 
Let $X/S$ be a proper SNCL scheme and let $D$ be a relative SNCD on $X/S$. 
For a point $\os{\circ}{s} \in \os{\circ}{S}$, denote by 
$s$ the log scheme whose underlying scheme is $\os{\circ}{s}$ 
and whose log structure is the pull-back of 
the log structure of $S$. 
Let $X(s)$ and $\os{\circ}{D}(\os{\circ}{s})$ be the fibers of 
$X$ and $\os{\circ}{D}$ at $s$ and $\os{\circ}{s}$, respectively. 
Assume that, for a point $\os{\circ}{s}\in \os{\circ}{S}$, 
there exist a henselian discrete valuation ring 
${\cal V}(\os{\circ}{s})$ with residue field 
$\kap(\os{\circ}{s})$ and a proper strict semistable family 
${\cal X}(s)$ with canonical log structure and with 
a horizontal relative SNCD $\os{\circ}{\cal D}(\os{\circ}{s})$ over 
${\cal V}(\os{\circ}{s})$ such that 
$({\cal X}(s),\os{\circ}{\cal D}(\os{\circ}{s})) 
\times_{({\rm Spec}({\cal V}(\os{\circ}{s})),
{\cal V}(\os{\circ}{s})\setminus \{0\})}
({\rm Spec}(\kap(\os{\circ}{s})),{\mab N}\oplus \kap(\os{\circ}{s})^*) 
=(X(s),\os{\circ}{D}(\os{\circ}{s}))$. 
Then the spectral sequence {\rm (\ref{eqn:lehkd})} 
degenerates at $E_2$ modulo torsion. 
\end{theo}
\begin{proof} 
Because the $E_r$-terms $(r\geq 1)$ are smooth 
on $\os{\circ}{S}$ (the proof of (\ref{prop:sms})) and 
because $\os{\circ}{S}$ is connected, 
it suffices to prove the $E_2$-degeneration for $(X(s), D(s))/s$. 
Hence we may assume that $S=s$ and 
we omit to write 
$s$ and $\os{\circ}{s}$ in the notations $X(s)$, 
$\os{\circ}{D}(\os{\circ}{s})$, ${\cal V}(\os{\circ}{s})$, 
$\kap(\os{\circ}{s})$, ${\cal X}(s)$ and 
$\os{\circ}{\cal D}(\os{\circ}{s})$.  
Thus we may assume that we 
are given a proper strict semistable family ${\cal X}$ 
with canonical log structure and with 
a horizontal SNCD $\os{\circ}{\cal D}$ 
over a henselian discrete valuation ring ${\rm Spec}({\cal V})$ with 
separably closed residue field $\kap$ such that 
the log special fiber obtained by $({\cal X},{\cal D})$ is $(X,D)$.  
Let $\eta$ be the generic point of ${\rm Spec}\,{\cal V}$ 
and $\ol{\eta}$ the geometric generic point of ${\rm Spec}({\cal V})$. 
Let $G$ be the absolute Galois group of $K$. 
Let $R\os{\circ}{\Psi}({\mab Z}/l^n)$ be 
the classical nearby cycle 
sheaf in $D^+(\os{\circ}{\ol{X}},G,{\mab Z}/l^n)$. 
Then 
$K^{\bul}_{l^n}(\ol{X}_{\frac{1}{l^{\infty}}}/\ol{s})=R\os{\circ}{\Psi}({\mab Z}/l^n)$ 
by the latter part of the proof of \cite[(1.9)]{nd}.  
We have the following commutative diagram 
\begin{equation*} 
\begin{CD} 
A^{\bul}_{l^n}({\cal X}_{\ol{\cal V}}/\ol{\cal V}) @>{\sim}>> 
A^{\bul}_{l^n}(\ol{X}_{\frac{1}{l^{\infty}}}/\ol{s}_{\frac{1}{l^{\infty}}}) \\ 
@A{\simeq}AA @AA{\simeq}A \\ 
R\os{\circ}{\Psi}({\mab Z}/l^n) @>{\sim}>>K^{\bul}_{l^n}(\ol{X}_{\frac{1}{l^{\infty}}}/\ol{s}_{\frac{1}{l^{\infty}}}) 
\end{CD} 
\tag{14.11.1}\label{cd:comvc} 
\end{equation*} 
(see [loc.~cit.], though the boundary morphisms of our complexes 
$A^{\bul}_{l^n}({\cal X}_{\ol{\cal V}}/\ol{\cal V})$ and 
$A^{\bul}_{l^n}(\ol{X}_{\frac{1}{l^{\infty}}}/\ol{s}_{\frac{1}{l^{\infty}}})$ 
are different from those of the corresponding complexes in [loc.~cit.]). 
Hence we have a canonical isomorphism 
\begin{equation*} 
A^{\bul}_{l^n}({\cal X}_{\ol{\cal V}}/\ol{\cal V})\otimes_{{\mab Z}/l^n}
\os{\circ}{{\ol{\iota}}}{}^{*}(M^{\bul}_{l^n}(\os{\circ}{\cal D}/\os{\circ}{S}),R)  
\os{\sim}{\lo} 
(A^{\bul}_{l^n}((\ol{X}_{\frac{1}{l^{\infty}}},\ol{D}_{\frac{1}{l^{\infty}}})/\ol{s}_{\frac{1}{l^{\infty}}}),
P^{\ol{D}_{\frac{1}{l^{\infty}}}})  
\end{equation*} 
in $D^+(\os{\circ}{X}_{\rm et},{\mab Z}/l^n)$. 
Let  $R\os{\circ}{\Psi}_{\os{\circ}{\cal D}_{\ol{\cal V}}}({\mab Z}/l^n)$ 
be the classical nearby cycle sheaf of ${\mab Z}/l^n$ on 
$\os{\circ}{\cal D}_{\ol{\cal V}}$. 
By (\ref{eqn:nbcyb}) 
we obtain 
\begin{align*} 
s(A^{\bul}_{l^n}({\cal X}_{\ol{{\cal V}}}/\ol{\cal V})
\otimes_{{\mab Z}/l^n}{\rm gr}^R_k
M^{\bul}_{l^n}(\os{\circ}{\cal D}_{\ol{{\cal V}}}/\ol{{\cal V}})) 
=c^{(k)}_{*}(R\os{\circ}{\Psi}_{\os{\circ}{{\cal D}_{\ol{\cal V}}}}({\mab Z}/l^n)(-k)
\otimes_{\mab Z}
\varpi^{(k)}_{\rm et}(\os{\circ}{D}_{\ol{{\cal V}}}/\ol{{\cal V}}))[-k] 
\tag{14.11.2}\label{eqn:nbacyb}\\ 
\end{align*} 
as in the proof of (\ref{lemm:grc}) (1). 
Hence the $E_1$-term $E_1^{-k,q+k}$ of 
(\ref{eqn:lehkd}) is isomorphic to  
$H^{q-k}_{\rm et}({\cal D}^{(k)}_{\ol{\eta}},{\mab Z}_l)(-k)$. 
\par 
First consider the case where $\eta$ is of characteristic $p>0$. 
As in the proof of 
(\ref{prop:stc}), the specialization argument and  
the purity of the Frobenius 
show the degeneration at $E_2$ modulo torsion. 
\par 
Next consider the case where $\eta$ is of characteristic $0$. 
In this case, the Lefschetz principle, the comparison theorem between 
the \'{e}tale cohomology and the Betti cohomology and the Hodge theory 
tell us the degeneration at $E_2$. 
\end{proof} 




\begin{theo}\label{theo:sprme} 
Let the notations be as in {\rm (\ref{theo:xpssf})}.  
Assume that ${\cal V}$ is of characteristic $p>0$. 
Assume that $\os{\circ}{\cal X}$ is proper over ${\cal V}$.
Then there exists a monodromy filtration 
$M$ on $H^q_{\rm ket}((\ol{X}_{\frac{1}{l^{\infty}}},\ol{D}_{\frac{1}{l^{\infty}}}),{\mab Q}_l)
(=H^q({\cal U}_{\ol{\eta}},{\mab Q}_l))$ $(q\in {\mab N})$ 
relative to $P^{\ol{D}_{\frac{1}{l^{\infty}}}}$. The relative monodromy filtration 
$M$ is equal to $P$. 
\end{theo} 
\begin{proof} 
By the proof of (\ref{theo:mimr}) it suffices to prove that 
(\ref{conj:log}) for $D^{(k)}$ $(k\in {\mab N})$ is true ((\ref{eqn:mwd})). 
This follows from \cite[(6.1)]{itp}.
\end{proof}


\begin{coro}\label{coro:nci}
Let the notations be as in {\rm \S\ref{sec:fpl}}. 
Assume that $\os{\circ}{X}$ is proper over $\os{\circ}{S}$. 
If, for each connected component $S'$ of $S$, there exists an exact closed point $s\in S'$ 
such that the fiber $(X_s,D_s)/s$ of $(X,D)/S$ at $s$ is the log special fiber 
of a proper strict semistable family over a henselian discrete valuation ring 
of equal characteristic, then there exists a relative monodromy filtration on 
$R^qf_{(X_{\frac{1}{l^{\infty}}},D_{\frac{1}{l^{\infty}}})/S_{\frac{1}{l^{\infty}}}*}({\mab Q}_l)$ with respect to $P^D$ and 
it is equal to $P$. 
\end{coro}
\begin{proof}
(\ref{coro:nci}) follows from the proof of (\ref{prop:sms}) and (\ref{theo:sprme}). 
\end{proof}

\begin{coro}\label{theo:genfs}  
Let the notations be as in {\rm (\ref{theo:func})}.  
Assume that $S$ and $S'$ are log points.  
Assume that 
$g^* \col H^q_{\rm ket}((\ol{X}_{\frac{1}{l^{\infty}}},D_{\frac{1}{l^{\infty}}}),{\mab Q}_l) 
\lo H^q_{\rm ket}((\ol{Y}_{\frac{1}{l^{\infty}}},E_{\frac{1}{l^{\infty}}}),{\mab Q}_l)$ 
has a section of ${\mab Q}_l$-vector spaces 
which are strictly compatible with respect to 
$P^{\ol{D}_{\frac{1}{l^{\infty}}}}$ and $P^{\ol{E}_{\frac{1}{l^{\infty}}}}$, and two $P$'s on 
$H^q_{\rm ket}((X_{\frac{1}{l^{\infty}}},D),{\mab Q}_l)$ 
and $H^q_{\rm ket}((Y_{\frac{1}{l^{\infty}}},E),{\mab Q}_l)$.  
Assume also that this section is compatible with 
$N$'s on $H^q_{\rm ket}((X_{\frac{1}{l^{\infty}}},D_{\frac{1}{l^{\infty}}}),{\mab Q}_l) 
\lo H^q_{\rm ket}((Y_{\frac{1}{l^{\infty}}},E_{\frac{1}{l^{\infty}}}),{\mab Q}_l)$.  
Furthermore, assume that $(Y,E)/T$ 
is a special fiber of proper strict semistable family with a horizontal 
SNCD over a henselian discrete valuation ring of 
characteristic $p>0$.  
Then there exists a monodromy filtration 
$M$ on $H^q_{\rm ket}((X_{\frac{1}{l^{\infty}}},D),{\mab Q}_l)$ 
relative to $P^{\os{\circ}{D}}$. 
The relative monodromy filtration $M$ is equal to $P$. 
\end{coro} 
\begin{proof} 
(\ref{theo:genfs}) immediately follows from 
(\ref{theo:sprme}) and (\ref{theo:func}). 
\end{proof}

\begin{rema}\label{rema:gfs}    
(1) Let the notations be as in the introduction. 
In some cases in mixed characteristics,  
in \cite{fuhs} K.~Fujiwara has raised a strategy 
for the conjecture (\ref{eqn:nkkph}): 
if $X$ is the log special fiber of 
a log proper strict semistable family ${\cal X}$ 
over $S$ in the case where $S$ is 
of mixed characteristics and, also, 
the special fiber of a log proper strict semistable family ${\cal X}'$ 
over the spectrum of a henselian discrete valuation ring of 
equal characteristic with canonical log structure, 
then the monodromy-weight conjecture 
for $H^q(\os{\circ}{\cal X}_{\ol{\eta}},{\mab Q}_l)$ follows from (\ref{eqn:fki}). 
\par 
If $(X,D)$ comes from a 
proper semistable family $({\cal X},{\cal D})$ with a horizontal SNCD over 
a henselian discrete valuation ring of 
mixed characteristics and, also, a proper semistable family 
with a horizontal SNCD over a henselian discrete valuation ring of 
equal characteristic, then 
the relative monodromy filtration on 
$H^q((\os{\circ}{\cal X}_{\ol{\eta}}\setminus 
\os{\circ}{\cal D}_{\ol{\eta}}),{\mab Q}_l)$ 
relative to the filtration $P^{\os{\circ}{\cal D}}$ exists by (\ref{theo:sprme}).  
\par 
(2) In the workshop ``Algebraic geometry 2000'' at Azumino, 
the author heard from T.~Katsura that 
F.~Oort has raised the following problem: 
for a proper smooth scheme $Z$ over a perfect field, 
is there an alternation $Z' \lo Z$ such 
that $Z'$ is the special fiber of a proper smooth scheme 
over a complete discrete valuation ring of mixed characteristics?
One year later(=2001),  N.~Tsuzuki asked to me the following: 
is there a surjective morphism $Z' \lo Z$ such 
that $Z'$ is the special fiber of a proper smooth scheme 
over a complete discrete valuation ring of mixed characteristics? 
\par 
One can ask the analogous problem  
in the context of log geometry: 
for a proper log smooth scheme $Z$ over 
the log point of a perfect field, 
is there a surjective morphism $Z' \lo Z$ such 
that $Z'$ is the special fiber of a proper log smooth scheme 
over the spectrum of 
a complete discrete valuation ring of mixed characteristics 
with canonical log structure.  
Influenced by Fujiwara's strategy in (1), 
C.~Nakayama kindly suggested to me in 2007  that,  
replacing ``mixed characteristics'' by ``equal characteristics'' 
has an application to the log $l$-adic monodromy-weight conjecture 
(\ref{conj:log}).  
However it seems to me that the analogous problem in the log case cannot 
be solved affirmatively in general 
because there exist log special fibers of rigid analytic analogues of non-K\"{a}hler 
elliptic surfaces for which the monodromy filtrations on $H^1$'s and 
the weight filtrations on them do not coincide (\cite[(6.5), (6.7)]{ndeg}). 
\end{rema} 
 
\bigskip
\bigskip
\parno
\begin{center}
{{\rm \Large{\bf Appendix}}}
\end{center}

\section{Edge morphisms between 
the $E_1$-terms of $l$-adic weight spectral sequences}
\label{sec:lbddmlif} 
In this section we prove that 
the edge morphisms $d_1^{\bul \bul}$ 
between the $E_1$-terms of 
(\ref{eqn:lledd}) are described by 
\v{C}ech Gysin morphisms and 
the induced \v{C}ech morphisms of pull-back morphisms 
by closed immersions. 
We also prove that 
the edge morphisms $d_1^{\bul \bul}$ 
between the $E_1$-terms of 
(\ref{eqn:lehkd}) are described by Gysin morphisms.  
\par 
For the time being, consider the case 
$\os{\circ}{D}=\emptyset$. 
\par 
Let $\os{\circ}{X}:=\bigcup_{i \in I(\os{\circ}{X})}\os{\circ}{X}_i$ 
be the union of smooth components of $\os{\circ}{X}$ over $\os{\circ}{S}$, 
where $I(\os{\circ}{X})$ is a set of indexes.  
We fix a total order on $I(\os{\circ}{X})$.  
Let $k\geq 2$ be an integer. 
Set 
$I_k(\os{\circ}{X})
:=\{(i_0, \ldots, i_{k-1})\in I(\os{\circ}{X})^k~\vert~
i_m< i_{m'}~(0\leq m<m' \leq k-1)\}$ 
and $\ul{i}:=(i_0, \ldots, i_{k-1})\in I_k(\os{\circ}{X})$. 
For an integer $0 \leq m \leq k-1$, 
set $\ul{i}_m:=(i_0, \ldots, \hat{i}_m, \ldots, i_{k-1})$. 
Set 
$\os{\circ}{X}_{\ul{i}}:=\os{\circ}{X}_{i_0} \cap \cdots \cap \os{\circ}{X}_{i_{k-1}}$ 
and $\os{\circ}{X}_{\ul{i}_m}:=\os{\circ}{X}_{i_0} \cap \cdots \cap 
\hat{\os{\circ}{X}}_{i_m} \cap \cdots \cap \os{\circ}{X}_{i_{k-1}}$.  
Let $\iota_{\ul{i}}^{\ul{i}_m} \col \os{\circ}{X}_{\ul{i}} 
\os{\sus}{\lo} \os{\circ}{X}_{\ul{i}_m}$ 
be the natural closed immersion and let 
$G_{\ul{i}}^{\ul{i}_m} \col 
\iota_{\ul{i}*}^{\ul{i}_m}({\mab Z}/l^n(-1))_{\os{\circ}{X}_{\ul{i}}}\{-1\} \lo 
({\mab Z}/l^n)_{\os{\circ}{X}_{\ul{i}_m}}[1]$ 
be the Gysin morphism  of the closed immersion 
$\iota_{\ul{i}}^{\ul{i}_m}$. 
\par 
Set 
$$\vp_{\ul{i}{\rm et}}(\os{\circ}{X}/\os{\circ}{S}):=
\vp_{i_0 \cdots i_{k-1}{\rm et}}(\os{\circ}{X}/\os{\circ}{S})$$ 
and
$$\vp_{\ul{i}_m{\rm et}}(\os{\circ}{X}/\os{\circ}{S}):=
\vp_{i_0 \cdots \wh{i}_m 
\cdots i_{k-1}{\rm et}}(\os{\circ}{X}/\os{\circ}{S}).$$ 
\par 
The exact closed immersion 
$\iota^{\ul{i}_m}_{\ul{i}} 
\col X_{\ul{i}} \os{\subset}{\lo} X_{\ul{i}_m}$
induces a morphism  
\begin{equation*}
(-1)^m\iota_{\ul{i}}^{\ul{i}_m*}
\col 
({\mab Z}/l^n)_{\os{\circ}{X}_{\ul{i}_m}} 
\otimes_{\mab Z}
\vp_{\ul{i}_m{\rm et}}(\os{\circ}{X}/\os{\circ}{S}) 
\lo \iota_{\ul{i}*}^{\ul{i}_m}
(({\mab Z}/l^n)_{\os{\circ}{X}_{\ul{i}}})\otimes_{\mab Z}
\vp_{\ul{i}{\rm et}}(\os{\circ}{X}/\os{\circ}{S})
\tag{15.0.1}\label{eqn:defcbd}
\end{equation*} 
defined by
$x\otimes (i_0 \cdots \wh{i}_m \cdots i_{k-1}) 
\lom (-1)^m\iota_{\ul{i}{\rm et}}^{\ul{i}_m*}
(x)\otimes (i_0 \cdots  i_{k-1})$. 
Let $a_{\ul{i}_m} \col \os{\circ}{X}_{\ul{i}_m} \os{\sus}{\lo} \os{\circ}{X}$ 
be the natural closed immersion. 
Set 
\begin{equation*}
\iota^{(k-1)*}:=
\sum_{\{(i_0, i_1, \ldots, i_{k-1})~\vert~i_{m} 
\not= i_{m'}~(0\leq m \not=m' \leq k-1)\}}
\sum_{m=0}^{k-1}
a_{\ul{i}_m*} \circ ((-1)^m\iota_{\ul{i}}^{\ul{i}_m*}) \col
\tag{15.0.2}\label{eqn:descbd}
\end{equation*}
$$a_{*}^{(k-1)}
(({\mab Z}/l^n)_{\os{\circ}{X}{}^{(k-1)}}
\otimes_{\mab Z}\vp^{(k-1)}_{\rm et}(\os{\circ}{X}/\os{\circ}{S}))
\lo a_{*}^{(k)}
(({\mab Z}/l^n)_{\os{\circ}{X}{}^{(k)}}\otimes_{\mab Z}
\vp^{(k)}_{\rm et}(\os{\circ}{X}/\os{\circ}{S})).$$
Set 
$L^{\bul}_{l^n}(X):= {\rm MF}_{l^n}(T-1)$ 
((\ref{ali:mftkdf})). 

\par
In \cite{ndeg} we have proved the following two lemmas: 

\begin{lemm}[{\bf {\cite[(5.1)]{ndeg}}}]\label{lemm:lbdtrho} 
Let the notations be as above. 
Then the following diagram
\begin{equation*}
\begin{CD}
{\cal H}^k(L^{\bul}_{l^n}(X))(k)   
@>{\theta}>> {\cal H}^{k+1}(L^{\bul}_{l^n}(X))(k+1) \\ 
@V{{\rm (\ref{eqn:tkx})}}V{\simeq}V  
@V{{\rm (\ref{eqn:tkx})}}V{\simeq}V \\
a^{(k)}_*(({\mab Z}/l^n)_{\os{\circ}{X}{}^{(k)}}
\otimes_{\mab Z}\vp^{(k)}_{\rm et}(\os{\circ}{X}/\os{\circ}{S})) 
@>{\iota^{(k)*}}>> 
a^{(k+1)}_{*}(({\mab Z}/l^n)_{\os{\circ}{X}{}^{(k+1)}}
\otimes_{\mab Z}\vp^{(k+1)}_{\rm et}(\os{\circ}{X}/\os{\circ}{S})) 
\end{CD}
\tag{15.1.1}\label{cd:ttarho}
\end{equation*}  
is commutative. 
\end{lemm}

\par 
We fix an isomorphism 
\begin{equation*}
\vp_{i_m{\rm et}}(\os{\circ}{X}/\os{\circ}{S})
\otimes_{\mab Z} 
\vp_{\ul{i}_m{\rm et}}(\os{\circ}{X}/\os{\circ}{S}) 
\os{\sim}{\lo}
\vp_{\ul{i}{\rm et}}(\os{\circ}{X}/\os{\circ}{S})
\tag{15.1.2}\label{eqn:colxts}
\end{equation*}
by the following morphism
$$(i_m)\otimes (i_0\cdots  
\wh{i}_m \cdots i_{k-1})
\lom (-1)^m(i_0\cdots i_{k-1}).$$
In (\ref{eqn:colxts}) we have omitted to 
write the direct images of the closed immersions 
$X_{\ul{i}} \os{\sus}{\lo} X_{\ul{i}_m}$ 
and $\iota^{\ul{i}_m}_{\ul{i}}$.  
We identify 
$\vp_{i_m{\rm et}}(\os{\circ}{X}/\os{\circ}{S})\otimes_{\mab Z} 
\vp_{\ul{i}_m{\rm et}}(\os{\circ}{X}/\os{\circ}{S})$ with 
$\vp_{\ul{i}{\rm et}}(\os{\circ}{X}/\os{\circ}{S})$ 
by this isomorphism.
We also have the following composite morphism
\begin{align*}
(-1)^mG_{\ul{i}}^{\ul{i}_m}  &
\col 
\iota^{\ul{i}_m}_{\ul{i}*}
(({\mab Z}/l^n)_{\os{\circ}{X}_{\ul{i}}})
\otimes_{\mab Z}
\vp_{\ul{i}{\rm et}}(\os{\circ}{X}/\os{\circ}{S}) 
\tag{15.1.3}\label{eqn:m1gom} \\ 
& \os{\sim}{\lo}  \iota^{\ul{i}_m}_{\ul{i}*}
(({\mab Z}/l^n)_{\os{\circ}{X}_{\ul{i}}})
\otimes_{\mab Z}\vp_{i_m{\rm et}}(\os{\circ}{X}/\os{\circ}{S})
\otimes_{\mab Z} 
\vp_{\ul{i}_m{\rm et}}(\os{\circ}{X}/\os{\circ}{S}) \\ 
{} & =\iota^{\ul{i}_m}_{\ul{i}*}
(({\mab Z}/l^n)_{\os{\circ}{X}_{\ul{i}}}) 
\otimes_{\mab Z} 
\vp_{\ul{i}_m{\rm et}}(\os{\circ}{X}/\os{\circ}{S}) \\ 
{} & \os{G_{\ul{i}}^{\ul{i}_m}\otimes 1}{\lo}  
({\mab Z}/l^n)_{\os{\circ}{X}_{\ul{i}_m}}
\otimes_{\mab Z}\vp_{\ul{i}_m{\rm et}}
(\os{\circ}{X}/\os{\circ}{S})[1]\{1\} 
\end{align*} 
defined by 
\begin{equation*}
{\rm ``} x\otimes (i_0 \cdots  i_{k-1}) \lom 
(-1)^mG_{\ul{i}}^{\ul{i}_m}(x)\otimes 
(i_0\cdots \wh{i}_m \cdots i_{k-1}){\rm "}.
\tag{15.1.4}\label{eqn:ooolll}
\end{equation*}
Set 
\begin{equation*} 
G^{(k)}:=\sum_{\ul{i} \in I_k(\os{\circ}{X})}
\sum_{m=0}^{r}(-1)^m
G_{\ul{i}}^{\ul{i}_m} \col 
a^{(k)}_*(({\mab Z}/l^n)_{\os{\circ}{X}{}^{(k)}}(-1)
\otimes_{\mab Z}\vp^{(k)}_{\rm et}(\os{\circ}{X}/\os{\circ}{S})) 
\end{equation*} 
$$\lo  a^{(k-1)}_*(({\mab Z}/l^n)_{\os{\circ}{X}{}^{(k-1)}}
\otimes_{\mab Z}\vp^{(k-1)}_{\rm et}(\os{\circ}{X}/\os{\circ}{S})) 
[1]\{1\}. $$

\begin{lemm}[{\bf {\cite[(5.4)]{ndeg}}}]\label{lemm:gsta} 
Let the notations and the assumptions be as in 
{\rm (\ref{lemm:lbdtrho})}. 
Let 
\begin{equation*}
d\col {\cal H}^{k}(L^{\bul}_{l^n}(X))\{-k\}  \lo 
{\cal H}^{k-1}(L^{\bul}_{l^n}(X))\{-(k-1)\}[1]  
\end{equation*}
be the boundary morphism 
of the following triangle 
\begin{equation*}
{\rm gr}^{\tau}_{k-1}L^{\bul}_{l^n}(X) 
\lo (\tau_{k}/\tau_{k-2})(L^{\bul}_{l^n}(X)) 
\lo {\rm gr}^{\tau}_{k}L^{\bul}_{l^n}(X) 
\os{+1}{\lo} 
\end{equation*}
by using the Convention {\rm (4)}.  
Then  the following diagram
\begin{equation*}
\begin{CD}
{\cal H}^{k}(L^{\bul}_{l^n}(X))\{-k\}  
@>{d}>>  \\
@V{{\rm (\ref{eqn:tkx})}}V{\simeq}V \\
a^{(k)}_*(({\mab Z}/l^n)_{\os{\circ}{X}{}^{(k)}}(-k)
\otimes_{\mab Z}\vp^{(k)}_{\rm et}(\os{\circ}{X}/\os{\circ}{S})) 
\{-k\} 
@>{-G^{(k)}}>>  \\
{\cal H}^{k-1}(L^{\bul}_{l^n}(X))\{-(k-1)\}[1] \\ 
@V{{\rm (\ref{eqn:tkx})}}V{\simeq}V \\ 
a^{(k-1)}_*
(({\mab Z}/l^n)_{\os{\circ}{X}{}^{(k-1)}}(-(k-1))
\otimes_{\mab Z}\vp^{(k-1)}_{\rm et}(\os{\circ}{X}/\os{\circ}{S}))
\{-(k-1)\}[1]
\end{CD}
\tag{15.2.1}\label{cd:dgsn}
\end{equation*}  
is commutative. 
\end{lemm}

By (\ref{lemm:lbdtrho}) and (\ref{lemm:gsta}) we obtain the following: 

\begin{coro}[{\bf cf.~{\cite[(5.5)]{ndeg}}}]\label{coro:bdel}
The edge morphism 
$d^{-k, q+k}_1 \col E_{1,l}^{-k, q+k} 
\lo E_{1,l}^{-k+1, q+k}$ 
of the spectral sequence of {\rm (\ref{eqn:lledd})} 
for the case $\os{\circ}{D}=\emptyset$ is 
identified with the following morphism$:$
\begin{equation*}
\sum_{j\geq {\rm max}\{-k, 0\}}
(G^{(k)}+\iota^{(k)*}). 
\tag{15.3.1}\label{eqn:lbd}
\end{equation*}
\end{coro}

\par 
Next consider the general case where the horizontal SNCD 
$\os{\circ}{D}$ is not necessarily empty.  
Let $\os{\circ}{D}=\bigcup_{i\in I(\os{\circ}{D})}\os{\circ}{D}_i$ 
be a union of smooth divisors, 
where $I(\os{\circ}{D})$ is a set of indexes.   
Fix a total order on $I(\os{\circ}{D})$. 
\par 
Let $k$ be a positive integer. 
Set 
$I_k(\os{\circ}{D})
:=\{(i_0, \ldots, i_{k-1})\in 
I(\os{\circ}{D})^k~\vert~i_m< i_{m'}~(0\leq m<m' \leq k-1)\}$ 
and $\ul{i}:=(i_0, \ldots, i_{k-1})$. 
For an integer $0 \leq m \leq k-1$, 
set $\ul{i}_m:=(i_0, \ldots, \hat{i}_m, \ldots, i_{k-1})$. 
Set 
$\os{\circ}{D}_{\ul{i}}:=\os{\circ}{D}_{i_0} \cap \cdots \cap \os{\circ}{D}_{i_{k-1}}$ 
and $\os{\circ}{D}_{\ul{i}_m}:=\os{\circ}{D}_{i_0} \cap \cdots \cap 
\hat{\os{\circ}{D}}_{i_m} \cap \cdots \cap \os{\circ}{D}_{i_{k-1}}$.

For a nonnegative integer $e$, 
let $\iota_{\ul{i}}^{\ul{i}_m}\vert_{\os{\circ}{X}{}^{(e)}} 
\col \os{\circ}{X}{}^{(e)}  \cap \os{\circ}{D}_{\ul{i}}
\os{\sus}{\lo}  \os{\circ}{X}{}^{(e)}\cap \os{\circ}{D}_{\ul{i}_m}$ 
be the restriction of $\iota_{\ul{i}}^{\ul{i}_m}$ to 
$\os{\circ}{X}_{\ul{i}}\cap \os{\circ}{D}{}^{(e)}$  and let 
$G_{\ul{i}}^{\ul{i}_m} \col 
({\mab Z}/l^n(-1))_{ \os{\circ}{X}{}^{(e)} \cap \os{\circ}{D}_{\ul{i}}}\{-1\} \lo 
({\mab Z}/l^n)_{\os{\circ}{X}{}^{(e)}\cap \os{\circ}{D}_{\ul{i}_m}}[1]$ 
be the Gysin morphism  of the closed immersion 
$\iota_{\ul{i}}^{\ul{i}_m}\vert_{\os{\circ}{X}_{\ul{i}}\cap \os{\circ}{D}{}^{(e)}}$. 
For a nonnegative integer $e$, 
let $\iota_{\ul{i}}^{\ul{i}_m}(\os{\circ}{D})\vert_{\os{\circ}{X}{}^{(e)}}  
\col \os{\circ}{X}{}^{(e)}\cap D_{\ul{i}}
\os{\sus}{\lo} \os{\circ}{X}{}^{(e)}\cap D_{\ul{i}_m}$
be the natural closed immersion and let 
\begin{equation*}
G_{\ul{i}}^{\ul{i}_m}(\os{\circ}{D}) \col 
({\mab Z}/l^n(-1))_{\os{\circ}{X}{}^{(e)}\cap {D_{\ul{i}}}}\{-1\} \lo 
({\mab Z}/l^n)_{\os{\circ}{X}{}^{(e)}\cap D_{\ul{i}_m}}[1]
\tag{15.3.2}\label{eqn:gdfs}
\end{equation*} 
be the Gysin morphism of 
$\iota_{\ul{i}}^{\ul{i}_m}(\os{\circ}{D})\vert_{\os{\circ}{X}{}^{(e)}}$. 
\par 
Let $k$ and $q$ be integers.   
Consider the following three morphisms for 
$k' \leq k$ and $j\geq \max\{-k',0\}$ appearing 
in the edge morphism 
$E^{-k,q+k}_1 \lo E^{-(k-1),q+k}_1$ of 
the spectral sequence  {\rm (\ref{eqn:lledd})}:  
\begin{align*} 
G^{(k),(k')} & :=\sum_{\ul{i} \in I_{2j+k'}(\os{\circ}{X})}
\sum_{m=0}^{2j+k'}(-1)^mG_{\ul{i}}^{\ul{i}_m} 
\col R^{q-2j-k}f_{\os{\circ}{X}{}^{(2j+k')}\cap 
\os{\circ}{D}{}^{(k-k')}/\os{\circ}{S} {\rm et}*}({\mab Z}_l
\otimes_{\mab Z}  \\ 
{} & (\varpi^{(2j+k')}_{\rm et}(\os{\circ}{X}/\os{\circ}{S}))
\vert_{\os{\circ}{X}{}^{(2j+k')}\cap 
\os{\circ}{D}{}^{(k-k')}}\otimes_{\mab Z}
\varpi^{(k-k')}_{\rm et}(\os{\circ}{D}/\os{\circ}{S})
\vert_{\os{\circ}{X}{}^{(2j+k')}\cap 
\os{\circ}{D}{}^{(k-k')}})(-j-k) 
\lo \\ 
& R^{q+1-2j-(k-1)}f_{\os{\circ}{X}{}^{(2j+k'-1)}\cap 
\os{\circ}{D}{}^{(k-k')}/S {\rm et}*}({\mab Z}_l
\otimes_{\mab Z}(\varpi^{(2j+k-1')}_{\rm et}(\os{\circ}{X}/\os{\circ}{S}))
\vert_{\os{\circ}{X}{}^{(2j+k-1')}\cap 
\os{\circ}{D}{}^{(k-k')}}\otimes_{\mab Z}  \\ 
{} & \varpi^{(k-k')}_{\rm et}(\os{\circ}{D}/\os{\circ}{S})
\vert_{\os{\circ}{X}{}^{(2j+k'-1)}\cap 
\os{\circ}{D}{}^{(k-k')}})(-j-(k-1)), 
\end{align*} 
\begin{align*} 
\iota^{(k),(k')*}
& :=\sum_{\ul{i} \in I_{2j+k'}(\os{\circ}{X})}
\sum_{m=0}^{2j+k'}(-1)^m\iota_{\ul{i}}^{\ul{i}_m*}  
\col R^{q-2j-k}f_{\os{\circ}{X}{}^{(2j+k')}\cap 
\os{\circ}{D}{}^{(k-k')}/S {\rm et}*}({\mab Z}_l
\otimes_{\mab Z}  \\ 
{} & (\varpi^{(2j+k')}_{\rm et}(\os{\circ}{X}/\os{\circ}{S}))
\vert_{\os{\circ}{X}{}^{(2j+k')}\cap 
\os{\circ}{D}{}^{(k-k')}}\otimes_{\mab Z}
\varpi^{(k-k')}_{\rm et}(\os{\circ}{D}/\os{\circ}{S})
\vert_{\os{\circ}{X}{}^{(2j+k')}\cap 
\os{\circ}{D}{}^{(k-k')}})(-j-k) 
\lo \\ 
{} & R^{q+1-2(j+1)-(k-1)}f_{\os{\circ}{X}{}^{(2j+k'+1)}\cap 
\os{\circ}{D}{}^{(k-k')}/S {\rm et}*}({\mab Z}_l
\otimes_{\mab Z}  \\ 
{} & (\varpi^{(2j+k'+1)}_{\rm et}(\os{\circ}{X}/\os{\circ}{S}))
\vert_{\os{\circ}{X}{}^{(2j+k'+1)}\cap 
\os{\circ}{D}{}^{(k-k')}}\otimes_{\mab Z}
\varpi^{(k-k')}_{\rm et}(\os{\circ}{D}/\os{\circ}{S})
\vert_{\os{\circ}{X}{}^{(2j+k'+1)}\cap 
\os{\circ}{D}{}^{(k-k')}})(-j-k) 
\end{align*} 
and 
\begin{align*} 
&G^{(k),(k')}(\os{\circ}{D})
 :=\sum_{\ul{i} \in I_{k-k'}(\os{\circ}{D})}
\sum_{m=0}^{k-k'}(-1)^mG_{\ul{i}}^{\ul{i}_m}(\os{\circ}{D})   
\col R^{q-2j-k}f_{\os{\circ}{X}{}^{(2j+k')}\cap 
\os{\circ}{D}{}^{(k-k')}/S {\rm et}*}({\mab Z}_l
\otimes_{\mab Z}  \\ 
{} & (\varpi^{(2j+k')}_{\rm et}(\os{\circ}{X}/\os{\circ}{S}))
\vert_{\os{\circ}{X}{}^{(2j+k')}\cap 
\os{\circ}{D}{}^{(k-k')}}\otimes_{\mab Z}
\varpi^{(k-k')}_{\rm et}(\os{\circ}{D}/\os{\circ}{S})
\vert_{\os{\circ}{X}{}^{(2j+k')}\cap 
\os{\circ}{D}{}^{(k-k')}})(-j-k) \lo \\ 
& 
R^{q+1-2j-(k-1)}f_{\os{\circ}{X}{}^{(2j+k')}\cap 
\os{\circ}{D}{}^{((k-1)-k')}/S {\rm et}*}({\mab Z}_l
\otimes_{\mab Z}  \\ 
{} & (\varpi^{(2j+k')}_{\rm et}(\os{\circ}{X}/\os{\circ}{S}))
\vert_{\os{\circ}{X}{}^{(2j+k')}\cap 
\os{\circ}{D}{}^{(k-k'-1)}}\otimes_{\mab Z}
\varpi^{((k-1)-k')}_{\rm et}(\os{\circ}{D}/\os{\circ}{S})
\vert_{\os{\circ}{X}{}^{(2j+k')}\cap 
\os{\circ}{D}{}^{(k-k'-1)}})(-j-(k-1)).  
\end{align*}

By noting the sign arising from the Convention (4) and the definitions 
(\ref{eqn:tdef}) and (\ref{eqn:bapt}), we obtain the following: 

\begin{coro}\label{coro:lbdl}
The edge morphism 
$d^{-k, q+k}_1 \col E_{1,l}^{-k, q+k} 
\lo E_{1,l}^{-k+1, q+k}$ 
of the spectral sequence  {\rm (\ref{eqn:lledd})} is 
identified with the following morphism$:$
\begin{equation*}
\sum_{k'\leq k}\sum_{j\geq {\rm max}\{-k', 0\}}
\{G^{(k),(k')}+\iota^{(k),(k')*}+
(-1)^{2j+k'+1}G^{(k),(k')}(\os{\circ}{D})\}. 
\tag{15.4.1}\label{eqn:glbd}
\end{equation*}
\end{coro}
\begin{proof} 
Because $(A^{\bul}_{l^n}((X_{\frac{1}{l^{\infty}}},D_{\frac{1}{l^{\infty}}})/S_{\frac{1}{l^{\infty}}}),P)$ 
is the single complex of the triple complex 
$$(A^{\bul \bul \bul}_{l^n}((X_{\frac{1}{l^{\infty}}},D_{\frac{1}{l^{\infty}}})/S),P)
=(A^{\bul \bul}_{l^n}(X_{\frac{1}{l^{\infty}}}/S_{\frac{1}{l^{\infty}}}),P)
\otimes_{{\mab Z}/l^n}(M^{\bul}_{l^n}(\os{\circ}{D}/\os{\circ}{S}),Q),$$ 
we obtain (\ref{coro:lbdl}) by (\ref{coro:bdel}) 
as in \cite[(10.1)]{ndw}. 
\end{proof} 

\begin{prop}\label{prop:gy}
Let 
\begin{equation*}
G^{(k)}(D)_{\ul{i}}^{\ul{i}_m} \col  
({\mab Z}/l^n(-1))_{D_{\ul{i}}}\{-1\} \lo 
({\mab Z}/l^n)_{D_{\ul{i}_m}}[1]
\end{equation*} 
be the Gysin morphism 
 of the closed immersion 
$\iota_{\ul{i}}^{\ul{i}_m}(D)\col D_{\ul{i}}\os{\sus}{\lo} D_{\ul{i}_m}$ 
obtained by $\iota_{\ul{i}}^{\ul{i}_m}(\os{\circ}{D})$. 
Then the edge morphism $d_1^{-k,q+k} \col E_1^{-k,q+k}\lo  E_1^{-k+1,q+k}$ of 
{\rm (\ref{eqn:lehkd})} is expressed as 
$\sum_{\ul{i} \in I_{k}(\os{\circ}{D})}\sum(-1)^mG^{(k)}(D)_{\ul{i}}^{\ul{i}_m}$. 
\end{prop} 
\begin{proof} 
This is obvious. 
\end{proof}

\bigskip
\parno
Yukiyoshi Nakkajima 
\parno
Department of Mathematics,
Tokyo Denki University,
5 Asahi-cho Senju Adachi-ku,
Tokyo 120-8551, Japan.

\end{document}